\documentclass[a4paper,11pt]{article}
\usepackage[utf8]{inputenc}
\usepackage{latexsym}
\usepackage{amsfonts}
\usepackage{amssymb}

\usepackage{xcolor}

\usepackage{enumitem} 
\usepackage{braket}
\usepackage{hyperref}

\usepackage{amsmath}
\usepackage{amssymb}
\usepackage{amsfonts}
\usepackage{amsthm}
\usepackage{indentfirst}
\usepackage{psfrag}
\usepackage{amscd,stmaryrd,latexsym}
\usepackage[margin=1.3in]{geometry}
\usepackage[pdftex]{graphicx}
\newtheorem{thm}{Theorem}[section]
\newtheorem{lem}{Lemma}[section]

\newtheorem{Prop}{Proposition}[section]
\newtheorem*{St*}{Statement}

\newtheorem{Dfi}{Definition}

\newtheorem*{theorem*}{Theorem}
\theoremstyle{remark}

\theoremstyle{definition}

\theoremstyle{remark}
\newtheorem{oss}{Remark}[section]

\newcommand{\be}{\begin{equation}}
\newcommand{\ee}{\end{equation}}
\newcommand{\R}{\mathbb{R}}
\newcommand{\N}{\mathbb{N}}

\newcommand{\He}{\mathbb{H}_{\varepsilon}}
\newcommand{\G}{\Gamma}

\newcommand{\ca}[2]{\mathcal{#1}_{#2}}
\newcommand{\Greg}{\text{gen-reg}\,}

\newcommand{\spt}[1]{\text{spt}\,\|#1\|}

\newcommand\res{\mathop{\hbox{\vrule height 7pt width .5pt depth 0pt
\vrule height .5pt width 6pt depth 0pt}}\nolimits}

\def\eps{\mathop{\varepsilon}}
\def\e{\mathop{\varepsilon}}

\def\Oc{\mathop{\mathcal{O}}}

\def\Fce{\mathop{\mathcal{F}_{\varepsilon}}}
\newcommand{\Het}{\mathbb{H}}
\newcommand{\OHet}{\overline{\mathbb{H}}}

\def\Fce{\mathop{\mathcal{F}_{\varepsilon}}}
\def\Ece{\mathop{\mathcal{E}_{\varepsilon}}}

\def\Hc{\mathop{\mathcal{H}}}

\def\t{\tau}

\def\e{\varepsilon}

\def\t{\tau}

\def\s{\sigma}

\def\om{\omega}

\def\p{\partial}

\def\dm{d_{\overline{M}}}

\newcommand{\Rc}[1]{\text{Ric}_{#1}} 

\def\eps{\mathop{\varepsilon}}

\def\om{\omega}
\def\p{\partial}

\DeclareMathAlphabet{\mathscr}{OT1}{pzc}{m}{it}

\begin{document} 
\title{\textbf{The inhomogeneous Allen--Cahn equation and the existence of prescribed-mean-curvature hypersurfaces}}
\author{Costante Bellettini\thanks{Partially supported by the EPSRC under the grant EP/S005641/1.} \, and Neshan Wickramasekera}
\date{}

\maketitle

\abstract{\begin{footnotesize} We prove that for any given compact Riemannian manifold $N$ of dimension $n~+~1~\geq~3$ and any non-negative Lipschitz function $g$ on $N$, there exists a quasi-embedded, boundaryless 
hypersurface  $M \subset N,$ of class $C^{2, \alpha}$ for any $\alpha \in (0,1),$ such that $M$ is the image of a two-sided immersion whose mean curvature is given by $g\nu$
for an appropriate choice of continuous unit normal $\nu$ to the immersion; and moreover, the singular set $\Sigma = \overline{M}   \setminus M$ is empty if $2 \leq n \leq 6,$ finite if $n=7$ and satisfies ${\mathcal H}^{n-7 + \gamma}(\Sigma) = 0$ for every $\gamma >0$ if $n \geq 8$.  Here quasi-embedded means that 
near every non-embedded point, $M$ is the union of two embedded $C^{2, \alpha}$ disks intersecting 
tangentially with each disk lying on one side of the other. 
If $g >0$ then $\overline{M}$ is the closure of the reduced boundary of a Caccioppoli set. 


Our proof of this theorem is PDE theoretic, and considers first the case that $g$ is a positive function of class $C^{1, 1}$. The theorem for non-negative Lipschitz $g$ follows by approximation, based on the estimates we establish. In the case of positive $g$ of class $C^{1, 1}$, the argument is based on: (i) a construction, using a simple mountain pass lemma, of 
a min-max solution $u_{\e}$ to the inhomogeneous Allen--Cahn equation $-\e \Delta u + \e^{-1}W^{\prime}(u) = \sigma g$ satisfying appropriate bounds independent of $\e$, where $\e>0$ is small, $\sigma$ is a fixed normalising constant and $W$ is a fixed double-well potential; and (ii) a proof of regularity  of any limit-varifold $V$ that arises, in the limit $\e_{j} \to 0^{+},$  from a sequence $(u_{\e_{j}})$ of such solutions---in fact from any sequence of bounded solutions $u_{\e_{j}}$ with uniform Morse-index and energy upper bounds. Parts of $V$ may be minimal (i.e.\ have zero mean curvature), but regularity of $V$ ensures that minimal portions, if there are any, can be smoothly excised. The remaining part of $V$ is a 
mean-curvature $g$ hypersurface $M$ as desired, unless  $V$ is entirely supported on a minimal hypersurface, a possibility we do not rule out. 
If this possibility arises and $V$ is a minimal hypersurface $M_{0}$ (with multiplicity, necessarily, even integer valued and locally constant), we appeal to PDE techniques again. Working at the level of the Allen--Cahn approximation, we use a semi-linear gradient flow (the Allen--Cahn flow) with carefully chosen initial data constructed from $M_{0}$  to produce, for each $j$,  a stable solution $v_{\e_{j}}$ to the 
above Allen--Cahn equation (with $\e = \e_{j}$) in such a way that the sequence $(v_{\e_{j}})$ leads  to a non-trivial limit-varifold that is not entirely supported on a minimal hypersurface. This new limit-varifold is regular by (ii), and after excision of its minimal portions we again obtain a mean-curvature $g$ hypersurface $M$ as desired.\end{footnotesize}} 

\tableofcontents

\section{Introduction}
Our main purpose here is to establish the following theorem:

\begin{thm}
\label{thm:existence}
Let $N$ be a compact Riemannian manifold of dimension $n+1$ with $n\geq 2,$ and let $g:N\to [0,\infty)$ be a (non-negative) Lipschitz function. There exists a quasi-embedded hypersurface $M\subset N$ of class $C^{2, \alpha}$ for any $\alpha \in (0,1),$ with ${\rm dim}_{\mathcal H} \, \left(\overline{M}\setminus M\right)\leq n-7$ if $n \geq 8$, $\overline{M} \setminus M$ finite  if $n=7$ and $\overline{M} \setminus M = \emptyset$ if $2 \leq n \leq 6$, such that $M$ is the image of a two-sided immersion with mean curvature $H_{M}$ given by $H_{M} = g \nu$ for some choice of continuous unit normal $\nu$ to the immersion.
\end{thm}
\begin{oss}\label{quasi-embedded}
Here a quasi-embedded hypersurface is the image of a special type of immersion that is natural in the context of non-minimal prescribed-mean-curvature hypersurfaces. Specifically, in the case of $M$ as in Theorem~\ref{thm:existence}, quasi-embedded means that $M = \iota (S)$ for some $n$-dimensional $C^{2, \alpha}$ manifold $S$ and a $C^{2, \alpha}$ immersion 
$\iota \, : \, S \to N$, and for every $y\in M$ there exists a neighbourhood of $y$ (in $N$) in which $M$ is either an embedded $C^{2, \alpha}$ disk, or the union of two embedded $C^{2, \alpha}$ disks intersecting tangentially along a set containing $y$ with each disk lying on one side of the other and each disk equal to the image under $\iota$ of a disk in $S$ (see Definition~\ref{quasi-embedded-hyp}). The precise meaning of the assertion in the theorem concerning the mean curvature is that $\iota$ can be chosen to be a two-sided immersion with a choice of continuous (in fact $C^{1, \alpha}$) unit normal $\nu$ such that the mean curvature of $\iota$ is $g(\iota(x))\nu(x)$ for each $x \in S.$ 
\end{oss}

\begin{oss}
Theorem~\ref{thm:existence} remains valid if the prescribing function for mean curvature is a non-positive Lipschitz function rather than a non-negative one. For if $\widetilde{g}\,: \, N\to (-\infty,0]$ is Lipschitz, we may apply Theorem 1.1 with $g = -\widetilde{g}$. Reversing the choice of unit normal on the resulting quasi-embedded hypersurface  
provides the resolution of the existence question for $\tilde{g}$. 
\end{oss}

\begin{oss}
Once $C^2$ regularity is obtained for a two-sided immersion with mean curvature $g \nu$, higher regularity follows by standard elliptic theory, depending on the regularity assumed on $g$. In particular, if $g\in C^{k,\alpha}$ with $k\geq 1$ and $\alpha \in (0, 1)$, then the immersion is of class $C^{k+2, \alpha}$.
\end{oss}

\begin{oss}\label{positive-g}
We shall first prove the theorem assuming $g \in C^{1, 1}(N)$ and $g >0$. The theorem for non-negative Lipschitz $g$ will follow (in Section~\ref{extension}) from a straightforward approximation argument using the estimates we establish. 
In the case $g >0,$ we have the additional fact that $|M| = |\partial^{\star} E|$ for a Caccioppoli set $E \subset N$ (where $|T|$ denotes the multiplicity 1 varifold associated with an $n$-rectifiable set $T$) and that at every non-embedded point $y \in M$, the two disks (as in Remark~\ref{quasi-embedded}) whose union is $M$ in a neighborhood of $y$ intersect (only tangentially) along a closed set containing $y$ and contained in an $(n-1)$-dimensional embedded $C^{1}$ submanifold (see \cite[Remark~2.6]{BW1} and \cite[Remark~3.3]{BW2}). 
\end{oss}

\emph{In view of Remark~\ref{positive-g}, from now on until the end of Section~\ref{proof} we shall assume that $g \in C^{1, 1}(N)$ and $g >0$ unless stated otherwise explicitly.} 


\subsection{The role of PDE in the proof of the main theorem}
\medskip
Our approach to Theorem~\ref{thm:existence}  is by means of an inhomogeneous Allen– Cahn approximation scheme. 
The case $g\equiv 0$ of the theorem corresponds to the celebrated, long known existence theory for minimal hypersurfaces in compact Riemannian manifolds. In this case, by employing a homogeneous Allen--Cahn approximation process, a strikingly simple new proof of existence of minimal hypersurfaces has been obtained in the recent work of Guaraco (\cite{Gua}).   
The original proof of existence of minimal hypersurfaces was based on 
a Geometric Measure Theory construction, due to Almgren (\cite{Alm65}) and Pitts (\cite{Pit}), that produces a min-max critical point of the area functional. (See \cite{ColDeL} and \cite{DeLTas} for a nicer implementation of this construction with some technical simplifications). For regularity conclusions this method relied on several fundamental GMT results of independent interest:  the Federer--Fleming existence theory for solutions to the oriented Plateau problem (\cite{FF}), the interior regularity theory for codimension 1 locally area 
minimizing rectifiable currents (due to the combined work of De~Giorgi (\cite{DG}), Federer (\cite{Fed}), Fleming (\cite{Fle}) and Simons (\cite{SJ})) and the Schoen--Simon compactness theory for area-bounded stable minimal hypersurfaces with small singular sets (\cite{SS}).  In the new approach Guaraco employs an elementary PDE mountain-pass lemma, as a replacement for the Almgren--Pitts varifold construction, to produce min-max solutions to the Allen--Cahn equation satisfying appropriate bounds independent of the parameter $\e$ in the equation. Producing the desired minimal hypersurface from these solutions requires proving stationarity and regularity of the limit varifold along which the Allen--Cahn energy of the solutions concentrates as $\e \to 0^{+}$  sequentially. Both of these follow fairly directly from prior work: stationarity (together with integrality) from the work of Hutchinson--Tonegawa (\cite{HT}), and regularity from an advance in the regularity theory for stable stationary varifolds (\cite{Wic}) and its application to the limit varifolds arising from stable Allen--Cahn solutions (\cite{TonWic}).

For the case $g >0$, our starting point is still a classical PDE min-max construction (given in Section~\ref{minmax_setup}) of solutions of the relevant inhomogeneous Allen--Cahn equation, together with the refinements of \cite{HT} provided by the work of R\"oger--Tonegawa (\cite{RogTon}) and Tonegawa (\cite{Ton1}). In subsequent steps of the proof however, in contrast to the case $g \equiv 0$, key difficulties arise from two new phenomena. 
One  is the possibility that the corresponding limit varifold may now have regular pieces intersecting tangentially on a set as large as having codimension 1, such as in the case of two 
touching unit cylinders in a Euclidean space (leading to the 
quasi-embeddedness conclusion in Theorem~\ref{thm:existence} as opposed to embeddededness). Therefore, a more comprehensive theory than \cite{Wic}, with hypotheses that allow this possibility, is needed to prove regularity of the limit varifold. 
The second new aspect is that the inhomogeneous Allen--Cahn approximation process may lead to a loss of mean curvature information in the limit. The proof of the theorem has to account for the difficulties arising from this latter phenomenon, including the possibility that the min-max process may fail to produce the desired prescribed-mean-curvature hypersurface and may instead just produce a minimal hypersurface. 

Despite these two issues, the spirit of simplicity  afforded by the PDE techniques,  already seen in the minimal hypersurface case, continues in our proof of Theorem~\ref{thm:existence}, with PDE arguments featuring heavily in overcoming both difficulties. 

First we obtain optimal regularity of the limit varifolds arising from min-max (or more generally, Morse index bounded) Allen--Cahn solutions. This regularity result says that up to a lower dimensional ``genuine'' singular set, the limit varifold is the image of a two-sided $C^2$ immersion without transverse self-intersections, and it allows: (i) for sheets  to touch without coinciding, and (ii) for the mean curvature of each connected component of the immersion either to be prescribed by $g$ or to vanish identically. This is proved (in Section~\ref{stable-regularity}) by an inductive argument (inducting on the multiplicity of the limit varifold) that makes key use of certain quasilinear elliptic PDE arguments that adapt and extend ideas from [BelWic-1],  in conjunction with the $C^{1, \alpha}$ varifold theory of [BelWic-2] (reproduced in Section~\ref{preliminaries}). The proof is also aided by a simple slicing argument to rule out, subject to a Morse index bound,  singularity formation on a codimension 1 set, and Almgren's generalised stratification of singular sets (\cite{Alm})---an elementary, standard Geometric Measure Theory fact used to further reduce the dimension of the singular set.

The second difficulty, arising from the fact that the mean curvature may conceivably vanish everywhere on the limit varifold, is handled (in Section~\ref{proof} below) by employing PDE principles again, this time from semilinear parabolic theory. More precisely, if the initial min-max Allen--Cahn solutions end up producing just a (non-empty) minimal hypersurface $M_0$, we are able to use that knowledge in conjunction with standard 
long-time existence and convergence results for the parabolic Allen–Cahn equation to obtain the desired mean-curvature-$g$ hypersurface from a different sequence of Allen--Cahn solutions. The initial conditions for the Allen--Cahn flow are constructed out of $M_0$, capitalising on the lower dimensionality of its singular set. 

In summary, once a general GMT regularity framework ([BelWic-1], [BelWic-2]) is in place, our proof of Theorem~\ref{thm:existence} proceeds essentially entirely via PDE arguments. In the remainder of this introduction and in Section 2, we shall provide a more detailed outline of the proof including the intermediate results and key technical aspects of the proof. For the readers’ convenience we have also collected, in Section~\ref{preliminaries}, the main GMT ingredients needed, including those from [BelWic-2], [RogTon08], [Ton05-2]; the rest of the article consists mainly of PDE theoretic arguments which we have endeavoured to make self-contained.

\subsection{An outline of the proof}
In more concrete terms, our proof of Theorem~\ref{thm:existence} can be described as follows. Let $N$ be a compact Riemannian manifold (without boundary) as in the theorem. We consider, for small positive $\e$,  the functional (energy) ${\mathcal F}_{\e, \, \sigma g}:W^{1,2}(N) \to \R$ defined by 
$${\mathcal F}_{\e, \, \sigma g}(u) = \int_N \eps \frac{|\nabla u|^2}{2} + \int_N \frac{W(u)}{\eps} -\int_N \sigma g\, u,$$
where $g \, : \, N\to \R$ is a positive $C^{1, 1}$ function and $W:\R\to [0,\infty)$ is a ``double-well'' potential, i.e.\ a non-negative function of class $C^2$ having precisely two non-degenerate minima at $\pm 1$ with values $W(\pm 1) = 0$ and  appropriate growth outside $[-2,2];$ we shall in fact require that $c \leq W^{\prime\prime}(t) \leq C$ for some constant $C, c >0$ and all $t \in {\mathbb R} \setminus [-2, 2].$  Given such $W$, $\sigma$ is a fixed positive normalising constant depending only on $W$. The first two terms of ${\mathcal F}_{\e, \sigma g}$ add up to give the usual Allen--Cahn energy 
$$\Ece(u) = \int_N \eps \frac{|\nabla u|^2}{2} + \int_N \frac{W(u)}{\eps}.$$ 

 The (relatively straighforward) starting point is to choose a sequence $\e_{j} \to 0^{+}$ and construct, in $W^{1,2}(N)$, a sequence of 1-parameter min-max critical points  $u_{\e_{j}}$ of ${\mathcal F}_{\e_{j}, \sigma g}$  with $\sup_{N}|u_{\e_{j}}|$ uniformly bounded independently of $j$, and the Allen--Cahn energies ${\mathcal E}_{\e_{j}}(u_{\e_{j}})$ uniformly bounded from above and below by positive numbers independently of $j$.  This is done by applying a mountain pass lemma based on the fact that ${\mathcal F}_{\e_{j}, \sigma g}$ satisfies a Palais--Smale condition. The Morse index of $u_{\e_{j}}$  will then automatically be $\leq 1$. By general principles, the energy bounds imply the existence of a non-zero Radon measure $\mu$ on $N$ and a function $u_{\infty} \in BV(N)$ with $u_{\infty}(x) = \pm 1$ for a.e.\ $x \in N$ such that after passing to a subsequence without relabeling, $$(2\sigma)^{-1}\left(\frac{\e_{j}}{2} |\nabla u_{\e_{j}}|^{2} + \e_{j}^{-1}W(u_{\e_{j}})\right) \, d{\rm vol}_{N} \to \mu$$ weakly and  $u_{\e_{j}} \to u_{\infty}$ in $L^{1}.$  Set $E = \{u_{\infty} = 1\}$ and note that $E$ is a Caccioppoli set in $N$ with its reduced boundary $\partial^{\star} E \subset {\rm spt} \, \mu$. 

By the combined results and ideas in the work of Ilmanen (\cite{Ilm}), Sch\"atzle (\cite{Sch2}), Hutchinson--Tonegawa (\cite{HT}), R\"oger--Tonegawa (\cite{RogTon}) and Tonegawa (\cite{Ton1}) (see Theorem~\ref{Roger-Tonegawa} below),  for any sequence $(u_{\eps_{j}})$ of critical points of ${\mathcal F}_{\eps_{j}, \sigma g}$ subject only to uniform upper bounds on 
$\sup_{N} \, |u_{\eps_{j}}|$ and ${\mathcal E}_{\eps_{j}}(u_{\eps_{j}}),$ we have the following results: with no sign condition (and less regularity than $C^{1, 1}$) on $g$, the limit measure $\mu$ is the weight measure $\|V\|$ of a (possibly zero) integral $n$-varifold $V$ on $N$ with first variation $\delta \, V  = -H_{V} \|V\|$ where  $H_{V}$, the generalized mean curvature vector of $V$, is locally bounded; if $ g> 0$ then $E \neq N$ unless $V = 0;$ and moreover:
\begin{itemize}
\item [(a)] if $V$ is of multiplicity 1, then $E$ is non-empty, $V = |\partial^{\star}E|$ (the multiplicity 1 varifold associated with the reduced boundary $\partial^{\star} E$), and $H_{V} = g\nu$ for ${\mathcal H}^{n}$~a.e.\ point of $\partial^{\star} E$ where $\nu$ is the unit normal to $\partial^{\star} E$ pointing into $E$; in other words, in this case, $V$ is a critical point of the functional $A - {\rm Vol}_{g}$ where 
$A$ is the area functional $A(V) = \|V\|(N)$ ($={\mathcal H}^{n}(\partial^{\star} E)$) and ${\rm Vol}_{g}$ is the enclosed $g$ volume ${\rm Vol}_{g}(V) =  \int_{E} g d{\mathcal H}^{n+1}$.  
\item [(b)] if $g \equiv 0$ then regardless of whether $V$ is of multiplicity 1, $H_{V} = 0$, i.e.\ $V$ is a critical point of $A$. 
\end{itemize}

Putting aside for the moment the obvious regularity questions, in the \emph{absence} of the multiplicity 1 assumption on $V$ as in (a), a combination of the conclusions in (a) and (b) may conceivably occur (even if $g >0$); that is to say, $V$ may consist of a higher multiplicity part $V_{0}$ on which $H_{V} = 0$ a.e.\ and the multiplicity 1 part 
$V_{g} = |\partial^{\star} \, E|$ on which $H_{V} = g\nu$ a.e. The two parts $V_{0}$ and $V_{g}$ may merge together in a potentially very complicated manner on not too small a set (see Figure~\ref{fig:necks} in Section~\ref{overview}). This would prevent their separation from each other as 
critical points, in $N$, of $A$ and $A - {\rm Vol}_{g}$ respectively.

In light of this, to prove Theorem~\ref{thm:existence}, one approach would be to first rule out higher multiplicity in $V$ when $g >0$ and $u_{\e_{j}}$ are min-max critical points. If this succeeds we would be in case (a). As regards regularity in this case, 
the central difficulty would be that eventhough $V$ is of multiplicity 1, a priori $V$ may still have a large (but ${\mathcal H}^{n}$-null) set of singularities with higher multiplicity planar tangent cones. However this possibility can be ruled out by a direct application of  \cite[Theorem~5.1]{BW2}, \cite[Theorem~9.1]{BW2} and \cite[Theorem~6.4]{BW2}. From there, it is in fact a few easy steps to the full regularity conclusion claimed in Theorem~\ref{thm:existence}. 

In the generality in which we work here though, showing that no part of $V$ can develop higher multiplicity appears to be difficult. 
In the absence of Morse index control in the limiting process higher multiplicity can in fact occur (see the example in \cite[Section~6.3]{HT}; see also \cite{Sch1} for a related example, which according to \cite{Sch1} has originally appeared in \cite{GB}, where a sequence of CMC hypersurfaces of Euclidean space with fixed mean curvature converges as varifolds to a multiplicity 2 plane). In any case we here follow a strategy that avoids altogether the need to rule out higher multiplicity. 

At the core of this strategy is establishing regularity of $V$ \emph{allowing multiplicity}. For this we employ key results (Theorem~\ref{BWregularity} and Theorem~\ref{nonvar-SS} below) from the regularity theory in our recent work (\cite{BW1}, \cite{BW2}),  as well as extensions of ideas therein, which are in particular sufficiently general to account for the non-variational nature of $V$, i.e.\ the presence of both a region with $H_{V} = 0$ and a region with $H_{V} = g\nu$. This first step is perhaps of interest independently of Theorem~\ref{thm:existence}, and produces the following general result concerning Allen--Cahn limit varifolds $V$ arsing from Morse index bounded critical points, which says that $V_{0}$ and $V_{g}$ are indeed critical points, in $N,$ of 
$A$ and $A - {\rm Vol}_{g}$ respectively; moreover, they are embedded and quasi-embedded respectively up to a codimension $7$ singular set. (See Figure~\ref{fig:PMC_structures} for a depiction of the possible local structure of $V$ at a point away from the singular set.) Completeness of $N$ is not necessary for this result. 

\begin{thm}[Theorem~\ref{thm:regularity} below]\label{thm:AC-regularity}
Let $g \in C^{1, 1}(N)$ with $g >0$ and $N$ not assumed complete. Let $\mu$ be a non-zero Radon measure on $N$ such that 
$$\mu_{j} = (2\sigma)^{-1}\left(\frac{\e_{j}}{2} |\nabla u_{\e_{j}}|^{2} + \e_{j}^{-1}W(u_{\e_{j}})\right) \, d{\rm vol}_{N} \to \mu$$ 
where $\sigma = \int_{-1}^{1}\sqrt{\frac{W(s)}{2}} \, ds$, $\e_{j} \to 0^{+}$ and $u_{\e_{j}}$ is a critical point of ${\mathcal F}_{\e_{j}, \sigma g}$ for each $j$, with 
$\limsup_{j \to \infty} \, |u_{\e_{j}}| < \infty.$ Suppose further that for any open subset 
$U \subset N$ with compact closure, the Morse index of $u_{\e_{j}}$  on $U$ is bounded independently of $j.$ Then 
$$\mu = {\mathcal H}^{n} \res M_{g}  + q{\mathcal H}^{n} \res M_{0}$$
where:
\begin{itemize}
\item[{\rm (a)}] $M_{g}$ is a (possibly empty) immersed,  quasi-embedded hypersurface bounding a Caccioppoli set $E \subset N$  such that its mean curvature $H_{M_{g}} = g\nu$ where $\nu$ is the unit normal pointing into $E$, and such that its singular set $\overline{M_{g}} \setminus M_{g}$ is empty if $2 \leq n \leq 6$, discrete if $n=7$ and has Hausdorff dimension $\leq n-7$ if $n \geq 7;$ in fact $E = \{u_{\infty} = 1\}$ where $u_{\infty}$ is the $L^{1}_{\rm loc}$ limit of $u_{\e_{j}}$, with $u_{\infty}(x) = \pm1$ for a.e.\ $x \in N$ and $E \neq N$ (since $\mu \neq 0$). Moreover, the multiplicity 1 varifold $V = |M_{g}|$ has first variation $\delta \, V = -H_{M_{g}}\|V\|$ in $N$.    
\item[{\rm(b)}]  $M_{0}$ is a (possibly empty) embedded minimal hypersurface of $N$ with its singular set $\overline{M_{0}} \setminus M_{0}$ empty if $2 \leq n \leq 6$, discrete if $n=7$ and having Hausdorff dimension $\leq n-7$ if $n \geq 7;$ the multiplicity function $q$ is equal to a constant positive even integer on each connected component of $M_{0},$ and the multiplicity 1 varifold $|M_{0}|$ is stationary in $N;$ 
\end{itemize}
Additionally, we also have that $\overline{M_{0}} \cap \overline{M_{g}} = \left(M_{0} \cap M_{g}\right) \cup \left((\overline{M_{0}} \setminus M_{0}) \cap (\overline{M_{g}} \setminus M_{g})\right)$, i.e.\ $\overline{M_{0}}$ and that 
$\overline{M_{g}}$ intersect only at common regular (quasi-embedded) points  or common singular points; moreover, $M_{0}$ and $M_{g}$ can only have tangential intersection, with $M_{0} \cap M_{g}$ locally contained in the union of at most two embedded $(n-1)$-dimensional $C^{1}$ submanifolds. 
\end{thm}

The crucial advantage Theorem~\ref{thm:AC-regularity} affords is the following:  it says that in order to prove Theorem~\ref{thm:existence}, one only has to address the possibility that the limit varifold arising from $(u_{\e_{j}})$ is \emph{entirely} a minimal hypersurface with  a small singular set and multiplicity $\geq 2$, i.e.\ the possibility that 
$\mu = q{\mathcal H}^{n} \res M_{0},$ where $q \geq 2$ and $M_{0}$ is a minimal hypersurface embedded away from a closed set of codimension $\geq 7$.
For if this is not the case, then we will have produced the desired hypersurface $M$ in Theorem~\ref{thm:existence}, by taking $M = M_{g}$ with $M_{g}$ as in 
Theorem~\ref{thm:AC-regularity}. 

In the final step of our proof of Theorem 1.1, we show not that the possibility $\mu=q \mathcal{H}^n \res M_0$ cannot arise, but that \textit{if it does} then we can produce the desired hypersurface 
$M$ in Theorem \ref{thm:existence} by a different method: rather than insisting on obtaining it from saddle-type critical points of ${\mathcal F}_{\eps_{j}, \sigma g}$, we will obtain it from stable critical points $v_{\e_{j}}$ (as in Proposition~\ref{Prop:main-intro} below). The construction of these stable critical points uses a negative gradient flow of ${\mathcal F}_{\e_{j}, \sigma g}$ with a well-chosen initial condition built from $M_{0}$. For this construction embeddedness of $M_{0}$ away from a small set is important, and it is carried out in such way that there is a fixed non-empty open set $\Omega \subset \{v_{\e_{j}} >3/4\}$ for each $j,$ ensuring that $\Omega \subset \{v_{\infty} = 1\}$ and hence $\{v_{\infty} = 1\} \neq \emptyset$. We may then apply Theorem~\ref{thm:AC-regularity} again, this time to the sequence  
$(v_{\e_{j}}).$ The fact that the original functions $u_{\e_{j}}$  leading to $M_{0}$ are min-max critical points guarantees that the limit measure corresponding to $(v_{\e_{j}})$ is non zero. Thus 
$\partial^{\star} \, \{v_{\infty} = 1\} \neq \emptyset$, so by the regularity and separation property guaranteed by Theorem~\ref{thm:AC-regularity} for the limit varifold corresponding to $(v_{\e_{j}}),$ we can take the regular (i.e.\ quasi-embedded) part of 
$|\partial^{\star} \, \{v_{\infty} = 1\}|$ to be the desired hypersurface $M$ in Theorem~\ref{thm:existence}.

In summary, our proof of Theorem~\ref{thm:existence}  consists of the following three main steps: 
\begin{enumerate}
\item[{\sc step (i)}] a proof of Theorem~\ref{thm:AC-regularity} above (Theorem~\ref{thm:regularity} below).  

\item[{\sc step (ii)}] a min-max construction of a critical point $u_{\e}$ of ${\mathcal F}_{\e, \sigma g}$  for each sufficiently small $\e,$ with the property that $0 < L \leq \Ece(u_{\e}) \leq K < \infty$ for constants $L$, $K$, independent of $\e$ and $g$.  (Proposition~\ref{Prop:mountainpass}, Lemma~\ref{lem:upperbound} and Lemma~\ref{lem:lowerbound}). This step is based on standard PDE tools: an application of a mountain pass lemma based on the fact that ${\mathcal F}_{\e, \sigma g}$ satisfies the Palais--Smale condition. It is then automatically true that the Morse index of $u_{\e}$ is at most 1. Taking $\e = \e_{j}$ in this construction for a sequence $\e_{j} \to 0^{+}$, the uniform upper bound on ${\mathcal E}_{\e_{j}}(u_{\e_{j}})$ implies that along a subsequence $\mu_{j} \to \mu$ (notation as in Theorem~\ref{thm:AC-regularity}) for some Radon measure $\mu$ on $N,$ with $\mu \neq 0$ in view of the positive lower bound on ${\mathcal E}_{\e_{j}}(u_{\e_{j}}).$  If $E = {\rm int} \, \{u_{\infty} = 1\} \neq \emptyset$, Theorem~\ref{thm:existence} holds in view of Theorem~\ref{thm:AC-regularity}, by just setting $M$ to be the quasi-embedded part of $\partial E.$ 

\item[{\sc step (iii)}]  a proof of the following result (Proposition~\ref{Prop:main} below) which leads to the conclusion of Theorem~\ref{thm:existence} if in {\sc step (ii)} we have that $\{u_{\infty} = 1\} = \emptyset$.

\end{enumerate}

\begin{Prop}
\label{Prop:main-intro}
Let $(u_{\e_{j}})$ be the sequence as {\sc step (ii)}, constructed as described in Section~\ref{minmax_setup}. Let $\mu$ be the (non-zero) limit Radon measure corresponding to (a subsequence of) $(u_{\e_{j}}),$ and suppose that $\mu =  q{\mathcal H}^{n} \res M_{0}$ where  $M_{0}$ and $q$ are as in Theorem~\ref{thm:AC-regularity}(b).  
Then there exist $v_{{\eps}_{j}}:N\to \R$ satisfying $\ca{F'}{{\eps}_j, \sigma g}(v_{{\eps}_j})=0$ and $\ca{F''}{{\eps}_j, \sigma g}(v_{{\eps}_{j}})\geq 0$ (i.e.\ stable critical points of ${\mathcal F}_{\e_{j}, \sigma g}$) with $\liminf_{{\eps}_j\to 0} \ca{E}{{\eps}_{j}}(v_{{\eps}_{j}})>0$ and $\limsup_{{\eps}_{j}\to 0} \ca{E}{{\eps}_j}(v_{{\eps}_j})<\infty$; moreover, there exists a (fixed) non-empty open set contained in $\{v_{{\eps}_{j}}>\frac{3}{4}\}$ for all ${\eps}_{j}$.
\end{Prop}

In certain cases, for instance if $g$ is constant and $N$ has positive Ricci curvature, it follows from {\sc step (iii)} of our proof that the possibility $V_{g} = 0$ cannot happen for the sequence of min-max critical points $(u_{\e_{j}})$ constructed in {\sc step (ii)} (see Remark~\ref{oss:Ric_pos} below). It is also conceivable that this possibility does not occur for more general metrics, for instance for $n \leq 6$ and for metrics on $N$ for which all minimal hypersurfaces are non-generate (a dense subset of the set of smooth metrics by a theorem of White (\cite{Whi})). We emphasize that here we bypass this question altogether by proceeding as in {\sc step (iii)}.

We end this introduction with remarks on some recent and old work related to the present work.

\begin{oss}
By techniques very different from those used here---specifically, by an Almgren--Pitts min-max construction---the existence of hypersurfaces with mean curvature prescribed by an ambient function $g$ ($\not\equiv 0$) on a compact manifold $N^{n+1}$ has recently been addressed in the following cases: for $2\leq n \leq 6$ when $g$ is any non-zero constant by Zhou--Zhu (\cite{ZZ1}); for $2\leq n \leq 6$ when $g:N \to \R$ is smooth and satisfies certain conditions on its nodal set $\{g=0\}$ (the resulting collection of functions being generic in a Baire sense) by Zhou--Zhu (\cite{ZZ2}); for $n\geq 7$ and  
$g$ any non-zero constant by Dey (\cite{D}). 

For $n \leq 6$, the compactness theorem needed in these works to prove regularity of the minmax varifolds is a straighforward consequence of the pointwise curvature estimates due to Schoen--Simon--Yau (\cite{SSY}) (for $n \leq 5$) and Schoen--Simon (\cite{SS}) (for $n=6$). For general dimensions such estimates do not hold and the corresponding compactness theorems have to be obtained by 
different methods. The results of \cite{BW1}, \cite{BW2} and \cite{SS} provide a complete (regularity and) compactness theory in general dimensions needed for both the Almgren--Pitts approach as in \cite{D} as well as our approach in the present work. 

We point out some methodological differences between the Almgren--Pitts minmax construction and the Allen--Cahn method pursued here. Just like in the original 
Almgren--Pitts method used to construct minimal hypersurfaces, the arguments in 
\cite{ZZ1}, \cite{ZZ2}, \cite{D} seem to require a ``pull-tight'' procedure (to make up for the lack of a Palais--Smale condition), the fulfilment of an ``almost-minimizing'' condition (as a substitute for uniform Morse index control), and the use of ``stable replacements'' (to obtain the regularity conclusions). None of these steps are required in the Allen--Cahn approach, which instead capitalises on elementary, general PDE principles and a sharp varifold regularity theory of independent interest.

We further point out that our regularity results (specifically, Theorem~\ref{BWregularity} and Theorem~\ref{estimates} below) allow us to extend the existence theorem in a straightforward manner from the case of positive $g \in C^{1,1}(N)$ to the case of arbitrary non-negative Lipschitz $g$ by means of an approximation argument (see Section~\ref{extension}). This argument uses in particular the local uniform $C^{2, \alpha}$ estimates of Theorem~\ref{estimates} which assumes that $g \in C^{1, 1}$ but yields constants in its conclusion that depend on $g$ only through an upper bound on the $C^{1}$ norm of $g$. The applicability of these estimates relies crucially on the fact that the Allen--Cahn min-max solutions satisfy a uniform Morse index bound, which implies stability of the approximating hypersurfaces in fixed-sized balls. 

\end{oss}

\begin{oss}
Historically the use of the Allen--Cahn energy to approximate the area functional goes back to an idea of De Giorgi and to the work of Modica--Mortola \cite{MM}, which considered the case of minimizers. 
In particular minimizers of ${\mathcal F}_{\e_{j}, \sigma g}$ converge in $L^{1}$ to  a $BV$ limit taking values $\pm1,$ and whose $+1$ phase $E$ is a minimizer of the functional $\text{Per}(E)-\int_E g.$ Higher multiplicity issues (including the precense of a ``touching set'') that need to be addressed here do not arise in the case of minimizers, and regularity of the limit in the form of embeddedness away from a codimension 7 singular set follows from the work of 
Gonzalez--Massari--Tamanini (\cite{GMT1}, \cite{GMT2}).
\end{oss}

\section{The case $g \equiv 0$, new difficulties when $g >0$ and technical aspects of their resolution}
\label{overview}
In this section we provide a brief description of the Allen--Cahn approach to the case $g \equiv 0$ of Theorem~\ref{thm:existence} (i.e.\ the existence of minimal hypersurfaces), and a more detailed overview of {\sc step(i), step (ii)} and {\sc step (iii)}  above in the case $g >0$.

 For the case $g \equiv 0$ of Theorem~\ref{thm:existence}, the analogue of {\sc step (ii)} is carried out in the work of Guaraco \cite{Gua} which shows, via a classical mountain pass lemma, the existence of a critical point $u_{\eps}$ of $\Ece$ for small $\e \in (0, 1),$ with Morse index at most 1 and with uniform upper and lower energy bounds.  Choosing next $\e_{j} \to 0^{+}$ and applying a theorem of Hutchinson--Tonegawa (\cite{HT}), one obtains a sequence of $n$-varifolds $(V^{j})$ associated with the measures $\mu_{j} = (2\sigma)^{-1}\left(\frac{\e_{j}}{2}|\nabla u_{\e_{j}}|^{2}  + \e_{j}^{-1}W(u_{\e_{j}})\right) d{\rm vol}_{N}$ such that along a subsequence, $V^{j}$ converges in the first instance to a non-trivial area-stationary integral $n$-varifold $V$ on $N,$ with its weight measure $\|V\| = \mu = \lim_{j \to \infty} \, \mu_{j}$.  The proof is completed by deducing regularity of $\mu$ as follows: the uniform Morse index bound on $(u_{\e_{j}})$ implies 
 that for each point $p \in N$ there is a number $r>0$ such that for each $\delta \in (0, r)$, a subsequence of $(u_{\e_{j}})$ is stable (i.e.\ has index zero) in the annulus 
 ${\mathcal N}_{r}(p) \setminus \overline{{\mathcal N}_{\delta}(p)}.$ 
 Hence  by a theorem of Tonegawa (\cite{Ton}), locally near every point of ${\rm spt} \, \mu$ the embedded part of ${\rm spt} \, \mu$ is stable with respect to the area functional for compactly supported normal variations (of the embedded part). Regularity of $\mu$ (i.e.\ that $\mu = q{\mathcal H}^{n} \res M_{0}$ where $M_{0}$ is a minimal hypersurface smoothly embedded away from a closed singular set ($=\overline{M_{0}} \setminus M_{0}$) of Hausdorff dimension $\leq n-7$, and $q,$ the multiplicity function, is positive integer valued and locally constant on $M_{0}$) follows by using local stability of the sequence $(u_{\e_{j}})$ to rule out a certain very specific type of singularities of $\mu$ (classical singularities, where at least one tangent cone is made up of three or more half-hyperplanes meeting along a common axis) (\cite{TonWic}), and applying the regularity theory for stable codimension 1 stationary varifolds (\cite{Wic}). Note that whether the set $E = \{u_{\infty} = 1\}$ is empty or not is irrelevant in the case of minimal hypersurfaces, and hence an additional step like {\sc step (iii)} above is not necessary in this case. (We remark that in dimension $n=2$, more recent work of Chodosh--Mantoulidis (\cite{CM}) provides an alternative proof of the regularity of the limit surface by establishing strong convergence of the level sets of the Allen--Cahn solutions.)

For the case $g >0$, {\sc step (ii)} is again, in essence, an application of a standard PDE mountain pass lemma. For each small $\e \in (0, 1)$, this step produces two functions 
$a_{\e}, b_{\e} \in W^{1,2}(N)$ close to the constant functions $-1$ and $1$ respectively, and a critical point $u_{\e}$ of ${\mathcal F}_{\e, \sigma g}$ in $W^{1, 2}(N)$ such that ${\mathcal F}_{\e, \sigma g}(u_{\e})$ is the min-max width of ${\mathcal F}_{\e, \sigma g}$ over all continuous paths in $W^{1, 2}(N)$ joining $a_{\e}$ to $b_{\e},$ i.e.\ 
$${\mathcal F}_{\e, \sigma g}(u_{\e}) = \min_{\gamma \in \Gamma_{\e}} \max_{t \in [-1, 1]} {\mathcal F}_{\e, \sigma g}(\gamma(t))$$
where $\Gamma_{\e} = \{\mbox{continuous} \;\;  \gamma \, : \,  [-1, 1] \to W^{1, 2}(N)\;\;  \mbox{with $\gamma(-1) = a_{\e}, \gamma(1)  = b_{\e}$}\}.$ 

We shall now provide a discussion of the key new aspects of {\sc step (i)} in the case $g >0$, followed by an overview of {\sc step (iii)}.

The starting point of {\sc step (i)} is the work of R\"oger--Tonegawa (\cite{RogTon}) and Tonegawa (\cite{Ton1}) (Theorem~\ref{Roger-Tonegawa} below). These works, which extend \cite{HT}, imply that the limit varifold $V$ is an integral $n$-varifold having locally bounded generalized mean curvature $H_{V}$ in $N$, that the (possibly empty) set $E$ ($= \{u_{\infty} = 1\}$) is a Caccioppoli set in $N$ with reduced boundary $\partial^{\star} \, E \subset {\rm spt} \, \mu$, that $H_{V}(x)  = 0$ for $\mu$-a.e.\ $x \in N \setminus \overline{E}$ and $H_{V}(x) = g(x) \nu(x)$ for $\mu$ a.e.\ $x \in \partial^{\star} \, E$ where $\nu$ is the unit normal to $\partial^{\star} \, E$ pointing into $E,$ and that the density $\Theta \, (\mu, x)  = 1$ for $\mu$ a.e.\ $x \in \partial^{\star} \, E$; moreover, in the case of positive $g$ as considered here, we have that ${\rm spt} \, \mu \subset N \setminus E$ and $\Theta \, (\mu, x)$ is an even integer for $\mu$ a.e.\ $x \in {\rm spt} \, \mu \setminus \overline{E}.$

As regards regularity of $\mu$, the first important difference between the case $g \equiv 0$ and the case $g >0$ is that as mentioned above, a limit varifold $V$ in the latter case may not in its entirety be a solution to a variational problem (whereas in the former case it is, namely a critical point of $n$-dimensional area). As the next best option we would \emph{like} to say that 
the part of $V$ corresponding to the reduced boundary (i.e.\ the multiplicity 1 varifold $|\partial^{\star} \, E|$, called the phase boundary) and the complementary part (i.e.\ $V \res N \setminus \overline{E}$, called the hidden boundary) are \emph{separately} critical points, in $N$,  of 
the functionals $A - {\rm Vol}_{g}$ and $A$ respectively, where $A$ is the $n$-dimensional area functional and ${\rm Vol}_{g}$ is the enclosed $g$-volume. However even under a uniform Morse index bound on $u_{\e_{j}}$ this does not follow apriori from the results of \cite{RogTon}, \cite{Ton1}.  In addition to the lack of regularity at this stage, a serious difficulty impeding such a decomposition is the topologically complicated ways in which the two parts $|\partial^{\star}\, E|$ and $V \res (N \setminus \overline{E})$ may merge together on a set of points $y$ of positive $(n-1)$-dimensional Hausdorff measure where $V$ has planar tangent cones of mutiplicity $> 1$, e.g.\ as depicted schematically in Figure~\ref{fig:necks}. \

\begin{figure}[h]
\centering
\includegraphics[scale=0.2]{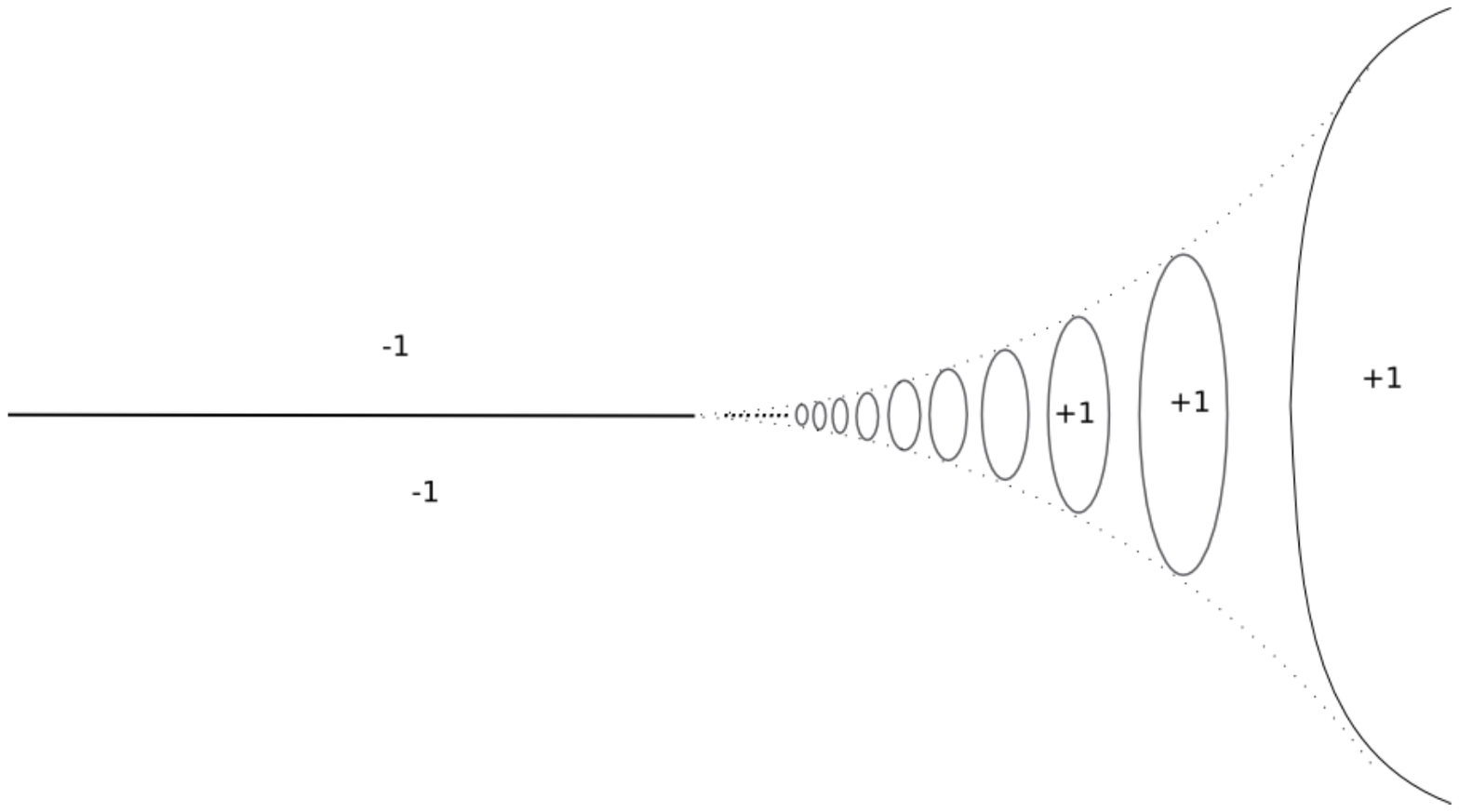}
\includegraphics[scale=0.2]{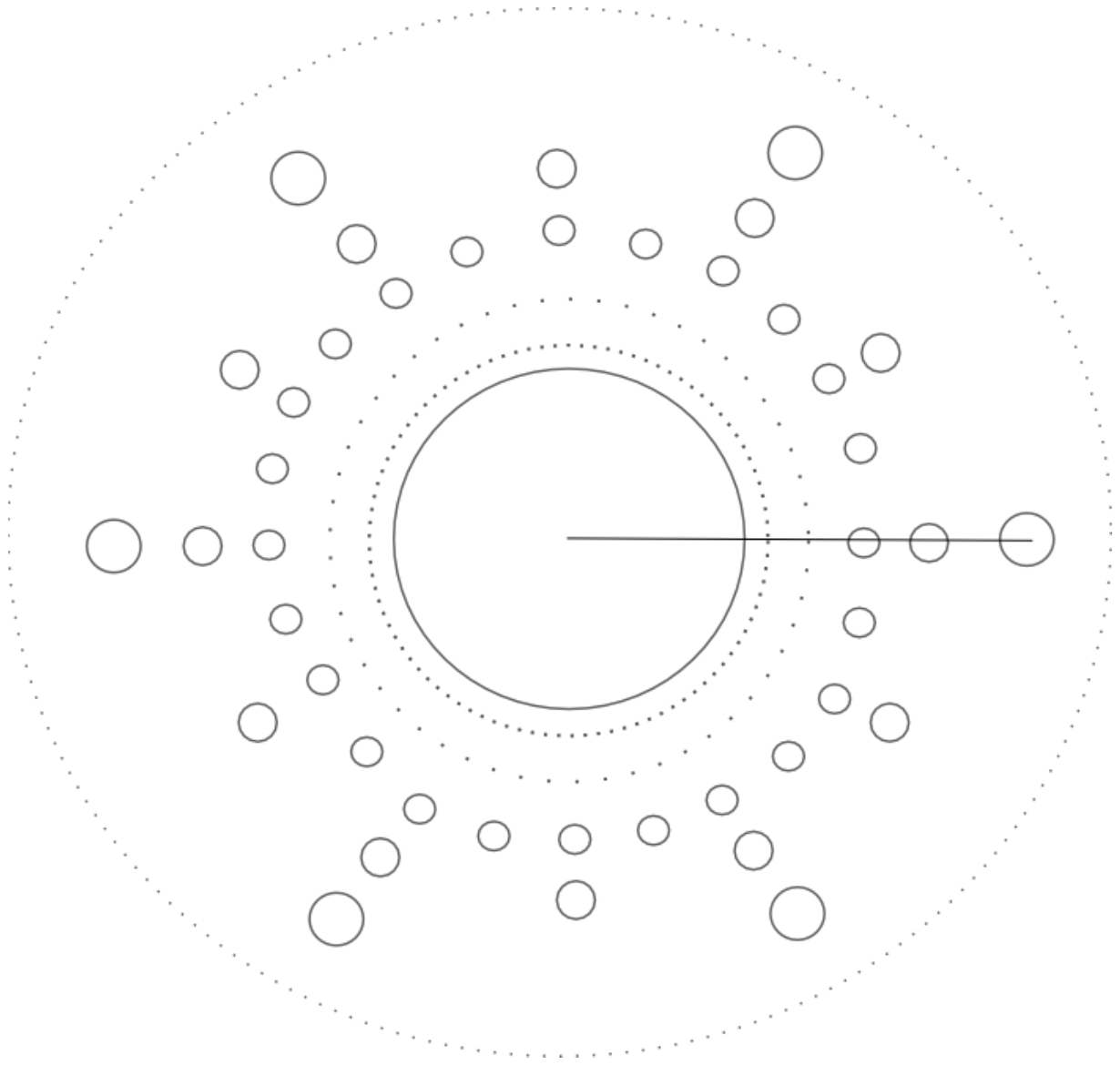}
\caption{Potential topological complexity (``neck accumulation'') at points of a limit varifold where pieces of zero and non-zero mean curvature meet. The picture on the right is the ``top'' view, i.e.\ the projection of the surface onto the tangent plane. The disk in the middle has multiplicity 2 and zero mean curvature. The picture on the left represents a slice along a half-plane (represented by the line segment in the picture on the right) perpendicular to the tangent plane. The numbers $\pm 1$ indicate the two phases $\{u_{\infty} = 1\}$ and $\{u_{\infty} = -1\}$.}
\label{fig:necks}
\end{figure}

In an inductive argument to resolve this regularity question, we first apply a key varifold regularity result (Theorem~\ref{BWregularity} below) from our recent work (\cite{BW1}, \cite{BW2})  to show (in Theorem~\ref{limit-regularity}) that subject to a uniform bound on the Morse index of $u_{\e_{j}}$, such topologically complicated behavior indeed does not occur; in fact an ordered $C^{1, \alpha}$ graph structure for $V$ must hold near such a point $y$.  Once this is achieved, we combine (in Theorem~\ref{thm:inductive_higher_reg}) local stability with an adaptation of a PDE argument from \cite{BW1} to prove $C^{2}$ (and hence $C^{3, \alpha}$) regularity of these graphs near $y$ (Theorem~\ref{limit-regularity}, part (ii)). The difficulty in this higher regularity question is that the union of the graphs need not be embedded; indeed, two graphs can touch a priori on a large (but ${\mathcal H}^{n}$-null) set, and the size of this coincidence set is only shown to be lower dimensional (in fact to be locally contained in an $(n-1)$-dimensional $C^{1}$ submanifold) after $C^{2}$ regularity of the graphs is established.

Once $C^{2}$ regularity is established in a neighbourhood of a point $y$ where $V$ has a planar tangent cone,  it follows that ${\rm spt} \, |\partial^{\star} \, E|$ and ${\rm spt} \, \mu \res (N \setminus \overline{E})$ must merge at $y$, if at all, smoothly (i.e.\ as $C^{3, \alpha}$ graphs each of which is entirely minimal or entirely PMC with mean curvature $g$) and tangentially; moreover, this can happen only in one of the two possible ways described in Definition~\ref{PMC-structure} parts (iv), (v) and Figure~\ref{fig:PMC_structures}. This result, together with the absence of classical singularities in $V$ (Theorem~\ref{limit-regularity}, part (i)), leads to the global decomposition and regularity conclusion for $\mu$ as in Theorem~\ref{thm:AC-regularity}. 

There is also a key difference between the case $g \equiv 0$ and the case $g >0$ in the way in which local stability of $u_{\e_{j}}$ (implied by the uniform Morse index bound on $u_{\e_{j}}$) is used in the  analysis of the limit varifolds $V$ described above. In either case, by \cite{Ton}, stability of (a subsequence of) $u_{\e_{j}}$ in a ball implies that the limit varifold $V$ in that ball admits a generalized second fundamental form  satisfying an ``ambient'' stability inequality, namely, inequality~(\ref{var-stability-reg}); this inequality has the  form of the usual ``intrinsic'' stability inequality for embedded stable hypersurfaces, but with one important difference: it is valid only for ambient $C^{1}$ compactly supported test functions with the \emph{ambient gradient} (of the test function) appearing where the intrinsic (hypersurface) gradient appears in the intrinsic stability inequality.  The ambient and intrinsic stability inequalities are equivalent on the embedded part of a hypersurface (any $C^{1}$ test function $\varphi$ supported on the embedded part has a $C^{1}$ compactly supported  extension $\widetilde{\varphi}$ to the ambient space such that on the hypersurface the ambient gradient of $\widetilde{\varphi}$ agrees with the intrinsic gradient of $\varphi$). For this reason, and also since the regularity theory of \cite{Wic} only requires stability of the embedded part of the varifold, ambient stability inequality for limit varifolds is sufficient in the case $g \equiv 0$.

This is not so in the case $g >0$. In this case, unlike in the case $g\equiv 0,$ two $C^{2}$ pieces of the limit $n$-varifold may intersect tangentially along a set (the coincidence set) of finite, positive $(n-1)$-dimensional measure 
(consider for instance two touching unit cylinders in Euclidean space). Because of this possibility stability of the embedded part alone cannot be expected to yield any useful regularity estimates for the $C^{2}$ quasi-embedded part. (A coincidence set of positive $(n-1)$-dimensional measure, or even infinite $(n-2)$-dimensional measure for that matter, is too large to be removable for the (intrinsic) stability inequality on the embedded part.)  Stability of the quasi-embedded part as an immersion, which would imply the intrinsic stability inequality on the quasi-embedded part, would suffice in the present context, but this does not follow from the ambient stability inequality (i.e.\ inequality (\ref{var-stability-reg}) below, which is implied by stability of $(u_{\e_{j}})$) eventhough the latter holds on the quasi-embedded part. (A $C^{1}$ function on a quasi-embedded hypersurface is not necessarily the restriction to the hypersurface of an ambient $C^{1}$ function.) The question then is which alternative stability condition, valid across the coincidence set, would suffice.

As we identify in the present work, of crucial significance is the fact that it suffices to have an ``ambient Schoen inequality'' on the quasi-embedded part. 
This is of course fully consistent with the case $g \equiv 0$. In that case the quasi-embedded part is the embedded part, and the (stronger) intrinsic Schoen inequality (\cite[Lemma 1]{SS}) on the embedded part is valid and is implied by the intrinsic stability inequality. This implication is the extent to which the intrinsic stability plays a role in the proof of the main estimate in 
\cite[Theorem 1]{SS}, and the application of \cite[Theorem 1]{SS} is primarily the way the stability condition enters the proof in \cite{Wic}. In the present case of $g >0$ we do not have an intrinsic Schoen inequality on the quasi-embedded part, but instead we obtain its ambient analogue where any ambient test function and its ambient gradient take the place of the intrinsic test function and its intrinsic gradient. This ambient Schoen inequality is obtained \textit{directly from the stability of the Allen--Cahn critical points $u_{\e}$}. For the Euclidean ambient space (and for $g\equiv 0$), the derivation of this inequality was carried out by Tonegawa in \cite{Ton}. We here  (in Lemma~\ref{Schoen-Tonegawa} below) generalise  this to Riemannian manifolds (and to $g\not \equiv 0$). (This generalisation  turns out to be somewhat subtle and require a careful choice of coordinates; we provide a full account of it in Lemma~\ref{Schoen-Tonegawa}.) 

In \cite{BW2} a ``non-variational'' version of our regularity theory (namely, \cite[Theorem~9.1 and Theorem~5.1]{BW2}, also recalled in Section~\ref{nonvar-regularity} below) is established taking this Schoen--Tonegawa inequality as the stability assumption on the quasi-embedded region.  With these results from \cite{BW2} in hand, as well as extensions of some of the arguments in \cite[Section 7]{BW1} 
to account for the presence of both minimal and PMC parts, in Sections~\ref{higher-reg} and~\ref{limit-reg-proof} below we complete the  proof of Theorem~\ref{thm:AC-regularity}.

Finally we briefly describe {\sc step (iii)}. This is the way our argument produces a hypersurface with mean curvature $g$ even in the event that the limit measure $\mu$ arising from the min-max critical points $u_{{\varepsilon}_j}$ constructed in {\sc step (ii)} ends up being supported entirely on a minimal hypersurface $M_0$. This step is in part inspired by the ideas in the recent work of the first author (\cite{B1}), where it is shown that if the ambient manifold has positive Ricci curvature, then any minimal hypersurface arising from a sequence of $1$-paramater min-max
critical points of the homogeneous Allen--Cahn equation must have multiplicity $1$. 
Our argument here requires the construction, for each sufficiently small $\eps$, of a function $h=h^{{\eps}}$. The construction takes $M_0$ (and its regularity) as starting point; $h$ is smooth for $n\leq 6$, while for $n\geq 7$ it is Lipschitz. The function $h$ (after an appropriate smoothing in case $n\geq 7$) is then used as initial data for a negative ${\mathcal F}_{\eps, \sigma g}$-gradient flow. This flow is then proved to converge to a stable solution $v_{\e}$ satisfying the properties given in Proposition~\ref{Prop:main-intro}. Special care has to be taken to ensure that $h$ meets certain conditions, namely:

\begin{itemize}
\item [(I)] a mean convexity condition, namely, $-\ca{F'}{\eps, \sigma g}(h)>0$ (where the first variation $-\ca{F'}{\eps, \sigma g}(h)=\eps \Delta h - \frac{W'(h)}{\eps} +g$ should be understood distributionally: when $n\geq 7$, $-\ca{F'}{\eps, \sigma g}(h)$ is in fact a Radon measure, while for $n\leq 6$ it is a smooth function);

\item [(II)] the existence (for all sufficiently small $\eps$) of a continuous path in $W^{1,2}(N)$ that joins $a_{{\eps}}$ (the valley point close to $-1$ identified in {\sc step (ii)}) to $h$, such that all along the path we have $(2\sigma)^{-1}{\mathcal F}_{\eps, \sigma g}  \leq 2\mathcal{H}^n(M_0)-\varsigma/2$, where $\varsigma>0$ depends only on the geometric data $M_0\subset N$. 
\end{itemize}

The construction of $h$ is based on a certain deformation of the minimal immersion $\iota$ of the oriented double cover of $M_0$ into $N$, that covers $M_0$ twice. This deformation is schematically depicted in Figure \ref{fig:avoid_peak}: in (i), two disks $D_1$ and $D_2$ are removed from $M_0$, leaving a double cover of $M_0\setminus (D_1 \cup D_2)$ (an immersion with boundary); in (ii), one of the disks, say $D_1$, is pasted back in (with multiplicity $2$); in (iii) the double disk $2|D_1|$ is deformed \textit{as an immersion}: since $M_0$ has vanishing mean curvature, and is therefore not stationary for the functional $\text{Area}-\text{Vol}_g$, this deformation decreases the value of $\text{Area}-\text{Vol}_g$; in (iv) and (v) the same type of operations are performed on $D_2$. The deformation represented in Figure \ref{fig:avoid_peak} is constructed in order to ensure that the value of $\text{Area}-\text{Vol}_g$ as we go from (i) to (v) stays strictly below $2\mathcal{H}^n(M_0)$, by an amount $\varsigma>0$ that only depends on the geometric properties of $M_0$ in $N$. Note that, for $\iota$, $\text{Vol}_g$ vanishes and therefore $\text{Area}-\text{Vol}_g$ is just $2\mathcal{H}^n(M_0)$. 

For any $\eps$ sufficiently small, we exhibit a continuous path in $W^{1,2}(N)$ using functions that replicate, in a suitable sense, the geometric behaviour identified by the deformations represented in Figure \ref{fig:avoid_peak}. These functions are identically $-1$ or identically $+1$ on certain regions and the transitions between the values $\pm 1$ happen in a small neighborhood of the hypersurfaces depicted in Figure \ref{fig:avoid_peak}. Moreover, the energy $(2\sigma)^{-1}{\mathcal F}_{\eps, \sigma g}$ of these functions is a very close approximation of the value attained by evaluating $\text{Area}-\text{Vol}_g+\frac{1}{2}\int_N g$ on the hypersurfaces depicted in Figure \ref{fig:avoid_peak}--- with an error of size $O(\eps|\log\eps|)$. An extremely effective feature of the Allen--Cahn framework is that the transitions from (i) to (ii) and from (iv) to (v) can be replicated \textit{continuously}, while the corresponding geometric deformations present discontinuities (with the sudden appearance of a portion). In similar spirit, we can find a continuous path that joins $a_{\eps}$ to the function corresponding to the immersion with boundary in (i). (The geometric counterpart is the sudden appearance of the hypersurface $M_0\setminus (D_1 \cup D_2)$ with multiplicity $2$.) The path ``from $a_{\eps}$ to (i)'' is produced so that ${\mathcal F}_{\eps, \sigma g}$ is almost increasing, i.e.~it is equal to an increasing function plus a function bounded in modulus by $O(\eps|\log\eps|)$. The function $h$ mentioned above corresponds to the immersion represented in (v) and it is close to $+1$ in the two domains obtained by deforming the double disks $D_1$ and $D_2$ and close to $-1$ in their complement. Moreover, the construction ensures the mean convexity condition in (I). The mean convexity implies that the negative ${\mathcal F}_{\eps, \sigma g}$-gradient flow starting at $h$ evolves towards a stable solution $v_{\eps}$ of 
$\ca{F'}{\eps, \sigma g}=0$, that has to be $\geq h$. The assumption that the min-max solutions give rise to $\mu$ supported entirely on $M_0$ with multiplicity $\geq 2$ implies that $v_{\eps}$ cannot be the second valley point $b_{\eps}$. Indeed, if $v_{\eps}=b_{\eps}$, we would have found a continuous path in $W^{1,2}(N)$ that connects $a_{\eps}$ to $b_{\eps}$  such that along this path, $(2\sigma)^{-1}{\mathcal F}_{\eps, \sigma g}$ stays below the minmax value $(2\sigma)^{-1}{\mathcal F}_{\eps, \sigma g}(u_{\eps})$ by a fixed amount $\varsigma/2$ (recall that we have assumed $(2\sigma)^{-1}{\mathcal F}_{\eps, \sigma g}(u_{\eps})\to |\mu|(N)+\frac{1}{2}\int_N g$); this contradicts the min-max characterisation. Very crucially, the fact that $v_{\e} \neq b_{\e}$ leads to a uniform positive lower bound on ${\mathcal E}_{\e}(v_{\e})$ independent of $\e$, while  the energy decreasing property of the flow implies a uniform upper bound on $\Ece(v_{\eps})$. 
Moreover, the condition $v_{\eps}\geq h$ guarantees the presence of a fixed open set (corresponding to the two domains obtained by deforming the double disks $D_1$ and $D_2$) on which $v_{\eps}\geq 3/4$ for all $\eps$, as claimed in Proposition~\ref{Prop:main-intro}. Then we can consider the varifold associated to $v_{\eps}$ and apply Theorem \ref{thm:AC-regularity} (with $v_{\eps_{j}}$ in place of $u_{\eps_{j}}$ for a sequence $\e_{j} \to 0^{+}$) to obtain the desired hypersurface with mean curvature $g$.

We remark that the arguments given in Section \ref{proof} for {\sc step (iii)} become considerably shorter if $n\leq 6$,
since in that case $M_0$ has no singularities. In general dimensions, while the basic geometric
idea remains the same, a finer analysis of the distance function to $M_0$ is necessary to handle the presence of a small singular set in $M_0$ (as well as some extra technical asides).

\section{Preliminaries: Allen--Cahn solutions, limit $(g,0)$-varifolds and GMT regularity results} \label{preliminaries}

In this section and in Section~\ref{stable-regularity} the ambient Riemannian manifold $N$ need not be complete. We shall continue to assume that ${\rm dim} \, N  = n+1$ with $n\geq 2$, and use notation as in the preceding sections.

\subsection{Allen--Cahn solutions and limit $(g, 0)$-varifolds}\label{ac-solutions}

 For $\eps\in(0,1)$ let $u_{\eps}:N\to \R$ satisfy $\sup_{N} \, |u_{\e}| + \Ece(u_{\eps})\leq K$, for $K>0$ independent of $\eps$. Let $w_{\eps} = \Phi(u_{\eps})$ for $\Phi(s) = \int_0^s \sqrt{\frac{W(\widetilde{s})}{2}}d\widetilde{s},$ and let $\sigma=\int_{-1}^{1} \sqrt{\frac{W(s)}{2}}ds$. The $BV$-norm of $w_{\eps}$ is bounded from above by $K$ (see \cite{HT} (and \cite{MM}). This allows to obtain a subsequential BV-limit $w_\infty$ of $w_{\eps}$, as $\eps\to 0$. We also denote $u_\infty = \Phi^{-1}(w_\infty)$, and we have $u_{\eps}\to u_\infty$ in $L^1$ (see section \ref{minmax_setup}).
We define, as in \cite{HT}, the $n$-varifold $V^{u_{\eps}}$ associated to $u_{\eps}$ by
$$V^{u_{\eps}}(A) = \int_{-\infty}^\infty|\{w_{\eps}=t\}|(A) dt, $$
where $|\{w_{\eps}=t\}|$ is the multiplicity-$1$ varifold associated to the level set $\{w_{\eps}=t\}$. 

\medskip

We recall the main results of \cite{RogTon}, and of \cite{Ton1} (which rely on \cite{HT}), concerning $V$ in the case $g >0.$

\begin{thm}(\cite[Theorem~3.2]{RogTon} and \cite[Theorem~2.2]{Ton1}) \label{Roger-Tonegawa} 
Let $g \in C^{1, 1}(N)$,  $g_{j} \in C^{1, 1}(N)$ with $g_{j} \to g$ locally in $C^{1, 1}$ and let $\e_{j} \to 0^{+}.$ Let 
$u_{\e_{j}} \in W^{1, 2}_{\rm loc}(N)$ be a critical point of ${\mathcal F}_{\e_{j}, \sigma g_{j}}$ for each $j$, and suppose that 
$\lim_{j \to \infty} \, V^{u_{\e_{j}}} = V$ for some $n$-varifold $V$ on $N$, where $V^{u_{\e_{j}}}$ is the $n$-varifold on $N$ associated with $u_{\e_{j}}$ as described above. Let $u_{\infty} \in BV_{\rm loc}(N; \{-1, 1\})$ be such that $u_{\e_{j}} \to u_{\infty}$ locally in $L^{1},$ and note that such $u_{\infty}$ exists possibly after passing to a subsequence of $(\e_{j}).$ 
Let  $E = \{x \in N \, : \, u_{\infty}(x) = 1\}$. We have the following: 
\begin{itemize} 
\item[{\rm (i)}] $\sigma^{-1}V$ is an integral $n$-varifold, with $V$ having locally bounded generalised mean curvature $H_{V}$ and first variation $\delta V = -H_{V} \|V\|$ in $N;$ the set $E$ is a Caccioppoli set in $N$ with  reduced boundary $\partial^{\star} E \subset {\rm spt} \, \|V\| \subset 
N \setminus {\rm int}\, E;$ 
\end{itemize}
if additionally $g >0$ and $V \neq 0$ then: 
\begin{itemize}
\item[{\rm (ii)}] $E \neq N;$
\item[{\rm (iii)}] $\sigma^{-1}\Theta \, (\|V\|, x)  = 1$ and $H_{V}(x) \cdot \nu(x) = g(x)$ for ${\mathcal H}^{n}$-a.e.\ $x \in \partial^{\star}E$ where $\nu$ is the inward pointing unit normal to $\partial^{\star} E$ (i.e.\ if $\nu = \frac{\nabla \, \chi_{E}}{|\nabla\chi_{E}|}$);
\item[{\rm (iv)}]  $H_{V}(x) = 0$ for ${\mathcal H}^{n}$-a.e.\ $x \in {\rm spt} \, \|V\|\setminus \partial^{\star} E;$
\item[{\rm (v)}] $\sigma^{-1}\Theta \, (\|V\|, x)$ is an even integer $\geq 2$ for ${\mathcal H}^{n}$-a.e.\ $x \in {\rm spt} \, \|V\| \setminus \partial^{\star} E.$ 
\end{itemize}
\end{thm}

\begin{oss} In \cite{RogTon} and \cite{Ton1}, these conclusions are established in the case of Euclidean ambient space. Adaptation of the arguments to the case of Riemannian ambient space is 
routine. 
\end{oss} 

Under the hypotheses of Theorem~\ref{Roger-Tonegawa}, no regularity result for the limit varifold $V$ is known beyond the fact that the regular set (i.e.\ the $C^{1, \alpha}$ embedded part of ${\rm spt} \, \|V\|$) is dense in ${\rm spt} \, \|V\|$ (which follows from Allard's regularity theorem). Theorem~\ref{Roger-Tonegawa}  says that the minimal portions (also referred to as 
``hiddden boundary''), if there are any,  always appear with even multiplicity and lie in the 
$\{u_\infty=-1\}$-phase whenever $g >0$. In principle, minimal and non-minimal portions may come together in irregular fashion (e.g.\ as depicted in Figure~\ref{fig:necks}). More threatening to the success of the 
min-max approach to Theorem~\ref{thm:existence} is the possibility that the limit interface ends up being completely minimal, i.e.\ the possibility that ~$u_\infty\equiv -1$ a.e.\ on $N$ and $\spt{V}$ all consists of hidden boundary. This possibility can in fact arise under the hypotheses of Theorem~\ref{Roger-Tonegawa} even when $g\equiv 1$ and $N=\R^n$; see \cite[Section~6.3]{HT}.

Throughout the rest of the article, it will be convenient to use the terminology defined as follows:

\noindent
\begin{Dfi}[{\sc limit $(g, 0)$-varifolds and stable limit $(g, 0)$-varifolds}]\label{limit-varifold}
Let $g \in C^{1, 1}(N)$.  We say that an $n$-varifold $V$ on $N$ is a \emph{limit $(g,0)$-varifold on $N$} if there are a sequence of numbers $\e_{j} \to 0^{+}$ and for each $j$, a function $g_{j} \in C^{1, 1}(N)$ and a critical point $u_{\e_{j}} \in W^{1, 2}_{\rm loc}(N)$ of ${\mathcal F}_{\e_{j}, \sigma g_{j}}$ such that $g_{j} \to g$ locally in $C^{1, 1}$ and $V = \lim_{j \to \infty} \, V^{u_{\e_{j}}},$ where $V^{u_{\e_{j}}}$ is the $n$-varifold on $N$ associated with $u_{\e_{j}}$ as described above.  We say that $V$ is a \emph{stable limit $(g,0)$-varifold on $N$} if $V$ is a limit $(g,0)$-varifold on $N$ and the associated critical points $u_{\e_{j}}$ of ${\mathcal F}_{\e_{j}, \sigma g_{j}}$ are stable in $N$, i.e.\ satisfy 
 $\left.\frac{d^{2}}{d s^{2}} \right|_{s=0} \, {\mathcal F}_{\eps_{j}, \sigma g_{j}}(u_{\eps_{j}} + s \varphi) \geq 0$ for 
each $\varphi \in C^{1}_{c}(N).$  
\end{Dfi}
This terminology is motivated by Theorem~\ref{Roger-Tonegawa}, according to which a varifold $V$ which is the limit of a sequence of varifolds $V^{u_{\e_{j}}}$ as described above admits generalised mean curvature, and consists of an oriented portion on which the generalized scalar mean curvature is equal to $g$ a.e., and a complementary portion where the generalized mean curvature is $0.$


\subsection{A non-variational varifold regularity theory} 
\label{nonvar-regularity}
The central ingredient of our proof of regularity of limit $(g, 0)$-varifolds associated with Morse index bounded Allen--Cahn solutions (Theorem~\ref{limit-regularity} and Theorem~\ref{estimates} below) is the general varifold regularity theory 
comprising Theorem~\ref{BWregularity} and Theorem~\ref{nonvar-SS} below. Both these theorems have non-variational hypotheses; that is to say, the varifolds to which the theorems are applicable are not assumed to be critical points of a functional. This is important for the present application since 
limit $(g, 0)$-varifolds may consist of both stationary (i.e.\ zero mean-curvature) and prescribed-mean-curvature parts merging together. 

Theorem~\ref{BWregularity}, which we state next, is a Riemannian counterpart of \cite[Theorem~9.1]{BW2} and it follows directly from the latter (see Remark~\ref{proof-rmk} below). In this theorem, an important role is played by a very specific type of singularities defined as follows: 

\begin{Dfi}[{\sc classical singularity}]\label{classical-sing} 
Let $V$ be an $n$-varifold on a Riemannian manifold $N$ of dimension $n+1$. A point $Y \in {\rm spt} \, \|V\|$ is a {\rm classical singularity} of $V$  if there exists $\rho >0$ such that, for some 
$\alpha \in (0, 1]$,  we have that $\spt{V} \cap {\mathcal N}_{\rho}(Y) = \cup_{j=1}^{k} M_{j}$ where: $k \in \N$, $k \geq 3$; each $M_{j}$ is an embedded $C^{1,\alpha}$ hypersurface-with-boundary in 
${\mathcal N}_{\rho}(Y);$ there is an $(n-1)$-dimensional embedded $C^{1, \alpha}$ submanifold $\gamma$ of ${\mathcal N}_{\rho}(Y)$ with $Y \in \gamma$ such that $\partial \, M_{j} = \gamma$  for $j=1, 2, \ldots, k;$ the $M_{j}$'s meet pairwise only along $\gamma$ with at least one pair meeting transversely everywhere along $\gamma$.  
\end{Dfi}

In the statement of Theorem~\ref{BWregularity}, it is convenient to use the following terminology associated with any integral $n$-varifold $V$ on $N$ such that $V$ has generalised mean curvature $H_{V} \in L^{p}_{\rm loc}(\|V\|)$ for some $p > n$ and first variation $\delta \, V = -H_{V} \|V\|$ in some open subset $U_{V} \subset N.$ 

\begin{Dfi}[{\sc ($q, \beta$)-separation property}] \label{qbseparation} Let $q \in {\mathbb N}$, and $\beta \in (0, 1)$. Let $V$ be as above. 
We say that $V$ has the \emph{$(q, \beta)$-separation property} provided the following implication holds: if 
\begin{itemize}
\item[{\rm (i)}] $X \in U_{V};$  $\rho \in (0, \min\{1, {\rm inj}_{X} N, {\rm dist} \, (X, \partial U_{V})\}];$ $Q \, : \, T_{X} \, N \approx  {\mathbb R}^{n+1} \to T_{X} \, N$ is an orthogonal rotation;
\item[{\rm (ii)}] the varifold  $\widetilde{V} \equiv \left(Q \circ {\rm exp}_{X}^{-1}\right)_{\#} \, V \res \left({\mathcal N}_{{\rm inj}_{X} N}(X) \cap U_{V}\right)$ satisfies  
$$\omega_{n}^{-1}\|\eta_{0, \rho \, \#} \, \widetilde{V}\|(B_{1}^{n+1}(0)) \leq q+ 1/2,$$ 
$$q-1/2 \leq \left(\frac{\omega_{n}}{2^{n}}\right)^{-1}\|\eta_{0, \rho \, \#} \widetilde{V}\|((B_{1/2}^{n}(0) \times {\mathbb R}) \cap 
B_{1}^{n+1}(0)) \leq q + 1/2,$$ and 
\begin{eqnarray*}
\int_{(B_{1/2}^{n}(0) \times {\mathbb R}) \cap B_{1}^{n+1}(0)} |x^{n+1}|^{2} \, d\|\eta_{0, \rho \#} \widetilde{V}\| && \\
&&\hspace{-3in} + \rho\left(\int_{(B_{1/2}^{n}(0) \times {\mathbb R}) \cap B_{1}^{n+1}(0)} |H_{V}(\exp_{X}(Q^{-1}(\rho Y))|^{p} \, d\|\eta_{0, \rho \, \#} \widetilde{V}\|(Y)\right)^{1/p} + \rho < \beta;
\end{eqnarray*} 
\item[{\rm (iii)}] $Y \in B^{n}_{\rho/2}(0) \times \{0\} \subset {\mathbb R}^{n+1}$; $\t \in (0, \rho/2];$  
\item[{\rm (iv)}]  $\Theta \, (\|V\|, \xi) < q$ for each $\xi \in {\rm exp}_{X}(B_{\t}^{n}(Y) \times {\mathbb R});$
\item[{\rm (v)}] the varifold $W \equiv \eta_{Y, \t \, \#} \, \widetilde{V}$ satisfies  
$$(\omega_{n})^{-1}\|W\|(B_{1}^{n+1}(0)) \leq q+ 1/4,$$ 
$$q-1/4 \leq \left(\frac{\omega_{n}}{2^{n}}\right)^{-1}\|W\|((B_{1/2}^{n}(0) \times {\mathbb R}) \cap 
B_{1}^{n+1}(0)) \leq q + 1/4,$$ and 
\begin{eqnarray*}
 \int_{(B_{1/2}^{n}(0) \times {\mathbb R}) \cap B_{1}^{n+1}(0)} |x^{n+1}|^{2} \, d\|W\| && \\
&&\hspace{-3in} + \rho\left(\int_{(B_{1/2}^{n}(0) \times {\mathbb R}) \cap B_{1}^{n+1}(0)} |H_{V}(\exp_{X}(Q^{-1}(\t Z))|^{p} \, d\|W\|(Z)\right)^{1/p}  < \beta/2;
\end{eqnarray*} 
\end{itemize}
then we have that $W \res ((B_{1/4}^{n}(0) \times {\mathbb R}) \cap B_{1}^{n+1}(0))  = \sum_{j=1}^{q} |{\rm graph} \, u_{j}|$ for some $u_{j} \in C^{2}(B_{1/4}^{n}(0))$, $j=1, 2, \ldots, q,$ with  
$u_{1} \leq u_{2} \leq \ldots \leq u_{q}.$ 
\end{Dfi}

\begin{thm}(\cite[Theorem~9.1]{BW2})\label{BWregularity}
Let $N$ be an $(n+1)$-dimensional  Riemannian manifold, $q$ be a positive integer, $\beta \in (0, 1)$ and $p >n$. Let ${\mathcal V}$ be a class of integral $n$-varifolds $V$ on $N$ satisfying the following properties ${\rm (a)}$-${\rm (c)}$: 
\begin{itemize}
\item[{\rm (a)}] corresponding to each $V \in {\mathcal V}$ there is an open set $U_{V} \subset N$ such that $V$ has generalised mean curvature $H_{V} \in L^{p}_{\rm loc} (\|V\|)$ and first variation $\delta \, V  = -H_{V} \|V\|$ with respect to $U_{V}$, i.e.\  $V$ satisfies 
$$\delta \, V({\psi}) = -\int_{N}<H_{V} , \psi> d\|V\|$$ 
for some $H_{V} \in L^{p}_{\rm loc}(\|V\|)$ and any compactly supported $C^{1}$ vector field $\psi \, : \, U_{V} \to TN;$ 
\item[{\rm (b)}] if $V \in {\mathcal V}$ then no point $Y \in {\rm spt} \, \|V\| \cap U_{V}$ with $\Theta \, (\|V\|, Y) = q$ is a classical singularity of $V$;
\item[{\rm (c)}]  if $V \in {\mathcal V}$ then $V$ satisfies the $(q, \beta)$-separation property; 
\end{itemize}
Conclusion : there exists $\epsilon = \epsilon (n, p, q, N, \beta, {\mathcal V}) \in (0, 1)$ such that 
if $V \in {\mathcal V}$, $X_{0} \in U_{V}$, $\widetilde{V} = \left(Q \circ {\rm exp}_{X_{0}}^{-1}\right)_{\#} \, V \res \left({\mathcal N}_{{\rm inj}_{X_{0}} N}(X_{0}) \cap U_{V}\right)$ for some orthogonal rotation $Q \, : \, T_{X_{0}} \, N \approx  {\mathbb R}^{n+1} \to T_{X_{0}} \, N$, and if $\rho \in (0, \min\{1, {\rm inj}_{X_{0}} N, {\rm dist} \, (X_{0}, \partial U_{V})\}]$, $$\omega_{n}^{-1}\|\eta_{0, \rho \, \#} \, \widetilde{V}\|(B_{1}^{n+1}(0)) \leq q+ 1/2,$$ 
$$q-1/2 \leq \left(\frac{\omega_{n}}{2^{n}}\right)^{-1}\|\eta_{0, \rho \, \#} \widetilde{V}\|((B_{1/2}^{n}(0) \times {\mathbb R}) \cap 
B_{1}^{n+1}(0)) \leq q + 1/2,$$ 
and 
\begin{eqnarray*}
\hat{E}_{\rho} \equiv \int_{(B_{1/2}^{n}(0) \times {\mathbb R}) \cap B_{1}^{n+1}(0)} |x^{n+1}|^{2} \, d\|\eta_{0, \rho \#} \widetilde{V}\| && \\
&&\hspace{-2.75in} + \rho\left(\int_{(B_{1/2}^{n}(0) \times {\mathbb R}) \cap B_{1}^{n+1}(0)} |H_{V}(\exp_{X}(Q^{-1}(\rho Y))|^{p} \, d\|\eta_{0, \rho \, \#} \widetilde{V}\|(Y)\right)^{1/p} + \rho < \e,
\end{eqnarray*}
then 
$$\eta_{0, \rho \, \#} \widetilde{V} \res ((B_{1/4}^{n}(0) \times {\mathbb R}) \cap B_{1}^{n+1}(0))  = \sum_{j=1}^{q} |{\rm graph} \, u_{j}|$$ 
for some $u_{j} \in C^{1, \alpha}(B_{1/4}^{n}(0))$, $j=1, 2, \ldots, q,$ with $u_{1} \leq u_{2} \leq \ldots \leq u_{q},$ 
where $\alpha = \alpha(n, p) \in (0, 1).$ Furthermore, we have that $$\|u_{j}\|_{C^{1, \alpha}(B_{1/4}(0))} \leq C \sqrt{\hat{E}_{\rho}}$$ 
for each $j \in \{1, 2, \ldots, q\},$ where $C = C(N, n, p, q).$ 
\end{thm} 

\begin{oss}The content of the preceding theorem (for a given varifold) can roughly be described as follows: fix a positive integer $q$. In the absence of classical singularities (as in hypothesis (b)), if an integral $n$-varifold with generalized mean curvature locally in $L^{p}$ for some $p > n$ has the property that its flat regions where density is $\leq q-1$ are ``well-behaved'' (as in hypothesis (c)), then the varifold is well-behaved also near any point where there is a tangent plane of multiplicity $q$.  Note the difference in the meaning of ``well-behaved'' in the hypotheses and in the conclusion: in the hypotheses (i.e.\ in hypothesis (c)) it means separation into ordered $C^{2}$ graphs, whereas in the conclusion it means separation into ordered $C^{1, \alpha}$ graphs. If $H_{V} = 0$ (the case handled in \cite{Wic}) the  Hopf boundary point lemma ensures $C^{1, \alpha} \implies C^{2}$, but simple examples show that in general $C^{2}$ conclusion need not hold, even if $|H_{V}| = 1$ (see e.g.\ \cite[Remark~2.14]{BW1}). In the present application (i.e.\ in Theorem~\ref{limit-regularity}) however, where $V$ is a limit $(g, 0)$-varifold arising from index bounded Allen--Cahn solutions, $C^{2}$ conclusion does hold (as we show in Section~\ref{higher-reg} below) because of the additional constraints on $V$ imposed by Theorem~\ref{Roger-Tonegawa}, including the fact that the set $E = \{u_{\infty} = 1\}$ (notation as in Theorem~\ref{Roger-Tonegawa}) is a Caccioppoli set.  
\end{oss}

\begin{oss}[\emph{\cite[Theorem~9.1]{BW2} $\implies$ Theorem~\ref{BWregularity}}]\label{proof-rmk}
The preceding theorem indeed follows very directly from \cite[Theorem~9.1]{BW2}, taken with ${\mathcal V}$ therein to be the collection of varifolds $\widetilde{\mathcal V}$ consisting of all varifolds 
$\widetilde{V} \res B_{1}^{n+1}(0),$ where (for $X_{0} \in N$ fixed as in the conclusion of the present theorem),  
$$\widetilde{V} = \left(\eta_{0, \rho} \circ Q \circ {\rm exp}_{X}^{-1}\right)_{\#} \, V \res {\mathcal N}_{{\rm inj}_{X_{0}} N}(X_{0})$$ 
for some $V \in {\mathcal V}$ (with ${\mathcal V}$ as in the present theorem), 
$X \in {\mathcal N}_{{\rm inj}_{X_{0}} N}(X_{0}) \cap U_{V}$, $Q \, : \, {\mathbb R}^{n+1} \to {\mathbb R}^{n+1}$ an orthogonal rotation, and $\rho \in (0, \min\{{\rm inj}_{X} N, {\rm dist} \, (X, \partial \, U_{V})\})$.  For the class $\widetilde{\mathcal V},$ hypothesis (a) of \cite[Theorem~9.1]{BW2} (with ${\hat H}_{\widetilde{V}} = \rho H_{V} \circ \exp_{X}$ for ${\widetilde V}$ as above, and for fixed constants $\kappa$, $\kappa_{1}$ depending only on $N,$ and for ``change-of-base-point'' diffeomorphisms $\varphi_{Y} \, : \, B_{1}^{n+1}(0) \to B_{1}^{n+1}(0)$, $Y \in B_{1/2}^{n}(0)$, defined by appropriate composition of scalings, rotations, exponential maps and their inverses as in \cite[Remark~3.1]{BW2}) follows from hypothesis (a) of the present theorem, as pointed out in \cite[Remark~9.1]{BW2}; hypothesis (b) of \cite[Theorem~9.1]{BW2}  follows directly from hypothesis (b) of the present theorem; and hypothesis (c) of \cite[Theorem~9.1]{BW2}  follows from hypothesis (c) of the present theorem, taking into account the following additional observation: let $u_{j}$ be as in hypothesis (c) of the present theorem. Then the requirement of hypothesis~(c) of \cite[Theorem~9.1]{BW2} that the $C^{1, \alpha_{1}}$ norm of $u_{j}$ are bounded (in the specified manner therein) for some fixed $\alpha_{1} \in (0, 1)$ holds automatically. Indeed, we have (by assumption) that $u_{j} \in C^{2}(B_{1/4}(0))$ for each $j$ and therefore that $M_{j} = \exp_{X} \circ Q^{-1} \circ \eta_{0, \rho^{-1}}({\rm graph}(u_{j}))$ is a $C^{2}$ hypersurface of $N;$ hence, by the discussion in \cite[Section~3.1]{BW2}, inequality $(A)$ in \cite[Section~5.1]{BW2} (taken with $\rho_{0} = 1$ and for $\kappa$, $\kappa_{1}$ fixed depending only on $N$) holds with $\eta_{0, 1/4 \, \#}|{\rm graph} \, u_{j}|$ in place of $V$  
\emph{for each $j$ separately}.  On the other hand, it is easy to see that the argument of Allard's regularity theorem (for the simpler case of $C^{2}$ graphs) carries over with inequality $(A)$ in place of the first variation hypothesis of Allard's theorem (requiring an $L^{p}$ bound on the mean curvature), leading to the estimate $$\|u_{j}\|_{C^{1, \alpha}(B_{1/8}^{n}(0))} \leq C \sqrt{\hat{E}_{\rho}}$$ 
for each $j$, where $\alpha = \alpha(n, p) \in (0, 1)$ and $C= (N, n, p).$ 
\end{oss}

We next state Theorem~\ref{nonvar-SS} which will serve to verify hypothesis (c) of Theorem~\ref{BWregularity} when applying Theorem~\ref{BWregularity} to certain limit $(g, 0)$-varifolds (in the proofs of Theorem~\ref{limit-regularity} and Theorem~\ref{estimates}). 
This result can be viewed as a Riemannian formulation of \cite[Theorem 5.1]{BW2} and it follows directly from the latter---see Remark~\ref{oss:5.1} below; \cite[Theorem~5.1]{BW2} in turn is a non-variational version of \cite[Theorem~1]{SS}, whose proof follows closely the argument of \cite[Theorem~1]{SS}. 

In Theorem~\ref{nonvar-SS} and subsequently, we shall use the terminology below (in Definitions~\ref{quasi-embedded}-\ref{reg}) associated with any integral $n$-varifold $V$ of a Riemannian manifold $N$ of dimemsion $n+1$: 

\begin{Dfi} [{\sc quasi-embedded point}] \label{quasi-embedded} A point $Y \in {\rm spt} \, \|V\|$ is said to be a {\rm quasi-embedded point} of $V$ if there exist $k \in \N$; $\rho>0;$ a hyperplane $L\subset \R^{n+1} \approx T_{Y} \, N$ with unit normal vector $\nu_L;$ $C^2$ functions $u_j:L\cap B_{\rho}^{n+1}(0) \to L^\perp$ for $j\in \{1, \ldots, k\}$ satisfying, whenever $k \geq 2$, $u_j\cdot \nu_L \leq u_{j+1}\cdot \nu_L$ for $j\in \{1, \ldots, k-1\},$ such that $$V \res {\mathcal N}_{\rho}(Y) 
= (\exp_{Y})_{\#} \left(\sum_{j=1}^{k} |\text{\rm graph}(u_j) \cap B_{\rho}^{n+1}(0)|\right).$$
\end{Dfi} 

\begin{Dfi}[{\sc generalised regular set} $\Greg \, V$]\label{greg}  The generalised regular set of $V$, denoted $\Greg \, V$, is the set of quasi-embedded points of $V$. 
\end{Dfi} 

\begin{Dfi}[{\sc singular set} ${\rm sing} \, V$]\label{sing} The singular set of $V$, denoted ${\rm sing} \, V$, is defined by ${\rm sing} \, V = {\rm spt} \, \|V\| \setminus \Greg \, V.$ 
\end{Dfi} 

\begin{Dfi}[{\sc regular set} ${\rm reg} \, V$] \label{reg} The regular set of $V$, denoted ${\rm reg} \, V,$ is the set of $C^{2}$ embedded points of ${\rm spt} \, \|V\|$. More precisely, $Y \in {\rm reg} \, V$ if $Y \in {\rm spt} \, \|V\|$ and there exist $\rho >0$ such that ${\rm spt} \, \|V\| \cap {\mathcal N}_{\rho}(Y)$ is an embedded, connected $C^{2}$ submanifold $M$ of ${\mathcal N}_{\rho}(Y).$ 
\end{Dfi}

\begin{thm}\cite[Theorem 5.1]{BW2}
\label{nonvar-SS}
Let $\lambda,  \overline{\lambda} \in [0,\infty)$, $p>n$, $\Lambda \in [0, \infty)$, $q \in \N$ and $\rho_{0} \in (0, \infty)$. There exist a constant $K= K(n, \lambda, \overline{\lambda}, \Lambda, N, \rho_{0})  \in (0, \infty)$ depending only on 
$n$, $\lambda$, $\overline{\lambda}$, $\Lambda,$ $N$ and $\rho_{0}$, and a constant $\e = \e(n, \lambda, \overline{\lambda}, \Lambda, p, q, n, N, \rho_{0}) \in (0, K\rho_{0}/2)$ depending only on $n$, $\lambda$, $\overline{\lambda}$, $\Lambda,$ $p$, $q$, $n,$ $N$ and $\rho_{0}$, such that the following holds. Suppose that $X_{0} \in N$, ${\rm inj}_{X_{0}} \, N \geq \rho_{0},$ $V$ is an integral $n$-varifold on ${\mathcal N}_{\rho_0}(X_{0})$ and that:  
\begin{itemize}
\item[(a)]  $V$ has  first variation $\delta \, V  = -H_{V} \|V\|$ in ${\mathcal N}_{\rho_{0}}(X_{0})$, i.e.\
$$\delta V (\psi) = - \int H_{V} \cdot \psi d\|V\|$$  
for every vector field $\psi \in C^1_c({\mathcal N}_{\rho_0}(X_{0}); T N),$ where the generalised mean curvature $H_{V} \in L^{p}(\|V\|)$ with 
$\left(\rho_{0}^{p-n}\int_{{\mathcal N}_{\rho_{0}}(X_{0})}|H_{V}|^{p} \, d\|V\|\right)^{1/p} \leq \Lambda$; 
\end{itemize}
additionally, suppose that $X \in {\mathcal N}_{\rho_{0}/2}(X_{0})$, $\rho \in (0, K^{-1}\e)$ and that:
\begin{itemize}
\item[(b)] ${\rm dim}_{\mathcal H} ({\rm sing} \, V \res {\mathcal N}_{\rho}(X)) \leq n-7$ if $n \geq 7$ and ${\rm sing} \, V \res {\mathcal N}_{\rho}(X) = \emptyset$ if $n \leq 6;$

\item[(c)] (Schoen--Tonegawa inequality) for every $y\in\spt{V} \cap {\mathcal N}_{\rho}(X)$ and every smooth, embedded $n$-dimensional disk $D \subset N$ containing $y,$ the following holds: if for some $\delta \in\left(0, \frac{1}{10}\text{dist}(y,\p {\mathcal N}_{\rho}(X))\right)$ and $r=\text{dist}(y,\p {\mathcal N}_{\rho}(X))-\delta,$ 
$\varphi_{y}: B^{n}_r(0) \times (-\delta, \delta) \to {\mathcal N}_{\rho}(X)$ is a coordinate chart induced by a choice of Fermi coordinates around $D$ with $\varphi_{y}(0)=y$ and with the pull back metric on $B^{n}_r \times (-\delta, \delta)$ coinciding at $0$, to first order, with the Euclidean metric, then
$$\int |A|^2 \zeta^2 d\|\widetilde{V}\| \leq \lambda  \int  (1-(\tilde{\nu} \cdot e_{n+1})^2)|\nabla \zeta|^2 d\|\widetilde{V}\| + \overline{\lambda} \int \zeta^2 d\|\widetilde{V}\|$$
for every $\zeta \in C^1_c\left((B^{n}_r(0) \times (-\delta, \delta))\setminus \varphi_{y}^{-1}({\rm sing} \, V)\right);$ here $\widetilde{V}=(\varphi_{y}^{-1})_{\#} V$, $\nabla$ is the gradient in $B^{n}_{r}(0) \times (-\delta, \delta)$, $A$ is the second fundamental form of $\widetilde{V}$ and $\tilde{\nu}$ is the unit normal to $\spt{\widetilde{V}}$ (all quantities with respect to the pull-back metric), and note that it suffices for $\tilde{\nu}$ to be defined up to sign and for all said quantities to be defined in the complement of $\varphi_{y}^{-1}({\rm sing} \, V)$;
\end{itemize}

\noindent
finally, suppose that with $\varphi_{X}$ equal to a chart as in (c) corresponding to $y=X$, and with $\widetilde{V}=(\varphi_{X}^{-1})_{\#} \, V$:
$$(\omega_{n})^{-1}\|\eta_{0, \rho \, \#} \, \widetilde{V}\|(B_{1}^{n+1}(0)) \leq q+ 1/2;$$ 
$$q-1/2 \leq \left(\frac{\omega_{n}}{2^{n}}\right)^{-1}\|\eta_{0, \rho \, \#} \, \widetilde{V}\|\left((B_{1/2}^{n}(0) \times {\mathbb R}) \cap 
B_{1}^{n+1}(0)\right) \leq q + 1/2$$ 
and 
\begin{eqnarray*}
E \equiv \int_{(B_{1/2}^{n}(0) \times {\mathbb R}) \cap B_{1}^{n+1}(0)} |x^{n+1}|^{2} \, d\|\eta_{0, \rho \, \#} \, \widetilde{V}\| && \\
&&\hspace{-2.5in} + \rho^{2}\left(\int_{(B_{1/2}^{n}(0) \times {\mathbb R}) \cap B_{1}^{n+1}(0)} |H_{V}\circ \varphi_{X_{0}}(Y)|^{p} \, d\|\eta_{0, \rho \, \#} \, \widetilde{V}\|(Y)\right)^{2/p} < \e.
\end{eqnarray*} 
 Then 
$$\eta_{0, \rho \, \#} \widetilde{V} \res ((B_{1/4}^{n}(0) \times {\mathbb R}) \cap B_{1}^{n+1}(0))  = \sum_{j=1}^{q} |{\rm graph} \, u_{j}|$$ 
for some $u_{j} \in C^{2}(B_{1/4}^{n}(0))$, $j=1, 2, \ldots, q,$ with $u_{1} \leq u_{2} \leq \ldots \leq u_{q};$ 
moreover,  
$$\|u_{j}\|_{C^{1, \alpha}(B_{1/4}(0))} \leq C \sqrt{E}$$
for each $j\in \{1, 2, \ldots, q\},$ where $\alpha = \alpha(n, p) \in (0, 1)$ and $C = C(p, q, N) \in (0, \infty).$ 
\end{thm}

\begin{oss}[\textit{\cite[Theorem 5.1]{BW2} $\Rightarrow$ Theorem \ref{nonvar-SS}}]
\label{oss:5.1}
Assumption (a) in Theorem \ref{nonvar-SS} implies condition (A)  of \cite[Theorem 5.1]{BW2}, as explained in \cite[Section 5.2]{BW2}. Assumption (c) in Theorem \ref{nonvar-SS} implies the condition (C) of \cite[Theorem 5.1]{BW2} by an elementary computation (see the next remark). The rest of the assumptions of Theorem~\ref{nonvar-SS} are the remaining assumptions of   \cite[Theorem 5.1]{BW2}.
\end{oss}

\begin{oss}
One needs to check that the inequality in (c), which is expressed in terms of Riemannian quantities, implies the validity of the same inequality, with an additional multiplicative constant on the right-hand-side that depends only on the metric, and with $A$, $\tilde{\nu}$, $\nabla$ and $\cdot$ replaced by their Euclidean counterparts. Indeed,   by virtue of the fact that the coordinates are Fermi, we have that everywhere in $B^n_r(0)\times (-\delta, \delta),$
the vector $\p_{n+1}$ is unit both with respect to both the Euclidean metric and the Riemannian metric, and moreover, $\p_{n+1}$ is the metric gradient, with respect to either metric, of the (coordinate) function $x^{n+1}$. Denote the metric gradients associated with the Euclidean and Riemannian metrics respectively by $\nabla^{\text{Eucl}}$ and $\nabla$, and let $
M={\rm reg}\,V.$ For any smooth function $f$, define $\nabla^{\text{Eucl},M} f$ and $\nabla^M f$ at a point on $X \in M$ as orthogonal projections of $\nabla^{\text{Eucl}} f$ and $\nabla f$ onto the tangent space to $T_{X} \, M$, where orthogonality is with respect to the corresponding metric (i.e.\ the Euclidean metric and the Riemannian metric respectively). Then we have (i) $1- (\nu \,\underline{\cdot} \, e_{n+1})^2= |\nabla^{\text{Eucl},M} x^{n+1}|^2$, where $\nu$ is the Euclidean unit normal to $M$, $\underline{\cdot}$ is the Euclidean scalar product and the norm on the right-hand-side is Euclidean, and (ii)  $1- (\tilde{\nu} \cdot e_{n+1})^2)^2= |\nabla^{M} x^{n+1}|^2$, where $\tilde{\nu}$ is the Riemannian unit normal to $M$, $\cdot$ is the Riemannian scalar product and the norm on the right-hand-side is Riemannian. On the other hand, $\frac{1}{\tilde{C}} |\nabla^{M} x^{n+1}|^2\leq |{\nabla}^{\text{Eucl}, M} x^{n+1}|^2\leq \tilde{C}|\nabla^{M} x^{n+1}|^2$ for a fixed positive constant $\tilde{C}$ depending only on the Riemannian metric. To this end we observe that ${\nabla}^{\text{Eucl}, M} f$ agrees with the gradient of $\left.f\right|_M$ with respect to the Riemannian metric on $M$ induced by the ambient Euclidean metric. Similarly, $\nabla^M f$ is the gradient of $\left.f\right|_M$ with respect to the metric on $M$ induced by the ambient Riemannian metric. In other words, $|\nabla^{M} x^{n+1}|^2$ and $|\nabla^{\text{Eucl}, M} x^{n+1}|^2$ can be computed pointwise on $M$ in intrinsic fashion (from the two metrics on $M$). The inequalities then follow, and equivalently we have $\frac{1}{\tilde{C}} \left(1- (\nu \,\underline{\cdot} \, e_{n+1})^2\right)\leq 1- (\tilde{\nu} \cdot e_{n+1})^2 \leq \tilde{C} \left(1- (\nu \,\underline{\cdot} \, e_{n+1})^2\right)$.
To complete the verification of the Schoen--Tonegawa inequality with Euclidean quantities, it suffices to observe the pointwise inequalities $|A|^2 \geq {\widetilde C}^{-1}|A_{{\rm Eucl}}|^{2}  - \widetilde{\gamma}$ and $|\nabla\phi|^2 \leq \widetilde{C} |\nabla^{{\rm Eucl}} \, \phi |^{2}$ where 
${\rm Eucl}$ denotes the Euclidean quantities, and $\widetilde{C}>1$ and $\widetilde{\gamma} >0$ are fixed constants that depend only on the Riemannian metric. ($\widetilde{C}$ can be made arbitrarily close to $1$ and $\widetilde{\gamma}$ arbitrarily close to zero by making the geodesic ball ${\mathcal N}_{\rho_{0}}(X_{0})$ in Theorem \ref{nonvar-SS} sufficiently small).


\end{oss}


\subsection{Further GMT preliminaries}
For an $m$-dimensional stationary integral cone ${\mathbf C}$ on ${\mathbb R}^{n+1}$, let $S({\mathbf C}) = \{Y  \in {\mathbb R}^{n+1}\, : \, \Theta \, (\|{\mathbf C}\|, Y) = \Theta \, (\|{\mathbf C}\|, 0)\}$. It is well-known that $\Theta \, (\|{\mathbf C}\|, Y) \leq \Theta \, (\|{\mathbf C}\|, 0)$ for every $Y \in {\mathbb R}^{n+1},$ 
$S({\mathbf C})$ is a vector subspace of ${\mathbb R}^{n+1}$ of some dimension $\leq m$ and that ${\mathbf C}$ is invariant under translation by any element in $S({\mathbf C}).$  
 
 We shall need the case $m=n$ of the following result:
 
\begin{lem}\label{stratification}{\rm [Generalised stratification of singular sets]} 
Let $V$ be an integral $m$-varifold on $N$ with generalised mean curvature $H_{V} \in L^{p}_{\rm loc}(\|V\|)$ for some $p >m$ and first variation $\delta \, V = -H_{V}\|V\|$ in $N$. 
For $k \in \{0, 1, 2, \ldots, m\}$, let ${\mathcal S}_{k} = \{Y  \in {\rm spt} \, \|V\|\, : \, {\rm dim} \, S({\mathbf C}) \leq k \;\; \mbox{for every tangent cone ${\mathbf C}$ to $V$ at $Y$}\}.$ 
Then ${\rm dim}_{\mathcal H} \, ({\mathcal S}_{k}) \leq k.$
\end{lem}

This result is due to Almgren; see \cite[Remark~2.28]{Alm}, where it is proved assuming $p = \infty$ and $N$ is a Euclidean space. In view of \cite[Remark~4.4]{A} and the validity of the 
monotonicity formula when $p >m$ (see \cite[Theorem~8.5]{A} or \cite[Section~4]{Sim}), the same argument establishes the above $L^{p}$ version (although for our purposes here the case $p=\infty$ suffices). See \cite[Chapter 4]{S} for a nice, concise exposition of the argument in the context of energy minimizing maps.   

We shall employ Lemma~\ref{stratification} together with the following result to deduce (in Theorem~\ref{limit-regularity}) the dimension bound on the singular set of limit $(g, 0)$-varifolds arising from index bounded Allen--Cahn solutions.

\begin{thm}\label{classification}{\rm [Classification of stable, stationary hypercones]}
Let ${\mathbf C}$ be a stationary integral hypercone on ${\mathbb R}^{n+1}$ such that ${\mathbf C}$ has no classical singularities and 
${\rm reg} \, {\mathbf C}$ (i.e.\ the embedded part of ${\rm spt} \, \|{\mathbf C}\|$) is stable in the sense that the stability inequality 
$\int |{\mathbf B}_{\mathbf C}|^{2} \varphi^{2} \, d{\mathcal H}^{n} \leq \int |\nabla^{{\rm reg} \, \mathbf C} \, \varphi|^{2} \, d{\mathcal H}^{n}$ holds for all 
$\varphi \in C^{1}_{c}({\rm reg} \, {\mathbf C}),$ where ${\mathbf B}_{\mathbf C}$ is the second fundamental form of ${\rm reg} \, {\mathbf C}$ and $\nabla^{{\rm reg}\, {\mathbf C}}$ is the gradient operator on ${\rm reg} \, {\mathbf C}.$ We then have the following:
\begin{itemize}
\item [{\rm (i)}] if $1 \leq n \leq 6$, then ${\mathbf C} = q|P|$ for some positive integer $q$ and a hyperplane $P$ of ${\mathbb R}^{n+1}$;
\item[{\rm (ii)}] if $n \geq 7$, then either ${\mathbf C} = q|P|$  for some positive integer $q$ and a hyperplane $P$ of ${\mathbb R}^{n+1},$ or ${\rm dim} \, S({\mathbf C}) \leq n-7;$ 
\item[{\rm (iii)}] if $n \geq 7$ then ${\rm dim}_{\mathcal H} \, ({\rm sing} \,{\mathbf C}) \leq n-7$.  
\end{itemize}  
\end{thm}
\begin{proof} The case $n=1$ is trivial, so assume $n \geq 2$. In case ${\rm sing} \, {\mathbf C} \subset \{0\},$ part (i) is a well-known theorem of Simons (\cite{SJ}). 
The general case follows from this and the results of \cite{Wic} as follows:  let $k$ be the smallest integer $\geq2$ for which there is a $k$-dimensional hypercone $\widetilde{\mathbf C}$ such that  the hypotheses of the theorem are satisfied with ${\mathbf C} = {\widetilde{\mathbf C}}$ but the conclusion in part (i) fails, i.e.\ $\widetilde{\mathbf C}$ is not supported on a hyperplane. If $Y \in {\rm spt} \, \|\widetilde{\mathbf C}\| \setminus \{0\}$ is arbitrary, then by \cite[Theorem~3.3]{Wic} and \cite[Theorem~3.4]{Wic}, any tangent cone ${\mathbf C}_{1}$ to $\widetilde{\mathbf C}$ at $Y$ has stable regular part and no classical singularities. Also, being invariant under translation along $\{tY \, : \, t \in {\mathbb R}\}$, ${\mathbf C}_{1}$ has the form  ${\mathbf C}_{1}^{(0)} \times {\mathbb R}$ 
after a rotation. Thus the ``cross-section'' ${\mathbf C}_{1}^{(0)}$ is a hypercone in ${\mathbb R}^{k}$ with stable regular part and no classical  singularities, and so by the definition of $k,$ the cone 
${\mathbf C}_{1}^{(0)}$ is supported on a $k-1$ dimensional plane and consequently ${\mathbf C}_{1}$ is supported on a $k$ dimensional plane. Hence by \cite[Theorem~3.3]{Wic} $Y$ is a regular point of $\widetilde{\mathbf C}.$ Thus ${\rm sing} \, \widetilde{\mathbf C} \subset \{0\}$ and hence by Simons' theorem, we must have that $k \geq 7$.  Part (ii) follows directly from part (i) by considering the cross-section of ${\mathbf C}$. Since (by \cite[Theorem~3.3]{Wic} and \cite[Theorem~3.4]{Wic}) every tangent cone to ${\mathbf C}$ has stable regular part and no classical singularities, part (iii) follows from part (ii), \cite[Theorem~3.3]{Wic} and Lemma~\ref{stratification}.
\end{proof}

\section{Regularity of limit $(g,0)$-varifolds}\label{stable-regularity} 
In this section we use the $C^{1, \alpha}$ varifold regularity theory described in the preceding section (namely, Theorem~\ref{nonvar-SS} and Theorem~\ref{BWregularity}), together with an adaptation of a PDE argument from \cite{BW1} (described fully in Section~\ref{higher-reg} below), and Lemma~\ref{stratification} and Theorem~\ref{classification},  to completely analyse the limit $(g,0)$-varifolds arising from Morse index bounded sequences of Allen--Cahn solutions. Results in this section may be of interest in contexts other than that of Theorem~\ref{thm:existence}. 

\subsection{Main theorems: smooth excision of hidden boundary}\label{smooth-excision}
Our main regularity theorem concerning limit varifolds (Theorem~\ref{thm:regularity} below) says that if $g \in C^{1,1}(N)$ is positive, then a limit $(g, 0)$-varifold $V$ for which the associated sequence $(u_{\e_{j}})$ is such that $u_{\e_{j}}$ has uniformly bounded Morse index with respect to ${\mathcal F}_{\e_{j}, \sigma g_{j}}$ is regular (i.e.\ has \emph{quasi-embedded $PMC \, (g, 0)$ structure} as in Definition~\ref{PMC-structure} below) away from a closed set of codimension $\geq 7$, and moreover, that the hidden boundary $V \res \left(N \setminus \overline{E}\right)$ (on which $H_{V} = 0$) and the phase boundary  $|\partial \, E|$ (where $H_{V} = g\nu$, with $\nu$ denoting the inward pointing unit normal to $\partial^{\star} E$) can be separated smoothly globally, i.e.\ neither the hidden boundary nor the phase boundary has singular first variation in $N$ (and each is separately regular away from a closed set of codimension $\geq 7$).

\noindent
\begin{Dfi}[{\sc quasi-embedded $PMC\, (g, 0)$ structure}]  \label{PMC-structure}
Let $g : N \to {\mathbb R}$ be a positive continuous function and let $V$ be an integral $n$-varifold on $N.$ We say that $V$ has \emph{quasi-embedded PMC$(g, 0)$ structure at $Y \in {\rm spt}  \, \|V\|$} if one of the following (depicted in figure~\ref{fig:PMC_structures}) holds:   
\begin{itemize}
\item[{\rm (i)}] $V$ in a neighborhood of $Y$ is equal to the multiplicity 1 varifold $|D|$ for some $C^{2}$ embedded disk $D$ with a choice of continuous unit normal with respect to which the scalar mean curvature of $D$ is equal to $g$ everywhere;
\item[{\rm (ii)}] $V$ in a neighborhood of $Y$ is equal to $q|D^{\prime}|$ for some even integer $q$ and a $C^{2}$ embedded minimal disk $D^{\prime}$;
\item[{\rm (iii)}] $\Theta \, (\|V\|, Y) = 2$ and $V$ in a neighborhood of $Y$ is equal to $|D_{1}| + |D_{2}|$ for two distinct $C^{2}$ embedded disks $D_{1}$, $D_{2}$ having only tangential intersection, with each disk lying on one side of the other, having mean curvature vector pointing away from the other and having scalar mean curvature (with respect to the unit normal in the direction of the mean curvature vector) equal to $g$ everywhere;
\item[{\rm (iv)}] $\Theta \, (\|V\|, Y) = q$ and $V$ in a neighborhood of $Y$ is equal to $|D| + (q-1)|D^{\prime}|$ for some odd integer $q \geq 3$ and two distinct $C^{2}$ embedded disks $D$, $D^{\prime}$ having only tangential intersection, with each disk lying on one side of the other, $D$ having mean curvature vector pointing away from $D^{\prime}$ and scalar mean curvature (with respect to the unit normal in the direction of the the mean curvature vector) equal to $g$ everywhere, and with $D^{\prime}$ being minimal;
\item[{\rm (v)}] $\Theta \, (\|V\|, Y) = q$ and $V$ in a neighborhood of $Y$ is equal to $|D_{1}| + |D_{2}| + (q-2)|D^{\prime}|$ for some even integer $q \geq 4$ and three distinct $C^{2}$ embedded disks $D_{1}$, $D_{2}$, $D^{\prime}$ with each pair of disks having only tangential intersection, and where $D_{1}$, $D_{2}$ are precisely as in {\rm (iii)} and $D^{\prime}$ is minimal and lies between $D_{1}$ and $D_{2}$.
\end{itemize} 
\end{Dfi} 

\begin{Dfi}[{\sc quasi-embedded hypersurface}]\label{quasi-embedded-hyp}
The image $M = \iota(S)$ of a $C^{2}$ immersion $\iota \, : \, S \to N$ of an $n$ dimensional $C^{2}$ manifold is said to be a {\rm quasi-embedded hypersurface} of $N$ if every point $Y \in {\rm spt} \, \|\iota_{\#} \, |S|\|$ is a quasi-embedded point of the varifold $\iota_{\#} \, |S|$ (see Definition~\ref{quasi-embedded}) and we may choose, in the notation of Definition~\ref{quasi-embedded}, $\Oc$ and $u_{j}$ for $j \in \{1, \ldots, k\}$ so that ${\rm graph} \, u_{j} = \iota(D_{j})$ for some $n$-dimensional open disk $D_{j} \subset S.$ 
\end{Dfi}

\begin{figure}[h]
\centering
\includegraphics[scale=0.4]{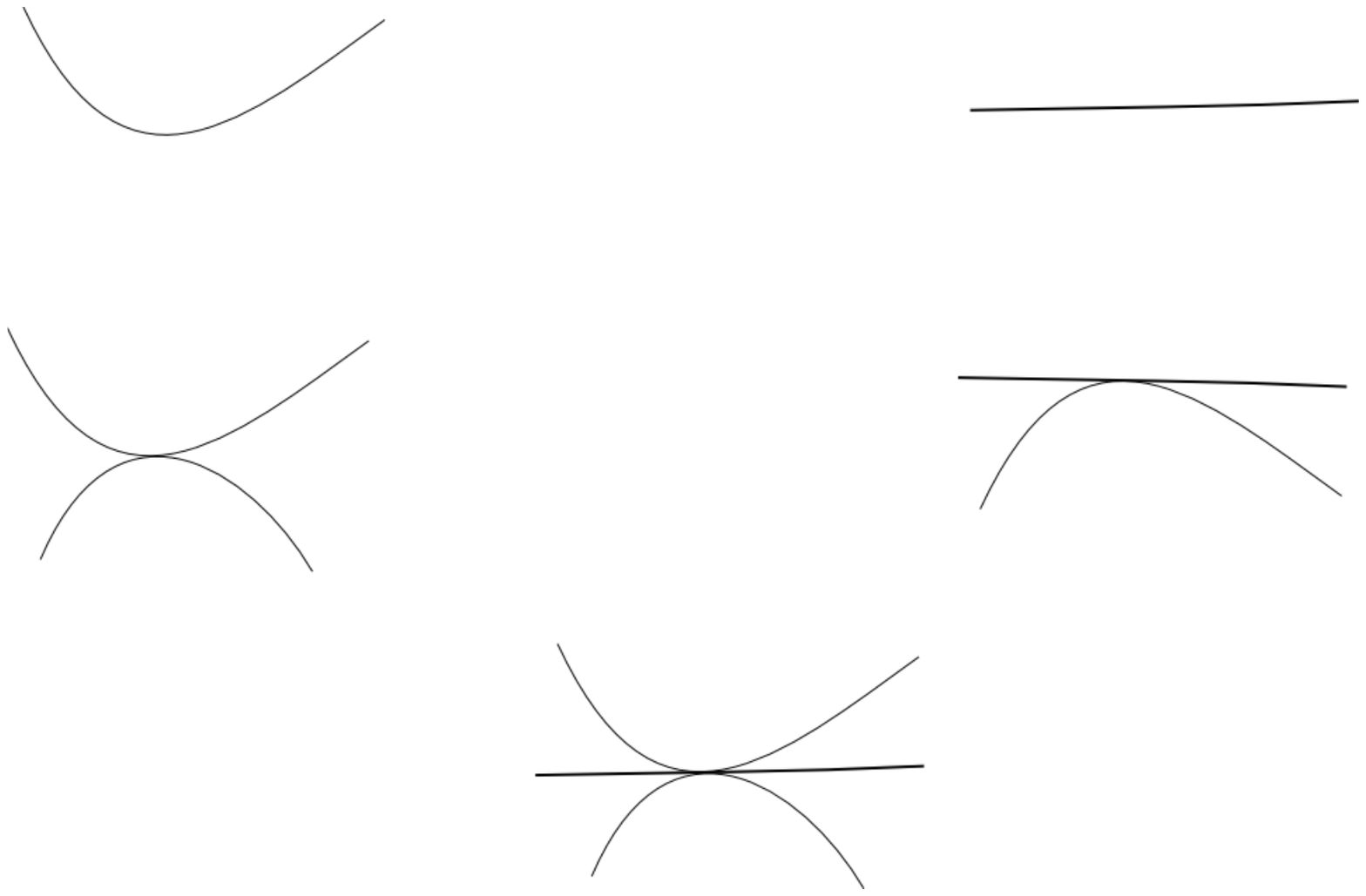}
\put (-275, 190){\tiny{$-1$}}
\put (-275, 208){\tiny{$+1$}}
\put (-60, 195){\tiny{$-1$}}
\put (-60, 211){\tiny{$-1$}}
\put (-283, 140){\tiny{$+1$}}
\put (-313, 125){\tiny{$-1$}}
\put (-258, 125){\tiny{$-1$}}
\put (-283, 110){\tiny{$+1$}}
\put (-60, 150){\tiny{$-1$}}
\put (-95, 135){\tiny{$-1$}}
\put (-60, 130){\tiny{$+1$}}
\put (-30, 135){\tiny{$-1$}}
\put (-155.5, 63){\tiny{$+1$}}
\put (-155.5, 40){\tiny{$+1$}}
\put (-129, 59){\tiny{$-1$}}
\put (-129, 49){\tiny{$-1$}}
\put (-189, 58){\tiny{$-1$}}
\put (-189, 48){\tiny{$-1$}}
\put (-280, 175){\tiny{(i)}}
\put(-65,175){\tiny{(ii)}}
\put(-280,90){\tiny{(iii)}}
\put(-65,90){\tiny{(iv)}}
\put(-160,10){\tiny{(v)}}
\caption{Quasi-embedded PMC$(g, 0)$ structure of a varifold $V$. In each case, $V$ corresponds to a certain number of disks, depicted by curves. The disks depicted by thin curves have multiplicity $1$ and mean curvature $g\nu$ for a choice of unit normal $\nu$, and those depicted by thick curves are minimal and have even multiplicity. The numbers $+1$ and $-1$ indicate the phase values, i.e.\ the values of $u_{\infty},$ when $V$ is the limit varifold arising from a sequence of Allen--Cahn critical points $u_{\e_{j}}$ with $g >0.$ In this case $\nu$ always points into the $+1$ phase.}
\label{fig:PMC_structures}
\end{figure}



\begin{thm}
\label{thm:regularity}
Let $g \in C^{1,1}(N)$ be positive, and let $V \neq 0$ be a limit $(g, 0)$-varifold on $N$ with associated sequences $\e_{j} \to 0^{+}$, $g_{j} \in C^{1, 1}(N)$ and $u_{\e_{j}} \in W^{1, 2}_{\rm loc}(N)$ where $g_{j} \to g$ locally in $C^{1, 1}$ and $u_{\e_{j}}$ is a critical point of ${\mathcal F}_{\e_{j}, \sigma g_{j}}$ for each $j.$  
Suppose further that the sequence $(u_{\e_{j}})$ converges to $u_\infty \in BV_{\rm loc}(N; \{-1, 1\})$ locally in $L^{1}$, noting that such $u_{\infty}$ exists possibly after passing to a 
subsequence of $(u_{\e_{j}})$. Let $E = \{x \in N \, : \, u_{\infty}(x) = 1\}$ and note, by Theorem~\ref{Roger-Tonegawa}, that $E$ is a Caccioppoli set with $E \neq N$. Finally  suppose that  that the Morse index of $u_{\e_{j}}$ with respect to ${\mathcal F}_{\e_{j}, \sigma g_{j}}$ is bounded from above independently of $j.$ Then we have $\sigma^{-1}V = V_{0} + V_{g}$, where: 
\begin{itemize}
\item[{\rm (i)}]  $V_{0}$ is a (possibly zero) stationary integral $n$-varifold on $N$ with ${\rm sing} \, V_{0} = \emptyset$ if $2 \leq n \leq 6$, ${\rm sing} \, V_{0}$ discrete if $n=7$ and  ${\rm dim}_{\mathcal H} \, \left({\rm sing} \, V_{0}\right) \leq n-7$ if $n \geq 8;$ moreover, $V_{0}$ has locally constant even multiplicity on ${\rm reg} \, V_{0} (=\Greg \, V_{0})$ and 
${\rm spt} \, \|V_{0}\| \subset N \setminus {\rm int} \, (E).$ 
\item[{\rm (ii)}] if $E = \emptyset$, then $V_{g} = 0;$ if $E \neq  \emptyset$  then $V_{g} = |\partial^{\star} \, E| \neq 0$ and we have that ${\rm sing} \, V_{g} = \emptyset$ if $2 \leq n\leq 6$, 
${\rm sing} \, V_{g}$ is discrete if $n=7$ and ${\rm dim}_{\mathcal H} \, ({\rm sing} \, V_{g}) \leq n-7$ if $n \geq 8;$ moreover, the (classical) mean-curvature of the immersed hypersurface ${\Greg} \, V_{g}$ is given by $H_{V_{g}} = g\nu$ where $\nu$ is the unit normal vector to ${\Greg} \, V_{g}$ pointing inward (i.e.\ towards $E$);  $H_{V_{g}}$ is also the generalized mean curvature of $V_{g}$ in $N$, and the first variation of $V_{g}$ is given by $\delta \, V_{g} = -H_{V_{g}} \|V_{g}\|$ on $N$. 
\end{itemize}
Additionally, $\sigma^{-1}V$ has quasi-embedded $PMC (g, 0)$ structure near each point 
$y \in \Greg \, V$ (see Figure~\ref{fig:PMC_structures} below), and $\Greg \, V_{g}$ is a quasi-embedded hypersurface of $N.$  
\end{thm}

For the proof of Theorem~\ref{thm:regularity},  it will be convenient to first establish Theorem~\ref{limit-regularity} below which collects several key facts concerning varifolds $V$ satisfying the hypotheses of Theorem~\ref{thm:regularity}. Theorem~\ref{thm:regularity} will be a direct consequence of Theorem~\ref{limit-regularity} and Theorem~\ref{Roger-Tonegawa} above.

\begin{thm}\label{limit-regularity}
 Let $g \in C^{1, 1}(N)
 $, $g >0$, and let $V$ be a limit $(g, 0)$-varifold on $N$ with associated sequences $\e_{j} \to 0^{+}$, $g_{j} \in C^{1, 1}(N)$ and $u_{\e_{j}} \in W^{1, 2}_{\rm loc}(N)$ where $g_{j} \to g$ locally in $C^{1, 1}$ and $u_{\e_{j}}$ is a critical point of ${\mathcal F}_{\e_{j}, \sigma g_{j}}$ for each $j.$  Suppose that  the Morse index of $u_{\e_{j}}$ with respect to ${\mathcal F}_{\e_{j}, \sigma g_{j}}$ is bounded from above independently of $j.$ Noting by Theorem~\ref{Roger-Tonegawa} that $\sigma^{-1}V$ is integral and that $V$ admits locally bounded generalized mean curvature, we have the following:   

\begin{itemize}
\item[{\rm (i)}] $V$ has no classical singularities; 
\item[{\rm (ii)}] if ${\mathbf C}$ is a tangent cone to $\sigma^{-1}V$, then ${\mathbf C}$ is a stationary integral hypercone with no classical singularities  and stable regular part;
\item[{\rm (iii)}] if $Y \in {\rm spt} \, \|V\|$ and if one tangent cone to $\sigma^{-1}V$ at $Y$ is supported on a hyperplane, then $\sigma^{-1}V$ has quasi-embedded PMC\,$(g, 0)$ structure at $Y$;
\item[{\rm (iv)}] ${\rm sing} \, V = \emptyset$ if $2 \leq n \leq 6$, ${\rm sing} \, V$ is discrete if $n=7$ and ${\rm dim}_{\mathcal H} \, ({\rm sing} \, V) \leq n-7$ if $n \geq 8$.
\end{itemize}
\end{thm}

When $V$ is a \emph{stable} limit $(g, 0)$-varifold, we have quantitative versions of parts (i) and (iii) of the preceding theorem, given by the uniform estimates in Theorem~\ref{estimates} below. These estimates will only be needed for the approximation argument we use in Section~\ref{extension} to extend Theorem~\ref{thm:existence} from the case of positive $g \in C^{1,1}$ to the case of non-negative Lipschitz $g$. 

\begin{thm} \label{estimates}
Let $q$ be a positive integer, $\Gamma >0$ and $\overline{\rho} \in (0, \infty)$. Let ${\mathbf C}$ be a stationary integral cone in ${\mathbb R}^{n+1}$ supported on a union of three or more $n$-dimensional half-hyperplanes meeting along  a common $(n-1)$-dimensional subspace. There are constants $\eta_{0} = \eta_{0}(n, {\mathbf C}, N, \Gamma, \overline{\rho}) \in (0, 1)$, $\e_{0} = \e_{0}(n, q, N, \Gamma, \overline{\rho}) \in (0, 1)$ and $\mu = \mu(n,N, \Gamma, \overline{\rho})$ such that whenever $g \in C^{1,1}(N)$ is a positive function with $\sup_{N} \, \left(|g| + |\nabla g| \right)\leq \Gamma$, $V$ is a stable limit $(g, 0)$-varifold on $N$, $y \in N$ with ${\rm inj}_{y} \, N \geq \overline{\rho}$ and 
$\widetilde{V} = \left(Q \circ \exp_{y}^{-1}\right)_{\#} \, \sigma^{-1} V \res {\mathcal N}_{\rho}(y)$ for some $\rho \in (0, {\rm inj}_{y} \, N)$ and some orthogonal rotation $Q \, : \, T_{y} \, N \approx {\mathbb R}^{n+1} \to T_{y} \, N$ (so that $\widetilde{V}$ is a varifold on 
$B_{\rho}^{n+1}(0) \subset {\mathbb R}^{n+1}$), we have the following conclusions:    
\begin{itemize}
\item[{\rm (i)}]  if $\frac{\|\widetilde{V}\|(B_{\rho}^{n+1}(0))}{\omega_{n}\rho^{n}} \leq \Theta \, (\|{\mathbf C}\|, 0) + 1/4$ then $$\mu\rho + \rho^{-1} {\rm dist}_{\mathcal H} \, ({\rm spt} \, \|{\mathbf C}\| \cap B_{\rho}^{n+1}(0), {\rm spt} \, \|\widetilde{V}\| \cap B_{\rho}^{n+1}(0))   \geq \eta_{0};$$  
\item[{\rm (ii)}] if $\frac{\|\widetilde{V}\|(B_{\rho}^{n+1}(0))}{\omega_{n}\rho^{n}} \leq q+ 1/2,$  
$q-1/2 \leq \frac{\|\widetilde{V}\|\left((B_{\rho/2}^{n}(0) \times {\mathbb R}) \cap B_{\rho}^{n+1}(0)\right)}{\omega_{n} (\frac{\rho}{2})^{n}} \leq q+1/2$ and  
$$E_{\rho} \equiv \mu\rho + \rho^{-n-2} \int_{(B_{\rho/2}^{n}(0) \times {\mathbb R}) \cap B_{\rho}^{n+1}(0)} |x^{n+1}|^{2} \, d\|\widetilde{V}\| < \e_{0}$$ then 
$$\widetilde{V} \res \left(B_{\rho/4}(0) \times {\mathbb R}\right) \cap B_{\rho}^{n+1}(0) = \sum_{j=1}^{q} |{\rm graph} \, u_{j}|$$ where $u_{j} \in C^{2, \alpha}(B_{\rho/4}^{n}(0))$ for any $\alpha \in (0, 1),$ $u_{1} \leq u_{2} \leq \ldots \leq u_{q}$ and \begin{align*} 
&\sup_{B_{\rho/4}^{n}(0)} \left(\rho^{-1} |u_{j}| + |Du_{j}| + \rho |D^{2}u_{j}|\right)\\
& \hspace{.5in}+ \rho^{1+\alpha} \sup_{x, y \in B_{\rho/4}^{n}(0), \, x \neq y} \frac{|D^{2}u_{j}(x) - D^{2}u_{j}(y)|}{|x - y|^{\alpha}} \leq C \sqrt{E_{\rho}}\\
\end{align*}
for some constant $C = C(n, q, \alpha, \Gamma)$ and each $j \in \{1, 2, \ldots, q\}.$
\end{itemize} 
\end{thm}

We shall give the proofs of Theorem~\ref{limit-regularity} and Theorem~\ref{estimates} in Section~\ref{limit-reg-proof} below after establishing some further necessary preliminary results. We point out next how Theorem~\ref{thm:regularity} follows from Theorem~\ref{limit-regularity}. 

\begin{proof}[Proof of Theorem~\ref{thm:regularity} assuming Theorem~\ref{limit-regularity}] 

By Theorem~\ref{limit-regularity}, part (iv) we have that ${\rm spt} \, \|V\| \setminus \Greg \, V = \emptyset$ if $2 \leq n \leq 6$, ${\rm spt} \, \|V\| \setminus \Greg \, V$ is discrete if $n=7$ and 
$${\rm dim}_{\mathcal H} \, \left({\rm spt} \, \|V\| \setminus \Greg\, V\right) \leq n-7$$ if $n \geq 8$. 
By the definition of $\Greg \, V$ and Theorem~\ref{Roger-Tonegawa} we have that for each $z \in \Greg \, V$, there is a ball $B_{\delta_{z}}(z) \subset N$ such that $\sigma^{-1}V \res B_{\delta_{z}}(z) = V_{0}^{(z)} + V_{g}^{(z)}$ where $V_{0}^{(z)}$ is a possibly zero stationary (i.e.\ zero-mean curvature) $n$-varifold on $B_{\delta_{z}}(z)$ with no singularities and constant even integer multiplicity and $V_{g}^{z} = |\partial ^{\star} \, \left(E \cap B_{\delta_{z}}(z)\right)|$ with $\left({\rm spt} \, \|V_{g}^{z}\| \setminus \Greg\, V_{g}^{(z)}\right) \cap B_{\delta_{z}}(z) = \emptyset$ and mean curvature $H$ of the immersion $\Greg\, V_{g}^{(z)}$ satisfying $H = g\nu$ where $\nu$ is the inward pointing unit normal to $\partial^{\star} \, E.$ 

Now define varifolds $V_{0}$, $V_{g}$ on $N \setminus ({\rm spt} \, \|V\| \setminus \Greg \, V)$ as follows: pick any $\varphi \in C^{0}_{c}\left(\left(N \setminus ({\rm spt} \, \|V\| \setminus \Greg \, V)\right)\times G(n)\right)$. If ${\rm spt} \, \varphi \subset B_{\delta_{z}}(z) \times G(n)$ for some $z \in \Greg \, V$, set $V_{0}(\varphi) = V_{0}^{(z)}(\varphi)$ and $V_{g}(\varphi) = V_{g}^{(z)}(\varphi),$ noting that these definitions are independent of the choice of $z$. For an arbitrary function $\varphi \in C^{0}_{c}\left( \left(N \setminus ({\rm spt} \, \|V\| \setminus \Greg \, V)\right) \times G(n)\right),$ set $V_{0}(\varphi) = \sum_{\alpha \in I} V_{0}(\psi_{\alpha} \varphi)$ and $V_{g}(\varphi) = \sum_{\alpha \in I}V_{g}(\psi_{\alpha} \varphi)$ where $\{\psi_{\alpha}\}_{\alpha \in I}$ is a smooth, locally finite partition of unity subordinate to the collection of open sets 
$$\left\{\{B_{\delta_{z}}(z) \, : \,  z \in \Greg\, V\}, N \setminus \overline{\cup_{z \in {\Greg\, V}}  B_{\delta_{z}}(z)}\right\},$$ where $I$ is some indexing set.  Since ${\rm spt} \, \|V_{0}\| \subset {\rm spt} \, \|V\|$, ${\rm spt} \, \|V_{g}\| \subset {\rm spt} \, \|V\|$ and  $V$ has locally bounded genralized mean curvature in $N$, it follows that $V_{0}$, $V_{g}$ have Euclidean volume growth everywhere. Since ${\mathcal H}^{n-1}({\rm spt} \, \|V\| \setminus \Greg \, V)  = 0$, it then follows that $V_{0}$ is stationary in $N$, and that $V_{g}$ has genralized mean curvature $H_{V_{g}}$ in $N$, given on $\Greg \, V_{g}$ by $H_{V_{g}} = g\nu$. We thus have that $V = V_{0} + V_{g}$ on $N$ with $V_{0}$, $V_{g}$ satisfying all of the requirements of the conclusion of the theorem.
\end{proof}

\begin{oss}
It follows from \cite[Remark~1.22]{BW2} that $\Greg \, V_{g}$ (where $V_{g}$ is as in Theorem~\ref{thm:regularity}) is the image of a $C^{3, \alpha}$ immersion. 
\end{oss}

As mentioned above, a key ingredient of the proofs of Theorem~\ref{limit-regularity} and Theorem~\ref{estimates} is Theorem~\ref{BWregularity}.  Though in its abstract form (given in Section~\ref{nonvar-regularity}) Theorem~\ref{BWregularity}  requires no stability hypothesis, in applying it to Theorem~\ref{limit-regularity} and Theorem~\ref{estimates} stability plays a key role. 
For instance, in the proof of Theorem~\ref{limit-regularity}, we employ an inductive argument (inducting on multiplicity $q$) in which Theorem~\ref{BWregularity} is applied to a limit $(g,0)$-varifold $V$ near a point where one of its tangent cones is a multiplicity $q$ hyperplane; the validity of hypotheses (b) and (c) of Theorem~\ref{BWregularity} in this application is a consequence of the uniform Morse index bound on the critical points $u_{\e_{j}}$ associated with $V$. 

Of fundamental importance in verifying hypothesis (c) in this setting is the existence of generalized second fundamental form of $V$ satisfying, locally near every point, a stability inequality (inequality~(\ref{var-stability}) below) as well as what we shall call the \emph{Schoen--Tonegawa inequality} (Lemma~\ref{Schoen-Tonegawa} below). These inequalities are respectively the analogues of the classical stability inequality and the Schoen inequality (\cite[Lemma~1]{SS}) that play a fundamental role in the regularity theory of embedded stable minimal hypersurfaces (\cite{SS}, \cite{Wic}). An important distinction between this classical setting and the present Allen--Cahn setting is that for stable limit $(g, 0)$-varifolds both inequalities hold for \emph{ambient} test functions, with their \emph{ambient gradient} (i.e.\ gradient on $N$) appearing where the intrinsic gradient (i.e.\ gradient on $V$) appears in the classical counterparts. Hence both inequalities are weaker in the present setting (although of course on the $C^{1}$ embedded part the ambient version is equivalent to the intrinsic one). Moreover, in the Allen--Cahn setting these inequalities are derived independently of each other unlike in the classical setting in which the Schoen inequality is derived from the stability inequality. These inequalities for the Allen--Cahn limit varifolds were first established by Tonegawa (\cite{Ton}), under the assumption that the ambient space is Euclidean and $g=0.$ 

We next discuss adaptation of the arguments of \cite{Ton} to the Riemannian ambient space (and for general $g$). While for the stability inequality this adaptation is standard and has appeared in several places in the literature, the derivation of the Schoen--Tonegawa inequality in the Riemannian setting is more subtle and requires care, and the right choice of coordinates. We provide a complete account of the latter in Lemma~\ref{Schoen-Tonegawa} below.

\subsection{Stability inequality and the Schoen--Tonegawa inequality}
\label{st-ineq}
Let $g \in C^{1, 1}(N)$ with $\sup_{N} \ |\nabla g| \leq \Gamma.$ \emph{In this subsection we assume no sign condition on $g$}. Let $V$ be a stable limit $(g,0)$-varifold on $N$, with associated sequence $\e_{j} \to 0^{+}$, functions $g_{\e_{j}} \in C^{1, 1}(N)$ converging locally in $C^{1, 1}$ to $g$, and associated sequence $(u_{\eps_{j}})$ of critical points of ${\mathcal F}_{\e_{j}, \sigma g_{\e_{j}}}.$ For each $j$, the function $u_{j}$ satisfies  the Euler--Lagrange equation 
\begin{equation}\label{ELeqn}
-\e_{j} \Delta \, u_{\e_{j}} + \e_{j}^{-1}W^{\prime}(u_{\e_{j}}) = g_{\e_{j}}
\end{equation}
weakly on $N.$ By elliptic regularity, $u_{\e_{j}} \in C^{3, \alpha}(N)$ for every $\alpha \in (0, 1)$, and the equation (\ref{ELeqn}) holds classically. Since $u_{\e_{j}}$ is stable, $u_{\e_{j}}$ additionally satisfies the inequality 
\begin{equation}\label{ACstability}
\int_{N} \e_{j} |\nabla \varphi|^{2} + \e_{j}^{-1} W^{\prime\prime}(u_{\e_{j}}) \varphi^{2} \geq 0
\end{equation} 
for all $\varphi \in C^{1}_{c}(N).$ 

We first outline the well-known derivation of the stability inequality on $V$ following \cite{Ton}. (See e.g.\ \cite{Hie} for details in the case of Riemannian ambient space). Replacing $\varphi$ in (\ref{ACstability}) with $|\nabla u_{\e_{j}}| \varphi$ (a valid choice by an approximation argument), and using in the resulting inequality 
the differentiated equation (\ref{ELeqn}) (i.e.\ $-\e_{j}^{2} \nabla \Delta \, u_{\e_{j}} + W^{\prime\prime}(u_{\e_{j}}) \nabla u_{\eps_{j}} = \e_{j} \nabla \, g$) and the Bochner identity 
$$\frac{1}{2} \Delta |\nabla \, u_{\e_{j}}|^{2} = |\nabla^{2}\, u_{\e_{j}}|^{2} + <\nabla \Delta \, u_{\e_{j}}, \nabla \, u_{\e_{j}}> + {\rm Ric}\left(\nabla \, u_{\e_{j}}, \nabla \, u_{\e_{j}}\right),$$ 
we obtain 
\begin{eqnarray}\label{pre-stability}
&&\int_{N} \left({\rm Ric}(\nabla u_{\e_{j}}, \nabla u_{\e_{j}}) + |\nabla^{2}\, u_{\e_{j}}|^{2}  - |\nabla|\nabla u_{\e_{j}}||^{2}\right) \e_{j}\varphi^{2}\nonumber\\ 
&&\hspace{1in}\leq \int_{N} |\nabla \varphi|^{2} \e_{j}|\nabla u_{\e_{j}}|^{2} + \int_{N} < \nabla g_{\e_{j}}, \nabla u_{\e_{j}}> \varphi^{2}.
\end{eqnarray}
Integrating by parts in the last term on the right, this leads to  

\begin{eqnarray}\label{stability}
&&\hspace{-.2in}\int_{N} \left(|\nabla^{2}\, u_{\e_{j}}|^{2}  - |\nabla|\nabla u_{\e_{j}}||^{2}\right) \e_{j}\varphi^{2} \leq c\int_{N} \varphi^{2} \e_{j} |\nabla u_{\e_{j}}|^{2} + \int_{N} |\nabla \varphi|^{2} \e_{j}|\nabla u_{\e_{j}}|^{2}\nonumber\\
&&\hspace{2.5in} -\int_{N} {\rm div} \, \left(\varphi^{2} \nabla g_{\e_{j}}\right) u_{\e_{j}}
\end{eqnarray}
where $c = c (N).$ This implies in the first instance that 
\begin{equation} \label{local-bounds}
\int_{B_{\rho_{0}}(0)} {\mathbf B}_{u_{\e_{j}}}^{2} \varphi^{2}\e_{j}|\nabla u_{\e_{j}}|^{2}  \leq C\sup\left( |\varphi|^{2} + |\nabla \varphi|^{2}\right)
\end{equation}
for each $\varphi \in C^{1}_{c}(N)$, where $C = C(N, g, \|V\|({\rm spt} \, \varphi)).$ Here, for a $C^{3}$ function $u$, the non-negative function ${\mathbf B}_{u}$ is defined by 
$${\mathbf B}_{u}^{2} = |\nabla u|^{-2}\left(|\nabla^{2} \, u|^{2}  - |\nabla|\nabla u||^{2}\right)$$  on $\{|\nabla u| > 0\}$ and ${\mathbf B}_{u} = 0$ on $\{|\nabla u| = 0\}$. It follows from (\ref{local-bounds}) (see also \cite{Hie}) that the limit varifold $V$ admits a (unique) generalized second fundamental form ${\mathbf B}_{V},$ and passing to the limit in (\ref{stability}) we obtain that 
$V$ satisfies the following ``ambient'' stability inequality:

\begin{equation}\label{var-stability}
\int_{N}|{\mathbf B}_{V}|^{2} \varphi^{2} d\|V\| \leq \int_{N} |\nabla \varphi|^{2} d\|V\| + c \int_{N}\varphi^{2} d\|V\|
\end{equation}
for all test functions $\varphi \in C^{1}_{c}(N)$ where $\nabla$ is the ambient gradient (i.e.\ the gradient on $N$), and where $c = c(N, \Gamma).$ 

\begin{oss}
Note that in the preceding argument we have used the hypotheses $g_{\e_{j}} \in C^{1, 1}(N)$ and $g_{\e_{j}} \to g$ in $C^{1, 1}$ to integrate by parts in the last term in inequality (\ref{pre-stability}) and to pass to the limit in inequality (\ref{stability}). These hypotheses will also be used again in a similar way in deriving  inequality (\ref{eq:limit_g_term}) below, a key step in the derivation of the Schoen--Tonegawa inequality (Lemma~\ref{Schoen-Tonegawa}). Note also that in passing to the limit in the last term in (\ref{stability}), we have used the fact that $u_{\e_{j}} \to u_{\infty}$ locally in $L^{1}$, $u_{\infty} = \pm1$ a.e.\ in $N$ and that $\{u_{\infty} = 1\}$ is a Caccioppoli set. 
\end{oss}

\begin{oss} Although we do not need it here, we note that by integrating by parts in the last term of (\ref{pre-stability}) and passing to the limit as $\e_{j} \to 0^{+}$, we obtain the following more precise ambient stability inequality: 
\begin{equation}\label{var-stability-reg}
\int_{N} \left({\rm Ric}(\nu, \nu) + |{\mathbf B}_{V}|^{2}\right) \varphi^{2} d\|V\| \leq \int_{N} |\nabla \varphi|^{2} d\|V\| + 2\int_{N \cap \partial^{\star} \{u_{\infty} = 1\}} \varphi^{2} D_{\nu}g d {\mathcal H}^{n} 
\end{equation}
for all $\varphi \in C^{1}_{c}\left(N \setminus {\rm sing} \, V\right),$ where $\nabla$ is the gradient on $N$, and $\nu$ is the unit normal to ${\rm spt} \, \|V\| \setminus {\rm sing} \, V$ which on $\partial^{\star} \{u_{\infty} = 1\}$ is ``outward pointing'' i.e.\ points away from $\{u_{\infty} = 1\}$.

\end{oss}

We next derive the Schoen--Tonegawa inequality, i.e.\ the Riemannian analogue of \cite[Theorem 4]{Ton}. For an arbitrary point $x\in N$ we consider in the first instance normal coordinates centred at $x$, pulling back the Riemannian metric to a ball $B_{{\rm inj}_{x}(N)}^{n+1}(0)\subset T_x N$ via the exponential map. Fix any unit vector $e \in T_{x} N$ and let $P\subset T_x N$ be the hyperplane $e^{\perp}$. Now we consider a Fermi system of coordinates centred at the disk $P\cap B_{\frac{3}{4}{\rm inj}_{x}(N)}^{n+1}(0)$. We will denote by $\{x_1, \ldots, x_{n+1}\}$ this coordinate system. These coordinates are chosen so that $\{x_1, \ldots, x_n, 0\}$ agree with the normal coordinates restricted to $P\cap B_{\frac{3}{4}{\rm inj}_{x}(N)}^{n+1}(0)$; on the other hand, the curve $s\in (-\sigma, \sigma) \to (x_1, \ldots, x_n, s)$ represents the geodesic orthogonal to $P\cap B_{\frac{3}{4}{\rm inj}_{x}(N)}^{n+1}(0)$ at the point $(x_1, \ldots, x_n, 0)$, and $\sigma>0$ is the semi-width of a tubular neighbourhood of $P\cap B_{\frac{3}{4}{\rm inj}_{x}(N)}^{n+1}(0)$. This coordinate system covers therefore an open subset (a cylinder) in $B_{{\rm inj}_{x}(N)}^{n+1}(0)\subset T_x N$. Note that $\p_{x_{n+1}}$ is a unit vector field (for the Riemannian metric induced from $N$), that is determined (up to sign) at $x$ by the initial choice of $e$. Moreover, the coordinate chart has differential equal to the identity at the point $(0, \ldots, 0)$ (which is mapped to $x$). 

We shall henceforth write $B^{n} = P\cap B_{\frac{3}{4}{\rm inj}_{x}(N)}^{n+1}(0).$ 
The following lemma is the Riemannian analogue of \cite[Proposition 4]{Ton}.

\begin{lem}\label{prelim-ST}
Let $x\in N$ and $e\in T_x N$ be an arbitrarily chosen unit vector and fix the coordinate chart $B^n \times (-\sigma, \sigma)$ described in the previous discussion, with $(0, \ldots, 0)$ mapped to $x$ and with $\p_{x_{n+1}}$ identified with $e$ at $0$. Let $g_{\e} \in C^{1, 1}(N)$ and $u_{\eps}:B^n \times (-\sigma, \sigma) \to \R$ satisfy $\ca{F'}{\eps, g_{\e}}(u_{\eps})=0$ and $\ca{F''}{\eps, g_{\e}}(u_{\eps})\geq 0$. Then we have for any $\phi \in C^\infty_c(B^n \times (-\sigma, \sigma))$
\begin{eqnarray}\label{eq:Tonegawa_type_ineq}
&&\hspace{-.2in}\int \eps \left(|\nabla^2 u_{\eps}|^2 -|\nabla(\nabla u_{\eps} \cdot \p_{n+1})|^2  - \left|\nabla |\nabla u_{\eps} - (\nabla u_{\eps} \cdot \p_{n+1})\p_{n+1} | \right|^2\right) \phi^2\nonumber\\
&&\leq \int \eps  (|\nabla u_{\eps}|^2-(\nabla u_{\eps} \cdot \p_{n+1})^2) |\nabla \phi|^2 + \int \phi^2 \nabla g_{\e} \cdot \left(\nabla u_{\eps} - (\nabla u_{\eps} \cdot \p_{n+1}) \p_{n+1}\right)\nonumber\\
&&\hspace{1in}+ \int \eps \phi^2 \Rc{N}(\nabla u_{\eps}, \nabla u_{\eps}) + \int \phi^2 E,
\end{eqnarray}
where $|E|\leq C_N \eps |\nabla u_{\eps}|^2$. Here $\nabla^2 u_{\e}$ denotes the Hessian (with respect to the Levi--Civita connection), $\nabla u_{\eps}$ and $\nabla g_{\e}$ are metric gradients and the norms and scalar products are taken with respect to the Riemannian metric.
\end{lem}

\begin{oss}
\label{oss:hessian}
Recall that $|\nabla^2 u|^2= \sum_{i=1}^{n+1}\sum_{j=1}^{n+1} \mathscr{g}^{ab} \mathscr{g}^{cd} {u}_{;ac} {u}_{;bd}$, where $\mathscr{g}$ stands for the Riemannian metric tensor, $u_{;ac}=\p^2_{ac}u - \Gamma_{ac}^s \p_s u$ and $\Gamma_{ac}^s$ are the Christoffel symbols for the Levi--Civita connection. 
\end{oss}

\begin{proof}
For notational ease we will write, within this proof, $u,$ $g$ instead of $u_{\eps}$, $g_{\e}$. Given $\phi \in C^\infty_c(B^n \times (-\sigma, \sigma))$, we use the stability assumption $\int  \eps |\nabla \varphi|^2 + \frac{W''(u)}{\eps}\varphi^2 \geq 0$ with the choice of test function $\varphi = \phi|\nabla u -(\nabla u \cdot \p_{n+1})\p_{n+1}|$. We remark that  the full justification of the argument would require the use of the test function $\varphi = \phi\sqrt{\delta^2 + |\nabla u|^2 -(\nabla u \cdot \p_{n+1})^2}$ for $\delta>0$, and then taking the limit $\delta\to 0$ (as in the proof of \cite[Proposition 4]{Ton}). With this understood, we will set straight away $\delta=0$, to make notation and ideas more transparent. We recall that $\p_{n+1}$ is a coordinate vector and has everywhere unit length. We get
\begin{eqnarray}\label{eq:ineq_1}
&&\int \eps \phi^2 \left|\nabla |\nabla u -(\nabla u \cdot \p_{n+1})\p_{n+1}|  \right|^2 + \int \eps |\nabla \phi|^2 \left(|\nabla u|^2-(\nabla u \cdot \p_{n+1})^2  \right)\nonumber\\
&&\hspace{.5in}+\int 2 \eps \phi \,|\nabla u -(\nabla u \cdot \p_{n+1})\p_{n+1}| \, \nabla \phi \cdot \nabla  |\nabla u -(\nabla u \cdot \p_{n+1})\p_{n+1}|\nonumber\\
&&\hspace{1in}+\int \frac{W''(u)}{\eps}\phi^2\left(|\nabla u|^2-(\nabla u \cdot \p_{n+1})^2  \right) \geq 0.
\end{eqnarray}
Differentiating the PDE $\eps \Delta u -  \frac{W'(u)}{\eps} = -g$ we have $\eps \nabla \Delta u -  \frac{W''(u)}{\eps}\nabla u  = -\nabla g$; taking the scalar product on both sides with $\nabla u -(\nabla u \cdot \p_{n+1})\p_{n+1}$ we obtain
\begin{eqnarray}\label{eq:diff_PDE}
&&\hspace{-.7in}\eps \nabla \Delta u \left(\nabla u -(\nabla u \cdot \p_{n+1})\p_{n+1}\right) + \nabla g \left(\nabla u -(\nabla u \cdot \p_{n+1})\p_{n+1}\right)\nonumber\\ 
&&\hspace{1.5in}= \frac{W''(u)}{\eps}\left(|\nabla u|^2-(\nabla u \cdot \p_{n+1})^2  \right).
\end{eqnarray}
We substitute (\ref{eq:diff_PDE}) in the last term of (\ref{eq:ineq_1}). Moreover we manipulate the third term of (\ref{eq:ineq_1}) by using the identity $f\nabla f = \frac{1}{2} \nabla f^2$ twice, and then performing an integration by parts: this term then becomes $-\int \eps \phi^2 \Delta \frac{|\nabla u|^2-(\nabla u\cdot \p_{n+1})^2}{2}$.
 This leads to
\begin{eqnarray}\label{eq:ineq_2}
&&\hspace{-.5in}\int \eps \phi^2 \left|\nabla |\nabla u -(\nabla u \cdot \p_{n+1})\p_{n+1}|  \right|^2 + \int \eps |\nabla \phi|^2 \left(|\nabla u|^2-(\nabla u \cdot \p_{n+1})^2  \right)\nonumber\\
&&\hspace{-.4in}-\int \eps \phi^2 \Delta \frac{|\nabla u|^2}{2}+\int \eps \phi^2\Delta \frac{(\nabla u\cdot \p_{n+1})^2}{2}+\int \eps \phi^2 \nabla \Delta u \cdot \nabla u\nonumber\\
&&\hspace{-.3in}-\int \eps \phi^2(\nabla u \cdot \p_{n+1}) \nabla \Delta u \cdot \p_{n+1}+\int\phi^2 \nabla g \left(\nabla u -(\nabla u \cdot \p_{n+1})\p_{n+1}\right) \geq 0.
\end{eqnarray}
Using the Bochner identity (for the third and fifth terms of (\ref{eq:ineq_2})) we get
\begin{eqnarray}\label{eq:ineq_3}
&&\hspace{-.5in}\int \eps \phi^2 \left|\nabla |\nabla u -(\nabla u \cdot \p_{n+1})\p_{n+1}|  \right|^2 + \int \eps |\nabla \phi|^2 \left(|\nabla u|^2-(\nabla u \cdot \p_{n+1})^2  \right)\nonumber\\
&&\hspace{-.4in}-\int \eps \phi^2 |\nabla^2 u|^2 -\int \eps \phi^2 \Rc{N}(\nabla u,\nabla u)+\int \eps \phi^2\Delta \frac{(\nabla u\cdot \p_{n+1})^2}{2}\nonumber\\
&&\hspace{-.3in}-\int \eps \phi^2(\nabla u \cdot \p_{n+1}) \nabla \Delta u \cdot \p_{n+1}+\int\phi^2 \nabla g \left(\nabla u -(\nabla u \cdot \p_{n+1})\p_{n+1}\right) \geq 0.
\end{eqnarray}
We will now focus on the sixth term in (\ref{eq:ineq_3}). Firstly, we compute the commutator $\nabla \Delta u \cdot \p_{n+1} - \Delta(\nabla u \cdot \p_{n+1})= \nabla_{n+1}( \Delta u ) - \Delta(\nabla_{n+1} u)$, where $\nabla_j$ stands for the covariant derivative along the (coordinate) vector field $\p_j$. Writing $\Delta = \mathscr{g}^{ij}\nabla_i \nabla_j$ (this is the expression of the rough Laplacian, which agrees with the Laplace-Beltrami when applied to functions), and writing $R$ for the curvature tensor, we compute
\begin{eqnarray*}
&&\Delta(\nabla_{n+1} u)- \nabla_{n+1}( \Delta u ) = \mathscr{g}^{ij}\nabla_i \nabla_j(\nabla_{n+1} u) - \nabla_{n+1} ( \mathscr{g}^{ij}\nabla_i \nabla_j u)\nonumber\\
&&\underbrace{=}_{\text{$\mathscr{g}$ is parallel}} \mathscr{g}^{ij}\nabla_i \nabla_j \nabla_{n+1} u - \mathscr{g}^{ij} \nabla_{n+1}  \nabla_i \nabla_j u\nonumber\\ 
&&\hspace{.2in}= \mathscr{g}^{ij}\nabla_i \nabla_j \nabla_{n+1} u - \mathscr{g}^{ij} \nabla_i  \nabla_{n+1} \nabla_j u + \mathscr{g}^{ij} R_{i (n+1) j}^a \nabla_a u\nonumber\\
&&\hspace{.2in}= \mathscr{g}^{ij}\nabla_i \nabla_j \nabla_{n+1} u - \mathscr{g}^{ij} \nabla_i  \nabla_{j} \nabla_{n+1} u + \mathscr{g}^{ij} R_{i (n+1) j}^a \nabla_a u= \mathscr{g}^{ij} R_{i (n+1) j}^a \nabla_a u.
\end{eqnarray*}
We can thus replace $\nabla \Delta u \cdot \p_{n+1}$ (in the sixth term of (\ref{eq:ineq_3})) with $\Delta(\nabla u \cdot \p_{n+1}) + \widetilde{E}$, where $|\widetilde{E}|\leq C_N |\nabla u|$. Working then on the fifth and sixth terms in (\ref{eq:ineq_3}), noting that $-\int \eps \phi^2 (\nabla u \cdot \p_{n+1}) \Delta (\nabla u \cdot \p_{n+1}) = -\int \eps \phi^2  \Delta \frac{ (\nabla u \cdot \p_{n+1})^2}{2} + \int \eps \phi^2 |\nabla (\nabla u \cdot \p_{n+1})|^2$, we obtain
\begin{eqnarray}\label{eq:ineq_4}
&&\hspace{-.4in}\int \eps \phi^2 \left|\nabla |\nabla u -(\nabla u \cdot \p_{n+1})\p_{n+1}|  \right|^2 + \int \eps |\nabla \phi|^2 \left(|\nabla u|^2-(\nabla u \cdot \p_{n+1})^2  \right)\nonumber\\
&&\hspace{-.2in}-\int \eps \phi^2 |\nabla^2 u|^2 -\int \eps \phi^2 \Rc{N}(\nabla u,\nabla u) + \int \eps \phi^2 |\nabla (\nabla u \cdot \p_{n+1})|^2\nonumber\\
&&-\int \eps \phi^2\widetilde{E} (\nabla u \cdot \p_{n+1}) +\int\phi^2 \nabla g \left(\nabla u -(\nabla u \cdot \p_{n+1})\p_{n+1}\right) \geq 0,
\end{eqnarray}
from which (\ref{eq:Tonegawa_type_ineq}) follows immediately.
\end{proof}

The next lemma provides the geometric significance of the integrand on the left-hand-side of (\ref{eq:Tonegawa_type_ineq}). Let $p\in B^n \times (-\sigma, \sigma)$ be such that $\nabla u(p)\neq 0$. Choose a normal system of coordinates $(y_1, \ldots, y_{n+1})$ centred at $p$, with $\p_{y_{n+1}} = \p_{x_{n+1}}$ at $p$ and $\nabla u \in \text{span}\{\p_{y_{n}}, \p_{y_{n+1}}\}$ at $p$. In analogy with the notation in \cite{Ton} we set 
$$|\widetilde{B}^{\eps}|^2(p)= \sum_{i=1}^{n-1}\sum_{j=1}^{n+1}|\p^2_{ij} u_{\eps}|^2(p),$$
where the partial derivatives are computed with respect to the coordinates $(y_1, \ldots, y_{n+1})$. Then $|\widetilde{B}^{\eps}|^2(p)$ is a well-defined function on $(B^n \times (-\sigma, \sigma)) \setminus \{\nabla u_{\eps}=0\}$ (it does not depend on the choice at $p$ of $\p_{y_j}$ for $j\in \{1, \ldots, n-1\}$). 

\begin{oss}
In the sequel, $|\widetilde{B}^{\eps}|^2$ will always appear multiplied by $|\nabla u_{\eps}|^2$ and integrated with respect to $d\mathcal{H}^{n+1} \res (B^n \times (-\sigma, \sigma))$. Therefore, we can view $|\widetilde{B}^{\eps}|^2$ as a summable function with respect to the measure $|\nabla u_{\eps}|^2\,d\mathcal{H}^{n+1} \res (B^n \times (-\sigma, \sigma))$ (regardless of the fact that it is undefined where $\nabla u_{\eps}=0$). In fact, later on, when letting $\eps\to 0^{+}$ in (\ref{eq:Tonegawa_type_ineq}), we will treat $|\widetilde{B}^{\eps}|^2$ and $\eps |\nabla u_{\eps}|^2$ as a measure-function pair (as in \cite{Ton}).
\end{oss}

\begin{lem}
 \label{lem:n+1_n-1_minor}
We have, on $(B^n \times (-\sigma, \sigma)) \setminus \{\nabla u_{\eps}=0\}$,
\begin{eqnarray}\label{eq:B_tilde_square}
&&\hspace{-.6in}|\nabla^2 u_{\eps}|^2 -|\nabla(\nabla u_{\eps} \cdot \p_{n+1})|^2  - \left|\nabla |\nabla u_{\eps} - (\nabla u_{\eps} \cdot \p_{n+1})\p_{n+1} | \right|^2\nonumber\\ 
&&\hspace{2in}=|\nabla u_{\eps}|^2|\widetilde{B}^{\eps}|^2+ \widetilde{E},
\end{eqnarray} 
where $|\widetilde{E}|\leq C_N |\nabla u_{\eps}|^2$.
\end{lem}

\begin{proof} We write $\nu=\frac{\nabla u}{|\nabla u|}$ in a neighbourhood of $p$ and let $\{e_1, \ldots, e_n, e_{n+1}\}$ be a smooth orthonormal frame in a neighbourhood of $p$, chosen so that $e_{n+1}$ agrees with $\p_{n+1}$ in the neighbourhood and $e_n$ is determined at $p$ by the condition $\nu \in \text{span}\{e_n(p), e_{n+1}(p)\}$. We set $\nu_j = \nu \cdot e_j$, so that $\nu = \sum_{j=1}^{n+1} \nu_j e_j$. Then
\begin{eqnarray*}
&&\hspace{-.4in}\left|\nabla |\nabla u_{\eps} - (\nabla u_{\eps} \cdot \p_{n+1})\p_{n+1} | \right|^2 = \left|\nabla \left(\sqrt{\sum_{j=1}^n \nu_j^2 } \,|\nabla u_{\eps}|\right) \right|^2\nonumber\\
&&\hspace{1.6in}=\left|\frac{\sum_{j=1}^n \nu_j \nabla \nu_j}{\sqrt{\sum_{j=1}^n \nu_j^2 }} \,|\nabla u_{\eps}| +    \sqrt{\sum_{j=1}^n \nu_j^2 } \,\nabla |\nabla u_{\eps}|  \right|^2.
\end{eqnarray*}
Evaluating at $p$, at which $\nu_j=0$ for $j\in\{1, \ldots, n-1\}$, this expression becomes
$$\left|\frac{\nu_n \nabla \nu_n}{|\nu_n|} \,|\nabla u_{\eps}| +    |\nu_n| \,\nabla |\nabla u_{\eps}|  \right|^2= \left| (\nabla \nu_n)  \,|\nabla u_{\eps}| +    \nu_n \,\nabla |\nabla u_{\eps}|  \right|^2.$$
This says (by writing $\nabla (\nabla u_{\eps} \cdot e_n) = \nabla (|\nabla u_{\eps}| \nu_n))$) that 
$$\left|\nabla |\nabla u_{\eps} - (\nabla u_{\eps} \cdot \p_{n+1})\p_{n+1} | \right|^2 (p)=|\nabla (\nabla u_{\eps} \cdot e_n)|^2 (p),$$ with the given choice of $\{e_1, \ldots, e_n, e_{n+1}\}$. Note that $|\nabla (\nabla u_{\eps} \cdot e_n)|^2(p)$ only depends on the choice of the (unit) vector field $e_n$ (rather than the whole frame). We have thus obtained that, when evaluated at $p$, each of the two terms that are substracted from $|\nabla^2 u_{\eps}|^2$ in (\ref{eq:B_tilde_square}) is of the form $|\nabla (\nabla_e u_{\eps})|^2(p)$, for a certain unit vector field $e$.

We write such a unit vector field $e$ in coordinates (around $p$) as $e=\eta^i \p_i$. Then 
$\nabla (\nabla_e u_{\eps}) = \nabla (\eta^i \p_i u) = \mathscr{g}^{ab} \p_b(\eta^i \p_i u) \p_a$. If the coordinate frame is orthonormal at $p$, taking the norm we get (with all the terms evaluated at $p$)
\begin{eqnarray*}
&&|\nabla (\nabla_e u_{\eps})|^2=\sum_{a=1}^{n+1} (\p_a \eta^i \p_i u + \eta^i \p^2_{ai} u)^2\nonumber\\
&&\hspace{1in}=\sum_{a=1}^{n+1}[ (\p_a \eta^i)^2 ( \p_i u)^2 +\p_a ({\eta^i})^2 \p_i u \p^2_{ai} u+(\eta^i)^2 (\p^2_{ai} u)^2].
\end{eqnarray*}
If the coordinate system is moreover chosen so that $\p_n(p) = e(p)$ we obtain at $p$
$$|\nabla (\nabla_e u_{\eps})|^2 = \sum_{a=1}^{n+1} (\p_a \eta^i)^2 ( \p_i u)^2 + 2\,  \p_a \eta^n \, \p_n u \,\p^2_{nn} u+\sum_{a=1}^{n+1}  (\p^2_{an} u)^2.$$
Finally, we note that if the coordinate system is chosen to be normal at $p$ (with coordinate frame orthonormal at $p$ and with $\p_n=e$ at $p$) we have $(\p_a \eta^n) (p)=0$. This follows by expanding $0=\p_a|e|^2 = \p_a(\mathscr{g}_{ij}\eta^i \eta^j)$ and using the vanishing of the partial derivatives of $\mathscr{g}_{ij}$ at $p$, which gives $2\mathscr{g}_{ij}\, \p_a\eta^i \, \eta^j =0$. Evaluating at $p$, where $\eta^i=0$ for $i\neq n$ and $\eta^n=1$, we find $(\p_a \eta^n)(p)=0$. In conclusion, at $p$ we have the following equality:
$$|\nabla (\nabla_e u_{\eps})|^2 = \sum_{a=1}^{n+1} (\p_a \eta^i)^2 ( \p_i u)^2 +\sum_{a=1}^{n+1}  (\p^2_{an} u)^2$$
in normal coordinates at $p$ chosen so that $e(p)=\p_n(p)$. With the same argument we get
$$|\nabla (\nabla_{e_{n+1}} u_{\eps})|^2 = \sum_{a=1}^{n+1} (\p_a \widetilde{\eta}^i)^2 ( \p_i u)^2 +\sum_{a=1}^{n+1}  (\p^2_{a(n+1)} u)^2$$
in normal coordinates at $p$ chosen so that $e_{n+1}(p)=\p_{n+1}(p)$ (here $\widetilde{\eta}^i$ denote the coordinates of $e_{n+1}$).
Recalling (see Remark \ref{oss:hessian}) that, in normal coordinates at $p$, we have $|\nabla^2 u_{\eps}|^2 = \sum_{i,j=1}^{n+1}  (\p^2_{ij} u)^2$, the claim is proved.
\end{proof}

We will now analyse the limiting behavior of (\ref{eq:Tonegawa_type_ineq}) as $\e  = \e_{j} \to 0^{+}$, under the assumption that $g_{\e_{j}} \to g$ locally in $C^{1, 1}$. We begin with the second term on the right-hand-side of (\ref{eq:Tonegawa_type_ineq}). Recall that $u_{\eps_{j}}\to u_\infty$ in $L_{\rm loc}^1(N)$ (up to a subsequence that we keep implicit) and that $u_\infty \in BV_{\rm loc}(N)$ and takes values in $\{-1,+1\}$. By direct computation, the second term on the right-hand-side of (\ref{eq:Tonegawa_type_ineq}) is equal to 
\begin{eqnarray*}
&&\int \phi^2 \nabla g_{\e} \cdot \nabla u_{\eps} -  \int \phi^2 (\nabla g_{\e}\cdot \p_{n+1}) (\nabla u_{\eps} \cdot \p_{n+1})\nonumber\\
&&\hspace{.5in}= \int \phi^2 \nabla g_{\e} \cdot \nabla u_{\eps} -  \int \phi^2 (\nabla g_{\e}\cdot \p_{n+1}) \text{div} \, (u_{\eps} \cdot \p_{n+1})\nonumber\\
&&\hspace{2in}+   \int \phi^2 (\nabla g_{\e}\cdot \p_{n+1})\,u_{\eps} \, \text{div}\,\p_{n+1}\nonumber\\
&&\hspace{.5in}=-\int \text{div}\,(\phi^2 \nabla g_{\e}) u_{\eps} +  \int u_{\eps} \,\nabla (\phi^2 (\nabla g_{\e}\cdot \p_{n+1})) \cdot \p_{n+1}\nonumber\\ 
&&\hspace{2.5in}+   \int \phi^2 (\nabla g_{\e}\cdot \p_{n+1})\,u_{\eps} \, \text{div}\, \p_{n+1}.
\end{eqnarray*}
Setting $\e = \e_{j}$ in this, we note that since $g_{\e_{j}} \to g$ in $C^{1,1}_{\rm loc}(N)$, in all of the terms the function $u_{\eps_{j}}$ is multiplied by an $L^{\infty}$-bounded function. Therefore the limit of the right side as $\eps_{j} \to 0^{+}$ is
\begin{eqnarray*}
&&-\int \text{div}(\phi^2 \nabla g) u_{\infty} +  \int u_{\infty} \,\nabla (\phi^2 (\nabla g\cdot \p_{n+1})) \cdot \p_{n+1}\nonumber\\ 
&&\hspace{1.5in}+   \int \phi^2 (\nabla g\cdot \p_{n+1})\,u_{\infty} \, \text{div}\,\p_{n+1}\nonumber\\
&&\hspace{.5in}=\int_{\p^*\{u_\infty=+1\}} \phi^2 \,\nabla g \cdot \hat{n} +  \int u_{\infty} \text{div} \left( \phi^2 (\nabla g\cdot \p_{n+1})  \p_{n+1}\right)\nonumber\\
&&\hspace{.5in}=\int_{\p^*\{u_\infty=+1\}} \phi^2 \,\nabla g \cdot \hat{n} +  \int_{\p^*\{u_\infty=+1\}}  \phi^2 (\nabla g\cdot \p_{n+1})  \p_{n+1}\cdot \hat{n}
\end{eqnarray*}
where  $\p^*$ denotes the reduced boundary and by $\hat{n}$ the measure theoretic normal. We have thus shown that 
\be
\label{eq:limit_g_term}
\lim_{\eps_{j} \to 0^{+}} \left|\int \phi^2 \nabla g_{\e_{j}} \cdot \left(\nabla u_{\eps_{j}} - (\nabla u_{\eps_{j}}\cdot \p_{n+1}) \p_{n+1}\right)\right|\leq C_{g} \int \phi^2 d\|V\|
\ee
where the constant $C_{g}$ depends only on an upper bound on $\sup_{N} |\nabla \, g|.$ 
\begin{oss}
\label{oss:error_terms}
For the last two terms on the right-hand-side of (\ref{eq:Tonegawa_type_ineq}), on the other hand, it suffices to recall that $\eps_{j} |\nabla u_{\eps_{j}}|^2 d\mathcal{H}^{n+1} \to \|V\|$, so that, the limit of these two terms as $\eps_{j} \to 0^{+}$ will be bounded from above by by $C_N \int \phi^2 d\|V\|$. Similarly, recalling Lemma \ref{lem:n+1_n-1_minor}, leaving $\int \eps |\nabla u_{\eps}|^2 \phi^2 |\widetilde{B}^{\eps}|^2$ on the left-hand-side of (\ref{eq:Tonegawa_type_ineq}) and moving the term $\int \eps  \phi^2 \widetilde{E}$ to the right-hand-side, we find that this latter term is also controlled by $C_N \int \phi^2 d\|V\|$ in the limit.
\end{oss}

\medskip

In view of (\ref{eq:limit_g_term}) and of Remark \ref{oss:error_terms}, we can rewrite (\ref{eq:Tonegawa_type_ineq}) as 
\be
\label{eq:Tonegawa_type_ineq_2}
\int \e_{j} |\nabla u_{\e_{j}}|^2 \phi^2 |\widetilde{B}^{\e_{j}}|^2 \leq  \int \e_{j}  (|\nabla u_{\e_{j}}|^2-(\nabla u_{\e_{j}} \cdot \p_{n+1})^2) |\nabla \phi|^2+\text{other terms}
\ee
and $\lim_{\eps_{j}\to 0^{+}}|\text{other terms}|\leq C_{N,g}\int \phi^2 d\|V\|$.
The only terms that are left to deal with, in taking the limit $\e_{j}\to 0^{+}$, are therefore $\int \e_{j} |\nabla u_{\e_{j}}|^2 \phi^2 |\widetilde{B}^{\e_{j}}|^2$ and the first term on the right-hand-side. The argument for these follows \cite{Ton} closely.

Using Lemma \ref{lem:n+1_n-1_minor} we have, in the first instance (arguing similarly to \cite[Proposition 5]{Ton}) that the measure-function pairs $(|\widetilde{B}^{\e_{j}}|, V^{\e_{j}})$ satisfy a uniform $L^2$ bound and therefore there exists a well-defined (subsequential) limit $(|\widetilde{B}|, V)$, where $V$ is the varifold limit of $V^{\e_{j}}$. Then arguing (pointwise) as in \cite[Lemma 1]{Ton} and using Lemma \ref{lem:n+1_n-1_minor} again, it follows that there exists $c=c(n)>0$ such that $|{\mathbf B}_{V}|^2 \leq c (|\widetilde{B}|^2 + g^2)$, where 
${\mathbf B}_{V}$ is the generalized second fundamental form of $V$ as in (\ref{var-stability}). For the first term on the right-hand-side of (\ref{eq:Tonegawa_type_ineq_2}) we also argue as in the (final line of the) proof of \cite[Proposition 5]{Ton}. These arguments, together with (\ref{eq:limit_g_term}) and Remark \ref{oss:error_terms}, give the following:

\begin{lem}[Schoen--Tonegawa inequality]\label{Schoen-Tonegawa}
Let $g \in C^{1, 1}(N)$ with $\sup_{N} \, |\nabla g| \leq \Gamma$, and let $V$ be a stable limit $(g,0)$-varifold on $N$. Let $y \in N$ and let $B^{n} \times (-\sigma, \sigma) \subset T_{y}N$ be a Fermi coordinate neighborhood as described in the paragraph preceding Lemma~\ref{prelim-ST}. Let $\widetilde{V}$ be the pullback of $V$ under this coordinate map. Then for any $\phi\in C^\infty_c(B^n \times (-\sigma, \sigma)),$
\be
\label{eq:Schoen-Tonegawa}
\int \phi^2 |{\mathbf B}_{\widetilde{V}}|^2 d\|\widetilde{V}\| \leq c \int |\nabla \phi|^2 (1-(\widetilde{\nu} \cdot \p_{n+1})^2)d\|\widetilde{V}\| + C\int \phi^2 d\|\widetilde{V}\|
\ee
where ${\mathbf B}_{\widetilde{V}}$ is the generalized second fundamental form of $\widetilde{V}$, $\widetilde{\nu}(x)$ is a choice of unit normal to $T_{x} \, \widetilde{V}$ (which exists for $\|\widetilde{V}\|$-a.e.\ $x \in {\rm spt} \, \|\widetilde{V}\|$),  and the constants $c = c(N)$ and $C = C(N, \Gamma);$ ${\mathbf B}_{\widetilde{V}}$, $\widetilde{\nu}$ as well as the inner product $\cdot$ on the right had side are all with respect to the pullback metric from $N$ under the coordinate map. 
\end{lem}


\subsection{Regularity of stable limit $(g,0)$-varifolds, Part I: $C^{1, \alpha} \implies C^{2}$}\label{higher-reg}

In this section we  show that if $g  \in C^{1, 1}(N)$ with $g >0$, and if a stable limit $(g, 0)$-varifold $V$ in some neighborhood of a point $y$ is given by a union of ordered $C^{1, \alpha}$ graphs for some $\alpha \in (0, 1)$, then each of the graphs is of class $C^{2}$ (and hence, by elliptic regularity, is of class $C^{3,\alpha}$ for any $\alpha \in (0, 1)$); we in fact show that $V$ has quasi-embedded PMC$(g, 0)$ structure near $y$ (see Definition~\ref{PMC-structure}). This is the content of Theorem~\ref{thm:inductive_higher_reg} below. 

The arguments we use to prove Theorem~\ref{thm:inductive_higher_reg} are adapted from those in \cite{BW1} and \cite{BW2}. The main difference between the present context and that of \cite{BW1}, \cite{BW2} is that here we must allow the full range of possibilities of quasi-embedded PMC\,$(g,0)$ structure, rather than restricting the scalar mean curvature to be prescribed by $g$ everywhere on the regular part as in \cite{BW2}. As a consequence, $3$-fold touching singularities are allowed here, as in the last picture in Figure \ref{fig:PMC_structures}, while only $2$-fold touching singularities can arise in \cite{BW1}, \cite{BW2}. On the other hand, we point out that by using the stability of the Allen--Cahn solutions $u_{{\eps}_j}$ corresponding to $V$, we can in fact simplify parts of the argument in \cite{BW1}, \cite{BW2}: specifically, we will be able to obtain the necessary $W^{2,2}$-estimates directly as a consequence of the local $L^{2}$ summability with respect to $\|V\|$ of the generalised second fundamental form of $V$ (implied by (\ref{var-stability})). 

\begin{oss}
In \cite[Section 7.4]{BW1} the corresponding $W^{2, 2}$ estimates require additional work (not aided by the stability assumption) since stability and the Schoen--Tonegawa inequalities in \cite{BW1}, \cite{BW2} are only assumed on $\Greg \, V$ (where there is $C^{2}$ regularity) and hence there is no \emph{a priori} information on the second fundamental form on approach to 
the part of the coincidence set away from $\Greg \, V$.
In the absence of a hypothesis such as the Allen--Cahn approximation, validity of the stability and the Schoen--Tonegawa inequalities only on $\Greg  \, V$ is the natural stability assumption to make since, as in well known, it is the most direct consequence of non-negativity of second variation (of the $C^{2}$ immersed part of the varifold) with respect to the functional $A - {\rm Vol}_{g}$ where $A$ is the mass and ${\rm Vol}_{g}$ is the enclosed $g$-volume (see \cite{BW2}). 
\end{oss}

\medskip

We begin by recalling that a smooth prescribed-mean-curvature hypersurface of a Riemannian manifold is locally given by the image under the exponential map of a graph of a function that solves a PDE (in Euclidean space) associated with a quasilinear second order operator of mean curvature type. Indeed, if $u:B^n_{\rho}(x) \to \R$ is $C^2$ and if there exists a local system of normal coordinates centred at $p\in N$ defined on a neighborhood $U \subset N$ such that $B^n_{\rho}(x) \times I \subset \text{exp}_p^{-1}(U)$ for some open interval $I$ that contains $[\min u, \max u]$, and if the mean curvature $H$ of 
$\text{exp}_p \left(\text{graph}(u)\right)$ satisfies $|H| = g$ everywhere or $H = 0$ everywhere, then $u$ satisfies one of the following three PDEs: for any $\zeta \in C^\infty_c(B^n_{\rho}(x))$

\be
\label{eq:PDE_pmc}
-\int F_j(x,u,Du) D_j \zeta + \int F_{n+1}(x,u, Du) \zeta = \left\{\begin{array}{ccc}
                                                                                                 \int g(x,u)\zeta\\
                                                                                                 \int -g(x,u)\zeta\\
                                                                                                 0
                                                                                                \end{array}
\right.,
\ee
where $F_j(x,u,Du)=\left(\frac{\p}{\p p_j} F\right)((x,u),(Du,-1))$ and $F:\text{exp}_p^{-1}(U) \times \R^{n+1} \to \R$ is as in \cite[Section 3]{BW2} and \cite{SS}, and satisfies conditions \cite[(1.2)-(1.5)]{SS}. The notation $D_j$ stands for $\frac{\p}{\p x_j}$ for $j\in \{1, \ldots, n\}$. Conditions \cite[(1.2)-(1.5)]{SS} guarantee in particular that $D_i F_j$ forms a positive definite matrix, so that the above are (uniformly) elliptic quasilinear PDEs. The function $g$ appearing on the right-hand-side of the first two PDEs in (\ref{eq:PDE_pmc}) is the $C^{1,1}$ function in $\text{exp}_p^{-1}(U)$ obtained by composing the original $\left. g\right|_U$ with $\text{exp}_p$, and multiplying by the Riemannian volume element in normal coordinates (see \cite[Section 3.1]{BW2}); \emph{by abuse of notation, we denote the resulting function by $g$}. The PDEs in (\ref{eq:PDE_pmc}) are written in their weak formulation (with an implicit summation over repeated indices); however they are satisfied in the strong sense as well, since $u$ is $C^2$. The three PDEs in (\ref{eq:PDE_pmc}) correspond respectively to the cases where ${\text{exp}_p}(\text{graph}(u))$ is a hypersurface with mean curvature $g \nu$, $-g \nu$ and $0$, where $\nu$ is the unit normal that has positive scalar product with the (pushforward via $\text{exp}_p$ of the) upwards direction in the cylinder $B^n_\rho(x) \times I$. For this reason, when $u$ satisfies the first, middle or the last of the PDEs in (\ref{eq:PDE_pmc}), we will say that ${\rm graph}(u)$ is \emph{PMC with mean curvature $g$, or PMC with mean curvature $-g$ or minimal} respectively.

The main result of this section is the following: 
\medskip

\begin{thm}
\label{thm:inductive_higher_reg} Let $g \in C^{1, 1}(N)$ be positive. Let $y \in N$ be an arbitrary point, $\Oc = B_{1}^{n} \times {\mathbb R} \subset T_{y} \, N \approx {\mathbb R}^{n+1}$ and let $V$ be the restriction to $\Oc$ of a varifold obtained by applying an appropriate rescaling to the pull back of a (part of a) stable limit $(g, 0)$-varifold $V_{1}$ on $N$ by the exponentinal map $\exp_{y}$. Suppose that $V =\sum_{j=1}^q |{\rm graph}\, u_j|$, where $q\in \N$, $u_j:B^n_1 \to \R$ are of class $C^{1,\alpha}$ for some $\alpha \in (0, 1)$ and $u_1\leq u_2 \leq \ldots \leq u_q.$
Suppose further that the density $\Theta \, (\|V\|, y) = q.$
Then, for each $j\in \{1, \ldots, q\}$, $u_j$ is of class $C^{2, \alpha}$ and the scalar mean curvature of ${\rm graph} \, u_j$ with respect to the upward pointing unit normal is either zero everywhere on ${\rm graph} \, u_{j}$, or is equal to $g(x)$ for every $x \in {\rm graph} \, u_{j},$ or is equal to $-g(x)$ for every $x \in {\rm graph} \, u_{j}$. Moreover, $\spt{V}$ is the union of the graphs of at most three $u_j$'s.
More precisely: 
\begin{itemize}
\item[{\rm (i)}] if $q$ is even, then either all of the $u_j$'s coincide and ${\rm graph} \, u_{j}$ are minimal; or ${\rm graph} \, u_1$ is PMC with mean curvature $-g$, ${\rm graph} \, u_q$ is PMC with mean curvature $g$ and (if $q \geq 4$) $u_2= \ldots =u_{q-1}$ with ${\rm graph} \, u_{j}$ minimal for $2 \leq j \leq q-1$; 
\item[{\rm (ii)}] if $q$ is odd, then either ${\rm graph} \, u_1$ is PMC with mean curvature $-g$ and (if $q \geq 3$) $u_2= \ldots =u_{q}$ with ${\rm graph} \, u_{j}$ minimal for $2 \leq j \leq q;$ or 
${\rm graph} \, u_q$ is PMC with mean curvature $g$ and (if $q \geq 3$) $u_1= \ldots =u_{q-1}$ with ${\rm graph} \, u_{j}$ minimal for $1 \leq j \leq q-1$.
\end{itemize}
Thus $V_{1}$ has quasi-embedded $PMC \, (g, 0)$ structure near $y$. 
\end{thm}

The strategy of the proof of Theorem \ref{thm:inductive_higher_reg} is to use the stability hypothesis and Theorem~\ref{Roger-Tonegawa} to show that each $u_{j}$, which satisfies one of the above PDEs where it is $C^2$, must in fact be a weak solution (to one of the above PDEs) on the entire domain. Once this is done, standard elliptic theory implies the conclusion. As a preliminary step, we will need the following lemma. 

\begin{lem}[Absence of $\ell$-fold touching singularities for $\ell\geq 4$]
\label{lem:no_fourfold}
Let $y \in N$ be an arbitrary point, $\Oc = B_{1}^{n} \times {\mathbb R} \subset T_{y} \, N \approx {\mathbb R}^{n+1}$ and let $V$ be the varifold on $\Oc$ obtained by applying an appropriate rescaling to the pull back of a (part of a) limit $(g, 0)$-varifold on $N$ by the exponentinal map $\exp_{y}$. Suppose that $V\res \Oc=\sum_{j=1}^q |{\rm graph} \, u_j|$, where $q\in \N$, $u_j:B^n_1 \to \R$ are of class $C^{1,\alpha}$ for some $\alpha \in (0, 1)$ and $u_1\leq u_2 \leq \ldots \leq u_q.$ Suppose further that there exists a point $y \in \spt{V} \cap \Oc$ such that density $\Theta \, (\|V\|, y) = q$ and that if $\Theta(\|V\|, x)\leq q-1$ then there exists a neighbourhood of $x$ in $\Oc$ in which $V$ has quasi-embedded $PMC \, (g,0)$ structure.

Then $\spt{V} \cap \Oc=\cup_{j=1}^{\widetilde{q}} {\rm graph}\, \widetilde{u}_j$, where $\widetilde{q}\in \{1, 2, 3\}$, $\widetilde{q}\leq q$, and $\widetilde{u}_j\in C^{1,\alpha}(B^n_1)$. Moreover, if $\widetilde{q}\geq 2$ then $\widetilde{u}_j\leq \widetilde{u}_{j+1}$ on $B^n_1$ for $j\in \{1, \widetilde{q}-1\}$, and there exists $x\in B^n_1$ such that $\widetilde{u}_j(x)< \widetilde{u}_{j+1}(x)$ for $j\in \{1, \widetilde{q}-1\}$. (In other words, $\spt{V} \cap \Oc$ is the union of the graphs of at most three distinct ordered $C^{1,\alpha}$ functions and is embedded in some non-empty open cylinder $\Omega \times {\mathbb R} \subset \Oc$.)
\end{lem}

\begin{proof} We note that the graph structure implies that $\Theta(\|V\|,\cdot)$ is integer-valued everywhere on $\spt{V}$ (not just almost everywhere); the proof of this is as in \cite[Lemma A.2]{BW1}. Let $\pi: \Oc\to B=B^n_1$ be the projection onto the first factor and let $C=\pi(\{x\in \spt{V} \cap \Oc:\Theta(x,\|V\|)=q\})$. The set $C$ is closed in $B$. If $B\setminus C= \emptyset$, then $u_j=u_1$ for all $j$ and $\spt{V}$ is a single graph, so the conclusion of the lemma holds with $\widetilde{q}=1$. From now we therefore assume that $B\setminus C \not= \emptyset$. We denote by $U$ the subset of $B\setminus C$ made up of points $p$ such that $\spt{V}$ is an embedded hypersurface at all the points in $\pi^{-1}(p)\cap \spt{V} = \cup_{j=1}^q \{(p,u_j(p))\}$ (there may be repeated points in the last expression). The set $U$ is open and, by the regularity assumption on points with density $\leq q-1$, $U$ is dense in $B\setminus C$ (the set of non-embedded points with multiplicity $\leq q-1$ projects locally to a finite union of submanifolds with dimension at most $n-1$). We define on $U$ the function $\widetilde{Q}$ that assigns to $p\in U$ the number of distinct points in the set $\pi^{-1}(p)\cap \spt{V}$. By the embeddedness requirement that characterizes $U$, the function $\widetilde{Q}$ is locally constant on $U$. Clearly $\widetilde{Q}\leq q$ on $U$. 

We claim, in a first instance, that $\widetilde{Q}$ extends to a locally constant function on $B\setminus C$. To see this, let $a\in B\setminus C$, $a \notin U$. Then there exists at least a point (and at most $q/2$ points) $x\in \pi^{-1}(a)\cap \spt{V}$ at which $\spt{V}$ is not embedded; at such a point $\Theta(x,\|V\|)\leq q-1$. By hypothesis, there exists a neighbourhood ${\Oc}_x$ of $x$ contained in $\Oc$ such that $V\res {\Oc}_x$ has one of the structures (iii), (iv) or (v) in Definition \ref{PMC-structure}. Denote by $\{x_1, \ldots x_K\}$ the distinct points in $\pi^{-1}(a)\cap \spt{V}$ at which $\spt{V}$ is not embedded and by $\{x_{K+1}, \ldots x_M\}$ the points in $\pi^{-1}(a)\cap \spt{V}$ at which $\spt{V}$ is embedded. Let $B_a$ be an open ball contained in $B$ such that $\spt{V} \res ( B_a \times I)$ is the union of exactly $M$ connected sets. Note that $M-K$ of these are embedded disks. From the characterization of the structures (iii), (iv) or (v) we can check the following fact: $\widetilde{Q}(p)$ is an integer independent of $p\in B_a \cap U$ (this follows by counting the distinct points in $\pi^{-1}(p)$ for each of the five possible structures). Therefore, recalling that $U$ is dense in $B\setminus C$, $\widetilde{Q}$ extends (from $U$) to a locally constant integer-valued function on $B\setminus C$, proving our first claim. 

We next define on each connected component $U_c$ of $U$ the functions $\underline{u}_1, \ldots \underline{u}_{\widetilde{Q}}$ such that $\underline{u}_j<\underline{u}_{j+1}$ for all $j\in \{1, \ldots \widetilde{Q}-1\}$ and such that $\spt{V} \cap (U_c \times I) = \cup_{j=1}^{\widetilde{Q}}\text{graph}(\underline{u}_j)$. (Note that $u_1$ necessarily agrees with $\underline{u}_1$.) Let $B_c$ be a connected component of $B\setminus C$. By our first claim, $\widetilde{Q}$ does not vary among the connected components of $U$ that lie in $B_c$. This implies that the functions $\underline{u}_1, \ldots \underline{u}_{\widetilde{Q}}$, initially defined only on $U\cap B_c$ for a common $\widetilde{Q}$, can be extended continuously from $U \cap B_c$ to $B_c$ giving rise to $\widetilde{Q}$ ordered functions that we still denote by $\underline{u}_1\leq \ldots \leq \underline{u}_{\widetilde{Q}}$. Moreover, $\underline{u}_1, \ldots, \underline{u}_{\widetilde{Q}}$ are $C^2$ on $B_c$ by the characterization of the structures (iii), (iv) or (v) of Definition \ref{PMC-structure}. A priori, $\widetilde{Q}$ may depend on $B_c$ and we will focus on a single connected component $B_c$. By the characterization of the structures (iii), (iv) or (v), each $\text{graph}(\underline{u}_{k})$ is either PMC with mean curvature $g$ on $B_c$, or PMC with mean curvature $-g$ on $B_c$, or is minimal on $B_c$. (These three options correspond respectively to the fulfillment, by $\underline{u}_k$, of the three PDEs in (\ref{eq:PDE_pmc}).)

Our next claim is that $\widetilde{Q}\leq 3$ on $B_c$. We let $b\in U\cap B_c$ and consider an open ball $D$ contained in $B_c$, centred at $b$ and such that $\p D \cap C$ is not empty. Let $p \in \p D \cap C$, then all the $\underline{u}_{j}$ extend continuously to $p$ with the same value ($=u_1(p)= \ldots =u_q(p)$). If $\widetilde{Q}\geq 4$ we find a contradiction to Hopf boundary point lemma, as follows. If $\widetilde{Q}\geq 4$ then 
there exist two indices $j_1\neq j_2$ for which $\underline{u}_{j_1}$ and $\underline{u}_{j_2}$ both solve the same of the three PDEs in (\ref{eq:PDE_pmc}). We let $v=\underline{u}_{j_1}-\underline{u}_{j_2}$ and compute, following a standard argument, the PDE satisfied by $v$ and obtain, for any $\zeta \in C^\infty_c(B_c)$

\be
\label{eq:PDE_difference_1_lemma}
-\int \left(\mathscr{a}_{ij} D_i v +  \mathscr{b}_i v\right) D_j \zeta + \int \left(\mathscr{d}_i D_i v + \mathscr{c}  v \right)\zeta  = \left\{\begin{array}{ccc}
                                                                                                 \int \mathscr{f}\,v \,\zeta\\
                                                                                                 
                                                                                                 0
                                                                                                \end{array}
\right.,
\ee
where 
$$\mathscr{a}_{ij}=\int_0^1 \frac{\p F_j}{\p p_i}(x, (1-s)\underline{u}_1+s \underline{u}_2, (1-s)D \underline{u}_1+s D \underline{u}_2) ds,$$
and $\mathscr{b}_i$, $\mathscr{c}$, $\mathscr{d}_i$ are similarly defined by integration of the composition of a smooth function (depending on $F$ or $DF$) with $(x,\underline{u}_{j_1}, D \underline{u}_{j_1},\underline{u}_{j_2}, D \underline{u}_{j_2})$. As such, $\mathscr{a}_{ij}$, $\mathscr{b}_i$, $\mathscr{c}$, $\mathscr{d}_i$ are in $C^{0,\alpha}(B_c)$. The left-hand-side of (\ref{eq:PDE_difference_1_lemma}) is the difference of the left-hand-sides of (\ref{eq:PDE_pmc}) for $\underline{u}_{j_1}$ and $\underline{u}_{j_2}$. The right-hand-side of (\ref{eq:PDE_difference_1_lemma}) is similarly obtained from the right-hand-sides of (\ref{eq:PDE_pmc}), an $\mathscr{f}$ is the integration of the composition of a $C^{0,1}$ function (depending on derivatives of $g$) with $(x,\underline{u}_{j_1}, \underline{u}_{j_2})$ and therefore $\mathscr{f}\in C^{0,1}(B_c)$. The function $\mathscr{a}_{ij}$ is moreover symmetric in $i,j$ (by the definition of $F_j$ in terms of $\frac{\p F}{\p p_j}$, see (\ref{eq:PDE_pmc}) above) and $\mathscr{a}_{ij}$ form a positive definite matrix by the conditions on $F$ \cite[(1.2)-(1.5)]{SS}, so the PDEs in (\ref{eq:PDE_difference_1_lemma}) are locally uniformly elliptic. The first case in (\ref{eq:PDE_difference_1_lemma}) arises when $\underline{u}_{j_1}$ and $\underline{u}_{j_2}$ either both solve the first or both solve the second of the PDEs in (\ref{eq:PDE_pmc}), while the second case in (\ref{eq:PDE_difference_1_lemma}) arises when $\underline{u}_{j_1}$ and $\underline{u}_{j_2}$ both solve the third PDE in (\ref{eq:PDE_pmc}). Since $v\in C^2(B_c)$, the PDE is satisfied in the strong sense.
Recall that all the $\text{graph}(u_j)$ intersect tangentially at the point $(p,u_1(p))$, therefore $v$ and $Dv$ extends continuously to $p$ with $v(p)=0$, $Dv(p)=0$. Hopf boundary point lemma (see Remark \ref{oss:Hopf} below), applied to either of the two PDEs in (\ref{eq:PDE_difference_1_lemma}), implies that $v\equiv 0$ on $D$, so $\underline{u}_{j_1}=\underline{u}_{j_2}$ on $D$, contradicting the initial definition of $\underline{u}_j$ on $U$.
We thus have $\widetilde{Q}\leq 3$ on any connected component $B_c$ of $B\setminus C$. 

We then set $\widetilde{q}=\max_{B\setminus C} \widetilde{Q}$. Since we are working under the assumption $B\setminus C\not= \emptyset$, 
we have $\widetilde{q}\geq 2$. The functions $\underline{u}_1$ and $\underline{u}_{\widetilde{Q}}$ (defined above on $B\setminus C$) agree respectively with (the restrictions to $B\setminus C$ of) $u_1$ and $u_q$. We set $\widetilde{u}_1=u_1$ and $\widetilde{u}_{\widetilde{q}} = u_q$ on $B$. 
If $\widetilde{q}= 2$, then $\spt{V}=\text{graph}(\widetilde{u}_{1}) \cup \text{graph}(\widetilde{u}_{\widetilde{q}})$ and all the conclusions of the lemma are satisfied. The last case to consider is $\widetilde{q}= 3$. In this case, we let $\widetilde{u}_2=\underline{u}_2$ on the (non-empty) open set $A$ given by the union of the connected components of $B\setminus C$ on which $\widetilde{Q}=3$. The function $\widetilde{u}_2$ is $C^2(A)$ and we extend it to $B$ by setting it equal to $\widetilde{u}_1$ on $B\setminus A$. The resulting function, still denoted $\widetilde{u}_2$, is in $C^{1,\alpha}(B)$: this only needs to be checked at points in $C$, since $\widetilde{u}_2$ is $C^2$ by construction on each connected components of $B\setminus C$. At point in $C$ the conclusion follows by recalling that all the $u_j$'s and their differentials agree on $C$, and that each $u_j$ is in $C^{1,\alpha}(B)$. By construction, there exists at least one connected component of $U$ on which $\underline{u}_1<\underline{u}_2<\underline{u}_3$, hence on the same non-empty open set $\widetilde{u}_1<\widetilde{u}_2<\widetilde{u}_3$.
\end{proof}

\begin{oss}\label{oss:Hopf}
We have used the Hopf boundary point lemma in its version that is valid regardless of the sign of the $0$-th order term, since we know that $v(p)=0$ (see the discussion that follows \cite[Theorem 3.5]{GT}). A more careful use of Hopf boundary point lemma gives additional information on $\widetilde{u}_{j_1}$ and $\widetilde{u}_{j_2}$, as we will point out now. While not necessary for the conclusion of Lemma \ref{lem:no_fourfold}, this will be used within the proof of Theorem \ref{thm:inductive_higher_reg} below. 
We assume that $\widetilde{u}_{j_1}\leq \widetilde{u}_{j_2}$ are $C^2$ on a ball $D$, $\widetilde{u}_{j_1}\not\equiv\widetilde{u}_{j_2}$ and that $\widetilde{u}_j$ and $D \widetilde{u}_j$, for $j\in \{j_1, j_2\}$, extend continuously to $p\in \p D$ with $\widetilde{u}_{j_1}(p)=\widetilde{u}_{j_2}(p)$ and $D\widetilde{u}_{j_1}(p)=D\widetilde{u}_{j_2}(p)$. We have ruled out, in Lemma \ref{lem:no_fourfold}, the possibility that these two functions solve the same PDE out of the three equations in (\ref{eq:PDE_pmc}). We now note, more precisely, 
that in fact the only possibilities for the right-hand-sides $\int h_j(x,u_j) \zeta$ in (\ref{eq:PDE_pmc}) are those for which $h_{j_1}(x,\widetilde{u}_{j_1}) < h_{j_2}(x,\widetilde{u}_{j_2})$. Here $h_j=g$, $h_j=-g$, or $h_j=0$ are the three options for $h_j$. This claim follows since if $h_{j_1}(x,\widetilde{u}_{j_1}) \geq h_{j_2}(x,\widetilde{u}_{j_2})$ for some $x$ then, since $g >0$, $h_{j_1}(x,\widetilde{u}_{j_1}) > h_{j_2}(x,\widetilde{u}_{j_2})$ for all $x$, so by taking the difference of the two PDEs gives for $v=\widetilde{u}_{j_1}-\widetilde{u}_{j_2}\leq 0$ 
\be
\label{eq:PDE_difference_1_oss}
-\int \left(\mathscr{a}_{ij} D_i v +  \mathscr{b}_i v\right) D_j \zeta + \int \left(\mathscr{d}_i D_i v + \mathscr{c}  v \right)\zeta  \geq 0,
\ee
whenever $\zeta\in C^\infty_c(D)$ and $\zeta\geq 0$. Here the coefficients $\mathscr{a}_{ij}$, $\mathscr{b}_{i}$, $\mathscr{d}_{i}$, $\mathscr{c}$ are as in the proof of Lemma \ref{lem:no_fourfold}. We have from (\ref{eq:PDE_difference_1_oss}) a contradiction to Hopf boundary point lemma, because $v(p)=0$, $Dv(p)=0$ and $v\not \equiv 0$ in $D$.
\end{oss}

\begin{oss}
We are not (yet) ruling out the possibility that $\widetilde{Q}$ in the proof of Lemma \ref{lem:no_fourfold} (i.e.~the number of distinct graphs needed to describe $\spt{V}$) may change from one connected component of $B\setminus C$ to another. This will be accomplished in the proof of Theorem~\ref{thm:inductive_higher_reg} based on the fact that $V$ is a limit $(g, 0)$-varifold (using Theorem~\ref{Roger-Tonegawa}).
\end{oss}

For the proof of Theorem \ref{thm:inductive_higher_reg} it is convenient to analyse the possible configuration that can arise depending on the maximum value of distinct graphs needed to describe $\spt{V}$, i.e.~the value of $\widetilde{q}$ in Lemma \ref{lem:no_fourfold}. 
In the proof of Theorem \ref{thm:inductive_higher_reg} we directly make use of the assumption that $V$ is a limit $(g, 0)$-varifold. This assumption was absent in Lemma \ref{lem:no_fourfold}, although it should be kept in mind that this lemma is used within an induction. To avoid confusion: the induction is carried out on $q$, however the cases below are differentiated depending on $\widetilde{q}$. For a fixed $q$, only some of the cases analysed below can actually happen. After the case by case analysis (on $\widetilde{q}$) is complete, we will show that in any of the possible configurations identified, one can extend the validity of the PDE of mean curvature type (i.e.\ one of the equations in (\ref{eq:PDE_pmc})) from the $C^2$ portion of each graph, to the full $C^{1,\alpha}$ graph in the weak form.

\medskip
\begin{proof}[Proof of Theorem~\ref{thm:inductive_higher_reg}]
Let $\widetilde{q}$ be as in Lemma~\ref{lem:no_fourfold}. Then $\widetilde{q} \leq 3$. Let $C =\pi(\{x\in \spt{V}:\Theta(x,\|V\|)=q\})$ where $\pi \, : \, B_{1}^{n} \times {\mathbb R} \to B_{1}^{n}$ 
is the orthogonal projection. We consider the cases $\widetilde{q}=1$, $\widetilde{q}=2$ and $\widetilde{q}=3$ separately.

\noindent \underline{$\widetilde{q}=1$}. In this case $\spt{V}$ is a single $C^{1,\alpha}$ graph, $\spt{V}=\text{graph}(\widetilde{u}_1)=\text{graph}(u_1)$. Then we will show by means of \cite{RogTon}---see Theorem \ref{Roger-Tonegawa} above---that it is either completely minimal, or completely PMC (the latter means that $\text{graph}(u_1)$ has either mean curvature everywhere equal to $g(x,u_1)$, or has mean curvature everywhere equal to $-g(x,u_1)$). Indeed, the phases of $u_\infty$ can only be: 

$+1$ on the supergraph of $u_1$ and $-1$ on the subgraph;

$-1$ on the supergraph of $u_1$ and $+1$ on the subgraph;

$+1$ on the supergraph of $u_1$ and $+1$ on the subgraph;

$-1$ on the supergraph of $u_1$ and $-1$ on the subgraph.

\noindent By Theorem \ref{Roger-Tonegawa}, the third option is not possible for $g>0$. In the first or second case, we have again by Theorem \ref{Roger-Tonegawa} the following two facts. The multiplicity of $V$ is a.e.~$1$ on the graph and the generalized mean curvature takes almost everywhere the value $g$, or almost everywhere the value $-g$, respectively in the first and second case. The first fact implies that the multiplicity is everywhere $1$, by the condition that the generalized mean curvature is in $L^p(\|V\|)$ (so this case can only happen for $q=1$). Using this fact and elliptic theory we get that the graph is $C^{2,\alpha}$, with mean curvature $g$ or $-g$ respectively. In the fourth case we get by Theorem \ref{Roger-Tonegawa} that the mean curvature is almost everywhere $0$ and the multiplicity is an even integer almost everywhere. Again, these fact imply that the graph is minimal (and smooth) and carries a constant even multiplicity (therefore this case can only arise for $q$ even). If $\widetilde{q}=1$, the proof of Theorem \ref{thm:inductive_higher_reg} is complete.

\medskip

\noindent \underline{$\widetilde{q}=2$}. This means that $\spt{V}=\cup_{j=1}^2 \text{graph}(\widetilde{u}_j)$, with $\widetilde{u}_1\leq \widetilde{u}_2$ and $\widetilde{u}_1< \widetilde{u}_2$ on some nonempty open set. The set $\widetilde{u}_1=\widetilde{u}_2$ is the set of points of multiplicity $q$ and away from it we have multiplicity $\leq q-1$ and (inductively) $C^2$ regularity. The set $\widetilde{u}_1=\widetilde{u}_2$ is the set $C$ from Lemma \ref{lem:no_fourfold} so we have that $\widetilde{u}_1$, $\widetilde{u}_2$ are $C^2$ on $B\setminus C$ and on each connected component of $B\setminus C$ the graph of $\widetilde{u}_j$ is either completely minimal or completely PMC. We distinguish two cases:

(I) $\text{graph}(\widetilde{u}_2) \setminus \text{graph}(\widetilde{u}_1)$ has some PMC $C^2$ embedded connected component; 

(II) $\text{graph}(\widetilde{u}_2) \setminus \text{graph}(\widetilde{u}_1)$ has no PMC $C^2$ embedded connected component.

\noindent We begin with (II). Then, by \cite{RogTon}, $\text{graph}(\widetilde{u}_2) \setminus \text{graph}(\widetilde{u}_1)$ is minimal and has even multiplicity, and moreover $u_\infty=-1$ on the supergraph of $\widetilde{u}_2$ and $u_\infty=-1$ on the set $\{(x,y)\in B\times I: \widetilde{u}_1(x) < y <\widetilde{u}_2(x)\}$. By considering a ball $D$ in an arbitrary connected component of $B\setminus C$, chosen so that a point on the boundary of $D$ belongs to $C$, we use Hopf boundary point lemma in $D$ (see Remark \ref{oss:Hopf}) to rule out the possibility that $\widetilde{u}_1$ is minimal on $D$, or PMC on $D$ with mean curvature $g$. The only possibility is therefore that $\widetilde{u}_1$ is PMC on $D$ with mean curvature $-g$, and therefore PMC on $B\setminus C$ with mean curvature $-g$. This also forces, again by \cite{RogTon}, that $u_\infty=+1$ on the subgraph of $\widetilde{u}_1$ and multiplicity to be $1$ at points in $\text{graph}(\widetilde{u}_1)\setminus \text{graph}(\widetilde{u}_2)$. (We are thus in a case that can only occur when $q$ is odd.) We next check that the set $\text{graph}(\widetilde{u}_1)\cap \text{graph}(\widetilde{u}_2)$ has vanishing $\mathcal{H}^n$-measure: if that were not the case, then we would contradict \cite{RogTon}, since almost every point on $\text{graph}(\widetilde{u}_1)\cap \text{graph}(\widetilde{u}_2)$ must lie in the interior of set $\{u_\infty=-1\}$, against our earlier conclusions that $u_\infty=+1$ on the subgraph of $\widetilde{u}_1$ and  $u_\infty=-1$ on the supergraph of $\widetilde{u}_2$. (Or, almost every point on $\text{graph}(\widetilde{u}_1)\cap \text{graph}(\widetilde{u}_2)$ must have even multiplicity, against the conclusion that $q$ is odd.)

\noindent We now consider case (I). Let $B_c$ be a connected component of $B\setminus C$ on which the graph of $\widetilde{u}_2$ is PMC. Then by considering a ball $D$ in $B_c$ chosen so that a point on the boundary of $D$ belongs to $C$, we use Hopf boundary point lemma (see Remark \ref{oss:Hopf}) to conclude that the mean curvature of $\text{graph}(\widetilde{u}_2)$ on $D$ is $g$ (and not $-g$). We therefore have that each connected components of $\text{graph}(\widetilde{u}_2)\setminus \text{graph}(\widetilde{u}_1)$ is either minimal or PMC with mean curvature $g$, and there is at least one PMC component. By \cite{RogTon} $u_\infty=+1$ on the supergraph of $\widetilde{u}_2$ and $u_\infty =-1$ on the set $\{(x,y)\in B\times I: \widetilde{u}_1(x) < y <\widetilde{u}_2(x)\}$. We check now that the set $\text{graph}(\widetilde{u}_1)\cap \text{graph}(\widetilde{u}_2)$ has vanishing $\mathcal{H}^n$-measure. If that were not the case, then by \cite{RogTon} we would have that almost every point on $\text{graph}(\widetilde{u}_1)\cap \text{graph}(\widetilde{u}_2)$ must lie in the interior of set $\{u_\infty=-1\}$, against our earlier conclusion. Moreover, by using \cite{RogTon} in the same way, almost all points in $\text{graph}(\widetilde{u}_1)\setminus \text{graph}(\widetilde{u}_2)$ must have multiplicity $1$ and mean curvature equal to $g$. We thus conclude that each connected components of $\text{graph}(\widetilde{u}_2)\setminus \text{graph}(\widetilde{u}_1)$ is PMC with mean curvature $g$. We now consider any connected component of $\text{graph}(\widetilde{u}_1)\setminus \text{graph}(\widetilde{u}_2)$. By using Hopf boundary point lemma as before, we conclude that each such connected component is either minimal or PMC with mean curvature $-g$. If one of them is minimal, then (arguing as above by means of \cite{RogTon}) $u_\infty = -1$ on the subgraph of $\widetilde{u}_1$ and this rules out the fact that any connected component of $\text{graph}(\widetilde{u}_1)\setminus \text{graph}(\widetilde{u}_2)$ is PMC. In other words, either all connected components of $\text{graph}(\widetilde{u}_1)\setminus \text{graph}(\widetilde{u}_2)$ are minimal, or they are all PMC with mean curvature $-g$. 

In conclusion, for the case $\widetilde{q}= 2$ we have two $C^{1,\alpha}$-functions $\widetilde{u}_1\leq \widetilde{u}_2$, that coincide on a set $C$ with $\mathcal{H}^n(C)=0$, are $C^2$ on $B\setminus C$ and

exactly one of these possible configurations occurs: 

$\text{graph}(\widetilde{u}_1)\setminus \text{graph}(\widetilde{u}_2)$ is minimal with even multiplicity, $\text{graph}(\widetilde{u}_2)\setminus \text{graph}(\widetilde{u}_1)$ has mean curvature $g$ and multiplicity $1$;

$\text{graph}(\widetilde{u}_1)\setminus \text{graph}(\widetilde{u}_2)$ has mean curvature $-g$ and multiplicity $1$, $\text{graph}(\widetilde{u}_2)\setminus \text{graph}(\widetilde{u}_1)$ has mean curvature $g$;

$\text{graph}(\widetilde{u}_1)\setminus \text{graph}(\widetilde{u}_2)$ has mean curvature $-g$ and multiplicity $1$, $\text{graph}(\widetilde{u}_2)\setminus \text{graph}(\widetilde{u}_1)$ is minimal with even multiplicity.

\medskip

\noindent \underline{$\widetilde{q}=3$}.  This means that $\spt{V}=\cup_{j=1}^3 \text{graph}(\widetilde{u}_j)$, with $\widetilde{u}_1\leq \widetilde{u}_2\leq \widetilde{u}_3$ and $\widetilde{u}_1< \widetilde{u}_2< \widetilde{u}_3$ on some nonempty open set.  The set $\widetilde{u}_1=\widetilde{u}_2=\widetilde{u}_3$ is the set of points of multiplicity $q$ and away from it we have multiplicity $\leq q-1$ and (inductively) $C^2$ regularity. The set $\widetilde{u}_1=\widetilde{u}_2=\widetilde{u}_3$ is the set $C$ from Lemma \ref{lem:no_fourfold} so we have that $\widetilde{u}_1$, $\widetilde{u}_2$, $\widetilde{u}_3$ are $C^2$ on $B\setminus C$ and on each connected component of $B\setminus C$ the graph of each $\widetilde{u}_j$ is either completely minimal or completely PMC.

\noindent We claim the following: $\widetilde{u}_1< \widetilde{u}_2< \widetilde{u}_3$ on an open subset of $B$ whose complement has vanishing $\mathcal{H}^n$ measure; the graphs of $\widetilde{u}_1$, $\widetilde{u}_2$, $\widetilde{u}_3$ have, on this open set, multiplicity respectively $1$, $q-2$ even, $1$; the graphs of $\widetilde{u}_1$, $\widetilde{u}_2$, $\widetilde{u}_3$ are, on $B\setminus C$, respectively PMC with mean curvature $-g$, minimal, PMC with mean curvature $g$. (In particular, this case can only occur for $q$ even and $\geq 4$.)

\noindent For the proof of our claim, we consider $B_c$, a connected component of $B\setminus C$ on which $\widetilde{u}_1< \widetilde{u}_2< \widetilde{u}_3$ and a ball $D$ in $B_c$ chosen so that a point on the boundary of $D$ belongs to $C$. Using Hopf boundary point lemma in $D$ as done in Remark \ref{oss:Hopf} we obtain that, on $B_c$, the top graph has necessarily mean curvature $g$, the bottom graph has mean curvature $-g$, the middle one is minimal. This forces (by \cite{RogTon}) the fact that $u_\infty=+1$ on the supergraph of $\widetilde{u}_3$ and on the subgraph of $\widetilde{u}_1$, while $u_\infty=-1$ on $\{(x,y)\in B\times I: \widetilde{u}_1(x) < y <\widetilde{u}_3(x)\}$. Always by \cite{RogTon}, the minimal portions have even multiplicity and the PMC portions have multiplicity $1$. Additionally, we also obtain that the set $Z$ of points where two of the three graphs agree is a set of vanishing $\mathcal{H}^n$-measure (for otherwise, almost everywhere on this set we would have multiplicity higher than $1$, hence even because $g>0$, and by \cite{RogTon} these points would have to belong to the interior of $\{u_\infty=-1\}$, against the previous conclusions). The $C^2$ regularity on $Z\setminus C$ completes the proof of the claim.

\medskip

In the remainder of this section we analyse the possible configurations of $V$ for $\widetilde{q}\in \{2, 3\}$. We will show that, for each configuration, each $\widetilde{u}_j$ is in $C^2(B)$ and ${\rm graph} \, u_{j}$ is either completely minimal, or completely PMC with mean curvature $g$, or  completely PMC with mean curvature $-g$. We consider  $\widetilde{u}_{j_1}$ and $\widetilde{u}_{j_2}$ for $j_1<j_2$ and let $v=\widetilde{u}_{j_2}-\widetilde{u}_{j_1}$. Then $v$ is a non-negative function and it is $C^2$ on $B\setminus C$. We compute the PDE satisfied by $v$ on $B\setminus C$, arguing as in Lemma \ref{lem:no_fourfold} and keeping in mind the possible configurations. We obtain, for all $\zeta \in C^\infty_c(B\setminus C)$,

\be
\label{eq:PDE_difference}
-\int \left(\mathscr{a}_{ij} D_i v +  \mathscr{b}_i v\right) D_j \zeta + \int\left(\mathscr{d}_i D_i v + \mathscr{c}  v \right)\zeta  = \left\{\begin{array}{ccc}
\int g(x,u_{j_2}) \zeta \\
\int g(x,u_{j_1}) \zeta\\
\int (g(x,u_{j_2})+g(x,u_{j_1})) \zeta
\end{array}
\right.,
\ee
where $\mathscr{a}_{ij}$, $\mathscr{b}_i$, $\mathscr{c}$, $\mathscr{d}_i \in C^{0,\alpha}(B) \cap C^1(B\setminus C)$ and $\mathscr{a}_{ij}$ is symmetric and positive definite. Which right-hand-side appears in (\ref{eq:PDE_difference}) depends on the right-hand-side of the PDEs for $\widetilde{u}_{j_1}$ and $\widetilde{u}_{j_2}$ in (\ref{eq:PDE_pmc}). The first corresponds to the case $\widetilde{u}_{j_1}$ minimal and $\widetilde{u}_{j_2}$ PMC (with mean curvature $g$). The second corresponds to the case $\widetilde{u}_{j_2}$ minimal and $\widetilde{u}_{j_1}$ PMC (with mean curvature $-g$). The third corresponds to the case in which $\widetilde{u}_{j_1}$ and $\widetilde{u}_{j_2}$ are PMC with mean curvature respectively $-g$ and $g$.

We will next prove that the PDE (\ref{eq:PDE_difference}) for $v$ extends (in its weak form) to the whole of $B$. Note that $v\in C^{1,\alpha}(B)$ by assumption and $v=0$ on $C$ ($v\geq 0$ on $B$) and that the right-hand-side (in all three cases) is a $C^{1,\alpha}$ function on $B$.

All possible cases are treated similarly and the third option in (\ref{eq:PDE_difference}) corresponds to the situation treated in \cite{BW2} (and in \cite{BW1} when $g\equiv \text{cnst}$). All cases follow a similar argument, that we carry out here only in the case corresponding to the first option in (\ref{eq:PDE_difference}). Recall, from our conclusions on the possible configurations when $\widetilde{q}\in \{2,3\}$, that $\mathcal{H}^n(C)=0$.

We note, first of all, that inequality~(\ref{var-stability}) provides an interior $L^2$-bound on $|D^2 \widetilde{u}_j|$ on $B\setminus C$, i.e.\ for each ball $B_{r} = B_{r}(0) \subset B$ with $r \in (0, 1)$, we have that 
$\int_{B_{r} \setminus C} |D\widetilde{u}_{j}|^{2} \leq \widetilde{c},$ where $\widetilde{c}$ depends on $\sup _{B_{r}} \, |Du_{j}|$, $\sup_{B_{r}} \, |Dg|$ and $r.$ This is because the generalized second fundamental form appearing in (\ref{var-stability}) agrees with the classical second fundamental form on $\Greg{V}$, and thus it bounds from above the second derivatives of $\widetilde{u}_j$ on $B_{r} \setminus C$ in terms of $\sup_{B_{r}} \, |Du_{j}|$. Thus $\widetilde{u}_j \in W^{2,2}(B_{r}\setminus C)$ and hence 
$v \in W^{2,2}(B_{r} \setminus C)$ for every $r \in (0, 1)$. Moreover, $Dv\in C^{0,\alpha}(B)$ and $Dv=0$ on $C$, with $Dv \in C^1(B\setminus C)$. Using now \cite[Lemma C.1]{BW1} we obtain that $Dv\in W^{1,2}_{\text{loc}}(B)$. We then adapt the argument in \cite[Section 7.5]{BW1}. We consider the function $f=\mathscr{a}_{ij} D_i v + \mathscr{b}_i v$ on $B$; $f\in C^{0,\alpha}(B)\cap C^1(B\setminus C)$, with $f=0$ on $C$ (since $v$ and $Dv$ vanish there). Recalling the structure of $\mathscr{a}_{ij}$ and $\mathscr{b}_i$, we find that $Df$ 
is in $L^1(B_{r}\setminus C)$ since $v \in W^{2,2}_{\rm loc}(B)$. Applying again \cite[Lemma C.1]{BW1} we conclude that $f\in W_{\rm loc}^{1,1}(B)$, with distributional derivative given by the $L^1$ function equal to $Df$ on $B\setminus C$ and $0$ on $C$. Using this fact, and recalling that the PDE for $v$ is satisfied strongly on $B\setminus C$, we compute, for $\zeta \in C^\infty_c(B)$:
\begin{eqnarray}\label{eq:extend_PDE_v}
&&\hspace{-.3in}-\int_B \left(\mathscr{a}_{ij} D_i v + \mathscr{b}_i v\right) D_j \zeta +  \left(\mathscr{d}_i D_i v + \mathscr{c}  v \right)\zeta=\int_B D_j\underbrace{\left(\mathscr{a}_{ij} D_i v + \mathscr{b}_i v\right)}_{f}  \zeta + \left(\mathscr{d}_i D_i v + \mathscr{c}  v \right)\zeta\nonumber\\
&&\hspace{1in}=\int_{B\setminus C} D_j \left(\mathscr{a}_{ij} D_i v + \mathscr{b}_i v\right)  \zeta + \left(\mathscr{d}_i D_i v + \mathscr{c}  v \right)\zeta\nonumber\\
&&\hspace{1.5in}= \int_{B\setminus C} g(x,\widetilde{u}_{j_2})\zeta=\int_{B} g(x,\widetilde{u}_{j_2})\zeta,
\end{eqnarray}
where we used (in the last equality) the fact that $\mathcal{H}^n(C)=0$. Equality (\ref{eq:extend_PDE_v}) says that the weak PDE for $v$ is valid on the whole of $B$.

\medskip

With the knowledge that the (weak) PDEs for $\widetilde{u}_{j_2}-\widetilde{u}_{j_1}$ extend from $B\setminus C$ to $B$, we are now ready to prove that $\widetilde{u}_j\in C^2(B)$ for $j\in\{1, 2, 3\}$. We will do so for the case $\widetilde{q}=3$ (thus with $q$ even); the case $\widetilde{q}=2$ can be treated analogously (and is more straightforward). We point out that the functions $F_j$ in (\ref{eq:PDE_pmc}) for $j\in \{1, \ldots n\}$ are odd in the third variable: this follows by recalling that $F(x,\nu)$ is even in $\nu$ (since $F$ is the integrand of the area functional). The first variation formula for $V$ gives (recalling that $\spt{V}=|\text{graph}(u_1)| + (q-2)|\text{graph}(u_2)|+|\text{graph}(u_3)|$ with $q$ even, in the case $\widetilde{q}=3$), for any $\zeta \in C^\infty_c(B)$

\begin{eqnarray}
\label{eq:PDE_firstvariation}
&&\sum_{j=1}^3 \widetilde{q}_j \left(-\int_B F_j(x,\widetilde{u}_j,D\widetilde{u}_j) D_j \zeta + \int F_{n+1} (x,\widetilde{u}_j,D\widetilde{u}_j) \zeta\right)\nonumber\\ 
&&\hspace{2in}= \int_B \left(g(x,\widetilde{u}_3) - g(x,\widetilde{u}_1)\right)\zeta,
\end{eqnarray}
with $\widetilde{q}_1=\widetilde{q}_3=1$, $\widetilde{q}_2 =q-2$ with $q$ even. We rewrite (\ref{eq:PDE_firstvariation}) as follows

\begin{eqnarray*}
&&\hspace{-.25in}q\left[\int_B -F_j(x,\widetilde{u}_1,D\widetilde{u}_1) D_j \zeta + F_{n+1} (x,\widetilde{u}_1,D\widetilde{u}_1) \zeta\right]\nonumber\\ 
&&\hspace{.25in}+(q-2)\left[\int_B -F_j(x,\widetilde{u}_2,D\widetilde{u}_2) D_j \zeta +F_{n+1} (x,\widetilde{u}_2,D\widetilde{u}_2) \zeta\right]\nonumber\\
 &&\hspace{.5in}-(q-2)\left[\int_B -F_j(x,\widetilde{u}_1,D\widetilde{u}_1) D_j \zeta +F_{n+1} (x,\widetilde{u}_1,D\widetilde{u}_1) \zeta\right]\nonumber\\ 
&&\hspace{.75in}+\left[\int_B  -F_j(x,\widetilde{u}_3,D\widetilde{u}_3) D_j \zeta +  F_{n+1} (x,\widetilde{u}_3,D\widetilde{u}_3) \zeta\right]\nonumber\\
&&\hspace{1in}-\left[\int_B -F_j(x,\widetilde{u}_1,D\widetilde{u}_1) D_j \zeta +F_{n+1} (x,\widetilde{u}_1,D\widetilde{u}_1) \zeta \right]\nonumber\\ 
&&\hspace{2.5in}= \int_B \left(g(x,\widetilde{u}_3) - g(x,\widetilde{u}_1)\right)\zeta.
\end{eqnarray*}
In the second and third square brackets we recognise (in their weak form) the left-hand-sides of the PDEs for $v_{21}=\widetilde{u}_2-\widetilde{u}_1$ and $v_{31}=\widetilde{u}_3-\widetilde{u}_1$; we obtained above (see (\ref{eq:extend_PDE_v})) the validity of these PDEs on $B$, so we obtain
\begin{eqnarray*}
&&q\left[\int_B -F_j(x,\widetilde{u}_1,D\widetilde{u}_1) D_j \zeta + F_{n+1} (x,\widetilde{u}_1,D\widetilde{u}_1) \zeta\right]\nonumber\\ 
&&\hspace{.5in}+(q-2)\int_B g(x,\widetilde{u}_1) \zeta +\int_B (g(x,\widetilde{u}_3)+g(x,\widetilde{u}_1))\zeta\nonumber\\ 
&&\hspace{2.5in}= \int_B \left(g(x,\widetilde{u}_3) - g(x,\widetilde{u}_1)\right)\zeta,
\end{eqnarray*}
which implies that (for every $\zeta\in C^\infty_c(B)$)
\be
\label{eq:PDE_u1}
\int_B -F_j(x,\widetilde{u}_1,D\widetilde{u}_1) D_j \zeta + F_{n+1} (x,\widetilde{u}_1,D\widetilde{u}_1) \zeta  =- \int_B g(x,\widetilde{u}_1) \zeta.
\ee
Similarly we can prove (using the PDEs for $v_{32}=\widetilde{u}_3-\widetilde{u}_2$ and $v_{31}=\widetilde{u}_3-\widetilde{u}_1$) that for every $\zeta\in C^\infty_c(B)$
\be
\label{eq:PDE_u3}
\int_B -F_j(x,\widetilde{u}_3,D\widetilde{u}_3) D_j \zeta + F_{n+1} (x,\widetilde{u}_3,D\widetilde{u}_3) \zeta  = \int_B g(x,\widetilde{u}_3) \zeta.
\ee
By standard elliptic theory the fulfilment (in weak sense) of the PDEs (\ref{eq:PDE_u1}) and (\ref{eq:PDE_u3}) implies that $\widetilde{u}_1$ and $\widetilde{u}_3$ are $C^2$ and solve these PDEs strongly. These facts, together with (\ref{eq:PDE_firstvariation}), also imply that $\widetilde{u}_2$ is $C^2$ and solves
$$D_j\left(F_j(x,\widetilde{u}_2,D\widetilde{u}_2)  \right) + F_{n+1} (x,\widetilde{u}_2,D\widetilde{u}_2)  =  0$$
on $B$. This completes the proof of Theorem \ref{thm:inductive_higher_reg}.
\end{proof}

\subsection{Regularity of stable limit $(g,0)$-varifolds, Part II: proofs of Theorem~\ref{limit-regularity}~and~Theorem~\ref{estimates}}\label{limit-reg-proof}

\begin{proof}[Proof of Theorem~\ref{limit-regularity}] 
Let the hypotheses be as in the theorem, so that $g \in C^{1, 1}(N),$ $g >0$ and $V$ is a limit $(g, 0)$-varifold on $N$ with associated sequences $(\e_{j}) \subset {\mathbb R}$ with $\e_{j} \to 0^{+}$, $(g_{j}) \subset C^{1, 1}(N)$ with $g_{j} \to g$ locally in $C^{1, 1}$ and $(u_{\e_{j}}) \subset W^{1, 2}_{\rm loc}(N)$ with $u_{\e_{j}}$ a critical point of 
${\mathcal F}_{\e_{j}, \sigma g_{j}}$ such that the Morse index of $u_{\e_{j}}$ with respect to ${\mathcal F}_{\e_{j}, \sigma g_{j}}$ is $\leq I$ for some fixed integer $I \in \N$ independent of $j$. 
We first state and prove the following:

\medskip
\noindent
\emph{Claim 1:} for each $y \in N$, there exists $\rho_{y}  \in (0, {\rm inj}_{y} \, N)$ and a (sub)sequence $u^{(\ell)} = u_{\e_{j_{\ell}}},$ $\ell=1, 2, 3, \ldots$ such that for each $\delta \in (0, \rho_{y})$ and all sufficiently large $\ell$ (depending on $\delta$) we have that $u^{(\ell)}$ is stable in ${\mathcal N}_{\rho_{y}}(y) \setminus \overline{{\mathcal N}_{\delta}(y)}.$ 

\medskip
To prove Claim 1, it is convenient to proceed by first establishing the following assertion:

\medskip
\noindent
\emph{Claim 2:} Let $y \in N$ and let $k$ be an integer $\geq 1.$ The following implication holds: 
\begin{eqnarray*}
&&\hspace{-.25in}\mbox{$\exists \, \rho>0$ and a (sub)sequence $u^{(\ell)} = u_{\e_{j_{\ell}}},$ $\ell=1, 2, \ldots,$ such that $\forall \, \delta \in (0, \rho)$ and $\forall \, \ell$}\nonumber\\ 
&&\hspace{-.25in}\mbox{sufficiently large (depending on $\delta$), the Morse index of $u^{(\ell)}$ in ${\mathcal N}_{\rho}(y) \setminus \overline{{\mathcal N}_{\delta}(y)}$ is}\nonumber\\ 
&&\hspace{-.25in}\mbox{$\leq k$}\nonumber\\ 
&&\implies\nonumber\\ 
&&\hspace{-.25in}\mbox{$\exists \, \rho^{\prime} \in (0, \rho)$ and a subsequence $(u^{(\ell^{\prime})})$ of $(u^{(\ell)})$ such that $\forall \, \delta \in (0, \rho^{\prime})$ and $\forall \, \ell^{\prime}$}\nonumber\\
&&\hspace{-.25in}\mbox{sufficiently large (depending on $\delta$), the Morse index of $u^{(\ell^{\prime})}$  in ${\mathcal N}_{\rho^{\prime}}(y) \setminus \overline{{\mathcal N}_{\delta}(y)}$ is}\nonumber\\ 
&&\hspace{-.2in}\mbox{$\leq k-1$.} 
\end{eqnarray*}

If this implication is false for some $k$, then there are a number $\rho >0$ and a subsequence $u^{(\ell)}  = u_{\e_{j_{\ell}}}$ of $(u_{\e_{j}})$ for which the hypothesis of the implication holds and yet the conclusion fails, allowing us to find sequences 
$\rho_{m} \to 0^{+}$ and $\delta_{m} \to 0^{+}$ with $0 < \rho_{m+1} < \delta_{m} < \rho_{m} < \rho$  for each $m,$ and a subsequence $(u^{(\ell^{\prime})})$ of $(u^{(\ell)})$ (arrived at by a diagonal sequence argument) such that for each $m$ and all sufficiently large $\ell^{\prime}$ (depending on $m$), the Morse index of $u^{(\ell^{\prime})}$ in ${\mathcal N}_{\rho_{m}}(y) \setminus \overline{{\mathcal N}_{\delta_{m}}(y)}$ is $\geq k$. Since by hypothesis of the implication we have that for each $m$ and sufficiently large $\ell^{\prime}$ the Morse index of $u^{(\ell^{\prime})}$ in ${\mathcal N}_{\rho}(y) \setminus \overline{{\mathcal N}_{\delta_{m}}(y)}$ is $\leq k,$ it follows that for each $m$ and sufficiently large $\ell^{\prime}$ the Morse index of $u^{(\ell^{\prime})}$ in ${\mathcal N}_{\rho}(y) \setminus  \overline{{\mathcal N}_{\rho_{m}}(y)}$ is zero. This says in particular that (a stronger form of) the conclusion of the implication holds. Claim 2 is thus established. 

To deduce Claim 1, we may assume $I \geq 1$ (since if $I = 0$ then Claim 1 holds trivially) and apply Claim 2 iteratively, starting with any $\rho >0$, $k=I$ and with the full sequence $(u_{\e_{j}})$ in place of $(u^{(\ell)})$.  After $I$ iterations we arrive at Claim 1.


Now fix an arbitrary point $y \in {\rm spt} \,\|V\|$ and let $\rho_{y} >0$ and $u^{(\ell)}$ be as given by Claim 1. Then for any $\delta \in (0, \rho_{y})$ and all sufficiently large $\ell$, 
$u^{(\ell)}$ is stable in ${\mathcal N}_{\rho_{y}}(y) \setminus \overline{{\mathcal N}_{\delta}(y)}$ and hence inequality (\ref{local-bounds}) holds 
with $u^{(\ell)}$ in place of $u_{\e_{j}}$ and $\varphi \in C^{1}_{c}({\mathcal N}_{\rho_{y}}(y) \setminus \overline{\mathcal N}_{\delta}(y))$. We may therefore argue as in the proof of (\cite[Proposition~3.2]{TonWic}) to show that, if $\delta \in (0, \rho_{y}/8)$,  there exists a set $\Sigma_{y, \delta}  \subset {\rm spt} \, \|V\| \cap \left({\mathcal N}_{\rho_{y}/2}(y) \setminus {\mathcal N}_{2\delta}(y)\right)$ with ${\rm dim}_{\mathcal H} \, (\Sigma_{y, \delta}) \leq n-2$ such that no tangent cone to $V$ at any point $y \in {\rm spt} \, \|V\| \cap \left({\mathcal N}_{\rho_{y}/2}(y) \setminus {\mathcal N}_{2\delta}(y)\right) \setminus \Sigma_{y, \delta}$ can be supported on a union of three or more half-hyperplanes meeting along a single $(n-1)$-dimensional subspace. Since we can choose $\delta$ arbitrarily small, and $y \in {\rm spt} \, \|V\|$ is arbitrary, we conclude that there is a set $\Sigma \subset {\rm spt} \, \|V\|$ with ${\rm dim}_{\mathcal H} \, (\Sigma) \leq n-2$ such that no tangent cone to $V$ at any point in ${\rm spt} \, \|V\| \setminus \Sigma$ can be supported on three or more half-hyperplanes meeting along a single $(n-1)$-dimensional  subspace. By the definition of classical singularity, this immediately implies that $V$ has no classical singularities anywhere. This is conclusion (i) of the theorem.

To see conclusion (ii), let $y \in {\rm spt} \, \|V\|$ be any point and let ${\mathbf C}$ be any tangent cone to $\sigma^{-1}V$ at $y.$ Then 
\begin{equation}\label{rescaling-seq}
\eta_{0, \sigma_{i} \, \#} \, (\exp_{Y}^{-1})_{\#} (\sigma^{-1}V) \to {\mathbf C}
\end{equation}
as varifolds on $T_{y} \, N \approx {\mathbb R}^{n+1}$ for some sequence of positive numbers $\sigma_{i} \to 0.$  Since $\sigma^{-1}V$ is an integral varifold with locally bounded generalised mean curvature, we have that ${\mathbf C}$ is a stationary integral hypercone. Passing to a subsequence of $\{\eps_{j}\}$ without relabelling, we may assume that $\widetilde{\eps}_{i} = \sigma_{i}^{-1}\eps_{i} \to 0$. Letting $\widetilde{u}_{\widetilde{\eps}_{i}}(X) = u_{\eps_{i}}(\exp_{y}(\sigma_{i}X))$, we see by the reasoning as in (\cite[p.\ 200]{TonWic}) with obvious modifications, that ${\mathbf C}$ is the limit varifold associated with the sequence $(\widetilde{u}_{\widetilde{\eps}_{i}})$ in the same way that the limit varifold $V$ is associated with the sequence $(u_{\eps_{i}}).$ Hence by the argument as in conclusion (i), we see that ${\mathbf C}$ has no classical singularities, proving the first assertion of conclusion (ii).
To see the second assertion of conclusion (ii), i.e.\ that ${\rm reg} \, {\mathbf C}$ 
is stable, let $\widetilde{\varphi} \in C^{1}_{c}({\mathbb R}^{n+1} \setminus \{0\})$ and choose $\delta >0$ such that $\widetilde{\varphi} \equiv 0$ on $B_{\delta}^{n+1}(0)$. Letting $\rho_{y}$, 
$u^{(\ell)} = u_{\e_{j_{\ell}}}$ be as given by Claim 1, and $(\sigma_{i})$ be the sequence for which (\ref{rescaling-seq}) holds, note that by Claim 1 for each fixed $i$ with $\sigma_{i}\delta < \rho_{y}$ and each sufficiently large $\ell$ (depending on $i$), $u^{(\ell)}$ is stable in ${\mathcal N}_{\rho_{y}}(y) \setminus \overline{{\mathcal N}_{\sigma_{i}\delta}(y)}$. Hence for such $\ell$, inequality (\ref{stability}) holds with $\widetilde{\varphi}(\sigma_{i}^{-1}\exp_{y}^{-1}(\cdot))$ in place $\varphi$ and $u^{(\ell)}$ in place of $u_{\e_{j}}$. Now choose a subsequence $(u^{(\ell_{i})})$ of $(u^{(\ell)})$ so that for each $i$, 
$u^{(\ell_{i})}$ is stable in  ${\mathcal N}_{\rho_{y}}(y) \setminus \overline{{\mathcal N}_{\sigma_{i}\delta}(y)}$ and $\overline{\e}_{i} = \sigma_{i}^{-1}\e_{j_{{\ell}_{i}}} \to 0$ as $i \to \infty$. 
Letting $\overline{u}_{\overline{\e}_{i}} (X) \equiv u_{\e_{j_{{\ell}_{i}}}}(\exp_{y}(\sigma_{i}X))$  and writing inequality (\ref{stability}) in terms of $\overline{u}_{\overline{\e}_{i}}$ and in local co-ordinates $X \in {\mathbb R}^{n}$ induced by the diffeomorphism $X \mapsto \exp_{Y}(\sigma_{i}X),$ and letting $i \to \infty$, 
we deduce (in the same way that (\ref{var-stability}) is derived from (\ref{stability})) that ${\mathbf C}$ has a generalised second fundamental form ${\mathbf B}_{\mathbf C}$ satisfying 
$$\int_{{\mathbb R}^{n+1}} |{\mathbf B}_{{\mathbf C}}|^{2} \widetilde{\varphi}^{2} d\|{\mathbf C}\| \leq \int_{{\mathbb R}^{n+1}} |\nabla^{{\mathbb R}^{n+1}} \widetilde{\varphi}|^{2} d\|{\mathbf C}\|$$
where $\nabla^{{\mathbb R}^{n+1}}$ is the gradient on ${\mathbb R}^{n+1}$. Now given any $\varphi \in C^{1}_{c}({\rm reg} \, {\mathbf C})$ (where ${\rm reg} \, {\mathbf C}$ is $C^{2}$ is the embedded part of 
${\mathbf C}$), we may take $\widetilde{\varphi}$ in the above inequality to be 
a compactly supported extension of $\varphi$ to ${\mathbb R}^{n+1} \setminus \{0\}$ which in a neighborhood of ${\rm spt} \, \varphi$ is constant in the normal direction to ${\rm reg} \, {\mathbf C}.$ From this we deduce, using also the fact that the multiplicity of ${\mathbf C}$ is constant on every connected component of ${\rm reg} \, {\mathbf C}$, the usual stability inequality 
$$\int_{{\rm reg}\, {\mathbf C}}  |{\mathbf B}_{\mathbf C}|^{2} \varphi^{2} d{\mathcal H}^{n} \leq \int_{{\rm reg} \, {\mathbf C}} |\nabla^{{\mathbf C}} \varphi|^{2} \, d{\mathcal H}^{n}$$
for all $\varphi \in C^{1}_{c}({\rm reg} \, {\mathbf C}),$ where ${\mathbf B}_{\mathbf C}$ is the (classical) second fundamental form of ${\rm reg} \, {\mathbf C}$ and $\nabla^{\mathbf C}$ denotes the gradient on ${\rm reg} \, {\mathbf C}.$  This is the second assertion of part (ii).

We next prove conclusion (iii). 
If a tangent cone ${\mathbf C}$ to $\sigma^{-1}V$ at a point $y \in {\rm spt} \, \|V\|$ is supported on a hyperplane $P$, then there is a positive integer $q$ such that 
${\mathbf C} = q|P|$. 
For each positive integer $q$, let $R_{q}(V)$ be the set of points $y \in {\rm spt} \, \|V\|$ such that one tangent cone to $\sigma^{-1}V$ at $y$ is $q|P|$ for some hyperplane $P.$
We wish to show by arguing by induction on $q$ that $R_{q}(V)  \subset \Greg \, V$ for all $q \geq 1$, which is the assertion in conclusion (iii). If $q=1$, this is true by Allard's regularity theorem (\cite[Theorem~8.19]{A}) and standard elliptic regularity theory. Fix an integer $q \geq 2$ and suppose that the following induction hypothesis holds: 
\begin{itemize}
\item[($\ast$)] $R_{q^{\prime}}(V) \subset \Greg \, V$ for any $q^{\prime} \in \{1, 2, \ldots, q-1\}$. 
\end{itemize}
Let $y \in R_{q}(V)$ and let $\rho_{y}$ be as given by Claim 1. We first wish to apply Theorem~\ref{BWregularity} with ${\mathcal V} = \{\sigma^{-1}V \res {\mathcal N}_{\rho_{y}/4}(y)\}$ and 
$U_{\sigma^{-1}V \res {\mathcal N}_{\rho_{y}/4}(y)} = {\mathcal N}_{\rho_{y}/4}(y).$  By Theorem~\ref{Roger-Tonegawa} and conclusion (i), ${\mathcal V}$ satisfies hypotheses (a), (b) of Theorem~\ref{BWregularity}, so we only need to verify condition (c) of Theorem~\ref{BWregularity}, i.e.\ that $V^{\prime} = \sigma^{-1}V \res {\mathcal N}_{\rho_{y}/4}(y)$ 
(taken with $U_{V^{\prime}} = {\mathcal N}_{\rho_{y}/4}(y)$) satisfies the $(q, \beta)$-separation property (as in Definition~\ref{qbseparation}) for some $\beta \in (0, 1)$.  If fact, we shall take 
$$\beta = (K + 1)^{-1}\epsilon$$ 
where $\epsilon,$ $K$ be the constants as 
in Theorem~\ref{nonvar-SS} taken with any $p > n$ and with $\rho_{0} = \rho_{y}/2$, $\Lambda = \frac{1}{2}\rho_{y}\sup_{N} \, |g|$, $\lambda = c$ and $\overline{\lambda} = C$ where $c$, $C$ are as in Lemma~\ref{Schoen-Tonegawa} (taken with $\Gamma = \sup_{N}|g| + \sup_{N} \, |\nabla g|$). 
With this choice of $\beta$, let $X \in {\mathcal N}_{\rho_{y}/4}(y)$, $\rho \in (0, {\rm min} \, \{1, {\rm inj}_{X}(N), {\rm dist} \, (X, \partial {\mathcal N}_{\rho_{y}/4}(y))\})$ and $Q \, : \, {\mathbb R}^{n+1} \to {\mathbb R}^{n+1}$ be an orthogonal rotation such that, writing $$\widetilde{V} = \left(Q \circ {\rm exp}_{X}^{-1}\right)_{\#} \, V^{\prime} \res \left({\mathcal N}_{{\rm inj}_{X} N}(X) \cap {\mathcal N}_{\rho_{y}/4}(y)\right),$$ 
the  conditions (i)-(v) of Definition~\ref{qbseparation} are satisfied; in particular, it is assumed that for some $Y \in B_{\rho/2}^{n}(0) \times \{0\} \subset {\mathbb R}^{n+1} \approx T_{X} \, N$ and $\tau \in (0, \rho/2]$, 
\begin{equation}\label{density-hyp}
\Theta \, (\|V^{\prime}\|, \xi) < q \;\; \mbox{for each} \;\; \xi \in {\rm exp}_{X}(B_{\t}^{n}(Y) \times {\mathbb R}).
\end{equation}
\noindent
We need to show that the conclusion of the implication in Definition~\ref{qbseparation} holds, namely, that 
\begin{equation}\label{hyp(c)}
\eta_{Y, \t \, \#} \widetilde{V} \res ((B_{1/4}^{n}(0) \times {\mathbb R}) \cap B_{1}^{n+1}(0))  = \sum_{j=1}^{q} |{\rm graph} \, u_{j}| 
\end{equation}
for some $u_{j} \in C^{2}(B_{1/4}^{n}(0))$, $j=1, 2, \ldots, q,$ with  $u_{1} \leq u_{2} \leq \ldots \leq u_{q}.$ 

It suffices of course to establish (\ref{hyp(c)}) with $(1-\delta)\t$ in place of $\t$ where $\delta \in (0, 1/2)$ is arbitrary. We shall do this by employing Theorem~\ref{nonvar-SS}, taken with arbitrary $p > n$, $\rho_{0} = \rho_{y}/2$, $X_{0} = y$, $\Lambda = \frac{1}{2}\rho_{y} \sup_{N} \, |g|$, $\rho = \t$, and with ${\rm exp}_{X}(Y)$ in place of $X$  and 
$\sigma^{-1}V \res {\mathcal N}_{\rho_{y}/2}(y)$ in place of $V$. It follows from Theorem~\ref{Roger-Tonegawa} that hypothesis (a) of Theorem~\ref{nonvar-SS} holds with these choices. To check that hypothesis (b) of Theorem~\ref{nonvar-SS} holds with these choices,  
note that in view of (\ref{density-hyp}) and the induction hypothesis ($\ast$), if a tangent cone to $\sigma^{-1}V$ at a point $Z \in {\rm spt} \, \|V\| \cap {\mathcal N}_{(1-\delta)\t}(\exp_{X}(Y))$  is supported on a hyperplane then $Z \in \Greg \, \sigma^{-1}V.$ Hence if $2 \leq n \leq 6,$ it follows from conclusion (ii) (of the present theorem) and Theorem~\ref{classification} that 
${\rm sing} \, V \res {\mathcal N}_{(1-\delta)\t}({\rm exp}_{X}(Y)) = \emptyset$, and if $n\geq 7,$  it follows from conclusion (ii), Theorem~\ref{classification} and Lemma~\ref{stratification} that 
${\rm dim}_{\mathcal H} \, ({\rm sing} \, V \res {\mathcal N}_{(1-\delta)\t}({\rm exp}_{X}(Y)) \leq n-7.$  Thus hypothesis (b) of Theorem~\ref{nonvar-SS} holds.  Finally,  to check that hypothesis (c) of Theorem~\ref{nonvar-SS} is satisfied, note that 
$y \not\in {\mathcal N}_{\t}({\rm exp}_{X}(Y))$ since $\Theta \,(\|V\|, y) = q$ and hence we can invoke Claim 1 with $\delta\t$ in place of $\delta$ to infer that 
$V \res {\mathcal N}_{(1- \delta)\t}({\rm exp}_{X}(Y))$ is a stable limit $(g, 0)$-varifold on ${\mathcal N}_{(1-\delta)\t}({\rm exp}_{X}(Y)).$ Consequently, Lemma~\ref{Schoen-Tonegawa}  
yields hypothesis (c) of Theorem~\ref{nonvar-SS} with $\lambda = c$ and $\overline{\lambda} = C$, where $c$ and $C$ are the constants as in Lemma~\ref{Schoen-Tonegawa}. Hence it follows  from Theorem~\ref{nonvar-SS} that (\ref{hyp(c)}) holds (first with $(1-\delta)\t)$ in place of $\t$ and hence also in the limit $\delta \to 0^{+}$),  i.e.\ that $V^{\prime}$ satisfies the $(q, \beta)$-separation property for the above choice of $\beta$. This verifies hypothesis (c) of Theorem~\ref{BWregularity}.

We may therefore choose a hyperplane $P$ such that $q|P|$ is a tangent cone to $V$ at $y$ and apply Theorem~\ref{BWregularity} to see that in an appropriately small neighborhood around the origin, the varifold $\left(\exp_{y}^{-1}\right)_{\#} V$ is given as the sum of the multiplicity 1 varifolds associated with graphs of $C^{1, \alpha}$ functions $u_{1} \leq u_{2} \leq \ldots \leq u_{q}$ defined over a ball $B \subset P.$  Since for any $\delta \in (0, \rho_{y})$ Claim 1 implies that $V \res \left({\mathcal N}_{\rho_{y}}(y) \setminus \overline{\mathcal N}_{\delta}(y)\right)$ is a stable limit $(g, 0)$-varifold, it follows from Theorem~\ref{thm:inductive_higher_reg} that for each $j$, the function $u_{j}$ is of class $C^{2}$ in $B \setminus \{0\}$ and solves one of the equations in 
(\ref{eq:PDE_pmc}). It is then straightforward to see that $u_{j}$ is a weak solution to the same equation on $B$, and hence by elliptic regularity that $u_{j} \in C^{2}$ for each $j$.  Thus 
$y \in \Greg \, V$, and this completes the proof of conclusion (iii). 

Conclusion (iv) in case $2 \leq n \leq 6$ follows readily from conclusion (iii), conclusion (ii) and Lemma~\ref{classification}. When $n \geq 7$ we see from conclusion (iii), conclusion (ii), Lemma~\ref{classification} and Lemma~\ref{stratification} that ${\rm dim}_{\mathcal H} \, ({\rm sing} \, V) \leq n-7$, giving in particular conclusion (iv) in dimensions $n \geq 8$. 
If $n=7$, conclusion (iv) makes the stronger assertion that ${\rm sing} \, V$ must be discrete. This follows by a standard argument which goes as follows: if this is false then there are points $z, z_{k} \in {\rm sing} \, V$ for $k=1, 2, 3, \ldots$ with $z_{k} \neq z$ and $z_{k} \to z$ as $k \to \infty$.  Rescaling $(\exp_{z}^{-1})_{\#} \, V$ about the origin (in $T_{z} \, N \approx {\mathbb R}^{8}$) by the sequence $\rho_{k} 
= |\exp_{z}^{-1}(z_{k})|$ produces, after passing to a subsequence, a tangent cone ${\mathbf C} = \lim_{k \to \infty} \, \eta_{0, \rho_{k} \, \#} \, (\exp_{z}^{-1})_{\#} \, V.$ By conclusion (ii) and 
Lemma~\ref{classification}, we have that ${\rm sing} \, {\mathbf C} \subset \{0\}$ (in fact ${\rm sing} \, {\mathbf C}  = \{0\}$ in view of conclusion (iii)).  By Claim 1, for each sufficiently large $k$ 
the varifold $V_{k} \equiv V \res ({\mathcal N}_{2\rho_{k}}(z) \setminus \overline{{\mathcal N}_{\rho_{k}/4}(z)})$ is a stable limit $(g, 0)$-varifold in 
${\mathcal N}_{2\rho_{k}}(z) \setminus \overline{{\mathcal N}_{\rho_{k}/4}(z)}.$ Since ${\rm dim}_{\mathcal H} \, ({\rm sing} \, V_{k}) = 0,$ $\eta_{0, \rho_{k} \, \#} \, (\exp_{z}^{-1})_{\#} \, V_{k} \to {\mathbf C} \res (B_{2}^{n+1}(0) \setminus \overline{B_{1/4}^{n+1}(0)})$ as varifolds and ${\mathbf C} \res (B_{2}^{n+1}(0) \setminus \overline{B_{1/4}^{n+1}(0)})$ is regular, we can apply Theorem~\ref{nonvar-SS} to conclude that ${\rm spt} \, \|V\| \cap \left({\mathcal N}_{3\rho_{k}/2}(z) \setminus \overline{{\mathcal N}_{\rho_{k}/2}(z)}\right) \subset \Greg V$, contrary to the the fact that $z_{k} \in {\rm sing} \, V.$ Hence ${\rm sing} \, V$ must be discrete if $n=7$, and the proofs of conclusion (iv) and the theorem are complete. 
\end{proof}

\begin{proof}[Proof of Theorem~\ref{estimates}]
 Let $\overline{{\mathcal V}}$ be the collection of all stable limit $(g, 0)$-varifolds on $N$ over all $g \in C^{1, 1}(N)$ such that $\sup_{N} \, |g|+ \sup_{N} \, |\nabla g| \leq \Gamma$. By 
 Theorem~\ref{limit-regularity},~part(i), no tangent cone to a varifolds in $\overline{{\mathcal V}}$ is supported on three or more half-hyperplanes meeting along a common $(n-1)$-dimensional subspace. 
 
First consider part (ii) of Theorem~\ref{estimates}. To establish this, we appeal to Theorem~\ref{BWregularity}, proceeding as in the argument of Theorem~\ref{limit-regularity}, part (iii) but now also keeping track of the estimate provided by Theorem~\ref{BWregularity}. (Thus we do not need \emph{a priori} the regularity provided by Theorem~\ref{limit-regularity}, part (iv)). To begin with, note  that  if $y \in N$, ${\rm inj}_{y} \, N \geq \overline{\rho}$, $\rho \in (0,{\rm inj}_{y} \, N),$ and if $\widetilde{V}$ is any varifold on $B_{\rho}^{n+1}(0) \subset T_{y} \, N \approx {\mathbb R}^{n+1}$ such that $\widetilde{V} = \left(Q \circ {\rm exp}_{y}^{-1}\right)_{\#} \, V \res {\mathcal N}_{\rho}(y)$ for some $V \in \overline{\mathcal V}$ and orthogonal rotation $Q \, : \, {\mathbb R}^{n+1} \to {\mathbb R}^{n+1}$, then 
$\hat H_{\widetilde{V}} \leq {\rm sup}_{N} \, |g|$,  where $\hat H_{\widetilde{V}} = |H_{V} \circ Q \circ \exp_{y}|.$ Hence, for $E_{\sigma}$ as in Theorem~\ref{estimates}, we have that the excess ${\hat E}_{\sigma}$ as in Theorem~\ref{BWregularity} satisfies 
$${\hat E}_{\sigma} \leq \int_{(B_{1/2}^{n}(0) \times {\mathbb R}) \cap B_{1}^{n+1}(0)} |x^{n+1}|^{2} \, d\|\eta_{0, \sigma \#} \widetilde{V}\| + (1 + \sup_{N} \, |g|) \sigma \leq E_{\sigma}$$ 
provided that (in the definition of $E_{\sigma}$) we choose $\mu \geq 1+ \Gamma$ $(\geq 1 + \sup_{N} \, |g|)$. 

First consider the case $q=1$. In this case, subject also to the rest of the hypotheses in part (ii), we have the validity of the assertion in part (ii) by the Allard regularity theorem and the $C^{2, \alpha}$ Schauder theory for uniformly elliptic equations.

 Now let $q \geq 2$, and assume by induction that part (ii) is valid with $q^{\prime}$ in place of $q$ for any $q^{\prime} \in \{1, 2, \ldots, q-1\}.$ By 
 arguing  exactly as in the proof of Theorem~\ref{limit-regularity}, part (iii), taking: (a) ${\mathcal V} = \overline{\mathcal V}$ with $U_{V} = N$ for each $V \in \overline{\mathcal V}$; (b) this induction hypothesis in place of the induction hypothesis ($\ast$) therein; (c) 
 $\overline{\rho}$ in place of $\rho_{y}$ and (d) $\delta = 0,$ we see (by employing Theorem~\ref{BWregularity} and Theorem~\ref{nonvar-SS} as in that argument) that part (ii) of the current theorem holds, in the first instance, with $C^{1, \alpha}$ functions $u_{j}$, $j=1, 2, \ldots q.$ By Theorem~\ref{thm:inductive_higher_reg} we see that these $u_{j}$ are in fact of class  $C^{2, \alpha}.$ Finally, the desired $C^{2, \alpha}$ estimate in part (ii) follows from standard Schauder estimates. 
 
 For part (i), note that by Theorem~\ref{limit-regularity} we have that any varifold in $\widetilde{\mathcal V}$ has quasi-embedded $PMC (g, 0)$ structure locally away from a closed set of Hausdorff dimension at most $n-7$. 
 The assertion of part (i) now follows from the slicing argument of \cite[pp.~785--787]{SS} (the first part of the proof of \cite[Theorem 2]{SS}), taking the just established conclusion (ii) in place of \cite[Theorem 1]{SS}. 
 \end{proof}

\section{A min-max construction of Allen--Cahn solutions}
\label{minmax_setup}

For notational convenience, we will work, in Sections~\ref{minmax_setup} and \ref{proof}, with the functional $\Fce=\ca{F}{\eps, g}$ given by 
$${\mathcal F}_{\e, \, g}(u) = \int_N \eps \frac{|\nabla u|^2}{2} + \int_N \frac{W(u)}{\eps} -\int_N g\, u$$
rather than with ${\mathcal F}_{\e, \, \sigma g}(u)$. This means that, given $g>0$ in $C^{1,1}(N)$, we will establish the existence of a quasi-embedded immersed hypersurface with mean curvature $\frac{g}{\sigma} \nu$. This is of course equivalent to proving Theorem~\ref{thm:existence} for 
$g >0$, $g \in C^{1, 1}(N)$.  Recall that  $W:\R\to [0,\infty)$ is a fixed double-well potential, i.e.\ a non-negative function of class $C^2$ having precisely two non-degenerate minima at $\pm 1$ with values $W(\pm 1) = 0$ and  we require that $c \leq W^{\prime\prime}(t) \leq C$ for some constant $C, c >0$ and all $t \in {\mathbb R} \setminus [-2, 2].$  

In this section we will choose two functions $a_{\eps}$ and $b_{\eps}$ as valley points for the functional $\ca{F}{\eps}$ and prove that the class of continuous paths in $W^{1,2}(N)$ that joins $a_{\eps}$ to $b_{\eps}$ satifies a suitable ``wall condition'' (mountain pass condition) to produce minmax critical points. 
The functions $a_{\eps}$ and $b_{\eps}$ converge uniformly on $N$ as $\eps\to 0$ respectively to the constants $-1$ and $+1$. The value of the functional at $-1$ and $+1$ is respectively $\Fce(-1) = \int_N g$ and $\Fce(+1) = -\int_N g$. We will consider in the next lemma affine subspaces of the form 
$$\Pi_\delta=\left\{u\in W^{1,2}(N): \int_N g\,u = -\int_N g + \delta\right\}.$$
In Proposition \ref{Prop:mountainpass} we will see that an affine subspace of this type provides a suitable ``wall'' for a mountain pass construction. We first introduce some one-dimensional profiles that play a role in the forthcoming arguments.

\medskip

\textit{One-dimensional profiles - single transition}. We denote by $\Het:\R\to \R$ the monotonically increasing solution to the Allen--Cahn ODE $u''-W'(u)=0$ such that $\lim_{r\to \pm\infty} \Het(r) = \pm 1$, with $\Het(0)=0$. (If we choose $W$ so that it agrees with the standard potential $\widetilde{W}(t) = \frac{(1-t^2)^2}{4}$ on $[-2, 2],$ then we have $\Het(r)=\tanh\left(\frac{r}{\sqrt{2}}\right)$; note that since $\tanh (\cdot) \in [-1, 1]$, modifying $\widetilde{W}$ outside the interval $[-2, 2]$ so as to arrange quadratic growth for $W$---as is needed here---does not affect this solution.) For any given $\eps>0$, the rescaled function $\He(r)=\Het\left(\frac{r}{\eps}\right)$ solves the ODE ${\eps} u''-\frac{W'(u)}{\eps}=0$. We will need a truncated version $\OHet$ of $\Het$ in our construction: this approximate solution is set to be constant ($\pm 1$) away from $[-6  |\log\eps|, 6  |\log\eps|]$. This is convenient for the construction of an ``Allen--Cahn approximation'' of a hypersurface, i.e.~a function on $N$ that takes on large sets the values $\pm 1$ and presents a single transition between these two values along the hypersurface in question, and such that its Allen--Cahn energy $\Ece$ approximates the area of the hypersurface. Similar truncations have often been used in the literature and we refer to \cite{B1} for further details. For $\Lambda=3|\log\eps|$ define

$$\OHet(r) = \chi(\Lambda^{-1} r -1)\Het(r) \pm (1-\chi(\Lambda^{-1}|r|-1)),$$
where $+1$ or $-1$ is chosen respectively on $r>0$, $r<0$ and $\chi$ is a smooth bump function that is $+1$ on $(-1,1)$ and has support equal to $[-2,2]$. With this definition, $\OHet=\Het$ on $(-\Lambda, \Lambda)$, $\OHet=-1$ on $(-\infty, -2\Lambda]$, $\OHet=+1$ on $[2\Lambda,\infty)$. Moreover the function $\OHet$ satisfies  $\|\OHet''-W'(\OHet)\|_{C^2(\R)} \leq C \eps^3$, for $C>0$ independent of $\eps$. (Note also that $\OHet''-W'(\OHet)=0$ away from $(-2\Lambda, -\Lambda) \cup (\Lambda,2\Lambda)$.)

For $\eps<1$, we rescale these truncated solutions and let $\OHet^{\eps}(\cdot)=\OHet\left(\frac{\cdot}{\eps}\right)$. Note that $\OHet^{\eps}$ solves $\|\eps\left(\OHet^{\eps}\right)''-\frac{W'(\OHet^{\eps})}{\eps}\|_{C^2(\R)} \leq C \eps^2$ and $\eps\left(\OHet^{\eps}\right)''-\frac{W'(\OHet^{\eps})}{\eps}=0$ on $(-\eps\Lambda, \eps\Lambda)$, $\OHet^{\eps}=+1$ on $(2\eps\Lambda, \infty)$, $\OHet^{\eps}=-1$ on $(-\infty,-2\eps\Lambda)$.

Using these facts and recalling that $\Ece(\Het^{\eps})=2\s$ we get $\Ece(\OHet^{\eps})=2\s+O({\eps}^2)$. (The function $O({\eps}^2)$ is bounded by $C{\eps}^2$ for all $\eps$ sufficiently small, with $C$ independent of $\eps$.)

\medskip

\begin{lem}
 \label{lem:wall_g_pos}
There exists $\delta\in(0,2\int_N g)$ such that the following is true. For any $\eps_j \to 0$ there exists $\delta_j \to \delta$ such that
 $$\liminf_{j\to \infty}\left(\inf_{u\in \Pi_{\delta_j}} \ca{F}{\eps_j}(u)\right)>\int_N g\,\,\left( = \ca{F}{\eps_j}(-1) \geq \ca{F}{\eps_j}(+1)\right).$$
\end{lem}

\begin{proof}
Within this proof we will write $\mathbb{I}_A$ for the characteristic function of $A$, and $|A|$ for $\mathcal{H}^{n+1}(A)$. Pick a point $z$ where $g$ achieves its maximum $g_M$ and consider a geodesic ball $B$ centred at $z$ and such that $g>g_M/2$ on $B$. Then $\mathbb{I}_B -\mathbb{I}_{N\setminus B}$ can be approximated by a function $v_\varepsilon$, for every $\eps$ small enough, as follows. For each sufficiently small $\varepsilon$ (we prescribe $120\eps|\log\eps|$ to be smaller than the radius of $B$) the function $v_{\varepsilon}$ is defined to be $\overline{\Het}^{\eps}\circ \text{dist}_{\p B}$, where $\text{dist}_{\p B}$ is the signed distance function to $\p B$, taken to be positive inside $B$. The distance function is Lipschitz on $N$ and satisfies $|\nabla \text{dist}_{\p B}|=1$; this permits to compute $\Ece(v_{\varepsilon})$ using the coarea formula (with slicing function given by $\text{dist}_{\p B}$) and the estimate on $\Ece(\overline{\Het}^{\eps})$ given before the statement of the lemma, to obtain that $\Ece(v_{\varepsilon})=2\sigma \mathcal{H}^n(\p B)+O(\eps|\log\eps|)$.\footnote{We will carry out more subtle computations in the same spirit in Section \ref{proof}.} We thus have $\frac{1}{2\s}\Ece(v_\varepsilon) \to {\Hc}^n(\p B)$ and $v_\varepsilon \stackrel{L^1}{\to} \mathbb{I}_B -\mathbb{I}_{N\setminus B}$. 
Then $\int_N g v_\varepsilon$ converges, as $\eps \to 0$, to $\int_N g(\mathbb{I}_B -\mathbb{I}_{N\setminus B})=-\int_N g + 2\int_B g$. We set $\delta_B=2\int_B g$, therefore $\int_N g v_\varepsilon \to -\int_N g+\delta_B$ and $|B| g_M\leq \delta_B \leq 2|B| g_M$. We will prove that for a sufficiently small choice of $B$ (if $g>g_M/2$ is true on a ball centred at $z$, it is true on any ball centred at $z$ and contained in the first), the Lemma holds with $\delta=\delta_B$ and with $\delta_j = \int_N g v_{\varepsilon_j} + \int_N g$ (for an arbitrary given sequence $\eps_j \to 0$). The choice of $\delta_j$ is made to ensure $v_{{\eps}_j}\in \Pi_{\delta_j}$, and $\delta_j \to \delta$.

Note that, for each $\eps_j$,
$$\inf_{u \in \Pi_{\delta_j}} \ca{F}{\eps_j}(u) = \left( \inf_{u \in \Pi_{\delta_j}} \ca{E}{\eps_j}(u)  \right) - \delta_j+\int_N g$$
and the minimizing sequences are the same for $\inf_{u \in \Pi_{\delta_j}} \ca{F}{\eps_j}(u)$ and for $\inf_{u \in \Pi_{\delta_j}} \ca{E}{\eps_j}(u)$ (because the two functionals differ by the constant $-\delta_j+\int_N g$ on $\Pi_{\delta_j}$). Since $\Ece(v_\varepsilon) \to 2\s{\Hc}^n(\p B)$ as $\eps \to 0$ and $v_{\eps_j}\in \Pi_{\delta_j}$ by construction, we can see that there exists an upper bound for $\inf_{u \in \Pi_{\delta_j}} \ca{F}{\eps_j}(u)$ (and for $\inf_{u \in \Pi_{\delta_j}} \ca{E}{\eps_j}(u)$) that is independent of $j$. Since $\Ece\geq 0$, we also have a lower bound for $\inf_{u \in \Pi_{\delta_j}} \ca{F}{\eps_j}(u)$ independently of $j$.

Pick $u_j$ such that $\ca{F}{\eps_j}(u_j) - \inf_{u \in \Pi_{\delta_j}} \ca{F}{\eps_j}(u)$ converges to $0$ (as $j \to \infty$). Then $\ca{F}{\eps_j}(u_j)$ is uniformly bounded above and therefore so is $\ca{E}{\eps_j}(u_j)$. The latter condition guarantees, thanks to a standard argument (\cite{HT}, \cite{MM}) that we now recall, that $u_j$ converges in $L^1$ to a $BV$ function $u_\infty$ that takes only the values $\pm 1$.

\medskip

The uniform bound on $\ca{E}{\eps_j}$ and the fact that $W(x) \geq \kappa (x-1)^2$, for some $\kappa>0$ and for $x \geq 2$, imply that $\|u_j\|_{L^2}$ is uniformly bounded. Set $\Phi(s) =\int_0^s \sqrt{\frac{W(\widetilde{s})}{2}}d\widetilde{s}$ and $\sigma=\int_{-1}^1 \sqrt{\frac{W(s)}{2}}ds$ (so that $\Phi(\pm 1)=\pm \s/2$) and let $w_j = \Phi(u_j)$. Then $\|w_j\|_{BV(N)}$ is uniformly bounded (with $\int_N|\nabla w_j| \leq \frac{\Ece(u_j)}{2}$) and therefore by the BV compactness theorem, upon passing to a subsequence that we do not relabel, there exists $w_\infty \in BV(N)$ such that $w_{j} \to w_{\infty}$ in $L^{1}(N)$ and $\int_{N} |Dw_{\infty}| \leq \liminf_{j \to \infty} \, \int_{N} |Dw_{j}|$. In particular $w_j \to w_\infty$ a.e., so that $u_j  = \Phi^{-1}(w_{j}) \to u_\infty:=\Phi^{-1}(w_\infty)$ a.e. The uniform bound on $\ca{E}{\eps_j}$ and Fatou's lemma imply that $\int_N W(u_\infty) \leq \liminf_{j\to \infty}\int_N W(u_j) =0$, from which it follows that $u_\infty=\pm1$ a.e. Then $\frac{2}{\s}w_\infty= u_\infty\in BV(N)$.

The convergence $w_j \stackrel{L^1}{\to} w_\infty$ implies in particular that $w_j \to w_\infty$ in measure. Since $\Phi^{-1}$ is uniformly continuous, for any $s>0$ we can choose $\widetilde{s}>0$ such that $|w_j-w_\infty|\leq \widetilde{s}$ implies $|u_j - u_\infty|\leq s$. Therefore for every $s>0$ we have $|\{|u_j-u_\infty|\geq s\}| \to 0$, i.e.~$u_j \to u_\infty$ in measure. We can then prove $u_j \stackrel{L^1}{\to} u_\infty$ as follows.

$$\int_N |u_j-u_\infty| = \int_{\{|u_j-u_\infty|\geq s\}}|u_j-u_\infty| + \int_{\{|u_j-u_\infty|<s\}}|u_j-u_\infty| \leq $$ $$\leq \int_{\{|u_j-u_\infty|\geq s\}}|u_j| + \int_{\{|u_j-u_\infty|\geq s\}}|u_\infty| + \int_N s \leq$$
$$ \leq |\{|u_j-u_\infty|\geq s\}|^{1/2}\left(\int_N |u_j|^2\right)^{1/2} + |\{|u_j-u_\infty|\geq s\}| + s|N|$$
and all three terms go to $0$ as $s \to 0$.

\medskip

Since $u_j \in \Pi_{\delta_j}$ we conclude that $\int_N g\, u_\infty = -\int_N g + \delta$. As $u_\infty=\pm 1$ a.e.~and is $BV$, we must have that there exists a set $D\subset N$ with finite perimeter such that $\mathbb{I}_D -\mathbb{I}_{N\setminus D}=u_\infty$ and $\int_D g = \int_B g$. This implies that $|D| \geq \frac{1}{2} |B|$ (by the choice of $B$). Denoting by $\overline{w_j}$ the average of $w_j$ and recalling the Sobolev--Poincar\'{e} inequality we have:

$$\frac{\Ece(u_j)}{2} \geq \int_N |\nabla w_j| = \int_N |\nabla (w_j - \overline{w_j})| \geq C_{SP}\left(\int_N  |w_j - \overline{w_j})|^{\frac{n+1}{n}}\right)^{\frac{n}{n+1}}.$$
By Fatou's lemma (and the $L^1$ and a.e.~convergence $w_j \to w_\infty=\mathbb{I}_D -\mathbb{I}_{N\setminus D}$, which also gives $\overline{w_j} \to \frac{1}{|N|}(2|D| - |N|)$), we get
$$\left(\int_N  \left|2\mathbb{I}_D-\frac{2|D|}{|N|}\right|^{\frac{n+1}{n}}\right)^{\frac{n}{n+1}} \leq \liminf_{j\to \infty}\left(\int_N  |w_j - \overline{w_j})|^{\frac{n+1}{n}}\right)^{\frac{n}{n+1}} .$$
Computing the left-hand-side (giving up the integral over $N \setminus D$) we get 
\be
 \label{eq:bound_D}
 \left(2-\frac{2|D|}{|N|}\right)|D|^{\frac{n}{n+1}}  \leq \frac{1}{2C_{SP}}  \liminf_{j\to \infty} \Ece(u_j).
\ee

\medskip

Assume that the conclusion of the lemma fails for a certain $\delta$ (i.e.~for a certain choice of $B$). In other words, we assume that for some ${\eps}_j\to 0$ we have (for the choice of $\delta_j$ specified above) $\liminf_{j\to \infty}\left(\inf_{u\in \Pi_{\delta_j}} \ca{F}{\eps_j}(u)\right)\leq \int_N g$. This means that there exists a subsequence (not relabeled) ${\eps}_j\to 0$ such that (for $u_j$ chosen above) $\lim_{j\to \infty} \ca{F}{\eps_j}(u_j)\leq \int_N g$. Then 
$$\int_N g \geq \liminf_{j\to \infty}\ca{F}{\eps_j}(u_j) = \liminf_{j\to \infty}\ca{E}{\eps_j}(u_j) +\int_N g -\delta, \,\,\,\,\text{i.e.}$$
\be
 \label{eq:sublineardelta}
 \liminf_{j\to \infty}\ca{E}{\eps_j}(u_j) \leq \delta.
\ee
Recalling $|D|\geq \frac{1}{2}|B|$ and $g_M |B| \leq \delta \leq 2g_M |B|$, (\ref{eq:bound_D}) and (\ref{eq:sublineardelta}) give that $\delta^{\frac{n}{n+1}}\lesssim \delta$, a contradiction if $B$ is chosen small enough (to make $\delta$ suitably small).
\end{proof}

\noindent
\textit{Choice of valley points $a_{\eps}$ and $b_{\eps}$}. There exist two functions $a_{\eps}$ and $b_{\eps}$ on $N$ that solve $\ca{F'}{\eps}=0$ with $-1<a_{\eps}<-1+\eps c$ and $b_{\eps}>+1$ and $a_{\eps} \to -1$ $b_{\eps}\to +1$ uniformly on $N$ as $\eps\to 0$, where $c>0$ depends on $W$ and on the maximum of $g$. To see this, consider the constant $-1$ and evaluate $-\ca{F'}{\eps}(-1)=\eps \Delta(-1)-\frac{W'(-1)}{\eps}+g=g>0$. For the constant $(-1+c\eps)$ on the other hand we have $-\ca{F'}{\eps}(-1+c\eps)=-\frac{W'(-1+c\eps)}{\eps}+g$. Recall that $W'(-1+c\eps)\approx c C_W \eps$ ($W$ is quadratic around $-1$); choosing $c$ sufficiently large (depending only on $W$ and $g$) we can ensure that $g<c C_W$, and therefore that $-\ca{F'}{\eps}(-1+c\eps)<0$. Therefore, by considering the negative gradient flow of $\Fce$ with initial condition given by the constant $-1$, we obtain a (stable) solution $a_{\eps}$ to $\ca{F'}{\eps}=0$ that lies between $-1$ and $-1+c\eps$ (the latter acts as an upper barrier by the maximum principle). Similarly, computing $-\ca{F'}{\eps}(+1)=g>0$ and $-\ca{F'}{\eps}(1+c\eps)=-\frac{W'(1+c\eps)}{\eps}+g<0$, we obtain that there is a (stable) solution $b_{\eps}$ to $\ca{F'}{\eps}=0$ that lies between $1$ and $1+c\eps$ and we can obtain $b_{\eps}$ via negative gradient flow of $\Fce$ with initial condition given by the constant $+1$. We will use the functions $a_{\eps}$ and $b_{\eps}$ as valley points for the class of admissible paths.

\begin{Prop}[Existence of a mountain pass solution]
\label{Prop:mountainpass}
For $\eps>0$ let $\Gamma$ denote the collection of all continuous paths $\gamma:[-1,1]\to W^{1,2}(N)$ such that $\gamma(-1)=a_{\eps}$ and $\gamma(1)=b_{\eps}$. Then there exists $\eps_0>0$ such that for each $\eps<\eps_0$
$$\inf_{\gamma \in \Gamma} \sup_{u\in \gamma([-1,1])} \ca{F}{\eps}(u) = \beta_\varepsilon$$ is a critical value, i.e.~there exists $u_\varepsilon \in W^{1,2}(N)$ that is a critical point of 
$\ca{F}{\eps}$ with ${\mathcal F}_{\eps}(u_{\eps}) = \beta_{\eps};$ moreover, $u_{\eps}$ has Morse index $\leq 1$.
\end{Prop} 

\begin{proof}
(i) We show that $\beta_\varepsilon > \ca{F}{\eps}(a_{\eps})$ and $\beta_\varepsilon > \ca{F}{\eps}(b_{\eps})$. By the choice of $a_{\eps}$ and $b_{\eps}$, we have $ \ca{F}{\eps}(a_{\eps})< \ca{F}{\eps}(-1)=\int_N g$ and $\ca{F}{\eps}(b_{\eps})< \ca{F}{\eps}(+1)=-\int_N g$, so it will suffice to prove that for all sufficiently small $\eps$ we have $\beta_\varepsilon > \int_N g$. Observe that $\int_N a_{\eps} g<-\int_N g +c \eps|N|\|g\|_{L^\infty}$ and $\int_N b_{\eps}>\int_N g$. The choice of $\delta \in(0,\int_N g)$ in Lemma \ref{lem:wall_g_pos} can be made independently of $\eps$, in particular $\delta> c \eps|N|\|g\|_{L^\infty}$ for $\eps$ sufficiently small. Then the continuity of $u\to \int_N u$ in $W^{1,2}(N)$ guarantees that any continuous path joining $a_{{\eps}_j}$ to $b_{{\eps}_j}$ must cross $\Pi_{\delta_j}$ (for $j$ large enough), thus ensuring the mountain pass condition.

\noindent (ii) We show that the Palais--Smale condition is satisfied (at fixed $\eps$) on $\ca{F}{\eps}$-bounded sequences, i.e.~that for any sequence $\{u_m\}_{m=1}^\infty$ such that $\ca{F}{\eps}(u_m)$ is uniformly bounded in $m$ and such that $\ca{F'}{\eps}(u_m) \to 0$ (as elements of the dual of $W^{1,2}(N)$) there exists a subsequence of $u_m$ converging strongly in $W^{1,2}(N)$. Note that for $|u|\geq 2$ we have $W(u) - gu \geq \kappa u^2 - \|g\|_\infty |u|$ for some $\kappa>0$ and thus for $|u|\geq C_{g,W}$ we have $W(u) - gu \geq \frac{\kappa}{2}u^2$. Therefore the assumption $\ca{F}{\eps}(u_m) \leq K$ implies that $\|\nabla u_m\|_{L^2(N)} \leq \frac{K}{\eps}$ and that $\|u_m\|_{L^2(N)} \leq C_{W,g,K}$. Rellich-Kondrachov theorem provides a subsequence (not relabeled) that converges weakly in $W^{1,2}(N)$ to a function $u$. Recall that $\ca{F'}{\eps}(u_m)(\psi) = \int_N \eps \nabla u_m \nabla \psi + \frac{W'(u_m)}{\eps} \psi - g \psi$. By the $L^2$-convergence of $u_m$ to $u$ and the fact that $W'$ is linear at $\pm \infty$ we obtain that $W'(u_m)\to W'(u)$ in $L^2$.
Using the weak convergence $u_m \stackrel{W^{1,2}}{\rightharpoonup} u$ we get $\lim_{m\to \infty} \ca{F'}{\eps}(u_m)(\psi)=\ca{F'}{\eps}(u)(\psi)$. On the other hand, the assumption on $\ca{F'}{\eps}(u_m)$ gives $\ca{F'}{\eps}(u)(\psi)=0$, i.e.~$u$ is a critical point of $\ca{F}{\eps}$. The boundedness of $u_m-u$ in $W^{1,2}$ gives that $\ca{F'}{\eps}(u_m) (u_m-u) \to 0$ and therefore $\ca{F'}{\eps}(u_m) (u_m-u)  - \ca{F'}{\eps}(u)(u_m-u) = \int_N \eps|\nabla(u_m - u)|^2 + \int_N\frac{1}{\eps}(W'(u_m)-W'(u))(u_m-u)-\int_N g(u_m-u) \to 0$. The second and third integrals go to $0$ by the strong $L^2$-convergence $u_m \to u$, therefore $\int_N |\nabla(u_m - u)|^2 \to 0$, concluding that $u_m \to u$ in $W^{1,2}$ (strongly).

\noindent (iii) The proposition now follows from standard minmax theory since the class of continuous paths $\gamma:[-1,1]\to W^{1,2}(N)$ such that $\gamma(-1)=a_{\eps}$ and $\gamma(1)=b_{\eps}$ is invariant under the flow induced by the negative gradient of $\ca{F}{\eps}$ (see e.g.~\cite{Struwe} or \cite{Gho}).
\end{proof}

What is left to do is to make sure that the energy $\ca{E}{\eps}(u_{\eps})$ associated to the mountain pass solution provided by Proposition \ref{Prop:mountainpass} stays uniformly (in $\eps$) bounded above and away from zero. This will guarantee that, as $\eps \to 0$ (subsequentially), the energy distribution of $u_{\eps}$ gives rise to a non-trivial varifold (with finite mass). In Lemma \ref{lem:upperbound} we discuss the upper bound. The lower bound (Lemma \ref{lem:lowerbound}) will be immediate from Lemma \ref{lem:wall_g_pos}.

\begin{lem}
 \label{lem:upperbound}
There exist $\eps_0>0$ and $K>0$ such that $\sup_{N} \, |u_{\eps}| + \ca{E}{\eps}(u_{\eps})\leq K$ for every $\eps<\eps_0$, where $u_{\eps}$ is as in Proposition \ref{Prop:mountainpass}.
\end{lem}

\begin{proof}
\noindent {\it Step 1}. We produce, for every $\eps$ sufficiently small, a continuous path joining $v_1$ to $v_2$ and such that the maximum of $\ca{F}{\eps}$ on the path is attained with a value that is bounded above independently of $\eps$. There is a fairly standard way to do this by using a sweepout via level sets of a Morse function, see e.g.~\cite{Gua}. For fixed $\varepsilon$, to each level set $\Sigma$ one associates a $W^{1,2}$ function $f_{\Sigma}$ on $N$ such that the Allen--Cahn energy $\Ece(f_{\Sigma})$ is approximately $2\sigma \mathcal{H}^n(\Sigma)$ (the difference is an infinitesimal of $\varepsilon$).\footnote{The function $f_{\Sigma}$ is constructed such that $\Sigma$ is its nodal set and such that the profile of the function in the normal direction to $\Sigma$ is given by $\Het^{\eps}$. More precisely, $f_{\Sigma}$ is the composition of $\Het^{\eps}$ with the signed distance function to $\Sigma$. The possible singularities in $\Sigma$ are isolated and have an explicit structure (thanks to the Morse condition), which permits to handle them in elementary fashion (\cite[Section 9.6]{Gua}). We do not give further details here, since in Section \ref{proof} we will carry out (explicitly) subtler constructions of this kind: we will need to handle a singular set of unknown structure and implement other operations that are not present in the construction just sketched.} Using this procedure, the sweepout (that is, the one-parameter family of level sets) induces a continuous path in $W^{1,2}(N)$, that joins $-1$ to $+1$ (see \cite[Section 7.4]{Gua}). One then modifies this path so that it joins $a_{\eps}$ to $b_{\eps}$ (remaining continuous): to this end, it suffices to compose with two short paths at the endpoints, one joining $a_{\eps}$ to $-1$ (through constant functions) and one joining $+1$ to $b_{\eps}$ (through constant functions). We denote the resulting path by $t \to u_t\in W^{1,2}(N)$, for $t\in[-1,1]$. The upper bound on the measure of the level sets of the Morse function then implies an upper bound on $\Ece(u_t)$, independently of $\varepsilon$. In order to infer from this an upper bound on $\Fce(u_t)$ (independent of $\varepsilon$) we note that the term $\int_N g u_t=\Ece(u_t) - \Fce(u_t)$ is bounded above and below independently of $\varepsilon$, because $u_t$ is bounded between $-1$ and $1$ by construction. The bound on the path just discussed then implies a uniform upper bound for $\ca{F}{\eps}(u_{\eps})$, independently of $\eps$, by the minmax characterization of $u_{\eps}$. 

\noindent {\it Step 2}. A critical point $u_{\eps} \in W^{1,2}(N)$ of $\ca{F}{\eps}$ solves the weak formulation of the semilinear elliptic PDE $\eps \Delta u - \frac{W'(u)}{\eps} = -g$. By elliptic theory $u_{\eps}$ is $C^{2,\alpha}$. Note (arguing as we did when we chose $a_{\eps}$ and $b_{\eps}$) that any constant smaller than or equal to $-1$ is a lower barrier and any constant larger than $1+c\eps$ is an upper barrier. The maximum principle therefore implies that any solution of the PDE (in particular $u_{\eps}$) is bounded between $-1$ and $1+c\eps$, which gives an $L^\infty$ bound $\|u_{\eps}\|_{L^\infty(N)}\leq 2$ uniformly in $\eps$.

\noindent {\it Step 3}. Since $\ca{E}{\eps}(u_{\eps}) \leq \ca{F}{\eps}(u_{\eps}) + |N| \|g\|_{\infty} \|u_{\eps}\|_{\infty}$, we conclude from steps 1 and 2 a uniform upper bound for $\ca{E}{\eps}(u_{\eps})$.
\end{proof}

\begin{lem}
 \label{lem:lowerbound}
There exist $\eps_0>0$ and $L>0$ such that $\ca{E}{\eps}(u_{\eps})\geq L$ for every $\eps<\eps_0$, where $u_{\eps}$ is as in Proposition \ref{Prop:mountainpass}.
\end{lem}

\begin{proof}
We have $\liminf_{\eps \to 0} \beta_{\eps} > \int_N g$ by definition of $\beta_{\eps}=\ca{F}{\eps}(u_{\eps})$ and by Lemma \ref{lem:wall_g_pos}. Then $\ca{E}{\eps}(u_{\eps})=\beta_{\eps}+\int_N g u_{\eps}$ and we have (see the proof of step 2 in Lemma \ref{lem:upperbound}) that $u_{\eps}>-1$, so $\int_N g u_{\eps}\geq -\int_N g$.
Then we conclude that $\liminf_{\eps\to 0}\ca{E}{\eps}(u_{\eps})>0$.
\end{proof}

\section{Proof of the existence theorem for positive $g \in C^{1, 1}$}
\label{proof}

\begin{oss}
Recall that, given $g>0$, $g\in C^{1,1}(N)$, we will prove the existence conclusion of Theorem~\ref{thm:existence} with $\frac{g}{\s}$ in place of $g$, because (for notational convenience) we work with the functional ${\mathcal F}_{\e, g}$ (rather than with ${\mathcal F}_{\e, \sigma g}$, which would lead to the existence conclusion with mean curvature $g$). To make notation lighter, in this section we will denote by $\Fce$ the functional ${\mathcal F}_{\e, g}$.
\end{oss}

\begin{proof}[First part of the proof of Theorem \ref{thm:existence}]
If there exists a sequence ${\eps}_{j}\to 0^{+}$ such that the min-max critical points $u_{{\eps}_j}$ given by Proposition \ref{Prop:mountainpass} taken with $\eps = \eps_{j}$ have the property that $u_{{\eps}_j}\to u_{\infty}$ in $L^1(N)$ and $u_\infty$ is not identically $-1$, then Theorem~\ref{thm:regularity} implies that $\p\{u_\infty=+1\}$ is a closed hypersurface that is quasi-embedded away from a possible singular set $\Sigma$ of Hausdorff dimension $\leq n-7$ and with mean curvature given on $\p\{u_\infty=+1\}\setminus \Sigma$ by $\frac{g}{\s}\nu$, where $\nu$ is the inward pointing normal to $\{u_\infty=+1\}$. In particular, $\p\{u_\infty=+1\}$ is, away from $\Sigma$, the image of a two-sided $C^2$ immersion, with unit normal $\nu$ and mean curvature $\frac{g}{\s}\nu$. Theorem \ref{thm:existence} is thus proved in the case in which $u_\infty\not\equiv -1$ for some sequence $(u_{{\eps}_j})$.
\end{proof}

In order to establish Theorem~\ref{thm:existence} for $g \in C^{1, 1}(N)$ with $g >0$, we only need to address the case in which $u_\infty\equiv -1$ for every convergent subsequence of every sequence of min-max critical points $u_{{\eps}_j}$ generated by Proposition~\ref{Prop:mountainpass}. In fact, to complete the proof, it will suffice to consider a single such convergent sequence $(u_{{\eps}_j})$ for which $u_\infty\equiv -1$. By possibly passing to a subsequence, we may also assume that $V^{{\eps}_j}\to V$ as varifolds, with $V$ a stationary integral $n$-varifold on $N$. By 
Theorem~\ref{thm:regularity} $V$ is, away from a singular set ${\rm sing} \, V$ of Hausdorff dimension $\leq n-7$ (with ${\rm sing} \, V$ empty if $2 \leq n \leq 6$ and finite if $n=7$), an embedded, smooth minimal hypersurface $M$ with locally constant even multiplicity. We will write $\overline{M}$ for the closure of $M$; note that $\overline{M}$ coincides with $\spt{V}$, $M={\rm reg} \, V$ and ${\rm sing} \, V = \overline{M}\setminus M$. Recall that we have, in this situation, that 
$$\lim_{j \to \infty} \frac{1}{2\s}\ca{F}{{\eps}_j}(u_{{\eps}_j}) =\|V\|(N)+\frac{1}{2\s}\int_N g \geq 2\mathcal{H}^n(M)+\frac{1}{2\s}\int_N g.$$ 

In the forthcoming sections, our goal will be to prove the following:

\begin{Prop}
\label{Prop:main}
Let $u_{{\eps}_j}$ be as in the preceding paragraph.
Then there exist $v_{{\eps}_j}:N\to \R$ that solve $\ca{F'}{{\eps}_j}(v_{{\eps}_j})=0$ and $\ca{F''}{{\eps}_j}(v_{{\eps}_j})\geq 0$ (i.e.~stable critical points of $\ca{F}{\varepsilon_j}$) with $\liminf_{{\eps}_j\to 0} \ca{E}{{\eps}_j}(v_{{\eps}_j})>0$ and $\limsup_{{\eps}_j\to 0} \ca{E}{{\eps}_j}(v_{{\eps}_j})<\infty$; moreover, there exists a (fixed) non-empty open set that is contained in $\{v_{{\eps}_j}>\frac{3}{4}\}$ for all ${\eps}_j$.
\end{Prop}

\begin{proof}[Second (final) part of the proof of Theorem \ref{thm:existence}, assuming Proposition \ref{Prop:main}]
Let $u_{{\eps}_j}$ be as in the paragraph preceding Proposition~\ref{Prop:main} and let $v_{{\eps}_j}$ be the functions given by Proposition \ref{Prop:main}. Owing to the condition that 
$\{v_{{\eps}_j}>\frac{3}{4}\}$ contains a fixed non-empty open set, we obtain that any (subsequential) $L^1$-limit $v_\infty$ of $v_{{\eps}_j}$ equals $+1$ on this open set. In view of the upper and lower bounds on $\ca{E}{{\eps}_j}(v_{{\eps}_j})$ and stability of $v_{\eps_{j}}$, we can apply Theorem \ref{thm:regularity} to any subsequential limit of the sequence of associated varifolds $V^{v_{{\eps}_j}}$. We obtain that the multiplicity-$1$ varifold associated with the reduced boundary of $\{v_\infty = +1\}$ is non-trivial and provides the prescribed-mean-curvature hypersurface needed to complete the proof of Theorem~\ref{thm:existence}.
\end{proof}

The proof of Proposition \ref{Prop:main} will be achieved, in the forthcoming sections, by exploiting the minmax characterisation of $u_{\eps}$ and the geometric fact that a minimal hypersurface with multiplicity $2$ is not a stationary point for the geometric functional $A-\text{Vol}_{\frac{g}{\s}}$, when viewed as an immersion from its double cover: the term $A$ measures the hypersurface area and the term $\text{Vol}_{\frac{g}{\s}}$ is the enclosed $\frac{g}{\s}$-volume, which is $\int_E \frac{g}{\s}$ when the hypersurface in question is the reduced boundary of a Caccioppoli set $E\subset N$. (See \cite{BW2} for the more general definition that applies also to the case in which the hypersurface is not a boundary. This is the natural functional whose critical points are hypersurfaces with scalar mean curvature prescribed by $\frac{g}{\s}$.) More specifically, we will proceed as follows. Recall that, under the hypotheses of Proposition~\ref{Prop:main}, we have a sequence $u_{{\eps}_j}$ of min-max critical points whose associated varifolds converge to a minimal hypersurface $M$ endowed with locally constant even multiplicity, and $\limsup_{{\eps}_j \to 0}\frac{1}{2\s}\ca{F}{{\eps}_j}(u_{{\eps}_j}) \geq  2\mathcal{H}^n(M)+\frac{1}{2\s}\int_N g$. We will exhibit, for all sufficiently small $\eps={\eps}_j$, a continuous path $\gamma:[-1,1]\to W^{1,2}(N)$ with $\gamma(-1)=a_{\eps}$, with the second endpoint $\gamma(+1)$ a stable solution to $\ca{F'}{\eps}=0$, and with an energy bound $\Fce(\gamma(t))\leq 2(2\s)\mathcal{H}^n(M) - c_M+\int_N g$ for all $t\in[-1,1]$ and for some $c_M>0$ independent of $\eps$ ($c_M$ will depend only on $M\subset N$).  Then the minmax characterisation of $u_{\eps}$ will imply that the second endpoint $\gamma(+1)$ is a function $v_{\eps}$ that cannot be $b_{\eps}$. Owing to the way in which we will construct the path, $v_{\eps}$ will satisfy the remaining conditions in Proposition \ref{Prop:main}.

\subsection{Preliminaries}
\label{prelim}

\textit{One-dimensional profiles - double transition}. 
The profiles $\Het^{\eps}$ introduced in Section \ref{minmax_setup} will be needed to write ``Allen--Cahn approximations'' of multiplicity-$1$ hypersurfaces. (Recall that $\Lambda$ is a shorthand notation for $3|\log\eps|$.) In order to deal with multiplicity-$2$ portions, instead, we define, for $\eps>0$, the function $\Psi:\R\to \R$
\be
\label{eq:Psi}
\Psi(r)=\left\{\begin{array}{ccc}
\OHet^{\eps}(r+2\eps \Lambda) & r\leq 0\\
\OHet^{\eps}(-r+2\eps \Lambda) & r>0
        \end{array}\right. .
\ee
(This function is smooth, since all derivatives of $\OHet^{\eps}$ vanish at $\pm 2\eps \Lambda$.) We have $\Ece(\Psi)=2(2\s) + O(\eps^2)$. Additionally, we will need a continuous family of one-dimensional profiles that will be employed to replicate, for functions, the geometric operation of continuously changing the weight of a hypersurface; in a similar spirit, this family will also be used to produce Allen--Cahn approximations of hypersurfaces-with-boundary endowed with multiplicity $2$. In view of this we define, for $t\in [0,\infty)$:
\be
\label{eq:family2}
\Psi_t(r):= \left\{ \begin{array}{ccc}
                      \OHet^{\eps}(r+2\eps \Lambda-t) & r\leq 0 \\
                      \OHet^{\eps}(-r+2\eps \Lambda-t) & r> 0 
                     \end{array}
\right. .
\ee
Note that $\Psi_0=\Psi$ and $\Psi_{t}\equiv -1$ for $t\geq 4\eps\Lambda$. For $t\in(0,4\eps\Lambda)$ the function $\Psi_t$ is equal to $-1$ on the set $\{r\in\R:|r|\geq 4\eps\Lambda-t\}$.
For each $t$ the function $\Psi_t$ is even and Lipschitz (and smooth away from $0$). The energy $\Ece(\Psi_t)$ is decreasing in $t$: indeed we have $\Ece(\Psi_t)=\Ece(\Psi)-\int_{-t}^t \eps |\Psi'|^2 + \frac{W(\Psi)}{\eps}$.

\medskip
 
\textit{Distance to $\overline{M}$}. Let $M$ be as in the beginning of Section \ref{proof} (just before the statement of Proposition \ref{Prop:main}). We denote by $\dm:N\to \R$ the Lipschitz function $\dm(x)=\text{dist}(x,\overline{M})$, where $\text{dist}$ is the Riemannian distance. By Hopf--Rinow theorem, $\dm(x)$ is always realized by at least one geodesic from $x$ to a point in $\overline{M}$; in our case, the endpoint of such a geodesic will always belong to $M$, see \cite[Lemma 3.1]{B1}. We let $\om\in(0,\text{inj}(N))$ and consider the open set $T_\omega=\{x:\dm(x)<\om\}$. We will restrict our analysis to this neighbourhood of $\overline{M}$. By the analysis in \cite[Section 3]{B1} (see also \cite{ManteMennu}), the function $\nabla \dm$ is in $SBV(T_{\om}\setminus \overline{M})$, i.e.~it is a $BV$-function whose distributional derivatives are Radon measures with no Cantor part. More precisely, we have that 
(see \cite[Lemma 3.2]{B1}) the distributional Laplacian $\Delta \dm$ restricted to $T_\omega\setminus \overline{M}$ is a Radon measure whose singular part is a negative. 
Moreover (see \cite{ManteMennu} and \cite[Proposition~3.1]{B1}), always restricting to $T_\omega \setminus \overline{M}$, the support of the singular part of $\Delta \dm$ is countably $n$-rectifiable and agrees with the so-called cut-locus of $M$, denoted by $\text{Cut}(M)$; away from the cut-locus of $M$, the Laplacian $\Delta \dm$ is smooth. We recall that, at $x\in \left(N \setminus \text{Cut}(M)\right)\setminus (\overline{M}\setminus M)$, $-\Delta d_{\overline{M}}(x)$ agrees with the (classical) scalar mean curvature of the level set of $d_{\overline{M}}$ that contains $x$, computed with respect to the normal that points away from $M$. For $x\in T_\om\setminus \text{Cut}(M) \setminus \overline{M}$ we have that there exists a unique geodesic from $x$ to $\overline{M}$ whose length realizes $\dm(x)$. This geodesic is completely contained (except for its endpoint, that lies in $M$) in the open set $T_\om\setminus \text{Cut}(M) \setminus \overline{M}$. This yields a retraction of $T_\om\setminus \overline{\text{Cut}(M)}$ onto $M$, see \cite[Remark 3.2]{B1}, with points moving towards $M$ at unit speed along the unique geodesic connecting them to $M$. Arguing as in \cite[Lemma 3.3]{B1} by means of Riccati's equation \cite[Corollary 3.6]{Gray}, and replacing the condition $\Rc{N}>0$ (valid in \cite[Lemma 3.3]{B1}) with $\Rc{N}>-C$ for some $C>0$ (valid on our compact manifold $N$), we obtain that $\Delta \dm \leq C \dm(x)$ on $T_\om\setminus \text{Cut}(M)\setminus \overline{M}$ (recall that $\dm$ is smooth on this open set). In fact, since $\dm$ is smooth on $M$ and $M$ is minimal, we have $\Delta \dm =0$ on $M$, so $\Delta \dm(x) \leq C \dm(x)$ on $T_\om\setminus \text{Cut}(M)\setminus (\overline{M}\setminus M)$. 

\begin{Prop}
\label{Prop:sign_lapl} 
Let $N$ be a compact $(n+1)$-dimensional Riemannian manifold and $M$ a smooth minimal hypersurface as in the beginning of Section \ref{proof}. Denote by $\dm$ the distance function to $\overline{M}$ and by $T_\om=\{x\in N:  \dm(x)<\om\}$, where $\om\in(0,\text{inj}(N))$. Then the distributional Laplacian $\Delta \dm$ on $T_\om$ is a Radon measure that satisfies $\Delta \dm \leq C\dm $, for\footnote{This should be interpreted as an inequality between measures: the function $C\dm$ is identified with the measure $C\dm \mathcal{H}^{n+1}\res T_{\om}$ and we know already that the distributional Laplacian $\Delta \dm$ is a Radon measure, therefore the inequality means that $C\dm \mathcal{H}^{n+1} -\Delta \dm$ is a positive (Radon) measure on $T_{\omega}$. Also recall that a distribution is said to be $\leq 0$ if for every non-negative test function the result is $\leq 0$ (similarly for $\geq 0$). A distribution that is $\geq 0$ or $\leq 0$ is necessarily a Radon measure, see e.g.~\cite[Theorem 1.39]{EvGa}.} $C=-\min_N \Rc{N}$.
\end{Prop}

\begin{proof}
We consider the distribution $\Delta \dm - C \dm$, defined by its action on $v\in C^\infty_c(T_\omega)$ by
$$(\Delta \dm - C \dm)(v)=-\int\nabla \dm \cdot \nabla v - C \int_N \dm v.$$
Note that this is a distribution on $T_{\om}$ of order at most $1$, because  $\nabla \dm\in L^{\infty}(T_\om)$ with $|\nabla \dm|\leq 1$ and hence we get $|(\Delta \dm - C \dm)(v)|\leq \mathcal{H}^{n+1}(N) (C\om +1)\|v\|_{C^1(T_\omega)}$.

Putting together $\Delta \dm(x) \leq C \dm(x)$ on $T_\om\setminus \text{Cut}(M)\setminus (\overline{M}\setminus M)$ with the sign condition on the singular part of $\Delta \dm$ obtained in \cite[Lemma 3.2]{B1} (and discussed above), we conclude that the restriction of the distribution $\Delta \dm - C \dm$ to $T_\om\setminus  (\overline{M}\setminus M)$ is $\leq 0$. This restriction is therefore a (negative) Radon measure on $T_\om\setminus  (\overline{M}\setminus M)$.

We argue via a capacity argument. For any $\delta>0$ we choose (see \cite[4.7]{EvGa}) $\chi \in C^\infty_c(T_\om)$ to be a function that takes values in $[0,1]$, is identically $1$ in an open neighbourhood of $\overline{M}\setminus M$, identically $0$ away from a (larger) neighbourhood of $\overline{M}\setminus M$ and such that $\int_{T_{\om}} |\nabla \chi|<\delta$. For $v\in C^\infty_c(T_\om)$, $v\geq 0$, we get

$$\int_{T_{\om}} (\Delta \dm - C\dm) v = \int_{T_{\om}} (\Delta \dm-C\dm) (1-\chi)v + \int_{T_{\om}} \Delta \dm \, \chi v - \int_{T_\omega} C\dm \chi v=$$
\be
\label{eq:capacity_delta}
= \int_{T_{\om}} (\Delta \dm-C\dm) (1-\chi)v - \int_{T_{\om}} \nabla \dm \, \nabla \chi\, v - \int_{T_{\om}} \nabla \dm \nabla v \, \chi - \int_{T_\omega} C\dm \chi v.
\ee
The second, third and fourth terms in the last identity tend to $0$ as $\delta \to 0$, because $\|\chi\|_{W^{1,1}(T_\omega)}\to 0$ as $\delta\to 0$ and $|\nabla \dm|\leq 1$. The first term is $\leq 0$ for any $\delta$, because $(1-\chi)v\geq 0$ and we saw that the distribution $\Delta \dm-C\dm$ is a negative Radon measure on the support of $(1-\chi)v$ (for any $\delta$). Taking the limit in (\ref{eq:capacity_delta}) as $\delta \to 0$ we therefore obtain that $\Delta \dm - C \dm$ is a negative distribution on $T_{\om}$, and therefore it is a (negative) Radon measure on $T_{\om}$.
\end{proof}

We will denote by $\widetilde{M}$ the oriented double cover of $M$, that is $\widetilde{M}=\{(y,v):y\in M, v \text{ is one of the two possible choices of unit normal to $M$ at $y$}\}$ ($\widetilde{M}$ naturally embeds in the unit sphere bundle on $N$). The standard projection $\widetilde{M}\to M$ is given by $(y,v)\to y$ and we denote by $\iota:\widetilde{M}\to N$ the composition of the standard projection with the embedding of $M$ into $N$. For $q=(y,v)\in \widetilde{M}$, for $y\in M$ and $v$ a choice of unit normal to $M$ at $y$, the geodesic $s\in (0,\text{inj}(N)) \to \text{exp}_{y}(s v)$ leaves $M$ orthogonally and is minimizing, between $y$ and the point $\text{exp}_{y}(t v)$, as long as $t$ is sufficiently small. We denote by $\sigma_{(y,v)}$ the (positive) number such that this geodesic is minimizing between $y$ and the point $\text{exp}_{y}(t v)$ for all $t\leq \sigma_{(y,v)}$ and is no longer minimizing if $t>\sigma_{(y,v)}$. As we are restricting to $T_{\om}$, we truncate $\sigma_{(y,v)}$ using the convention that $\sigma_{y,v}=\om$ if the geodesic is minimising for some $t>\om$. The function $\sigma_{(y,v)}$ on $\widetilde{M}$ is continuous (see \cite[Section 3]{B1} and \cite{ManteMennu} for details); moreover, we have a smooth diffeomorphism $F$
\be
\label{eq:diffeo_F}
F:\{((y,v),s):(y,v)\in \widetilde{M}, s\in(0,\sigma_{(y,v)})\} \to T_\om\setminus \text{Cut}(M) \setminus \overline{M} 
\ee
defined by $F((y,v),s)=\text{exp}_y(sv)$. This diffeomorphism extends by continuity to a map  
$$V_{\widetilde{M}}=\{((y,v),s):(y,v)\in \widetilde{M}, s\in[0,\sigma_{(y,v)})\} \to T_\om\setminus \left(\text{Cut}(M) \cup (\overline{M} \setminus M)\right).$$ This extended map will be still denoted by $F$ and is $2-1$ on $\widetilde{M}\times \{0\}$. (On $\widetilde{M}\times \{0\}$ this map can be identified with $\iota$.) Since $\sigma_{(y,v)}>0$, on any compact subset of $M$ there is a positive lower bound for $\sigma_{(y,v)}$ and therefore the map $F$ provides, around any compact set of $M$, a system of Fermi coordinates (tubular neighbourhood system).

\medskip

We next collect certain properties of the level sets $\Gamma_t=\{x\in N:\dm(x)=t\}$ for $t\in [0,\om)$. For $t\in(0,\om)$ we have that $\Gamma_t \setminus \text{Cut}(M)$ is smooth, with scalar mean curvature at $x$ given by $-\Delta\dm(x)$. For $\mathcal{H}^1$-a.e.~$t\in(0,\om)$ we have $\mathcal{H}^n\left(\Gamma_t \cap \text{Cut}(M)\right)=0$, since $\text{Cut}(M)$ has dimension $n$. Therefore we have that $\mathcal{H}^1$-a.e.~level set is $\mathcal{H}^n$-a.e.~smooth. For these level sets we can therefore compute the $\mathcal{H}^n$-measure by computing the measure of their smooth part. For this we argue as in \cite[Lemma 4.1]{B1} (to which we refer for further details). We use \cite[Theorem 3.11]{Gray} to compute the distortion of the area element as we move along a geodesic $s\in(0,\sigma_{(y,v)}) \to \text{exp}_y(sv)$, for $(y,v)\in \widetilde{M}$, the oriented double cover of $M$ (in other words, $y\in M$ and $v$ is one of the two choices of unit normal to $M$ at $y$). The distortion of the area element $\theta_s$ is ruled by the ODE $\frac{\p}{\p s} \log\theta_s  = -\vec{H}(y,s)\cdot \frac{\p}{\p s}$, where $\vec{H}_{(y,s)}$ is the mean curvature of the level set at distance $s$ evaluated at the point $(y,s)=\text{exp}_y(sv)$. Using Riccati's equation \cite[Corollary 3.6]{Gray}, and the bound $\Rc{N}\geq -C$ on $N$, we obtain that $H_{(y,s)}\geq - C s$, where $H(y,s)=\vec{H}_{(y,s)}\cdot \frac{\p}{\p s}$ is the scalar mean curvature of the level set $\Gamma_s$ at the point $(y,s)$, with respect to the unit normal that points away from $M$. We thus have $\frac{\p}{\p s} \log\theta_s  \leq Cs$. Integrating this inequality we obtain that, with coordinates chosen so that $\theta_0(x)=1$, the area element evolves with the bound $\theta_s(x)\leq e^{Cs^2/2}$. Therefore 
\begin{equation}
\label{eq:area_level_set}
 \mathcal{H}^n(\Gamma_t)\leq 2\mathcal{H}^n(M) e^{Ct^2/2} \text{ for almost every } t\in(0,\om)
\end{equation}
(level sets of $\dm$ are double covers of $M$, since $\dm$ is unsigned, hence the appearance of the factor $2$). Recall also that the scalar mean curvature of $\Gamma_t$ at the point $x\in \Gamma_t\setminus \text{Cut}(M)$ agrees with $-\Delta \dm(x)$. The following estimate is implicit in the previous discussion:
\begin{equation}
\label{eq:mean_curv_level_set}
x\in \Gamma_t\setminus\text{Cut}(M) \Rightarrow \Delta \dm(x)\leq Ct.
\end{equation}

\subsection{Allen--Cahn approximation of $2|M|$}
\label{2M}

In this section we will produce for all sufficiently small $\eps$, a function $G^{\eps}_0:N \to \R$ whose Allen--Cahn energy is\footnote{We will only be interested in obtaining a control from above of the energy by the area, up to a small error term. It is however true, as can be seen by computations similar to those that we give in this section, that a control from below is also valid.} approximately $2\mathcal{H}^n(M)$. While this function is not part of the path $\gamma$ that we aim to construct (see the discussion that precedes Section \ref{prelim}), by suitably deforming it we will construct two functions through which the path $\gamma$ will pass. 

Let $\eps>0$ be sufficiently small to ensure $12\eps|\log\eps|<\min\{\om,1\}$. The function $G^{\eps}_0$ has level sets coinciding with the level sets of $\dm$ and the one-dimensional profiles normal to these level sets are dictated by $\OHet^{\eps}$ (see (\ref{eq:Psi}) for the definition of $\Psi$):

\be
\label{eq:G0}
G^{\eps}_0(x) = \left\{\begin{array}{ccc}
                          -1 & \text{ for } x\in N\setminus T_\om \\
\Psi(\dm(x))& \text{ for } x\in T_\om
                         \end{array}\right. .
\ee
Equivalently, for $x\in T_\om$ we have $G^{\eps}_0(x) = \OHet^{\eps}(-\dm(x)+2\eps\Lambda)$. We will now compute $\Ece(G^{\eps}_0)$ and the first variation $\ca{E'}{\eps}(G^{\eps}_0)$.

We use the shorthand notation $\Lambda=3|\log\eps|$. By definition we have, on $T_{\om}$, that $\nabla G^{\eps}_0 = \Psi'(\dm)\nabla \dm$, while the energy of $G^{\eps}_0$ is $0$ in $N\setminus\{x\in N:\dm(x)\leq 4\eps\Lambda\}$; we use the coarea formula for the Lipschitz function $\dm$ (for which $|\nabla \dm|=1$) to get

$$\Ece(G_0^{\eps})=\int_{T_{\om}} \eps \frac{|\nabla G_0^{\eps}|^2}{2} + \frac{W(G_0^{\eps})}{\eps} = \int_{0}^{\om} \left(\int_{\Gamma_s} \frac{|\nabla G_0^{\eps}|^2}{2} + \frac{W(G_0^{\eps})}{\eps}\right) ds =$$ 
$$ =\int_{-2\eps\Lambda}^{2\eps\Lambda} \left(\int_{\Gamma_{2\eps\Lambda-s}} \eps \frac{({\OHet^{\eps}}'(s))^2}{2} + \frac{W({\OHet^{\eps}}(s))}{\eps}\right) ds \underbrace{\leq}_{(\ref{eq:area_level_set})} $$ 
$$\leq   2 e^{\frac{C}{2} (12\eps|\log\eps|)^2} \mathcal{H}^n(M) \left(\int_{\R}  \eps \frac{({\OHet^{\eps}}')^2}{2} + \frac{W({\OHet^{\eps}})}{\eps}\right)\leq 2(2\s)\mathcal{H}^n(M) + O(\eps|\log\eps|),$$
for $\eps$ in the chosen range. The Allen--Cahn first variation of $G^{\eps}_0$ (which is clearly $0$ outside $\{x\in N:\dm(x)\leq 4\eps\Lambda\}$) can be computed in $\{x\in N:\dm(x)<5\eps\Lambda\}$ as follows. The Radon measure $\Delta \dm$ satisfies Proposition \ref{Prop:sign_lapl}, recall moreover the error by which $\OHet^{\eps}$ fails to solve the Allen--Cahn ODE (Section \ref{prelim}). Then, in the distributional sense, we have
\be
\label{eq:first_var_G0_part1}
-\ca{E'}{\eps}(G^{\eps}_0)=\eps\Delta G^{\eps}_0 -\frac{W'(G^{\eps}_0)}{\eps} = 
\ee
$$=\eps {\OHet^{\eps}}''(-\dm+2\eps\Lambda) |\nabla \dm|^2 - \eps{\OHet^{\eps}}'(-\dm+2\eps\Lambda) \Delta \dm -\frac{W'(\OHet^{\eps}(-\dm+2\eps\Lambda))}{\eps}=$$
$$=\underbrace{\eps {\OHet^{\eps}}''(-\dm+2\eps\Lambda)-\frac{W'(\OHet^{\eps}(-\dm+2\eps\Lambda))}{\eps}}_{O({\eps}^2)}- \underbrace{\eps{\OHet^{\eps}}'(-\dm+2\eps\Lambda)}_{0\leq\, \cdot\,\leq 3} \underbrace{\Delta \dm}_{\leq C\dm},$$
Here $-\ca{E'}{\eps}(G^{\eps}_0)$ and $\Delta \dm$ are a Radon measures. The term $O({\eps}^2)$ in the last line is a Lipschitz function that we interpret as a density with respect to $\mathcal{H}^{n+1}$ and the last term is the measure $\Delta \dm$ multiplied by a bounded Lipschitz function (we have used $0\leq (\OHet^{\eps})' \leq 3/\eps$). Note that $\dm<5\eps\Lambda$ on the relavant domain. Therefore there exist $\eps_{N}\in(0,1)$ and $C_0$ (only depending on $N$) such that for all $\eps<\eps_{N}$ 
\be
\label{eq:first_var_G0}
-\ca{E'}{\eps}(G^{\eps}_0)\geq -C_0 \eps|\log\eps|,
\ee
interpreting the inequality as one between Radon measures.

\subsection{Immersions and signed distance} 
\label{signed_distance}

We will need a certain notion of signed distance. Fix a compact set $K\subset \widetilde{M}$ with non-empty interior; this\footnote{When $\overline{M}=M$, e.g.~for $n\leq 6$, one can choose $K=\widetilde{M}$ and the whole construction presented in the remaining sections can be shortened considerably.} set will be kept fixed throughout the construction in the coming sections. It is convenient to choose $K$ to be even, i.e.~$K=\iota^{-1}(\iota(K))$. Let $\widetilde{\phi}\in C^\infty_c(\widetilde{M})$, $\widetilde{\phi}\geq 0$, such that $\text{supp}\widetilde{\phi} \subset \subset \text{Int}(K)$. The continuous function $\sigma_{(y,v)}$ (see (\ref{eq:diffeo_F}) and the discussion preceding it) has a strictly positive minimum on $K$ and we choose $\sigma_K>0$ strictly smaller than this minimum. Consider, for $c\in(0,\sigma_K/3)$ and $t\in[0,\frac{\sigma_K}{3\max\widetilde{\phi}}]$, the following immersions:

$$p=(y,v)\in \text{Int}(K)\to \text{exp}_{\iota(p)}((c+t\widetilde{\phi}(y)) v).$$

The image of this immersion is a smooth embedded hypersurface. Note that the immersion extends smoothly up to the boundary, because in a neighbourhood of $\p K$ we have $\widetilde{\phi}=0$. We will denote by $K_{c,t,\widetilde{\phi}}$ the image of $K$ via the immersion. Note that the image of $K\setminus \text{supp}\widetilde\phi$ in contained in (the smooth part of) the level set $\Gamma_c$.

For a point $(q,s) \in \text{Int}(K)\times (0,\sigma_K)$ we define its signed distance to $\text{graph}(c+t\widetilde{\phi})$ to be negative for $s< (c+t\widetilde{\phi})(q)$, positive for $s>(c+t\widetilde{\phi})(q)$ and vanishing on $\text{graph}(c+t\widetilde{\phi})$, with absolute value equal to the Riemannian distance of $(q,s)$ to $\text{graph}(c+t\widetilde{\phi})$, where the Riemannian distance is the one induced by the pull-back via $F$ of the metric on $N$. This signed distance descends to a well-defined (smooth) signed distance $\text{dist}_{K_{c,t,\widetilde{\phi}}}$ on $F\left(\text{Int}(K)\times [0,\sigma_K)\right)\setminus M$.

Since $c+t\widetilde{\phi}$ is smooth up to $\p K$ (and extends smoothly to an open neighbourhood of $K$ with value $c$), there exists a tubular neighbourhood of $K_{c,t,\widetilde{\phi}}$ in which the nearest point projection onto $K_{c,t,\widetilde{\phi}}$ is well-defined. We denote this projection by $\Pi_{c,t}$. Upon choosing the tubular neighbourhood sufficiently small, we also ensure that in the tubular neighbourhood of $F((K\setminus \text{supp}\widetilde{\phi}) \times \{c\})$ the projection $\Pi_{c,t}$ agrees with the nearest point projection onto $\Gamma_c$, denoted by $\Pi_c$.  

We choose $c_1>0$, $t_1>0$ such that for all $c\in(0,c_1]$ and all $t\in [0,t_1]$ there exists a tubular neighbourhood of $K_{c,t,\widetilde{\phi}}$ of semi-width $c$ in which the nearest point projection $\Pi_{c,t}$ is well-defined, it coincides with $\Pi_c$ in the tubular neighbourhood of $F((K\setminus \text{supp}\widetilde{\phi}) \times \{c\})$, and moreover the following bounds hold. There exists $\kappa_K>0$ such that for all $x$ in the tubular neighbourhood of $K_{c,t,\widetilde{\phi}}$ of semi-width $c$
\be
\label{eq:JacPi_ct}
\left|\,\,|J\Pi_{c,t}|(x)-1\,\,\right|\leq \kappa_K d \,\,\text{ and }\,\,\left|\,\,\frac{1}{|J\Pi_{c,t}|(x)}-1\,\,\right|\leq \kappa_K d,
\ee
where $|J\Pi_{c,t}|=\sqrt{(D \Pi_{c,t})(D \Pi_{c,t})^T}$ and $d$ is the Riemannian distance of $x$ to $K_{c,t,\widetilde{\phi}}$.  

The way in which we will use these properties (in the forthcoming sections) is that for each fixed $\eps$ (sufficiently small) we will work with $c=4\eps\Lambda$ and with variable $t$, and with tubular neighbourhoods of semi-width $4\eps\Lambda$. The choice is made so that we can fit one-dimensional Allen-Cahn profiles in the normal bundle to $K_{c,t,\widetilde{\phi}}$. The estimates in (\ref{eq:JacPi_ct}) guarantee that the Allen--Cahn energy of the resulting function (defined in the tubular neighbourhood of $K_{c,t,\widetilde{\phi}}$) is very close to the area of $K_{c,t,\widetilde{\phi}}$ (up to the usual multiplicative constant $2\s$). The fact that $K_{c,t,\widetilde{\phi}}$ agrees with $\Gamma_c$ on its boundary will give that the function just constructed can be extended to a Lipschitz function on $N$, thanks to the properties of $\dm$.

We will additionally make use of the following fact. There exists a constant $\kappa_K>0$ (that we can assume is the same as the one appearing in (\ref{eq:JacPi_ct})) depending only on $K\subset N$ such that the nearest point projection $\Pi_K: K\times [0,\sigma_K) \to K$ satisfies
\be
\label{eq:JacPi_K}
\left|\,\,|J\Pi_{K}|(x)-1\,\,\right|\leq \kappa_K s \,\,\text{ and }\,\,\left|\,\,\frac{1}{|J\Pi_{K}|(x)}-1\,\,\right|\leq \kappa_K s,
\ee
where $|J\Pi_{K}|=\sqrt{(D \Pi_{K})(D \Pi_{K})^T}$ and $x=(q,s)$. The same notation $\Pi_K$ will be used to denote the nearest point projection from $F\left(K\times [0,\sigma_K)\right)$ onto $F(K)\subset M$ (recall that $s$ is the distance of $F(x)$ to $F(K)$). 

\subsection{Choice of $\eps$ and geometric quantities involved}
\label{epsilon}

Recall that our final aim, in order to achieve the proof of Proposition \ref{Prop:main}, is to produce, for each sufficiently small $\eps$, a continuous path $\gamma$ (with values in $W^{1,2}$) that joins $a_{\eps}$ to another stable solution (see the discussion after the statement of Proposition \ref{Prop:main}). The path itself will be exhibited for each $\eps<\eps_1$, for a certain ${\eps}_1>0$. Estimates on the energy $\Fce$ along the path will be obtained in terms of certain (fixed) geometric quantities (e.g.~$\mathcal{H}^n(M)$), that are independent of $\eps$, and error terms. For all $\eps\leq {\eps}_2$, for a certain ${\eps}_2\in(0,{\eps}_1]$, these error terms will be of the type $O(\eps|\log\eps|)$, i.e.~they will be bounded, in absolute value, by $C\eps|\log\eps|$ with $C>0$ independent of $\eps\in(0,{\eps}_2)$. Finally (this will only happen in Section \ref{conclusive}), for a certain ${\eps}_3\in (0,{\eps}_2]$, the terms $O(\eps|\log\eps|)$ will be absorbed in the geometric quantities, leading to an effective estimate on the energy $\Fce$ along the path, i.e.~an estimate that only depends on geometric quantities.

Rather than picking $\eps_1$ here, we will require a smallness condition on it repeatedly (finitely many times) throughout the forthcoming sections. One smallness requirement was made in Section \ref{2M}, $12 {\eps}_1 |\log {\eps}_1| < \min \{\om,1\}$. At every new requirement, we will implicitly assume that all those previously imposed remain valid. Similarly, upon estimating $\Fce$ for the functions that are constructed for $\eps<{\eps}_1$, we will (finitely many times) write the error terms in the form $O(\eps|\log\eps|)$ for $\eps\in(0,{\eps}_2)$; each time we specify a new ${\eps}_2$ we will implicity assume that we pick the smallest ${\eps}_2$ among all those identified until that moment. An initial requirement is $\eps_2\leq {\eps}_N$, for the ${\eps}_N$ chosen in Section \ref{2M}. At the end (Section \ref{conclusive}) we will choose ${\eps}_3\in (0,{\eps}_2]$ and restrict to $\eps\in (0,{\eps}_3]$: for this range, the energy estimates become effective and allow us to conclude the proof.

From now on we shall let the sequence $u_{\e_{j}}$ be as in the beginning of Section~\ref{proof}. Note however that until Section~\ref{conclusive}, all we need to know about $u_{\e_{j}}$ is that the corresponding sequence of varifolds $V^{u_{\e_{j}}}$ converges to $q|M|$ where $q$ is a locally constant function on $M$ taking even integer values and $|M|$ is  stationary (i.e.\ zero mean curvature); the fact that ${\mathcal F}_{\e_{j}}(u_{{\eps}_i})$ are min-max values does not play a role until Section \ref{conclusive}. The minimal hypersurface $M$  however plays a key role throughout. We fix the compact set $K$ (chosen in Section \ref{signed_distance}) and two geodesic balls $D_1, D_2 \subset M$ whose double covers have compact closures contained in $\text{Int}(K)$ and such that the respective concentric balls of half the radius (that will be denoted by $B_j \subset D_j$ for $j\in \{1,2\}$) have equal areas. Geometric quantities depending on $D_1, D_2, M, N$ will appear in the energy estimates. To avoid burdening the notation, we will often drop the explicit dependence on $\eps$ for the functions, and paths of functions, that we construct. With reference to the upcoming sections, $\varpi_{D_j, t}$, $\gamma_t$, $f_r$, $h_t$ all implicitly depend on $\eps$.

\subsection{One-parameter deformation, bumping outwards}
\label{bump_out}

Recall that $\iota:\widetilde{M}\to N$ is the (minimal) immersion of the oriented double cover $\widetilde{M}$ of $M$ induced by the standard projection $(y,v)\in \widetilde{M}\to y\in M$. (The image of $\iota$ is $M$ and is covered twice.)
Let $D\subset M$ be a geodesic ball, with radius denoted by $2R>0$, with closure contained in $\text{Int}(K)$ (where $K$ is the compact set fixed at the beginning of Section \ref{signed_distance}); here $D$ is either of the geodesic balls $D_1$, $D_2$ chosen in Section \ref{epsilon}, and with which we will work in Section \ref{two_parameter}. Let $B\subset D$ the concentric geodesic ball of radius $R$. Let $\chi_D\in C^\infty_c(D)$ with $\chi_D\geq 0$, $\chi_D=1$ on $B$ and $|\nabla \chi_D|\leq 2/R$ (in this section the latter condition is not needed and we only use $\chi_D>0$ on $B$, rather than $=1$). We will write $\widetilde{D}=\iota^{-1}(D)\subset \widetilde{M}$ and $\widetilde{B}=\iota^{-1}(B)\subset \widetilde{M}$. Define $\chi_{\widetilde{D}}=\chi_D\circ \iota \,\in C^\infty_c(\widetilde{D})$. We will consider the one-parameter (one-sided) smooth family of immersions $\iota_{D,t}:\widetilde{D}\to N$ defined as follows for $t\in [0,\frac{\sigma_K}{3 \max_D \chi}]$ (recall the choice of $\sigma_K$ at the beginning of Section \ref{signed_distance}):
\be
\label{eq:immersion_0t}
p=(y,v)\in \widetilde{D}\to \iota_{D,t}(p)=\text{exp}_{y}(t\,\chi_{\widetilde{D}}(p) \,v).
\ee
We have $\iota_{D,t}=\left.\iota\right|_{\widetilde{D}}$ on $\widetilde{D}\setminus \text{supp}\chi_{\widetilde{D}}$ for every $t$, and $\iota_{0,D}=\left.\iota\right|_{\widetilde{D}}$. The key property of interest to us for this (one-sided) deformation of $\left.\iota\right|_{\widetilde{D}}$ is going to be the value of the functional $J_{\frac{g}{\sigma}}(t)=J_{\frac{g}{\sigma}}(\iota_{t,D})$, where $J_{\frac{g}{\sigma}}=A-\text{Vol}_{\frac{g}{\sigma}}$. The evaluation of $\text{Vol}_{\frac{g}{\sigma}}$ requires the preliminary choice of an oriented open set $\Oc\subset N$ (that must contain the images of the immersions); in our case we choose $\Oc$ to be the cylinder $F(\widetilde{D}\times [0,\sigma_K))$, with $F$ defined in (\ref{eq:diffeo_F}). The value of $\text{Vol}_{\frac{g}{\sigma}}(t)$ is then given by $\int_{F(S_t)} {\frac{g}{\sigma}} d\mathcal{H}^{n+1}$, where $S_t=\{(y,s)\in \widetilde{D}\times [0,\sigma_K): s<t \chi_{\widetilde{D}}(y)\}$. The term $A(t)$ is the $n$-dimensonal area of the immersion $\iota_{t,D}$. We refer to \cite{BW2} for the general definition of $J_g$. 

It is readily checked that there exists $t_0 \in (0,\frac{\sigma_K}{3 \max_D \chi}]$ such that $J_{\frac{g}{\sigma}}(t)$ is decreasing on $t\in [0,t_0]$. This can be seen by considering the first variation $J_{\frac{g}{\sigma}}^\prime(0)$. The first variation of area is $0$ because $\iota$ is minimal, while the first variation of $\text{Vol}_{\frac{g}{\sigma}}$ is obtained by integrating on $\widetilde{D}$ the product of ${\frac{g}{\sigma}}$ with the normal speed $\left(\left.\frac{d}{dt}\right|_{t=0^+}\iota_{t,D}\right) \cdot v= \chi_{\widetilde{D}}$. We thus have $J_{\frac{g}{\sigma}}'(0)=-\int_{\widetilde{D}}{\frac{g}{\sigma}}\chi_{\widetilde{D}} <0$. Therefore, since $J_{\frac{g}{\sigma}}$ is $C^1$ on $[0,\frac{\sigma_K}{3 \max_D \chi}]$, there exists $t_0>0$ such that $J_{\frac{g}{\sigma}}(t)$ is decreasing on $t\in [0,t_0]$. 

\begin{oss}[smallness of $t_0$]
\label{oss:choice_t0}
We will need, for several reasons, to possibly make $t_0$ smaller. We impose $t_0<t_1$, where $t_1$ (together with $c_1$, which will be needed below) is as in Section \ref{signed_distance} with $\widetilde{\phi}$ replaced by the function $\widetilde{\chi}_D$, extended to be in $C^\infty_c(K)$, by setting it equal to $0$ in the complement of $\widetilde{D}$. This will be needed in order to use the bounds obtained in Section \ref{signed_distance}. Moreover, in view of Section \ref{flow}, we require that the mean curvature of the immersion (\ref{eq:immersion_0t}) is, in absolute value, smaller than $\frac{\min_N g}{20}$ for all $t\in [0,t_0]$. This is possible, since $\min_N g>0$, the values of the mean curvature change continuously in $t$ and at $t=0$ the immersion is minimal. In the next remark we impose one further condition, which may require to make $t_0$ smaller one last time.
\end{oss}

\begin{oss}[choice of $\tau_B$]
\label{oss:choiceoftauB}
By making $t_0$ smaller if necessary, we ensure that the (positive) quantity $\frac{J_{\frac{g}{\sigma}}(0)-J_{\frac{g}{\sigma}}(t_0)}{2}$ is smaller than $2\mathcal{H}^n(B)$. From now on we will work with this (positive) $t_0$ and denote by $\tau_B=\frac{J_{\frac{g}{\sigma}}(0)-J_{\frac{g}{\sigma}}(t_0)}{2}$. We then have $\tau_B<2\mathcal{H}^n(B)$; this will be used in Section \ref{two_parameter}.
\end{oss}

\medskip

We point out, for future reference, that $\iota_{t,D}$ and $\iota$ coincide in a neighbourhood of $\p \widetilde{D}$ in $\widetilde{D}$. We now consider, for $c\in [0,\frac{\sigma_K}{3}]$, $t\in [0,\frac{\sigma_K}{3 \max_D \chi}]$, the immersion from $\widetilde{D}$ into $N$ given by

\be
\label{eq:immersion_ct}
p=(y,v)\in \widetilde{D} \to \text{exp}_y((c+t\chi_{\widetilde{D}}(p))\,v).
\ee
(For $c>0$ the image of such an immersion is embedded.) We denote by $J_{\frac{g}{\sigma}}(c,t)$ the evaluation of $J_{\frac{g}{\sigma}}$ on the immersion (\ref{eq:immersion_ct}), with reference to the oriented open set $F(\widetilde{D}\times [0,\sigma_K))$. We recall that $J_{\frac{g}{\sigma}}(c,t)=A(c,t)-\text{Vol}_{\frac{g}{\sigma}}(c,t)$, where $A(c,t)$ is the area of the immersion (\ref{eq:immersion_ct}) and $\text{Vol}_{\frac{g}{\sigma}}(c,t)=\int_{F(S_{c,t})} {\frac{g}{\sigma}} d\mathcal{H}^{n+1}$, with $S_{c,t}=\{(y,s)\in \widetilde{D}\times [0,\sigma_K): s<c+t \chi_{\widetilde{D}}(y)\}$. (Note that $J_{\frac{g}{\sigma}}(0,t)$ agrees with $J_{\frac{g}{\sigma}}(t)$ used above.) By continuity in $c$ and $t$, and by the decreasing property of $J_{\frac{g}{\sigma}}(0,t)$ on $[0,t_0]$ obtained above, there exist $c_0\in (0,\frac{\sigma_K}{3}]$ such that for all $c\in [0,c_0]$ and for all $t\in[0,t_0]$
\be
\label{eq:J_bounds_ct}
J_{\frac{g}{\sigma}}(c,t)\leq J_{\frac{g}{\sigma}}(0,0) + \frac{\tau_B}{2}, \,\,\,\,J_{\frac{g}{\sigma}}(c,t_0)\leq J_{\frac{g}{\sigma}}(0,0) - \tau_B,
\ee
where $\tau_B>0$ was fixed in Remark \ref{oss:choiceoftauB}. Note that $J_{\frac{g}{\sigma}}(0,0)=2\mathcal{H}^{n+1}(D)$, since the term $\text{Vol}_{\frac{g}{\sigma}}$ vanishes for $\iota_{0,D}=\left.\iota\right|_{\widetilde{D}}$.

\begin{oss}[smallness of $c_0$]
By making $c_0$ smaller if necessary, we also assume that $c_0<c_1$, where $c_1$ is chosen as in Section \ref{signed_distance} with $\widetilde{\phi}$ replaced by the function $\widetilde{\chi}_D$, extended to be in $C^\infty_c(K)$, by setting it equal to $0$ in the complement of $\widetilde{D}$. This will permit the use of the bounds obtained in Section \ref{signed_distance}.
\end{oss}

The geometric deformation of $\left.\iota\right|_{\widetilde{D}}$ just described will now be replicated with a family of functions in $W^{1,2}(F(\widetilde{D}\times [0,\sigma_K)))$, with bounds on $\Fce$ that will replace (and be deduced from) the estimates in (\ref{eq:J_bounds_ct}). 
We define, initially, a one-parameter family of functions on $\widetilde{D}\times [0,\sigma_K) \subset V_{\widetilde{M}}$ (where $V_{\widetilde{M}}$ is the extended domain of the map $F$, see (\ref{eq:diffeo_F})), using coordinates $(q,s)$, and then we pass the definition to $F\left(\widetilde{D}\times [0,\sigma_K)\right)$. In the following we assume $\eps<{\eps}_1$, with $6{\eps}_1|\log{\eps}_1|<c_0$. We use the shorthand notation $2\eps\Lambda=6{\eps}|\log{\eps}|$. We set, for $t\in[0,t_0]$ and $(q,s)\in \widetilde{D}\times (0,\sigma_K)$ (and for each $\eps$ in the specified range)
\be
\label{eq:pre_bump_D}
\varpi_{D,t}(q,s) = \OHet^{\eps}\left(-\text{dist}_{\text{graph}(2\eps\Lambda+t\chi_{\widetilde{D}})}(q,s)\right).
\ee
Here we are using the signed distance to $\text{graph}(2\eps\Lambda+t\chi_{\widetilde{D}})$ that was discussed in Section \ref{signed_distance}\footnote{To be coherent with Section \ref{signed_distance}, $\chi_{\widetilde{D}}$ must be extended to a function in $C^\infty_c(K)$, by setting it equal to $0$ in the complement of $\widetilde{D}$ and the graph in question must be considered as a graph on $K$. Then the resulting distance, defined in Section \ref{signed_distance} on $K\times (0,\sigma_K)$, has to be restricted to $\widetilde{D}\times (0,\sigma_K)$.}.
Note that although the distance is not defined for $s=0$, the definition in (\ref{eq:pre_bump_D}) extends continuously from $s>0$ to $s\geq 0$ with value $1$ at $s=0$; this follows upon observing that $\liminf_{s\to 0} \text{dist}_{\text{graph}(2\eps\Lambda+t\chi_{\widetilde{D}})}(q,s) \leq -2\eps\Lambda$ and that $\OHet^{\eps}$ has value $1$ on $[2\eps \Lambda, \infty)$. More precisely, $\OHet^{\eps}$ has vanishing derivative at $2\eps \Lambda$. This guarantees that the function defined in (\ref{eq:pre_bump_D}), extended to $s=0$ with value $1$, passes to the quotient as a $C^1$ function on the open set $F\left( D\times [0,\sigma_K) \right)$ (this is a tubular neighbourhood of $D$). With slight abuses of notation, we use the notation $\varpi_{D,t}$ also to denote this quotient and we write $\varpi_{D,t}:F\left( D\times [0,\sigma_K) \right) \to \R$,
\be
\label{eq:bump_D}
\varpi_{D,t}(x) = \OHet^{\eps}\left(-\text{dist}_{K_{2\eps\Lambda,t,\chi_{\widetilde{D}}}}(x)\right).
\ee
Here $K_{2\eps\Lambda,t,\chi_{\widetilde{D}}}$ denotes, as in Section \ref{signed_distance}, the set $F\left(\text{graph}(2\eps\Lambda+t\chi_{\widetilde{D}})\right)$.\footnote{As in the previous footnote, the graph is taken over $K$ and $\chi_{\widetilde{D}}$ is extended to a function in $C^\infty_c(K)$.}
As before, the function $\varpi_{D,t}$ is extended in a $C^1$ fashion across $D$, with value $1$ on $D$.

We remark that for $t=0$ we have $\varpi_{D,0}=\left.G^{\eps}_0\right|_{F\left( D\times [0,\sigma_K) \right)}$, where $G^{\eps}_0$ was defined in (\ref{eq:G0}). 
Moreover, we point out that (for each $\eps$ considered) the assignment $t \in [0,t_0] \to \varpi_{D,t}\in W^{1,2}(F(\widetilde{D}\times [0,\sigma_K))$ is continuous.

\medskip

The Allen--Cahn energy of $\varpi_{D,t}$ can be computed using the coarea formula (either with respect to the distance, for which the Jacobian factor is $1$, or with respect to the nearest point projection, using (\ref{eq:JacPi_ct})) in a tubular neighbourhood of $K_{2\eps\Lambda,t,\chi_{\widetilde{D}}} \cap F\left( D\times [0,\sigma_K) \right)$ of semi-width $2\eps\Lambda$; away from this tubular neighbourhood, $\varpi_{D,t}$ is constantly $\pm 1$, hence there is no energy contribution. The coarea formula gives that there exists $\eps_2>0$ such that for all $\eps\leq \eps_2$ the following bound holds:

\be
\label{eq:E_energy_bump_D}
\Ece(\varpi_{D,t}) \leq (2\s) \mathcal{H}^n\left(K_{2\eps\Lambda,t,\chi_{\widetilde{D}}} \cap F\left( D\times [0,\sigma_K) \right)\right) +O(\eps|\log\eps|).
\ee
The set $K_{2\eps\Lambda,t,\chi_{\widetilde{D}}} \cap F\left( D\times [0,\sigma_K)\right)$ is the image of the immersion in (\ref{eq:immersion_ct}) with $c=2\eps\Lambda$ (for which we obtained the estimates in (\ref{eq:J_bounds_ct}) --- we will use these to obtain (\ref{eq:F_energy_bump_D}) below).

In order to get an estimate for $\Fce(\varpi_{D,t})$ we now consider $\int_{\widetilde{D}\times [0,\sigma_K)} g \varpi_{D,t} d\mathcal{H}^{n+1}$ and relate this quantity to $\text{Vol}_g(2\eps\Lambda,t)$. As in (\ref{eq:J_bounds_ct}), $\text{Vol}_g(c,t)$ is the evaluation of $\text{Vol}_g$ on the immersion (\ref{eq:immersion_ct}) and we are choosing $c=2\eps\Lambda$. Let $\varrho$ denote the function that is $+1$ on $\{(q,s)\in \widetilde{D}\times [0,\sigma_K) : \text{dist}_{\text{graph}(2\eps\Lambda+t\chi_{\widetilde{D}})}((q,s)) \leq 0\}$ and $-1$ on $\{(q,s)\in \widetilde{D}\times [0,\sigma_K) : \text{dist}_{\text{graph}(2\eps\Lambda+t\chi_{\widetilde{D}})}((q,s)) > 0\}$. Then $\int_{\widetilde{D}\times [0,\sigma_K)} g\, \varrho\, d\mathcal{H}^{n+1} = 2\int_{\{\varrho=+1\}} g \, d\mathcal{H}^{n+1} - \int_{\widetilde{D}\times [0,\sigma_K)} g\, d\mathcal{H}^{n+1} =2\text{Vol}_g(2\eps\Lambda,t) - \int_{\widetilde{D}\times [0,\sigma_K)} g\, d\mathcal{H}^{n+1}$. We let  
$$U_{1,0}=\{(q,s)\in \widetilde{D}\times [0,\sigma_K) : -2\eps\Lambda \leq \text{dist}_{\text{graph}(2\eps\Lambda+t\chi_{\widetilde{D}})}((q,s)) \leq 0\},$$
$$U_{0,-1}=\{(q,s)\in \widetilde{D}\times [0,\sigma_K) : 0\leq \text{dist}_{\text{graph}(2\eps\Lambda+t\chi_{\widetilde{D}})}((q,s)) \leq 2\eps\Lambda \};$$
the function $\varpi_{D,t}$ decreases, on these two sets, respectively from $1$ to $0$ and from $0$ to $-1$ (as the distance increases respectively from $-2\eps\Lambda$ to $0$ and from $0$ to $2\eps\Lambda$), so that $|\varpi_{D,t}-\varrho|\leq 1$ on $U_{0,-1} \cup U_{1, 0}$. On $\left(\widetilde{D}\times [0,\sigma_K)\right) \setminus (U_{0,-1} \cup U_{1, 0})$, on the other hand, we have $\varpi_{D,t}=\varrho$. We then have
$$\int_{\widetilde{D}\times [0,\sigma_K)} g \,\varpi_{D,t}\, d\mathcal{H}^{n+1}\geq $$ $$2\text{Vol}_g(2\eps\Lambda,t) - \int_{\widetilde{D}\times [0,\sigma_K)} g\, d\mathcal{H}^{n+1} - \mathcal{H}^{n+1}(U_{1,0}\cup U_{0,-1})|\sup_N g|.$$
Using the coarea formula (with respect to $\text{dist}_{\text{graph}(2\eps\Lambda+t\chi_{\widetilde{D}})}$ or with respect to $\Pi_{c,t}$, recalling (\ref{eq:JacPi_ct})), we obtain that $\mathcal{H}^{n+1}(U_{1,0})$ and $\mathcal{H}^{n+1}(U_{0,-1})$ are bounded above by $2\eps\Lambda \mathcal{H}^n\left(\left.\text{graph}(2\eps\Lambda+t\chi_{\widetilde{D}})\right|_{\widetilde{D}}\right) + O(\eps|\log\eps|)$. 
Moreover, from the area formula we obtain $\mathcal{H}^n\left(\left.\text{graph}(2\eps\Lambda+t\chi_{\widetilde{D}})\right|_{\widetilde{D}}\right)\leq \mathcal{H}^n(\widetilde{D}) (1+ C_{\chi_{\widetilde{D}},c_0,t_0,N})$. In conclusion, there exists $\eps_2>0$ such that for all $\eps\leq \eps_2$ we have

$$\int_{\widetilde{D}\times [0,\sigma_K)} g \,\varpi_{D,t}\, d\mathcal{H}^{n+1} \geq 2\text{Vol}_g(2\eps\Lambda,t) - \int_{\widetilde{D}\times [0,\sigma_K)} g\, d\mathcal{H}^{n+1} -  |O(\eps|\log\eps|)|.$$

Putting this together with (\ref{eq:E_energy_bump_D}) and (\ref{eq:J_bounds_ct}) we obtain that there exists $\eps_2>0$ such that for all $\eps\leq \eps_2$ (here the energy is computed on the domain of $\varpi_{D,t}$, that is on $F(D\times [0,\sigma_K))$)

$$\frac{1}{2\s}\Fce(\varpi_{D,t}) \leq \underbrace{J_{\frac{g}{\s}}(0,0)}_{=2\mathcal{H}^{n}(D)} + \int_{\widetilde{D}\times [0,\sigma_K)} \frac{g}{2\s}\, d\mathcal{H}^{n+1}+ \frac{\tau_B}{2} +O(\eps|\log\eps|) \,\,\,\text{ for all } t\in[0,t_0],$$
\be
\label{eq:F_energy_bump_D}
\frac{1}{2\s}\Fce(\varpi_{D,t_0}) \leq \underbrace{J_{\frac{g}{\s}}(0,0)}_{=2\mathcal{H}^{n}(D)}+ \int_{\widetilde{D}\times [0,\sigma_K)}\frac{g}{2\s}\, d\mathcal{H}^{n+1}- \tau_B  +O(\eps|\log\eps|).
\ee

\subsection{One-parameter deformation, bumping downwards}
\label{bump_down}

In this section $\widetilde{D}$ will be as in Section \ref{bump_out} and we will construct a deformation that continues the one provided in Section \ref{bump_out} to $t\leq 0$. More precisely, we will construct, for each $\eps\in (0,{\eps}_1]$ (for ${\eps}_1$ as in Section \ref{bump_out}) a one-parameter deformation 
$$t\in [-4\eps\Lambda, 0]\to \varpi_{D,t}\in W^{1,2}\left(\widetilde{D}\times [0,\sigma_K)\right).$$
At $t=0$ the resulting function will agree with $G^{\eps}_0$ and therefore with the function defined by (\ref{eq:bump_D}) for $t=0$ (justifying the notation). We will then check that the functions $\varpi_{D,t}$ (for $t\in [-4\eps\Lambda, 0]$) pass to the quotient as $W^{1,2}$ (actually $W^{1,\infty}$) functions on $F\left(\widetilde{D}\times [0,\sigma_K)\right)$ and that the resulting deformation $[-4\eps\Lambda, 0]\to W^{1,2}\left(\widetilde{D}\times [0,\sigma_K)\right)$ is continuous in $t$.

Let $\chi_{\widetilde{D}}:\widetilde{D}\to [0,1]$ be as in Section \ref{bump_out} and $\Psi_t:\R\to \R$ as in (\ref{eq:family2}). We define, for $t\in [-4\eps\Lambda, 0]$ and $(q,s)\in \widetilde{D}\times [0,\sigma_K)$,

\be
\label{eq:bump_down}
 \varpi_{D,t}(q,s)=\Psi_{-t\,\chi_{\widetilde{D}}(q)}(s).
\ee
Since $\chi_{\widetilde{D}}$ is even on $\widetilde{D}$ by construction, the function $\varpi_{D,t}(q,s)$ in (\ref{eq:bump_down}) passes to the quotient in $F\left(\widetilde{D}\times [0,\sigma_K)\right)$. We will now check that it is in fact Lipschitz on $F\left(\widetilde{D}\times [0,\sigma_K)\right)$. Note that at $t=0$ the function agrees with $\left.G^{\eps}_0\right|_{F\left(\widetilde{D}\times [0,\sigma_K)\right)}$ (see (\ref{eq:G0})).

We only need to check the Lipschitz property locally around a point $x\in D$, since the function is smooth in $F\left(\widetilde{D}\times (0,\sigma_K)\right)$. Let $(y,a)\in B_\rho(x)\times (-\sigma_K, \sigma_K)$ denote Fermi coordinates centred at geodesic ball around $x$ in $D$. Then the expression of the function obtained by passing $\varpi_{D,t}$ to the quotient is $\Psi_{-t\,\chi_{D}(y)}(a)$, because $\Psi_s$ is even for all $s\geq 0$. Moreover,  $\Psi_s$ is Lipschitz for all $s\geq 0$; this implies that the function $\Psi_{-t\,\chi_{D}(y)}(a)$ is Lipschitz with respect to the product metric on $B_\rho(x)\times (-\sigma_K, \sigma_K)$. The distortion factor between the Riemannian metric on $N$ and this product metric is bounded by a geometric constant (that is fixed by the choices of $K$ and $\sigma_K$ and only depends only on the geometry of $F(K\times [0,\sigma_K))$), hence the function $\varpi_{D,t}$ passes to the quotient as a Lipschitz function. We also point out that there exists a neighbourhood of $F\left(\p \widetilde{D} \times [0,\sigma_K)\right)$ in which the function agrees with $G^{\eps}_0$ for every $t\in [-4\eps\Lambda, 0]$.

Next we estimate the Allen--Cahn energy of $\varpi_{D,t}$. Denote by $\nabla_q$ the metric gradient in $\widetilde{D}\times [0,\sigma_K)$ projected onto the level set $\{s=\text{cnst}\}$. Then

$$\nabla_q  \varpi_{D,t} (q,s)=-t \left.\frac{d}{d a} \Psi_{a}(s)\right|_{a=-t\chi_{\widetilde{D}}(q)}\nabla_q \chi_{\widetilde{D}}(q,s),$$
where we think temporarily of $\chi_{\widetilde{D}}(q,s)=\chi_{\widetilde{D}}(q)$ as a function on $\widetilde{D}\times [0,\sigma_K)$ that only depends on the variable $q$. Recalling (\ref{eq:family2}), we have $\left|\frac{d}{d a} \Psi_{a}(s)\right|=|\Psi'(a+|s|)|\leq \frac{3}{\eps}$. Moreover we ensured $|\nabla \chi_{\widetilde{D}}|\leq \frac{2}{R}$, as a function on $\widetilde{D}$, and therefore $|\nabla_q \chi_{\widetilde{D}}(q,s)|\leq \frac{C_K}{R}$, for a constant $C_K$ that depends only on the Riemannian metric on $K\times [0,\sigma_K)$. Therefore

\be
\label{eq:tangential_gradient}
\eps |\nabla_q  \varpi_{D,t}|^2 \leq \frac{9 C_K^2 t^2}{\eps R^2}\leq  \frac{9\cdot 16\, C_K^2 t^2}{R^2} \eps \Lambda^2=C \eps |\log\eps|^2,
\ee
for a constant $C>0$ that depends only on the choices of $K$ and $D$ (in particular, it does not depend on $\eps$). Recalling that (for the Riemannian metric) the unit vectors $\frac{\p}{\p s}$ are orthogonal to the level sets $\{s=\text{cnst}\}$ we can write $|\nabla \varpi_{D,t}|^2=|\nabla_q \varpi_{D,t}|^2 + \left|\frac{\p}{\p s}\varpi_{D,t}\right|^2$. We then compute the Allen--Cahn energy of $\varpi_{D,t}$ by using the coarea formula with respect to $\Pi_{K}$, recalling that $\nabla_q \varpi_{D,t}=0$ on $\widetilde{B}\times [0,\sigma_K)$ (since $\chi_{\widetilde{D}}=1$ on $\widetilde{B}$):  

\be
\label{eq:energy_varpi_t}
\int_{\widetilde{D}\times (0,\sigma_K)} \eps\frac{|\nabla \varpi_{D,t}|^2}{2} + \frac{W(\varpi_{D,t})}{\eps} 
\ee
$$=\int_{\widetilde{B}} \left(\int_0^{\sigma_K} \frac{1}{|J\Pi_K|(q,s)}\left(\eps \left|\frac{d}{d s} \Psi_{|t|}(s)\right|^2 +  \frac{W( \Psi_{|t|}(s))}{\eps}\right) ds\right)dq\,\,\,+$$
$$+\int_{\widetilde{D}\setminus \widetilde{B}} \left( \int_{(0, {\sigma_K})}\frac{1}{|J\Pi_K|(q,s)}\left( \eps \left|\frac{\p}{\p s} \Psi_{-t\chi_{\widetilde{D}}(q)}(s)\right|^2 + \frac{W(\Psi_{-t\chi_{\widetilde{D}}(q)}(s))}{\eps} \right)  ds\right) dq +$$ $$+ \int_{(\widetilde{D}\setminus \widetilde{B})\times (0,{\sigma_K})}  \eps |\nabla_q\varpi_{D,t}|^2 .$$
We consider the right-hand-side. The first term bounded by $(1+k_K\eps\Lambda)\mathcal{H}^n(\widetilde{B})(2\s)$, by the observation following (\ref{eq:family2}) and by the bounds on $\Pi_K$ in (\ref{eq:JacPi_K}). For the same reason, the second term is bounded by $(1+k_K\eps\Lambda)\mathcal{H}^n(\widetilde{D}\setminus \widetilde{B})(2\s)$. In view of (\ref{eq:tangential_gradient}) the third term is $O(\eps|\log\eps|)$ (in fact, it is $O(\eps^2 |\log\eps|^3)$, by noticing that the integrand vanishes on $(\widetilde{D}\setminus \widetilde{B})\times (4\eps\Lambda, \sigma_K)$). We thus have that there exists ${\eps}_2\in(0,{\eps}_1]$ such that, for all ${\eps}\in(0,{\eps}_2]$, the following bound holds independently of $t$: 

\be
\label{eq:E_energy_bump_down_D_t}
\Ece(\varpi_{D,t})\leq (2\s) 2\mathcal{H}^n(D) +O(\eps|\log\eps|).
\ee
One can easily obtain a finer bound by estimating the first term more precisely. We only need to do so for $t=-4\eps\Lambda$; in this case the first term vanishes, therefore

\be
\label{eq:E_energy_bump_down_D}
\Ece(\varpi_{D,-4\eps\Lambda})\leq (2\s) 2\mathcal{H}^n(D\setminus B) +O(\eps|\log\eps|).
\ee

To conclude this section, we estimate the energy $\Fce$ of the functions $\varpi_{D,t}$. A very rough estimate will suffice for our purposes: since $\varpi_{D,t} \geq -1$ we have 

$$\int_{\widetilde{D}\times [0,\sigma_K)} \varpi_{D,t}\, g\, d\mathcal{H}^{n+1}\geq -\int_{\widetilde{D}\times [0,\sigma_K)}g \,d\mathcal{H}^{n+1};$$
together with (\ref{eq:E_energy_bump_down_D_t}) and (\ref{eq:E_energy_bump_down_D})
this gives, for all ${\eps}\leq{\eps}_2$,
\be
\label{eq:F_energy_bump_down_D}
\frac{1}{2\s}\Fce(\varpi_{D,-4\eps\Lambda})\leq 2\mathcal{H}^n(D)-2\mathcal{H}^n(B)+\int_{\widetilde{D}\times [0,\sigma_K)}\frac{g}{2\s} \,d\mathcal{H}^{n+1} +O(\eps|\log\eps|),
\ee

$$\frac{1}{2\s}\Fce(\varpi_{D,t})\leq \underbrace{2\mathcal{H}^n(D)}_{=J_{\frac{g}{\s}}(0,0)}+\int_{\widetilde{D}\times [0,\sigma_K)}\frac{g}{2\s} \,d\mathcal{H}^{n+1} +O(\eps|\log\eps|) \text{ for all } t\in[-4\eps\Lambda,0].$$

\subsection{Avoiding the value $2\mathcal{H}^n(M)+\frac{1}{2\sigma}\int_N g$}
\label{two_parameter}

The value $2\mathcal{H}^n(M)+\frac{1}{2\sigma}\int_N g$ is the peak that we wish to avoid with our path, see the discussion before and after the statement of Proposition \ref{Prop:main}.

Given a geodesic ball $D\subset M$ (and denoting by $\widetilde{D}\subset K$ its double cover), in Sections \ref{bump_out} and \ref{bump_down} we produced, for $t_0>0$ (depending on $D$) and for every sufficiently small $\eps\leq \eps_2$, a continuous family $t\in [-4\eps\Lambda, t_0]\to  W^{1,2}\left(\widetilde{D}\times [0,\sigma_K)\right)$. This family descends to a continuous one $[-4\eps\Lambda, t_0]\to  \varpi_{D,t}\in W^{1,2}\left(F(\widetilde{D}\times [0,\sigma_K))\right)$; more precisely, the functions in the image of this curve are in $W^{1,\infty}\left(F(\widetilde{D}\times [0,\sigma_K))\right)$. At $t=0$ we have that $\varpi_{D,0}$ agrees with the restriction of $G^{\eps}_0$ (defined in (\ref{eq:G0})). In other words, we have a continuous two-sided deformation of $G^{\eps}_0$ in $F(\widetilde{D}\times [0,\sigma_K))$. By (\ref{eq:F_energy_bump_D}) and (\ref{eq:F_energy_bump_down_D}) the energy $\frac{1}{2\s}\Fce$ stays below $$J_{\frac{g}{\s}}(0,0)+\int_{\widetilde{D}\times [0,\sigma_K)}\frac{g}{2\s} \,d\mathcal{H}^{n+1} +\frac{\tau_B}{2}+O(\eps|\log\eps|)$$ 
for all $t\in [-4\eps\Lambda, t_0]$ and, moreover, at the endpoints $t=-4\eps\Lambda$ and $t=t_0$ we have that the energy $\frac{1}{2\s}\Fce$ is at most (recalling from Remark \ref{oss:choiceoftauB} that $\tau_B<2\mathcal{H}^n(B)$)
$$J_{\frac{g}{\s}}(0,0)-\tau_B +\int_{\widetilde{D}\times [0,\sigma_K)}\frac{g}{2\s}  \,d\mathcal{H}^{n+1} +O(\eps|\log\eps|).$$

In this section we will consider two distinct geodesic balls and, around each of them, we will produce a two-sided deformation of $G^{\eps}_0$ as we did above. By suitably combining them, and extending to $N$, we will produce a continuous curve into $W^{1,2}(N)$ (the functions are moreover Lipschitz on $N$) with the key property that the energy $\frac{1}{2\s}\Fce$ stays always below the value $2\mathcal{H}^n(M)-\varsigma+\int_{N} \frac{g}{2\sigma} \,d\mathcal{H}^{n+1} +O(\eps|\log\eps|)$, for some $\varsigma>0$ that only depends on quantities determined by $\iota$, by the choice of $K$ and of the two geodesic balls; in particular, the energy bound holds independently of $\eps$, for all sufficiently small $\eps$.

\medskip

Let $D_1\subset \subset M$ and $D_2\subset \subset M$ be the geodesic balls chosen in Section \ref{epsilon}. We denote respetively by $B_1$ and $B_2$ the concentric geodesic balls with half the radius. The balls are chosen so that $\mathcal{H}^n(B_1)=\mathcal{H}^n(B_2)$. We denote the respective double covers by $\widetilde{D}_1$, $\widetilde{D}_2$, $\widetilde{B}_1$, $\widetilde{B}_2$ and we assumed that $\widetilde{D}_j\subset \subset K$ for $j\in\{1,2\}$. For each $D_j$ we can repeat the construction in Sections \ref{bump_out}, \ref{bump_down}. We let ${\eps}_1$ denote the smallest of the two ${\eps}_1$ identified for $j\in\{1,2\}$, and $t_0^{(j)}$ the final time of the deformation identified for $D_j$ (which was denoted by $t_0$ for $D$). We thus obtain, for each $\eps\leq {\eps}_1$ two continuous curves into $W^{1,2}\left(F(\widetilde{D}_j\times [0,\sigma_K))\right)$, respectively for $j=1,2$:
$$t\in [-4\eps\Lambda, t_0^{(1)}] \to \varpi_{D_1,t}\,\,\,\,\,\,,\,\,\,\, t\in [-4\eps\Lambda, t_0^{(2)}]\to \varpi_{D_2,t}.$$ 
We define, for $t\in [-4\eps\Lambda - t_0^{(1)}, 4\eps\Lambda + t_0^{(2)}]$, the following continuous\footnote{Continuity is immediate from the continuity of each $\varpi_{D_j,t}$ in their respective (disjoint) domains.} one-parameter family of $W^{1,2}$ functions on $F\left(\widetilde{D}_1 \times [0,\sigma_K) \right)\cup F\left(\widetilde{D}_2 \times [0,\sigma_K) \right)$:

\be
\label{eq:two-param}
\underline{\gamma}_t = \left\{ \begin{array}{ccc}
\varpi_{D_1,t+t_0^{(1)}}  &\text{ on } F\left(\widetilde{D}_1 \times [0,\sigma_K)\right) &\text{ for }  t\in [-4\eps\Lambda - t_0^{(1)},-t_0^{(1)}]\\
\varpi_{D_2,-4\eps\Lambda}  &\text{ on }  F\left(\widetilde{D}_2 \times [0,\sigma_K)\right)  &\text{ for }  t\in [-4\eps\Lambda - t_0^{(1)},-t_0^{(1)}]\\
\varpi_{D_1,t+t_0^{(1)}}  &\text{ on } F\left(\widetilde{D}_1 \times [0,\sigma_K)\right) &\text{ for }  t\in [- t_0^{(1)},0]\\
\varpi_{D_2,-4\eps\Lambda}  &\text{ on }  F\left(\widetilde{D}_2 \times [0,\sigma_K)\right)  &\text{ for }  t\in [-t_0^{(1)},0]\\
\varpi_{D_1,t_0^{(1)}}  &\text{ on } F\left(\widetilde{D}_1 \times [0,\sigma_K)\right) &\text{ for }  t\in [0,4\eps\Lambda]\\
\varpi_{D_2,t-4\eps\Lambda}  &\text{ on }  F\left(\widetilde{D}_2 \times [0,\sigma_K)\right)  &\text{ for }  t\in [0,4\eps\Lambda]\\
\varpi_{D_1,t_0^{(1)}}  &\text{ on } F\left(\widetilde{D}_1 \times [0,\sigma_K)\right) &\text{ for }  t\in [4\eps\Lambda,4\eps\Lambda+t_0^{(2)} ]\\
\varpi_{D_2,t-4\eps\Lambda}  &\text{ on }  F\left(\widetilde{D}_2 \times [0,\sigma_K)\right)  &\text{ for }  t\in [4\eps\Lambda,4\eps\Lambda+t_0^{(2)} ]
                        \end{array}
\right. .
\ee

\begin{figure}[h]
\centering
\includegraphics[scale=0.4]{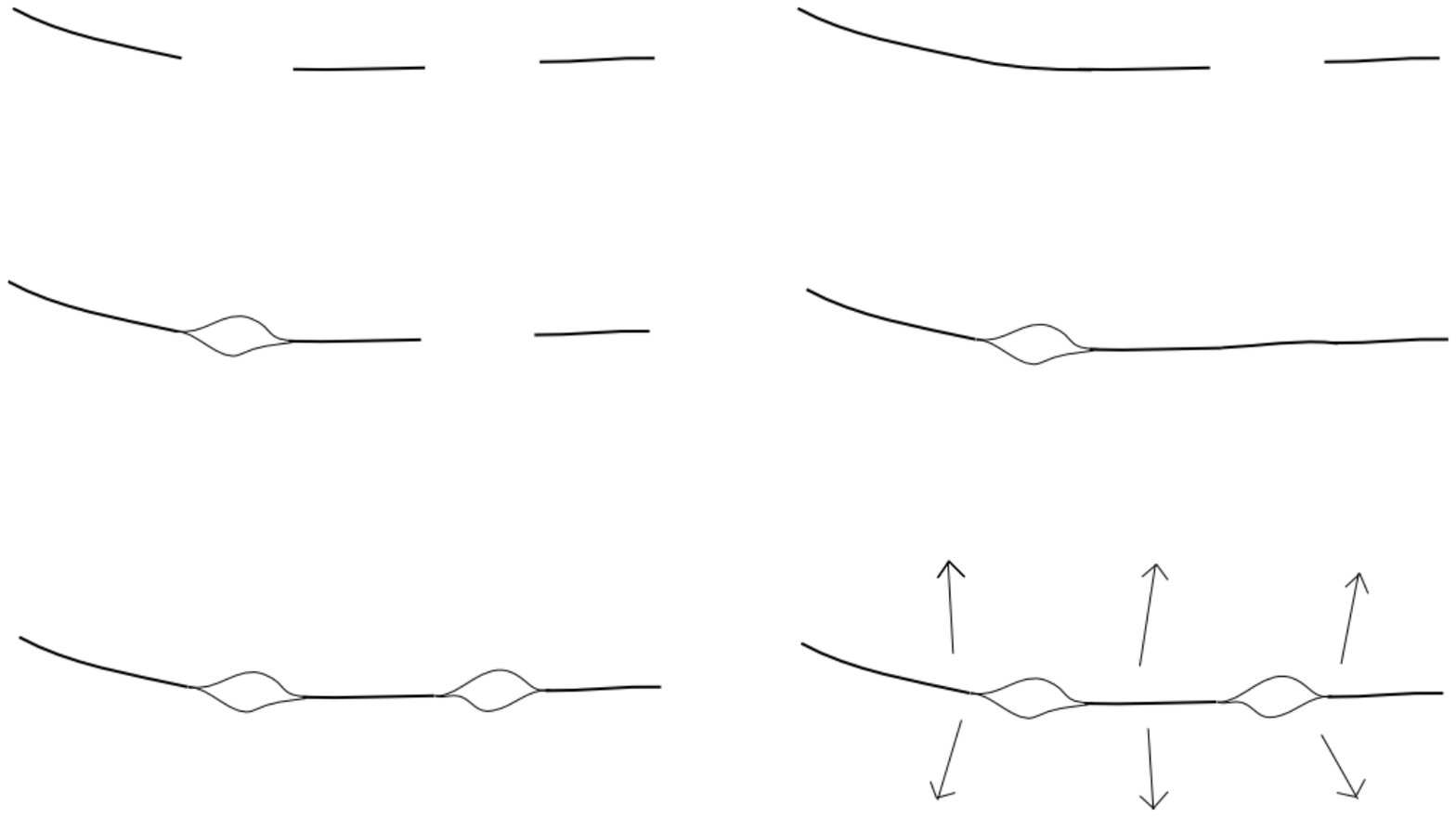}
\put (-240, 180){\tiny{(i)}}
\put(-75,180){\tiny{(ii)}}
\put(-240,120){\tiny{(iii)}}
\put(-75,120){\tiny{(iv)}}
\put(-240,20){\tiny{(v)}}
\put(-75,20){\tiny{(vi)}}
\caption{Schematic picture representing the $\e \to 0^{+}$ limit of the path $\underline{\gamma}_{t}$ defined above. The order (i) $\to$ (ii) $\to$ (iii) $\to$ (iv) $\to$ (v) $\to$ (vi) represents increasing $t$. 
The deformation (i) $\to$ (ii), for instance, represents the $\e \to 0^{+}$ limit of the top two lines of the piecewise definition of $\underline{\gamma}_{t}$; likewise, (ii) $\to$ (iii) corresponds to the third and fourth lines of the definition, etc.  The deformations (i) $\to$ (ii), and (iii) $\to$ (iv) are instantaneous ``jumps'' in the $\e \to 0^{+}$ limit (and therefore discontinuous), but the corresponding deformation $t \mapsto \underline{\gamma}_{t}$ are continuous in $t$ for fixed $\e = \e_{j}>0.$ The evolution $t \mapsto \underline{\gamma}_{t}$ (for fixed $\e = \e_{j}>0$) corresponding to ``(vi) onwards'' in the picture is carried out by the negative gradient flow of ${\mathcal F}_{\e_{j}}$, and converges to the stable solution $v_{\e_{j}}$, as described in Section~\ref{flow} below.     
}
\label{fig:avoid_peak}
\end{figure}

The idea is that for every $t$, in one of the two subdomains $F(\widetilde{D}_j \times [0,\sigma_K))$, the function agrees with one of the endpoints of the curve $\varpi_{D_j,t}$. Recall that $\tau_{B_j}<2\mathcal{H}^n(B_j)$ by the choice made in Section \ref{bump_out}, moreover we ensured $\mathcal{H}^n(B_1)=\mathcal{H}^n(B_2)$. Thanks to (\ref{eq:F_energy_bump_D}) and (\ref{eq:F_energy_bump_down_D}) we can estimate, for $\eps\leq {\eps}_2$, the energy $\frac{1}{2\s}\Fce(\underline{\gamma}_t)$ on the set $F\left(\widetilde{D}_1 \times [0,\sigma_K) \right)\cup F\left(\widetilde{D}_1 \times [0,\sigma_K) \right)$ from above with the quantity 
\be
\label{eq:E_bound_two_param}
2\mathcal{H}^n(D_1) + 2\mathcal{H}^n(D_2)+\int_{\widetilde{D_1}\times [0,\sigma_K)}\frac{g}{2\s} \,d\mathcal{H}^{n+1}+\int_{\widetilde{D_2}\times [0,\sigma_K)}\frac{g}{2\s} \,d\mathcal{H}^{n+1}$$ $$-\min\left\{\frac{\tau_{B_1}}{2}, \frac{\tau_{B_2}}{2} \right\} +O(\eps|\log\eps|).
\ee

Each $\underline{\gamma}_t$ is in fact, on $F\left(\widetilde{D}_1 \times [0,\sigma_K) \right)\cup F\left(\widetilde{D}_1 \times [0,\sigma_K) \right)$, a Lipschitz function. Our next aim is to extend $\underline{\gamma}_t$, for each $t$, to a Lipschitz function on $N$, obtaining a continuous curve from $[-4\eps\Lambda - t_0^{(1)}, 4\eps\Lambda + t_0^{(2)}]$ into $W^{1,2}(N)$.
In order to do that, we recall that there exists a neighbourhood of $F\left(\p \widetilde{D}_1 \times [0,\sigma_K) \right)\cup F\left(\p \widetilde{D}_1 \times [0,\sigma_K) \right)$ in which $\underline{\gamma}_t$ agrees with $G^{\eps}_0$ for every $t$.
This implies that we can define a Lipschitz function on $N$ that extends $\underline{\gamma}_t$ by setting

\be
\label{eq:gamma_t}
\gamma_t (x) = \left\{\begin{array}{ccc}
                      \underline{\gamma}_t(x) & \text{ for } x\in F\left(\widetilde{D}_1 \times [0,\sigma_K) \right)\cup F\left(\widetilde{D}_2 \times [0,\sigma_K) \right)\\
                      G^{\eps}_0(x) & \text{ for } x\in N\setminus \left(F\left(\widetilde{D}_1 \times [0,\sigma_K) \right)\cup F\left(\widetilde{D}_2 \times [0,\sigma_K) \right)\right)
                      \end{array}
 \right. .
\ee
The continuity of $t\in[-4\eps\Lambda - t_0^{(1)}, 4\eps\Lambda + t_0^{(2)}]\to \gamma_t \in W^{1,2}(N)$ is also immediate from the continuity of $\underline{\gamma}_t$.

\medskip

We next estimate the Allen--Cahn energy of $\gamma_t$. For that purpose, we first give a lower bound on the Allen--Cahn energy of $G^{\eps}_0$ on the sets $F\left(\widetilde{D}_1 \times [0,\sigma_K) \right)$ and $F\left(\widetilde{D}_2 \times [0,\sigma_K) \right)$. We follow the argument in Section \ref{prelim} that led to (\ref{eq:area_level_set}), this time using the bound $\Rc{N}\leq C$; this is true for some $C>0$ because $N$ is compact. Integrating Riccati's equation we get $H_t\leq \sqrt{C}\tan(\sqrt{C}t)$, where $H_t$ denotes the scalar mean curvature of the level set of $\dm$ at distance $t$ (computed with respect to the unit normal that points away from $M$, equivalently with respect to $\frac{\p}{\p s}$ in $K\times [0,\sigma_K)$). The ODE for the area element $\theta_s$ then leads to the following bound for the evolution of $\theta_s$ along a geodesic orthogonal to $\widetilde{M}$: $\theta_s\geq \theta_0\left(1-\frac{C s^2}{2}\right)$. Therefore 
$$\mathcal{H}^n\left(\widetilde{D}_j \times \{s\}\right)\geq \left(1-\frac{C s^2}{2}\right)\mathcal{H}^n\left(\widetilde{D}_j \times \{0\}\right).$$
Then the coarea formula, used for the function $\dm$ gives for $j\in \{1,2\}$ (similarly to Section \ref{2M}):
\be
\label{eq:E_G0_D}
\int_{F\left(\widetilde{D}_j \times [0,\sigma_K) \right)} \eps \frac{|\nabla G_0^{\eps}|^2}{2} + \frac{W(G_0^{\eps})}{\eps} = \int_{0}^{\sigma_K} \left(\int_{\widetilde{D}_j \times \{s\}} \frac{|\nabla G_0^{\eps}|^2}{2} + \frac{W(G_0^{\eps})}{\eps}\right) ds =
\ee 
$$ =\int_{-2\eps\Lambda}^{2\eps\Lambda} \left(\int_{\widetilde{D}_j \times \{2\eps\Lambda-s\}} \eps \frac{({\OHet^{\eps}}'(s))^2}{2} + \frac{W({\OHet^{\eps}}(s))}{\eps}\right) ds \geq $$ 
$$\geq   2  \left(1-\frac{C (4\eps\Lambda)^2}{2}\right) \mathcal{H}^n(D_j) \left(\int_{-2\eps\Lambda}^{2\eps\Lambda}  \eps \frac{({\OHet^{\eps}}')^2}{2} + \frac{W({\OHet^{\eps}})}{\eps}\right)\geq 2 (2\sigma)\mathcal{H}^n(D_j)-|O(\eps|\log\eps|)|.$$
The last inequality holds for $\eps\leq {\eps}_2$ for a suitable choice of ${\eps}_2$. The bound for $\Ece(G^{\eps}_0)$ obtained in Section \ref{2M}, together with (\ref{eq:E_G0_D}), gives an upper bound for $G^{\eps}_0$ on the set $N\setminus \left(F\left(\widetilde{D}_1 \times [0,\sigma_K) \right)\cup F\left(\widetilde{D}_1 \times [0,\sigma_K) \right)\right)$: the Allen--Cahn energy of $G^{\eps}_0$ on this set is at most 
$$  2 (2\sigma)\mathcal{H}^n(M)- 2 (2\sigma)\mathcal{H}^n(D_1)- 2 (2\sigma)\mathcal{H}^n(D_2) + O(\eps|\log\eps|).$$
In order to estimate $\Fce$ on this same set we note that $G^{\eps}_0 \geq -1$ and therefore
\begin{eqnarray*}
&&\int_{N\setminus \left(F\left(\widetilde{D}_1 \times [0,\sigma_K) \right)\cup F\left(\widetilde{D}_1 \times [0,\sigma_K) \right)\right)} G^{\eps}_0\, g\, d\mathcal{H}^{n+1}\nonumber\\
&&\hspace{1.75in}\geq -\int_{N\setminus \left(F\left(\widetilde{D}_1 \times [0,\sigma_K) \right)\cup F\left(\widetilde{D}_1 \times [0,\sigma_K) \right)\right)}g \,d\mathcal{H}^{n+1}.
\end{eqnarray*}
Recalling (\ref{eq:gamma_t}) and the bound (\ref{eq:E_bound_two_param}), the last two estimates imply that there exists $\eps_2>0$ such that for all $\eps\leq \eps_2$ the family $t\in [-4\eps\Lambda - t_0^{(1)}, 4\eps\Lambda + t_0^{(2)}] \to \gamma_t \in W^{1,2}(N)$ (recall that $\gamma_t=\gamma_t^{\eps}$ is constructed for each fixed $\eps$ in the chosen range) satisfies

\be
\label{eq:F_bound_gammat}
\frac{1}{2\s}\Fce(\gamma_t)\leq 2 \mathcal{H}^n(M) +\int_N \frac{g}{2\s}\, d\mathcal{H}^{n+1}-\underbrace{\min\left\{\frac{\tau_{B_1}}{2}, \frac{\tau_{B_2}}{2} \right\}}_{:=\,\varsigma\,>0} +O(\eps|\log\eps|)
\ee
for all $t\in [-4\eps\Lambda - t_0^{(1)}, 4\eps\Lambda + t_0^{(2)}]$, as claimed in the beginning of this section.

\subsection{Path to $a_{\eps}$}
\label{path_to_a}

In this section we exhibit, for each sufficiently small $\eps$, a continuous path in $W^{1,2}(N)$ that connects the function $a_{\eps}$ (the first endpoint of the class of admissible paths, see Section \ref{minmax_setup}) to the function $\gamma_{-4\eps\Lambda - t_0^{(1)}}$ obtained in Section \ref{two_parameter}, ensuring that $\Fce$ along this path remains bounded by the right-hand-side of (\ref{eq:F_bound_gammat}). To ease notation, in this section we will denote simply by $f_0:N\to \R$ the function $\gamma_{-4\eps\Lambda - t_0^{(1)}}$. 
We rewrite the definition of $f_0$ as follows, recalling (\ref{eq:bump_down}), (\ref{eq:two-param}) and (\ref{eq:gamma_t}):

\be
\label{eq:f}
f_0(x)=\left\{
\begin{array}{ccc}
\Psi_{4\eps\Lambda \,\chi_{\widetilde{D_j}}(q)}(s) & \text{ for } x=F(q,s), (q,s) \in \widetilde{D_j}\times [0,\sigma_K),\,\,\,j\in\{1,2\}\\
\Psi_{0}(\dm(x)) & \text{ for } x \in N\setminus\left( F(\widetilde{D_1}\times [0,\sigma_K)) \cup F(\widetilde{D_2}\times [0,\sigma_K))\right)
\end{array}\right. .
\ee
We define, for $r\in [0, 4\eps\Lambda]$

\be
\label{eq:fr}
f_r(x)=\left\{
\begin{array}{ccc}
\Psi_{4\eps\Lambda \,\chi_{\widetilde{D_j}}(q)+r}(s) & \text{ for } x=F(q,s), (q,s) \in \widetilde{D_j}\times [0,\sigma_K),\,\,\,j\in\{1,2\}\\
\Psi_{r}(\dm(x)) & \text{ for } x \in N\setminus\left( F(\widetilde{D_1}\times [0,\sigma_K)) \cup F(\widetilde{D_2}\times [0,\sigma_K))\right)
\end{array}\right. .
\ee
The first line is well-defined thanks to the fact that $\chi_{\widetilde{D}_j}$ is even on $\widetilde{M}$; note that for $r=0$ the definition in (\ref{eq:fr}) agrees with the one in (\ref{eq:f}), justifying the notation. Since $\chi_{\widetilde{D_j}} \in C^\infty_c(\widetilde{D}_j)$ for $j\in\{1, 2\}$, the definition in the first line agrees with $\Psi_{r}(\dm(x))$ when $(q,s)$ is in a neighbourhood of $\p\widetilde{D_j}\times [0,\sigma_K)$. Thanks to the fact that $\Psi_r$ is Lipschitz, one can check that each $f_r$ is a Lipschitz function on $N$. Moreover, the mapping $r\in [0,4\eps\Lambda]\to f_r\in W^{1,2}(N)$ is continuous. Note that $f_{4\eps\Lambda}\equiv -1$. 

To compute $\Ece(f_r)$ we employ the coarea formula, with respect to the function $\Pi_K$ on $\cup_{j=1}^2 F\left(\widetilde{D_j}\times [0,4\eps\Lambda)\right)$ and with respect to $\dm$ on the complementary domain in $\{\dm<4\eps\Lambda\}$ (there is no energy contribution in $\{\dm\geq 4\eps\Lambda\}$, since $f_r=-1$ there). The key observation is that the one-dimensional profile that appears for $f_r$ (with $r>0$) in the normal bundle to $M$ carries less energy than the one that appears for $f_0$. Indeed, (\ref{eq:f}) and (\ref{eq:fr}) show that the profile $\Psi_0$ for $f_0$ is replaced by $\Psi_r$ for $f_r$, and the profile $\Psi_{4\eps\Lambda \,\chi_{\widetilde{D_j}}(q)}$ for $f_0$ is replaced by $\Psi_{4\eps\Lambda \,\chi_{\widetilde{D_j}}(q)+r}$ for $f_r$. Then a computation analogous to (\ref{eq:energy_varpi_t}) in combination with the observation that follows (\ref{eq:family2}) gives that $\Ece(f_r)\leq \Ece(f_0)$ in the domain $F\left((\widetilde{D}_1 \times [0,\sigma_K)) \cup (\widetilde{D}_2 \times [0,\sigma_K))\right)$; similarly, using the coarea formula for $d_{\overline{M}}$, we get $\Ece(f_r)\leq \Ece(f_0)$ in the complementary domain. Moreover, we note that $f_r\geq -1$ and therefore $\int_N f_r\, g\, d\mathcal{H}^{n+1} \geq -\int_N g \,d\mathcal{H}^{n+1}$. Therefore we obtain that there exists $\eps_2>0$ such that for all $\eps\leq \eps_2$ the following bound holds

\be
\label{eq:F_energy_fr}
\frac{1}{2\s}\Fce(f_r)\leq 2\mathcal{H}^n(M)-2\mathcal{H}^n(B_1)-2\mathcal{H}^n(B_2)+\int_{N}\frac{g}{2\s} \,d\mathcal{H}^{n+1} +O(\eps|\log\eps|)
\ee
for all $r\in [0,4\eps\Lambda]$. (It should be kept in mind that $f_r=f_r^{\eps}$ is built for each fixed $\eps$ in the chosen range.)

We next connect $f_{4\eps\Lambda}\equiv -1$ to $a_{\eps}$, continuously in $W^{1,2}$. For this it suffices to recall that $a_{\eps}$ was defined as limit of the negative $\Fce$-gradient flow with initial condition $-1$; we denote by $[0,T_a]$ the time interval\footnote{Although irrelevant for our construction, it is easy to check that $T_a$ tends to $0$ as $\eps\to 0$.} on which this flow is defined. The same flow (translated by $4\eps\Lambda$) provides the continuation of $f_r$ to an interval $[0,4\eps\Lambda + T_a]$ with the property that 
the family $r\in  [0, 4\eps\Lambda + T_a] \to f_r \in W^{1,2}(N)$ is a continuous path that satisfies the bound in (\ref{eq:F_energy_fr}) for all $r$ and such that $f_0=\gamma_{-4\eps\Lambda - t_0^{(1)}}$ and $f_{4\eps\Lambda + T_a}= a_{\eps}$.

\subsection{Flow to a stable solution}
\label{flow}

In this section we produce a continuous path in $W^{1,2}(N)$ that starts at the function $\gamma_{4\eps\Lambda + t_0^{(2)}}$ obtained in Section \ref{two_parameter} and ends at a stable solution $v_{\eps}$ of $\ca{F'}{\eps}=0$.

To ease notation, we denote by $h$ the function $\gamma_{4\eps\Lambda + t_0^{(2)}}$. Recalling (\ref{eq:bump_D}) and (\ref{eq:gamma_t}), $h$ is given by

\be
\label{eq:h}
h(x)=\left\{
\begin{array}{ccc}
\OHet^{\eps}(-\text{dist}_{K_{2\eps\Lambda, t_0^{(j)}, \chi_{\widetilde{D}_j}}}(x)) & \text{for } x=F(q,s), (q,s) \in \widetilde{D_j}\times [0,\sigma_K), j\in\{1, 2\},\\
\Psi_{0}(\dm(x)) & \text{for } x \in N\setminus ( F(\widetilde{D_1}\times [0,\sigma_K)) \cup F(\widetilde{D_2}\times [0,\sigma_K))).
\end{array}\right. 
\ee
The first line is well-defined thanks to the fact that $\chi_{\widetilde{D}_j}$ is even on $\widetilde{M}$. Since $\chi_{\widetilde{D_j}} \in C^\infty_c(\widetilde{D}_j)$ for $j\in\{1, 2\}$, the definition in the first line agrees with $\Psi_{0}(\dm(x))$ (and thus with $G^{\eps}_0$) when $(q,s)$ is in a neighbourhood of $\p\widetilde{D_j}\times [0,\sigma_K)$. In view of this, we will compute separately the first variation of $h$ with respect to $\Ece$ in the domains that appear in the first and second line of (\ref{eq:h}). For the second line, it suffices to recall (\ref{eq:first_var_G0}). For the first, we need to estimate, for any sufficiently small $\eps$, the mean curvature of the embedded hypersurfaces given by level sets of the distance to $F\left(\text{graph}(2\eps\Lambda + t_0^{(1)} \chi_{\widetilde{D}_1} + t_0^{(2)} \chi_{\widetilde{D}_2})\right),$ where the graph is intended over $\text{Int}(K)$. Recall Remark \ref{oss:choice_t0} and the fact that $M$ is minimal. Then, choosing $\eps_2$ sufficiently small, we can ensure that for all $\eps\leq\eps_2$
the mean curvature of the embedded hypersurfaces $\{\text{dist}_{K_{2\eps\Lambda, t_0^{(j)}, \chi_{\widetilde{D}_j}}}=d\}$ for $d\in [-2\eps\Lambda, 2\eps\Lambda]$ are, in absolute value, smaller than $\frac{\min_N g}{9}$. 
With a computation similar to (\ref{eq:first_var_G0_part1}), with the distance $\text{dist}_{K_{2\eps\Lambda, t_0^{(j)}, \chi_{\widetilde{D}_j}}}$ in place of $\dm$, we obtain that, on the set $F(\widetilde{D_1}\times [0,\sigma_K)) \cup F(\widetilde{D_2}\times [0,\sigma_K))$, and for all $\eps\leq \eps_2$,

$$\ca{E'}{\eps}(h)\geq -\frac{\min_N g}{3}-|O(\eps|\log\eps|)|.$$
Together with (\ref{eq:first_var_G0}), this implies that $\ca{F'}{\eps}(h)\geq \frac{2 \min_N g}{3}-|O(\eps|\log\eps|)|$ on $N$ for all $\eps<\eps_1$ and therefore there exists $\eps_2>0$ such that for all $\eps\leq \eps_2$

\be
\label{eq:F_firstvar_h}
\ca{F'}{\eps}(h)\geq \frac{\min_N g}{2}>0 \,\, \text{ on } N.
\ee
(As in (\ref{eq:first_var_G0}), this is understood to be an inequality between Radon measures.) 

We will now produce the path mentioned at the beginning of this section, by means of a negative gradient flow (with respect to the functional $\Fce$). If $n\leq 6$, then the function $h$ is smooth (because so it $\dm$ in a tubular neighbourhood of $M=\overline{M}$) and we can use it as initial condition for the flow. For the general case, we first need a smoothing of $h$ with the key property that it preserves the positivity of the first variation (\ref{eq:F_firstvar_h}). We refer to \cite[Appendix A]{B1} for the definition of the mollifiers $\rho_\delta:N\times N\to \R$ used for the smoothing operation. For each sufficiently small $\eps$ there exists $\delta>0$ such that the convolution $h_{\delta}=h\star \rho_\delta$ is a smooth function on $N$ and satisfies 
\be
\label{eq:F_firstvar_h_delta}
-\ca{F'}{\eps}(h_\delta)=\eps \Delta h_\delta - \frac{W'(h_\delta)}{\eps} + g >0
\ee
(see \cite[Lemma A.2]{B1}, this uses that $h$ is Lipschitz). Moreover, $r\in (0,\delta]\to h_r = h\star \rho_r$ is continuous in $W^{1,2}(N)$ and extends by continuity at $r=0$ with $h_0=h$ (see \cite[Lemma A.1]{B1}). Still by continuity in $r$ of the convolution (\cite[Lemma A.1]{B1}), upon choosing $\delta$ sufficiently small we can also ensure that a bound of the form (\ref{eq:F_bound_gammat}) continues to hold for $h_r$, i.e.

\be
\label{eq:bound_hr}
\frac{1}{2\s}\Fce(h_r)\leq 2 \mathcal{H}^n(M)+\int_N \frac{g}{2\s}\, d\mathcal{H}^{n+1}-\frac{2}{3}\varsigma +O(\eps|\log\eps|)
\ee
for all $r\in [0,\delta]$. 

We need to ensure two more conditions on $h_{\delta}$, for which we may have to make $\delta$ smaller (the choice of $\delta$ is allowed to depend on $\eps$). Recall that by construction $h=1$ on the fixed non-empty open set $$F\left(\{(q,s):q\in \widetilde{D}_1\cup \widetilde{D}_2, 0\leq s<t_0^{(1)} \chi_{\widetilde{D}_1} + t_0^{(2)} \chi_{\widetilde{D}_2}\}\right).$$ By continuity, if $\delta$ is chosen sufficiently small, then $h\star \rho_\delta >3/4$ on the same open set. By construction, $h\leq 1$. Recall that the stable solution $b_{\eps}$ to $\ca{F'}{\eps}=0$ lies (strictly) above $1$. Therefore, with a suitable small choice of $\delta$, $h\star \rho_\delta  <b_{\eps}$.

\medskip

We will use the function $h_\delta$ just identified as initial condition for the negative gradient flow (with respect to $\Fce$), defined for $t\geq \delta$ by

\be
\label{eq:F_flow}
\left\{\begin{array}{ccc}
\eps \p_t h_t = \eps\Delta h_t -\frac{W'(h_t)}{\eps} + g\\
\left. h_t \right|_{t=\delta}= h\star \rho_\delta = h_\delta
       \end{array}\right. .
\ee
This flow is well-defined and smooth for all times by standard semi-linear parabolic theory. Moreover, it is mean-convex with respect to $\Fce$, i.e.
$$\ca{F'}{\eps}(h_t)>0 \text { for all }t\in [\delta,\infty).$$
To see this, notice that $F_t=\eps\Delta h_t -\frac{W'(h_t)}{\eps} + g$ is smooth on $N$ for all $t\geq \delta$ and $F_{\delta}>0$ by (\ref{eq:F_firstvar_h_delta}). Since $h_t$ solves (\ref{eq:F_flow}), then the PDE $\p_t U_t= \Delta U_t -\frac{W''(h_t)}{\eps^2}  U_t$ is solved by $U_t=F_t$. Another solution is given by $U_t\equiv 0$. By the parabolic maximum principle, the condition $F_t>0$ holds for all $t\geq \delta$, since it holds for $t=\delta$.

The mean-convexity guarantees, in a first instance, that the functions $h_t:N\to \R$ are increasing in $t$ and the limit $h_\infty$ for $t\to \infty$ is well-defined. The function $h_\infty$ is a solution of the elliptic PDE $\ca{F'}{\eps}(h_\infty)=0$; the mean-convexity further gives the following.

\begin{lem}
 \label{lem:stability}
The function $h_\infty = \lim_{t\to \infty} h_t$ is a stable solution of $\ca{F'}{\eps}=0$. Morover, there exists a fixed non-empty open set (independent of $\eps$) that is contained in $\{h_\infty>\frac{3}{4}\}$.
\end{lem}

\begin{proof}
To prove stability, we recall that the second variation of $\Fce$ at $h_\infty$ is given by the quadratic form $Q(\phi,\phi)=\int_N \eps|\nabla \phi|^2 - \frac{W''(h_\infty)}{\eps}\phi^2$ and that the associated Jacobi operator is $-\eps\Delta \phi + \frac{W''(h_\infty)}{\eps}\phi$. Letting $\rho_1$ denote the first eigenfunction and $\lambda_1$ the associated eigenvalue, if $\lambda_1<0$ we can find $s>0$ sufficiently small so that $\ca{F'}{\eps}(h_\infty-s\rho_1)<0$ on $N$. Since $h_{\delta}<h_\infty$, we can also ensure that $h_{\delta}<h_\infty-s\rho_1$. We then consider the solution to (\ref{eq:F_flow}) and let $T$ be the first time such that $h_T=h_\infty-s\rho_1$ at a point $x\in N$. At this point we must then have $\Delta h_T\leq \Delta(h_\infty-s\rho_1)$ and $W'(h_T)=W'(h_\infty-s\rho_1)$, therefore $\ca{F'}{\eps}(h_\infty-s\rho_1)\geq \ca{F'}{\eps}(h_{T})>0$ at $x$. This contradicts $\lambda_1< 0$. 

To prove the second statement, we recall that $h_{\delta} >3/4$ on the (fixed) non-empty open set $F\left(\{(q,s):q\in \widetilde{D}_1\cup \widetilde{D}_2, 0\leq s<t_0^{(1)} \chi_{\widetilde{D}_1} + t_0^{(2)} \chi_{\widetilde{D}_2}\}\right)$. The condition that $h_t$ increases along the flow guarantees the conclusion.
\end{proof}

\begin{oss}
 \label{oss:bound_hr}
As $\Fce$ decreases along the flow, (\ref{eq:bound_hr}) holds for all $r\geq 0$.
\end{oss}

\begin{oss}
By construction $h_\delta  <b_{\eps}$, and since $b_{\eps}$ is a stationary solution to the PDE in the first line of (\ref{eq:F_flow}), $b_{\eps}$ acts as a barrier for the flow $h_t$. In particular, $h_\infty \leq b_{\eps}$.
\end{oss}

\subsection{Concluding argument for the proof of Proposition \ref{Prop:main}}
\label{conclusive}

By reversing the path $f_r$ in Section \ref{path_to_a} and composing it with the one produced in Section \ref{two_parameter} and the one in Section \ref{flow}, we obtain a continuous path in $W^{1,2}(N)$ that starts at $a_{\eps}$ and ends at a stable critical point $h_\infty$ of $\Fce$. This path can be produced for all $\eps$ sufficiently small. Moreover, from (\ref{eq:F_bound_gammat}), (\ref{eq:F_energy_fr}), (\ref{eq:bound_hr}) and Remark \ref{oss:bound_hr}, there exists $\eps_2>0$ such that, for all $\eps<\eps_2$ we have, all along this path, the upper bound $\frac{1}{2\s} \Fce \leq 2\mathcal{H}^n(M)+\int_N \frac{g}{2\s} d\mathcal{H}^{n+1} -\frac{2}{3}\varsigma + O(\eps|\log\eps|)$,
with $\varsigma>0$ independent of $\eps$. By choosing a suitable $\eps_3\leq \eps_2$, we then have, for all $\eps\leq \eps_3$, and all along the path,

\be
\label{eq:F_bound_whole_path}
\frac{1}{2\s} \Fce \leq 2\mathcal{H}^n(M)+\int_N \frac{g}{2\s}\, d\mathcal{H}^{n+1} -\frac{\varsigma}{2}.
\ee
We now proceed to conclude the proof of Proposition \ref{Prop:main}. Recalling the assumption set up at the beginning of Section \ref{proof}, the min-max solutions $u_{\eps_{j}}$ (obtained in Proposition \ref{Prop:mountainpass}) have the property that $\lim_{j\to 0}u_{{\eps}_j}=u_\infty\equiv -1$ and
$$\liminf_{\varepsilon_j \to 0} \frac{1}{2\s}\ca{F}{\eps}(u_{{\eps}_j}) \geq 2 \mathcal{H}^n(M) + \int_N \frac{g}{2\s}\, d\mathcal{H}^{n+1}.$$
This implies that, for $\eps={\eps}_j$, the function $h_\infty$ cannot be $b_{\eps}$; otherwise, the continuous path in $W^{1,2}(N)$ that starts at $a_{\eps}$ and ends at $b_{\eps}$ and satisfies (\ref{eq:F_bound_whole_path}) would contradict the minmax characterization of $u_{\eps}$. Therefore $h_\infty$ is a stable solution to $\ca{F'}{\eps}=0$ that does not coincide with $b_{\eps}$. We now make the dependence on $\eps$ explicit, and denote $h_\infty$ by $v_{\eps}$: we check next that the stable solutions $v_{\eps_{j}}$ indeed complete the proof.

The solutions $v_{\eps_{j}}$ have a uniform upper bound on ${\mathcal E}_{\eps_{j}}(v_{\eps_{j}})$, because 
\begin{equation}\label{stable-mass-bound-1}
\frac{1}{2\s} \Fce(v_{\eps_{j}}) \leq 2\mathcal{H}^n(M)+\int_N  \frac{g}{2\s}\, d\mathcal{H}^{n+1} -\frac{\varsigma}{2}
\end{equation}
and 
\begin{equation}\label{stable-mass-bound-2}
-1\leq v_{\eps_{j}}\leq b_{\eps_{j}} \leq 1+c_{W} \varepsilon_{j}.
\end{equation} 
The condition in Lemma \ref{lem:stability} gives a fixed non-empty open set on which $v_{{\eps}_j}>\frac{3}{4}$ for all $j$. 

It only remains to prove that ${\mathcal E}_{\e_{j}}(v_{\eps_{j}})$ is bounded away from $0$.  To see this, we will prove first of all that if ${\mathcal E}_{\e_{j}}(w_{j})\to 0$ with $\e_{j} \to 0$ for a sequence $w_{j} \in W^{1, 2}(N)$ satisfying $\ca{F'}{\eps_{j}, g}(w_{j})=0$ and $\limsup_{j \to \infty} \,  \sup_{N} |w_{j}| < \infty$, then we must have that $\{w_{j} = 0\}=\emptyset$ for all sufficiently large $j$. If this is false, then passing to a subsequence without relabelling, we find points 
$y_{j}$ such that $w_{j}(y_{j}) = 0.$ Let $r_0\in(0,\text{inj}(N))$. Let $\widetilde{w}_j: B^{n+1}_{\frac{r_0}{{\eps}_i}}(0)\to \R$ be defined by $\widetilde{w}_j(x)=w_{j}(\text{exp}_{y_j}({\eps}_j x))$. Then $\widetilde{w}_j$ solves the PDE $\Delta \widetilde{w}_j - W'(\widetilde{w}_j) = -\eps_j g$ on $B^{n+1}_{\frac{r_0}{{\eps}_j}}(0)$, where $\Delta$ is the Laplace-Beltrami operator with respect to the metric obtained by pulling-back the Riemannian metric to $B_{\frac{r_0}{\e_{i}}}^{n+1}(0)$ via the map $x\to \text{exp}_{y_j}({\eps}_j x)$. Note that $\widetilde{w}_j(0)= 0$ for all $j$ and that  
for any compact set $K\subset \R^{n+1}$ we have that $\int_K \frac{1}{2}|\nabla \widetilde{w}_j|^2 + W(\widetilde{w}_j) \to 0$ as $j\to \infty$. Since $w_{j}$ is bounded, it follows from the 
De~Giorgi--Nash--Moser estimates that $\widetilde{w}_{j}$ is locally uniformly H\"older continuous on ${\mathbb R}^{n+1}$. By Schauder theory, we then have in fact that $\widetilde{w}_{j}$ 
is locally uniformly bounded in $C^{2, \alpha}$ for any $\alpha \in (0, 1)$. 
Using a diagonal argument, we then obtain an entire $C^{2, \alpha}$ solution $\widetilde{w}$ to $\Delta \widetilde{w} - W'(\widetilde{w}) =0$ on $\R^{n+1}$ such that $\int_{{\mathbb R}^{n}} \frac{1}{2}|\nabla \widetilde{w}|^2 + W(\widetilde{w}) = 0$. This forces $\widetilde{w}\equiv 1$ or $\widetilde{w}\equiv -1$.  But this is impossible since we must also have $\widetilde{w}(0) = 0.$ Thus we have $\{w_{j} = 0\} = \emptyset$ as asserted. 
Returning to the sequence $(v_{\e_{j}})$, if $\liminf_{j \to \infty} \,  {\mathcal E}_{\e_{j}}(v_{\e_{j}}) = 0,$ then passing to subsequence and taking $w_{j} = v_{\e_{j}}$ in the preceding discussion, we see that since $v_{{\eps}_j}$ is somewhere positive by Lemma \ref{lem:stability}, we must have that $\inf_{N} \, v_{{\eps}_j} > 0$ for all sufficiently large $j$.  Since $v_{\e_{j}}$ satisfies $\e_{j} \Delta v_{\e_{j}} - \e_{j}^{-1}W^{\prime}(v_{\e_{j}})  +g = 0$ on $N,$ evaluating at any $z \in N$ with $v_{\e_{j}}(z) = \inf_{N} \, v_{\e_{j}}$ we see that $W^{\prime}\left(\inf_{N} \, v_{\e_{j}}\right) = \e_{j}^{2}\Delta v_{\e_{j}}(z) + \e_{j}g(z) > 0$. Since $W^{\prime}(t) \leq 0$ for $t \in [0, 1]$ and $\inf_{N} \, v_{\e_{j}} > 0$ for all sufficiently large $j$, this implies that $\inf_{N} \, v_{\e_{j}} >1$ for all sufficiently large $j$.    
To complete the argument, note that the negative gradient flow of $\ca{F}{{\eps}_j}$ with initial condition $1$ tends to $b_{{\eps}_j}$ (by the definition of the latter). 
We consider the negative gradient flow for $\ca{F}{{\eps}_j}$ with initial condition $v_{{\eps}_j}$: this flow is time-independent. Recalling that $1\leq v_{{\eps}_j}\leq b_{{\eps}_j}$, the parabolic maximum principle then implies that $v_{{\eps}_j}=b_{{\eps}_j}$.

We have thus shown that if $\liminf_{i \to \infty} \,  {\mathcal E}_{\e_{j}}(v_{{\eps}_j}) = 0$ then along a subsequence we have $v_{{\eps}_j}=b_{{\eps}_j}$, a possibility we have excluded earlier in this section. Therefore we must have that $\liminf_{j \to \infty} \, {\mathcal E}_{\e_{j}}(v_{{\eps}_j}) >0$. This concludes the proof of Proposition \ref{Prop:main}.

\begin{oss}[\textit{The case $\text{Ric}_N>0$, $g\equiv \text{cnst}$}]
 \label{oss:Ric_pos}
In the case in which $N$ has positive Ricci curvature and $g\equiv \lambda\in (0,\infty)$, it follows from the expression of the second variation of $\ca{F}{\eps,\lambda}$ that stable solutions to $\ca{F'}{\eps,\lambda}=0$ must be constant functions on $N$ (see e.g.~\cite[Proposition 7.1]{B1}). It is then easy to check that there are only two stable solutions to $\ca{F'}{\eps,\lambda}=0$, one of which is a constant close to $-1$ and the other is a constant close to $+1$. Then the first solution has to be the function that we denoted by $a_{\eps}$, while the second is the function that we denoted by $b_{\epsilon}$. In particular, in the argument given at the beginning of this section, the function $h_\infty=v_{\epsilon}$ had to agree with $b_{\eps}$, giving a contradiction. We therefore conclude that when $N$ has positive Ricci curvature and $g\equiv \lambda\in (0,\infty)$ the minmax solutions $u_{\eps}$ produced by Proposition \ref{Prop:mountainpass} cannot yield (through their associated varifolds) a completely minimal hypersurface, in other words the open set $\{u_\infty=+1\}$ is non-empty.
\end{oss}

\section{Extension to the case of non-negative Lipschitz~$g$}\label{extension}

We shall continue to assume that $N$ is a compact Riemannian manifold. Having completed (in Section \ref{proof}) the proof of Theorem~\ref{thm:existence} for $g\in C^{1,1}(N)$ with $g>0$, we can now use a fairly straightforward approximation argument, based on the estimates of Theorem~\ref{estimates}, to generalise Theorem~\ref{thm:existence} to the case of Lipschitz $g$ with $g\geq 0$. This will establish Theorem \ref{thm:existence} in the stated generality. 

Let $g\, : \, N \to [0, \infty)$ be Lipschitz. Choose $g_{j}\in C^\infty(N)$ such that $g_{j} >0$ and $g_j \to g$ in $C^{0}(N)$ as $j \to \infty$, with $\sup_j \|g_j\|_{C^{1}(N)}\leq \Gamma$, where $\Gamma = \sup_{N}\, |g| + {\rm Lip} \, (g) + 1$.  (To this end, it suffices to consider $g+\delta_j$ for a sequence of positive numbers $(\delta_j)$ with $\delta_j \to 0^+$ as $j\to \infty$ and then define $g_j$ to be the convolution of $g+\delta_j$ with an appropriate mollifier.) By applying the main results of Sections~\ref{stable-regularity}, \ref{minmax_setup} and \ref{proof} taking $g_{j}$ in place of $g$ we get that for every $j,$ there exist a sequence $(\e_{k}^{j})$ with ${\eps}_k^j\to 0^+$ as $k\to \infty$ and a sequence of critical points $u_{{\eps}_k^j}:N\to \R$ of $\ca{F}{{\eps}_k^j, \sigma g_{j}}$ with Morse index at most $1$ with the following property: the associated varifolds $V_k^j=V^{u_{{\eps}_k^j}}$ satisfy $V_k^j \to \sigma V^j$ as $k\to \infty$, where $V^{j}$ is a non-zero integral $n$-varifold that can be written in the form $V^j= V^j_0 + V^j_{g_j}$ such that the conclusions of Theorem~\ref{thm:regularity} hold with $V^j_0$, $V^j_{g_j}$, $\sigma V^j$ in place of $V_0, V_g, V$ respectively,  and with $V^j_{g_j}\neq 0$. Indeed for each $j,$ we have one of the following two possibilities for the sequence $(u_{\eps_{k}^{j}})_{k=1}^{\infty}$: 
\begin{itemize}
\item[(i)] for each $k$, $u_{{\eps}_k^j}$ is a min-max critical point of ${\mathcal F}_{\e_{k}^{j}, \sigma g_{j}}$ given by Proposition \ref{Prop:mountainpass} taken with $\e = \e_{k}^{j}$, or,  
\item[(ii)] for each $k$, $u_{\e_{k}^{j}} = v_{\e_{k}^{j}}$ where $v_{\e_{k}^{j}}$ is the stable critical point of ${\mathcal F}_{\e_{k}^{j}, \sigma g_{j}}$ given by Proposition \ref{Prop:main} taken with $\e_{k}^{j}$ in place of $\e_{j}$.
\end{itemize}

To produce the desired hypersurface with mean curvature prescribed by $g$, we wish to take the varifold limit of (a subsequence of) the sequence $(V^j_{g_j})$. However, since we have not established a uniform upper bound on the Morse index of $\Greg{V^j_{g_j}}$ (with respect to the functional $A-\text{Vol}_{g_{j}}$), there is not enough information in the sequence $(V^{j}_{g_{j}})$ to immediately yield regularity of $\lim_{j\to \infty} V^j_{g_j}$. To circumvent this issue we proceed slightly differently as in the following outline: we take the varifold limit of (a subsequence of) the sequence $(V^j)$ (with $V^{j}$ including \emph{both} $V^{j}_{g_{j}}$ and the possible minimal portions $V_{0}^{j}$), and exploit the index bound on $u_{{\eps}_k^j}$ to deduce regularity of $V=\lim_{j\to \infty} V^j$ (namely, that away from a genuine singular set ${\rm sing} \, V$ of Hausdorff dimension $\leq n-7$, $\spt{V}$ is locally given by the union of a finite number of---in fact at most three---$C^2$ embedded ordered graphs) and that the convergence $V^{j} \to V$ is locally graphical and in $C^{2}$ away from ${\rm sing} \, V$ and away from possibly one additional point. 
It then readily follows that $V_g^j \to V_g$ locally in $C^{2}$ away from a closed set of 
singularities ${\rm sing} \, V_{g}$ of Hausdorff dimension $\leq n-7$ and away from one possible additional point, and that 
$\spt{V_g} \setminus {\rm sing} \, V_{g}$ is the desired immersed, quasi-embedded hypersurface. The fact that $V_g\neq 0$ is an easy consequence of the monotonicity formula. 
We now provide a more detailed exposition of these steps.

\medskip

By construction, the varifolds $V^j$ are integral, with generalised mean curvature uniformly bounded by $\sup_{N} |g| + 1.$ Moreover $V^{j}$ have mass $\|V^{j}\|(N)$ uniformly bounded from above; this follows from Lemma~\ref{lem:upperbound} in case the sequence $u_{\eps_{k}^{j}}$ corresponding to $V^{j}$ consists of min-max critical points as in (i) above, or from 
inequalities (\ref{stable-mass-bound-1}) and (\ref{stable-mass-bound-2}) in case the sequence $u_{\eps_{k}^{j}}$ corresponding to $V^{j}$ consists of stable critical points as in (ii) above. Hence by Allard's compactness theorem, upon extracting a subsequence that we do not relabel, we have $V^j \to V$ with $V$ a non-zero integral varifold and having first variation in 
$L^{\infty}(\|V\|).$ 

\begin{lem}\label{tangent-cones-to-V}
Let $V$ be as above, and let ${\mathbf C}$ be a tangent cone to $V$ at a point $p \in {\rm spt} \, \|V\|$. Then ${\mathbf C}$ is a stationary integral $n$-varifold on $T_{p} \, N \approx {\mathbb R}^{n+1}$ with  stable regular part and no classical singularities. 
\end{lem} 

\begin{proof} Let $(V^{j})$ be the sequence of limit$(g_{j},0)$-varifolds as above so that $V = \lim_{j \to \infty} \, V^{j}$ on $N$. It follows from this and the definition of tangent cone that 
there is a sequence of numbers $\rho_{\ell} \to 0^{+}$ and a subsequence $(V^{j_{\ell}})$ such that  
\begin{equation}\label{tangent-limit}
(\eta_{0, \rho_{\ell}} \circ \exp_{p}^{-1})_{\#} \, V^{j_{\ell}} \to {\mathbf C} 
\end{equation} 
as varifolds on ${\mathbb R}^{n+1}.$  For each $\ell,$ there is a sequence $\left(\eps^{j_{\ell}}_{k}\right)_{k=1}^{\infty}$ with $\eps^{j_{\ell}}_{k} \to 0^{+}$ as $k \to \infty$ and critical points 
$u_{k}^{(\ell)} = u_{\e_{k}^{j_{\ell}}}$ of ${\mathcal F}_{\e_{k}^{j_{\ell}}, \sigma g_{j_{\ell}}},$ as in (i) or (ii) above (with $j_{\ell}$ in place of $j$),  such that $V^{j_{\ell}}$ is the limit varifold corresponding to the sequence $(u_{k}^{(\ell)})_{k=1}^{\infty}$. Since the Morse index of $u_{k}^{(\ell)}$ is $\leq 1$, we have that either (a) a subsequence of $(u_{k}^{(\ell)})_{k=1}^{\infty}$ is stable in 
${\mathcal N}_{\rho_{\ell}/2}(p)$, or (b) a subsequence of $(u_{k}^{(\ell)})_{k=1}^{\infty}$ is stable in $N \setminus \overline{{\mathcal N}_{\rho_{\ell}/2}(p)}.$ Since (a) or (b) must hold for infinitely many $\ell$, 
we may pick an appropriate diagonal subsequence $(u_{k_{i}}^{(\ell_{i})})_{i=1}^{\infty}$ and argue exactly as in the proof of Theorem~\ref{limit-regularity}, part (ii), to first conclude that either ${\mathbf C} \res B_{1/2}^{n+1}(0)$ has stable regular part and no classical singularities, or that ${\mathbf C} \res \left({\mathbb R}^{n+1} \setminus \overline{B_{1/2}^{n+1}(0)}\right)$ has stable regular part and no classical singularities. Since ${\mathbf C}$ is a cone, either of these possibilities implies the full conclusion of the lemma.\end{proof}

In order to establish regularity of $V$, it is convenient to first prove the following lemma which concerns the ``stable case;'' this lemma can then be used in a fairly standard way to handle the Morse index $\leq 1$ situation at hand.  


\begin{lem}
\label{lem:extension_stablecase}
Let $g$, $g_{j}$ be as above so that $g_j \in C^{1,1}(N)$ with $g_j>0$ for $j \in \N$, $\sup_{j} \, \|g_{j}\|_{C^{1}(N)} \leq \Gamma$ for some $\Gamma > 0$ and $g_{j} \to g$ in $C^{0}(N).$ Let $V^{j}$, $V$ be the varifolds as above, and let $U \subset N$ be open. For each $j$, assume that $V^j \res U$ is a stable limit $(g_j,0)$-varifold in $U$ (in the sense that there is a sequence of critical points $u_{\eps_{k}^{j}}$ of ${\mathcal F}_{\e_{k}^{j}, \sigma g_{j}}$ corresponding to $V^{j}$ such that  $u_{\eps_{k}^{j}}$ are stable in $U$ for $k=1, 2, 3, \ldots$). 
Then ${\rm sing} \, V \cap U \left(\equiv ({\rm spt} \, \|V\| \setminus \Greg \, V) \cap U\right)$  has Hausdorff dimension $\leq n-7$ (and is empty if $2 \leq n \leq 6$). Moreover, the convergence $V^{j} \to V$ is graphical and in 
$C^{2}$ in any compact subset of $U \setminus {\rm sing} \, V.$ 
\end{lem}

\begin{proof}
First consider a point $p\in \spt{V} \cap U$ such that at least one tangent cone to $V$ at $p$ is supported on a plane $P.$ Let $q = \Theta(\|V\|, p)$ (a positive integer).
Let $\e_{0} = \e_{0}(n, q, N, \G)$ and $\mu = \mu(n, N, \Gamma)$ be the constants as in Theorem~\ref{estimates}, taken with $\overline{\rho} = {\rm inj} \, N$. Identify $T_{p} \, N$ with ${\mathbb R}^{n+1}$ such that $P$ is identified with the hyperplane ${\mathbb R}^{n} \times \{0\}$. By the definition of tangent cone, we can choose sufficiently small $r=r(p)>0$ with $r<\frac{1}{2}\min\{\text{inj}(N),\text{dist}(p,\p U)\}$ so that if $\widetilde{V}$ is the pull back of $V \res B_{2r}(p)$ by the exponential map at $p$, then $\widetilde{V}$ satisfies 
$\frac{\|\widetilde{V}\|(B_{2r}^{n+1}(0))}{\omega_{n}(2r)^{n}} \leq q+ 1/4,$  
$q-1/4 \leq \frac{\|\widetilde{V}\|\left((B_{r}^{n}(0) \times {\mathbb R}) \cap B_{2r}^{n+1}(0)\right)}{\omega_{n}r^{n}} \leq q+1/4$ and  
$$2\mu r + (2r)^{-n-2} \int_{(B_{r}^{n}(0) \times {\mathbb R}) \cap B_{2r}^{n+1}(0)} |x^{n+1}|^{2} \, d\|\widetilde{V}\| < \e_{0}/2.$$
Then by varifold convergence, it follows that for all sufficiently large $j$, 
$$\frac{\|\widetilde{V}^{j}\|(B_{2r}^{n+1}(0))}{\omega_{n}(2r)^{n}} \leq q+ 1/2,$$  
$q-1/2 \leq \frac{\|\widetilde{V}^{j}\|\left((B_{r}^{n}(0) \times {\mathbb R}) \cap B_{2r}^{n+1}(0)\right)}{\omega_{n}r^{n}} \leq q+1/2$ and  
$$2\mu r + (2r)^{-n-2} \int_{(B_{r}^{n}(0) \times {\mathbb R}) \cap B_{2r}^{n+1}(0)} |x^{n+1}|^{2} \, d\|\widetilde{V}^{j}\| < \e_{0}$$
where $\widetilde{V}^{j}$ is the pull back of $V^{j} \res B_{2r}(p)$ by the exponential map at $p$. 
We can therefore apply Theorem \ref{estimates}, part (ii)  to conclude that for any given $\theta \in (0, 1)$ and all $j$ sufficiently large, $\widetilde{V}^j \res (B_{r/2}^{n}(0) \times {\mathbb R}) \cap B_{2r}^{n+1}(0)$ is made up of graphs of $q$ ordered functions over $B_{r/2}^{n}(0)$, of class $C^{2,\theta}$ with $C^{2,\theta}$ norm bounded from above independently of $j$. (To this end, note that the constant $C$ on the right-hand-side of the estimate in Theorem~\ref{estimates}, part (ii) depends on $g_{j}$ only in terms of an upper bound on $\|g_j\|_{C^{1}(N)}$, and we have by assumption that $\|g_j\|_{C^{1}(N)} \leq \G$.) 
By the Arzel\`a--Ascoli theorem, it follows that $\widetilde{V} \res (B_{r/2}^{n}(0) \times {\mathbb R}) \cap B_{2r}^{n+1}(0)$ is the sum of multiplicity 1 varifolds associated with graphs of $q$ ordered functions over $B_{r/2}(0)$ of class $C^{2,\theta}.$ Thus $p \in \Greg \, V$. 

Let $\Sigma \subset {\rm spt} \, \|V\| \cap U$ be the set of points where no tangent cone to $V$ is supported on a hyperplane. Then by the above $\Sigma = {\rm sing} \, V \cap U$, and by Lemma~\ref{tangent-cones-to-V}, Theorem~\ref{classification} and Lemma~\ref{stratification} we have that $\Sigma = \emptyset$ if $2\leq n \leq 6$ and ${\rm dim}_{\mathcal H} \, (\Sigma) \leq n-7$ if $n \geq 7$. The final conclusion of the lemma asserting the $C^{2}$ convergence $V^{j} \to V$ locally away from ${\rm sing} \, V$ follows from the discussion in the preceding paragraph.
\end{proof}

Returning to the analysis of $V^j\to V$ on $N$, we next fix a small arbitrary $\tau \in(0,\text{inj}(N))$ and a (finite) cover of ${\rm spt} \, \|V\|$ by open balls ${\mathcal N}_{\tau/2}(p_{k}^{\tau}),$ $k=1, 2, \ldots, N(\tau),$ with radius $\tau/2$ and centre $p_{k}^{\tau} \in {\rm spt} \, \|V\|$. 
%
We then consider the following two possibilities, one of which must hold:
\begin{itemize}
\item[{\rm (a)}] there exists a subsequence $(j_\ell)_{\ell=1}^{\infty}$ of the sequence $(j)_{j=1}^{\infty}$ such that for every $k \in \{1, 2, \ldots, N(\tau)\}$ and every $\ell \geq 1$, the restriction $V^{j_\ell} \res {\mathcal N}_{\tau}(p_{k}^{\tau})$ is a stable limit $(g_{j_\ell},0)$-varifold in ${\mathcal N}_{\tau}(p_{k}^{\tau})$;
\item[{\rm (b)}] there is $j_0 = j_{0}(\tau) \in \N$ such that for every $j\geq j_0$ there is $k(j) \in \{1, 2, \ldots, N(\tau)\}$ such that $V^{j} \res {\mathcal N}_{\tau}(p_{k(j)}^{\tau})$ is not a stable limit $(g_j,0)$-varifold in ${\mathcal N}_{\tau}(p_{k(j)}^{\tau})$.
\end{itemize} 

\begin{oss}
\label{oss:compact_ST}
If for some fixed $j$ and some point $p \in N$ the varifold $V^{j} \res {\mathcal N}_{\tau}(p)$ is not a stable limit $(g_{j}, 0)$-varifold in ${\mathcal N}_{\tau}(p)$, then no subsequence of the sequence $(u_{{\eps}_k^j})_{k=1}^{\infty}$ can be stable 
(with respect to ${\mathcal F}_{\eps_{k}^{j}, \sigma g_{j}}$) in ${\mathcal N}_{\tau}(p)$. In this case, since $u_{{\eps}_k^j}$ has Morse index $\leq 1$ in $N,$ it follows that $u_{{\eps}_k^j}$ is stable in $N\setminus \overline{{\mathcal N}_{\tau}(p)}$ for all sufficiently large $k$, and hence $V^{j} \res (N \setminus \overline{{\mathcal N}_{\tau}(p)})$ is a stable limit $(g_j,0)$-varifold in $N\setminus \overline{{\mathcal N}_{\tau}(p)}$.
\end{oss}

If case (a) occurs for some $\tau \in (0, {\rm inj} \, (N))$, then  for every $k \in \{1, 2, \ldots, N(\tau)\}$ we can apply Lemma~\ref{lem:extension_stablecase} with $U = {\mathcal N}_{\tau}(p_{k}^{\tau})$ to conclude the following: 

\begin{itemize}
\item [{\rm (a$^{\prime}$)}] ${\rm sing} \, V$ has Hausdorff dimension $\leq n-7$ if $n\geq 7$ and ${\rm sing} \, V =\emptyset$ if $2 \leq n \leq 6$; moreover, there is a subsequence $(V^{j_{\ell}})$ of $(V^{j})$ such that the convergence $V^{j_{\ell}} \to V$ is in $C^2$ (as ordered graphs) in each compact subset $K \subset N\setminus {\rm sing} \, V$. 
\end{itemize} 
\begin{oss}
In the last part of this section we will point out that ${\rm sing} \, V$ is in fact finite if $n=7.$ 
\end{oss}

If case (a) fails for every $\tau \in (0, {\rm inj} \, (N))$, then for an arbitrary sequence of positive numbers $(\tau_{\ell})_{\ell=1}^{\infty}$ with $\tau_{\ell} \to 0,$ case (b) holds with 
$\tau = \tau_{\ell}$ for every $\ell \geq 1$.  In this case, by Remark~\ref{oss:compact_ST}, there is a subsequence $(V^{j_{\ell}})$ of the sequence $(V^{j})$ such that for each $\ell \geq 1$ and some $p_{\ell} \in {\rm spt} \, \|V\|$, 
$V^{j_{\ell}} \res (N \setminus \overline{{\mathcal N}_{\tau_{\ell}}(p_{\ell})})$ is a stable limit $(g_{j_{\ell}}, 0)$-varifold in $N \setminus \overline{{\mathcal N}_{\tau_{\ell}}(p_{\ell})}$. Then, since ${\rm spt} \, \|V\|$ is compact, there is a point  $p_\infty \in {\rm spt} \, \|V\|$ such that passing to a subsequence which we index by $\ell$ again, we have that $p_{\ell} \to p_{\infty}$ whence, for any $\delta >0$, 
 the ball ${\mathcal N}_{\delta}(p_{\infty})$ contains ${\mathcal N}_{\tau_{\ell}}(p_{\ell})$ for all sufficiently large $\ell.$ Consequently, $V^{j_{\ell}} \res (N\setminus \overline{{\mathcal N}_{\delta}(p_\infty)})$ is a stable limit $(g_{j_{\ell}},0)$-varifold for each $\delta >0$ and sufficiently large $\ell$ (depending on $\delta$).  Since $\delta >0$ is arbitrary, we may again apply Lemma~\ref{lem:extension_stablecase} to conclude the following:

 \begin{itemize}
 \item[{\rm (b$^{\prime}$)}] there is a point $p_{\infty} \in {\rm spt} \, \|V\|$ and a subsequence $(V^{j_{\ell}})$ of the sequence $(V^{j})$ such that for each $\delta > 0$ and sufficiently large $\ell$ (depending on $\delta$), the varifold $V^{j_{\ell}}$ is a stable limit $(g_{j_{\ell}}, 0)$-varifold in $N \setminus {\mathcal N}_{\delta}(p_{\infty});$ consequently, the convergence  $V^{j_{\ell}} \to V$ is in $C^{2}$ (as ordered graphs) in each compact set $K \subset N \setminus ({\rm sing} \, V \cup \{p_{\infty}\});$ moreover, if $n \geq 7$, then ${\rm sing} \, V$ (which may or may not include $p_{\infty}$) has Hausdorff dimension  $\leq n-7;$ if $2 \leq n \leq 6$, then ${\rm sing} \, V \subset \{p_{\infty}\}.$ 
 \end{itemize}
\begin{oss} 
In the last part of this section we will improve this to say that ${\rm sing} \, V$ is finite if $n=7$ and  empty if $2 \leq n \leq 6.$
\end{oss}

Note that since (a) or (b) holds, either (a$^{\prime}$) or (b$^{\prime}$) must hold. For notational convenience, let us relabel the subsequence $(V^{j_{\ell}})_{\ell =1}^{\infty}$ in either case as $(V^{j})_{j=1}^{\infty}$, and the subsequence $(g_{j_{\ell}})_{\ell=1}^{\infty}$ as $(g_{j})_{j=1}^{\infty}$. 

Now let $V_{g}$ be the varifold limit in $N$ of the sequence $(V_{g_{j}}^{j})$, and note that $V_g\neq 0$ because $V_{g_j}^j\neq 0$ for all $j$ and, in view of the uniform $L^{\infty}$ bound on $(g_{j})$, there is a uniform positive lower bound for $\|V^j_{g_j}\|(N)$ by the monotonicity formula applied in a ball of radius $\text{inj} \, (N)$ around an arbitrary point of 
$\spt{V^j_{g_j}}$. Moreover, the local $C^{2}$ convergence $V^j \to V$ away from ${\rm sing} \, V$ (in case (a$^{\prime}$)) or away from 
${\rm sing} \, V \cup \{p_{\infty}\}$ (in case (b$^{\prime}$)) implies that the convergence $V^j_{g_j} \to V_g$ is also locally in $C^{2}$ away from ${\rm sing} \, V$ (in case (a$^{\prime}$)) or away from ${\rm sing} \, V\cup \{p_{\infty}\}$ (in case (b$^{\prime}$)). 
 In view of this local $C^{2}$ convergence, and the fact that the $C^{2}$ graphs locally describing 
 $V^j_{g_j}$ on $N \setminus {\rm sing} \, V$ (in case (a$^{\prime}$)) or on 
$N \setminus \left({\rm sing} \, V \cup \{p_{\infty}\}\right)$ (in case (b$^{\prime}$)) have scalar mean curvature $g_j$, we conclude that $V_g$ is, away from ${\rm sing} \, V$, locally given by a union of $C^2$ graphs each having mean curvature $g\, \hat{n}$ for one of the two choices of unit normal $\hat{n}$ on each graph. Since $g$ is Lipschitz, it follows from this that $V_g$ is of class $C^{2, \alpha}$ for any $\alpha \in (0, 1)$ away from ${\rm sing} \, V.$ It also follows from this and the Hopf boundary point lemma that $\Greg \, V_{g} = {\rm spt} \, \|V_{g}\| \setminus {\rm sing} \, V$ is quasi-embedded (in the sense defined in Remark~\ref{quasi-embedded}). 

It now follows from Theorem~\ref{Roger-Tonegawa} that $V_{g}$ has multiplicity 1 on ${\rm reg}\,V_g \cap \{g>0\}$ (where  ${\rm reg} \, V_{g}$ is the embedded part of ${\rm spt} \, \|V_{g}\|$). For if not, then there is a point $y \in {\rm reg} \, V_{g} \cap \{g >0\}$ and a ball $B = {\mathcal N}_{\rho}(y)$ with ${\rm spt} \, \|V_{g}\| \cap {\mathcal N}_{2\rho}(y) \subset {\rm reg} \, V_{g}$, with $p_{\infty} \in N \setminus {\mathcal N}_{2\rho}(y)$ in case (b$^{\prime})$, and with $g >0$ on $B$ such that $V_{g} \res B$ 
has integer multiplicity $k \geq 2$ and ${\rm spt} \, \|V_{g}\| \cap B$ has mean curvature $g \hat{n}$ for some choice of continuous unit normal $\hat{n}$. By the $C^{2}$ convergence 
$V^{j}_{g_{j}} \to V_{g}$ in $B$, 
it follows that there is an open set $\Omega \subset N$ with ${\rm spt} \, \|V_{g}\| \cap {\mathcal N}_{\rho/2}(y) \subset \Omega$ and $V^{j}_{g_{j}} \res \Omega$ consisting of $k$ normal graphs $G_{j}^{\ell}$, $\ell =1, 2, \ldots, k$, of class $C^{2},$ over ${\rm spt} \, \|V_{g}\| \cap B_{\rho/2}(y)$, with each  $G_{j}^{\ell}$ converging in $C^{2}$ to ${\rm spt} \, \|V_{g}\| \cap B_{\rho/2}(y)$ as $j \to \infty.$ Furthermore, since 
$g_{j} >0$ in $B$, it follows by Theorem~\ref{Roger-Tonegawa} (iii) that for each $j$, the graphs $G_{j}^{\ell},$ $\ell = 1, 2, \ldots, k,$ are distinct, and  have mean curvature given by 
$g_{j} \nu_{j}^{\ell}$ with the unit normal vectors $\nu_{j}^{\ell}$ to $G_{j}^{\ell}$ close to $\hat{n}$ (by the $C^{2}$ convergence). Since $\cup_{\ell=1}^{k}G_{j}^{\ell}$ is contained in the phase boundary, this contradicts Theorem~\ref{Roger-Tonegawa} (iii) which says that the mean curvature vector must point into the $+1$ phase everywhere on the phase boundary.  Thus $V_{g}$ has multiplicity 1 on ${\rm reg}\,V_g \cap \{g>0\}$ as claimed.

The local $C^{2}$ convergence $V^j_{g_{j}} \to V_g$ on $N \setminus {\rm sing} \, V$ (in case (a$^{\prime}$)) or on $N\setminus ({\rm sing} \, V \cup\{p_\infty\})$ (in case (b$^{\prime}$)) 
further gives that condition (\textbf{T}) of \cite[Section~1.3]{BW2} is satisfied by $V_g.$ 
We can then apply \cite[Theorem 4.1]{BW2} and \cite[Remark 4.6]{BW2} (the validity of which extends, with the same reasoning as in \cite{BW2}, to the case of Lipschitz  $g$) to conclude that the varifold $V_{g} \res (N \setminus {\rm sing} \, V)$ in case (a$^{\prime}$) or 
$V_g\res (N\setminus ({\rm sing} \, V \cup \{p_\infty\}))$ in case (b$^{\prime}$) can be realized as the pushforward of an oriented $n$-manifold via a two-sided immersion with mean curvature given by $g \nu$, where $\nu$ is a choice of unit normal to the immersion. Since both ${\rm sing} \, V$ in case (a$^{\prime}$) and ${\rm sing} \, V \cup \{p_\infty\}$ in case (b$^{\prime}$) are lower dimensional, 
$V_g$ is realized by the same pushforward. 

If case (a$^{\prime}$) arises, then the proof of Theorem~\ref{thm:existence} is now complete with $M  = {\rm spt} \, \|V_{g}\| \setminus {\rm sing} \, V$ except when $n=7$ (in which case it remains to show that ${\rm sing} \, V$ is finite).  

If case (b$^{\prime}$) arises and $n \geq 8$, then the proof is complete with $M = {\rm spt} \, \|V_{g}\| \setminus {\rm sing} \, V.$
 
If $n=7$, ${\rm sing} \, V$ must be finite (in either of the cases (a$^{\prime}$) and (b$^{\prime}$)) by essentially the same argument as for the corresponding claim in Theorem~\ref{limit-regularity}, part (iv). This goes as follows: if ${\rm sing} \, V$ is an infinite set, then since ${\rm spt} \, \|V\|$ is compact, there are points $y, y_{k} \in {\rm sing} \, V$ for $k=1, 2, 3, \ldots$ 
with $y_{k} \neq y$ and $y_{k} \to y$ as $k \to \infty$.  Rescaling $(\exp_{y}^{-1})_{\#} \, V$ about the origin (in $T_{y} \, N \approx {\mathbb R}^{8}$) by the sequence $\rho_{k} 
= |\exp_{y}^{-1}(y_{k})|$ produces, after passing to a subsequence, a tangent cone ${\mathbf C} = \lim_{k \to \infty} \, \eta_{0, \rho_{k} \, \#} \, (\exp_{y}^{-1})_{\#} \, V.$ 
Note that we have of course that $p_{\infty} \not\in {\mathcal N}_{2\rho_{k}}(y) \setminus \overline{{\mathcal N}_{\rho_{k}/4}(y)}$ for all  sufficiently large $k$, and hence, for each $k$ and all sufficiently large $j$, the varifolds $V^{j}$ are stable limit $(g_{j}, 0)$-varifolds in 
${\mathcal N}_{2\rho_{k}}(y) \setminus \overline{{\mathcal N}_{\rho_{k}/4}(y)}.$ Thus if ${\mathbf C} \res (B_{3/2}^{n+1}(0) \setminus \overline{B_{1/2}^{n+1}(0)})$ were regular,  it would follow from Theorem~\ref{estimates}, part (ii) that $V \res ({\mathcal N}_{3\rho_{k}/2}(y) \setminus \overline{{\mathcal N}_{\rho_{k}/2}(y)})$ would be regular, contrary to the the fact that 
$y_{k} \in {\rm sing} \, V.$ Hence there is a point $z \in {\rm sing} \, {\mathbf C} \setminus \{0\}$ whence the ray $\{\lambda z \, : \, \lambda \geq 0\} \subset {\rm sing} \, {\mathbf C}$. This is impossible by 
Lemma~\ref{tangent-cones-to-V} and Theorem~\ref{classification} according to which ${\rm sing} \, {\mathbf C}$ must be $0$-dimensional. 


The only remaining case to analyse is when (b$^{\prime}$) arises and $n \in \{2, 3, \ldots, 6\}.$  In this case, the proof of Theorem~\ref{thm:existence} would be complete with $M = {\rm spt} \, \|V_{g}\|$ provided we can check that $p_{\infty} \in \Greg\, V.$ 
To see this,  consider any tangent cone ${\mathbf C}$ to $V$ at $p_{\infty}$.  
By Lemma~\ref{tangent-cones-to-V} and Theorem~\ref{classification}, 
\begin{equation*}\label{tangent-is-planar}
{\mathbf C} = q|P|
\end{equation*}
for some hyperplane $P$ of $T_{p_{\infty}} N$ and a positive integer $q$. 
With $q$, $P$ as above, we can now argue essentially as in the proof of Theorem~\ref{limit-regularity}, part (iii). First apply Theorem~\ref{BWregularity} with ${\mathcal V} = \{V\},$ $U_{V} = N$, $X_{0} = p_{\infty}$ and $\beta = (1 + \mu)^{-1}\epsilon_{0}$ where 
 $\mu = \mu(n, N, \Gamma)$, $\epsilon = \epsilon_{0}(n, q, N, \Gamma)$ are as in Theorem~\ref{estimates} taken with $\overline{\rho} = {\rm inj} \, N$, to conclude that $(\exp_{p_{\infty}}^{-1})_{\#} \, V$ near the origin 
 $0 \in {\mathbb R}^{n+1} \approx T_{p_{\infty}} \, N$ is the sum of $q$ multiplicity 1 varifolds associated with ordered $C^{1, \alpha}$ functions $u_{1} \leq u_{2} \leq \ldots \leq u_{q}$ over a ball $B_{\rho_{1}}^{n}(0) \subset P \approx {\mathbb R}^{n} \times \{0\}$. Note that in this application of Theorem~\ref{BWregularity}, hypothesis (b) holds because all tangent cones to $V$ (including at $p_{\infty}$) are supported on hyperplanes, and hypothesis (c) (for the above choice of $\beta$) is verified by applying Theorem~\ref{estimates}, part (ii) to the approximating varifolds $V^{j}$. Indeed, since $\Theta \, (\|V\|, p_{\infty}) = q$, the density assumption in hypothesis (c), namely, that  $\Theta (\|\eta_{0, \rho \, \#} \widetilde{V}\|, Y) < q$ for all $Y \in B_{1}^{n+1}(0)$ where 
 $\widetilde{V} = \left(\Gamma \circ \exp_{X}^{-1}\right)_{\#} V \res {\mathcal N}_{\rho}(X)$ (notation as in hypothesis (c)), guarantees that $p_{\infty} \not\in {\mathcal N}_{\rho}(X),$ ensuring that for each $\delta \in (0, 1/2)$ and sufficiently large $j$ (depending on $\delta$ and $\rho$), $V^{j}$ are stable limit $(g_{j}, 0)$-varifolds in ${\mathcal N}_{(1-\delta)\rho}(X)$. Hence Theorem~\ref{estimates} is applicable to $V^{j}$ in  the ball ${\mathcal N}_{(1-\delta)\rho}(X)$ for any $\delta \in (0, 1/2)$, and letting first $j \to \infty$ and then $\delta \to 0$ (for fixed $\rho$) we see that $V$ satisfies hypothesis (c) of Theorem~\ref{BWregularity}. 
 Finally, since $V$ is of class $C^{2}$ away from $p_{\infty}$, we see that the functions $u_{j}$ are of class $C^{2}$ away from $p_{\infty}$. It then follows by a standard cut-off function argument that each $u_{j}$ separately is a weak solution to one of the three equations in (\ref{eq:PDE_pmc}) on $B_{\rho_{1}}^{n}(0)$, and hence, since $g$ is Lipschitz, that $u_{j} \in C^{2, \alpha}$ for each $j$ and each $\alpha \in (0, 1)$.  This concludes the proof of everywhere $C^{2, \alpha}$ regularity of $V$, and hence of $V_{g}$, in case (b$^{\prime}$) and in dimensions $n \in \{2, 3, \ldots, 6\}$. The proof of Theorem~\ref{thm:existence} is now complete.

\section{Appendix: index of notation and definitions}
 The following notation and definitions are used throughout the article. Here $N$ is an $(n+1)$-dimensional Riemannian manifold.
\begin{itemize}
\item ${\rm inj}_{X} \, N$: injectivity radius of $N$ at $X \in N$.
\item ${\rm inj} \, N$: the injectivity radius of $N$. 
\item $\overline{M}$: closure in $N$ of $M \subset N$. 
\item $\omega_{n}$: Lebesgue measure of the unit ball in ${\mathbb R}^{n}$. 
\item ${\mathcal H}^{k}$: $k$-dimensional Hausdorff measure on $N$ with respect to the Riemannian metric on $N$.
\item ${\mathcal H}^{n} \res M$: the restriction of ${\mathcal H}^{n}$ to $M \subset N$, defined by ${\mathcal H}^{n} \res M (A) = {\mathcal H}^{n} \, (A \cap M)$ for $A \subset N$. 
\item ${\rm dim}_{\mathcal H} \, (A)$: Hausdorff dimension of $A \subset N$.  
\item $|M|$: multiplicity 1 $n$-varifold on $N$ associated with the $n$-rectifiable subset $M \subset N$. 
\item $\|V\|$: weight-measure on $N$ associated with the varifold $V$ on $N$.
\item $V \res \Oc$: restriction of varifold $V$ on $N$ to open subset $\Oc \subset N$.
\item $f_{\#} \, V$: push-forward of varifold $V$ on $N$ by diffieomorphism $f \, : N \to N^{\prime}$ between manifolds.  
\item $\partial^{\star} E$: reduced boundary of the Caccioppoli set $E \subset N$. 
\item ${\mathcal N}_{\rho}(X)$: normal coordinate ball in $N$ of radius $\rho >0$ and centre $X$. 
\item limit $(g, 0)$-varifold: Definition~\ref{limit-varifold} (Section~\ref{ac-solutions}).
\item stable limit $(g, 0)$-varifold: Definition~\ref{limit-varifold} (Section~\ref{ac-solutions}).
\item classical singularity: Definition~\ref{classical-sing} (Section~\ref{nonvar-regularity}). 
\item $(q, \beta)$-separation property: Defintion~\ref{qbseparation} (Section~\ref{nonvar-regularity}). 
\item quasi-embedded point: Definition~\ref{quasi-embedded} (Section~\ref{nonvar-regularity}). 
\item $\Greg \, V$: Definition~\ref{greg} (Section~\ref{nonvar-regularity}).
\item ${\rm sing} \, V$: Definition~\ref{sing} (Section~\ref{nonvar-regularity}).
\item ${\rm reg} \, V$: Definition~\ref{reg} (Section~\ref{nonvar-regularity}).
\item quasi-embedded $PMC \, (g, 0)$ structure: Definition~\ref{PMC-structure} (Section~\ref{smooth-excision}).
\item $\Lambda$ in Sections \ref{minmax_setup} and \ref{proof}: shorthand notation for $3|\log\eps|$.
\item $\Het$, $\Het^{\eps}$: functions defined in Section \ref{minmax_setup}.
\item $\Psi$, $\Psi_t$: functions defined in (\ref{eq:Psi})-(\ref{eq:family2}) (Section \ref{prelim}).
\item $d_{\overline{M}}$: (unsigned) Riemannian distance to $\overline{M}$ (Section \ref{prelim}).
\item $\widetilde{M}$: oriented double cover of $M$ (Section \ref{proof}).
\item $\iota: \widetilde{M}\to N$: (minimal) immersion of $\widetilde{M}$ into $N$.
\item diffeomorphism $F$ and its domain $V_{\widetilde{M}}$: Section \ref{prelim}.
\item $K$: compact subset of $M$ chosen in Section \ref{signed_distance}.
\item $B_j\subset D_j$ ($j\in\{1,2\}$): geodesic balls in $K$ chosen in Section \ref{epsilon}.
\end{itemize} 

\bigskip
\hskip-.1in\vbox{\hsize2.5in\obeylines\parskip -1pt 
\footnotesize
Costante Bellettini
Department of Mathematics
University College London
London WC1E 6BT, United Kingdom
\vspace{4pt}
{\tt C.Bellettini@ucl.ac.uk}} 
\vbox{\hsize2.5in
\obeylines 
\parskip-1pt 
\footnotesize
Neshan Wickramasekera
DPMMS 
University of Cambridge 
Cambridge CB3 0WB, United Kingdom
\vspace{4pt}
{\tt N.Wickramasekera@dpmms.cam.ac.uk}}

\end{document}